\documentclass[10pt]{article}

\usepackage[utf8]{inputenc}
\usepackage{amsmath,amsfonts,amssymb}
\usepackage{amsthm}
\usepackage[a4paper]{geometry}
\usepackage{color}
\usepackage{transparent}
\usepackage[dvipsnames]{xcolor}
\usepackage{graphicx}
\usepackage{tabularx,float}
\usepackage[hidelinks]{hyperref}
\usepackage{enumerate,mathrsfs}
\usepackage{mathtools}
\usepackage{enumitem}
\usepackage{orcidlink}
\usepackage{titling}
\usepackage{cleveref}
\usepackage[misc]{ifsym}
\usepackage[titletoc]{appendix}
\usepackage[intoc]{nomencl}
\usepackage{caption}
\usepackage{overpic}
\usepackage{tikz}
\usetikzlibrary{arrows.meta,calc,positioning,decorations.pathreplacing,patterns,shapes.geometric,intersections,3d,backgrounds}
\usepackage{pgfplots}
\pgfplotsset{compat=1.18}
\usepgfplotslibrary{fillbetween}

\definecolor{ddeNeg}{RGB}{170,20,40}
\definecolor{ddePos}{RGB}{90,190,120}
\definecolor{ddeIndigo}{RGB}{51,34,136}
\definecolor{ddeCyan}{RGB}{51,170,221}
\definecolor{ddeCyanLight}{RGB}{170,220,240}
\definecolor{ddePurple}{RGB}{170,68,153}
\definecolor{ddeOchre}{RGB}{204,153,0}
\definecolor{ddeTeal}{RGB}{34,136,136}
\definecolor{ddeSpecial}{RGB}{230,130,20}
\definecolor{ddeGray}{RGB}{140,140,140}
\definecolor{ddeGrayDark}{RGB}{64,64,64}
\definecolor{ddeGrayLight}{RGB}{225,225,225}
\tikzset{
	nn/.style={>={Stealth[length=2.0mm]}},
	nnmain/.style={circle, fill=ddeGrayLight, draw=ddeGray, line width=0.8pt, minimum size=6mm, inner sep=0pt},
	nnin/.style={nnmain},
	nnout/.style={nnmain},
	nnpre/.style={circle, draw=ddeGray, fill=white, minimum size=6mm, inner sep=0pt},
	nnedge/.style={->, very thick, ddeGray, shorten >=1.5pt, shorten <=1.5pt},
	nnar/.style={->, semithick, ddeGray, shorten >=2.5pt, shorten <=2.5pt},
	nnhdr/.style={draw=ddeGray, rectangle, inner sep=4pt, font=\small},
	nncn/.style={font=\scriptsize},
	nnrl/.style={font=\scriptsize, anchor=east, text=ddeGray},
	lg/.style={>={Stealth[length=2.0mm]}, node distance=14mm, very thick, shorten >=5pt, shorten <=5pt},
	lglayer/.style={rectangle, rounded corners=1.5pt, fill=ddeGrayLight, draw=ddeGray, line width=0.8pt, minimum size=5.5mm, inner sep=0pt},
	lglayerin/.style={lglayer},
	lglayerout/.style={lglayer},
	lglbl/.style={font=\normalfont},
	axisarrow/.style={-{Latex[length=2.4mm,width=1.7mm]}, thick},
	guide/.style={densely dotted, thick, black},
	fillforest/.style={fill=ddePos, opacity=0.35},
	lineforest/.style={draw=ddePos, thick},
	fillorange/.style={fill=ddeNeg, opacity=0.35},
	lineorange/.style={draw=ddeNeg, thick},
}
\newcommand{\ddeXmax}{6.4}
\newcommand{\ddeYmax}{4.4}
\newcommand{\ddeTval}{3.0}
\newcommand{\blobcurve}{plot [smooth cycle, tension=0.75] coordinates
	{(0.3,2.2) (0.7,3.4) (1.7,3.9) (2.6,3.6) (2.85,2.9) (2.35,2.45) (2.75,1.7) (2.1,1.1) (1.0,1.15)}}
\newcommand{\noncompactcurve}{(0.9,0) .. controls (1.3,0.9) and (2.6,0.8) .. (2.9,1.7)
	.. controls (3.15,2.45) and (2.2,2.9) .. (2.6,3.7)
	.. controls (2.8,4.0) and (3.1,4.2) .. (3.5,4.3)}
\usepackage[subrefformat=parens,labelformat=parens]{subcaption}
\makenomenclature

\usepackage{etoolbox}
\renewcommand\nomgroup[1]{%
	\ifstrequal{#1}{A}{\vspace{20pt} \item[\textbf{Abbreviations}]}{%
		\ifstrequal{#1}{B}{\vspace{20pt} \item[\textbf{Basic Mathematics}]}{%
			\ifstrequal{#1}{D}{\vspace{20pt} \item[\textbf{Feed-Forward Neural Networks}]}{%
				\ifstrequal{#1}{E}{\vspace{20pt} \item[\textbf{Delay Differential Equations}]}{%
					\ifstrequal{#1}{F}{\vspace{20pt} \item[\textbf{Neural DDEs}]}{%
						\ifstrequal{#1}{C}{\vspace{20pt} \item[\textbf{Function Spaces and Norms}]}{%
							\ifstrequal{#1}{G}{\vspace{20pt} \item[\textbf{Other Symbols}]}{}}}}}}}%
}

\numberwithin{equation}{section}

\reversemarginpar

\newcommand*{\dd}{\mathrm{d}}

\newtheorem{theorem}{Theorem}[section]

\newtheorem{definition}[theorem]{Definition}
\newtheorem{example}[theorem]{Example}
\newtheorem{proposition}[theorem]{Proposition}
\newtheorem{corollary}[theorem]{Corollary}
\newtheorem{lemma}[theorem]{Lemma}
\newtheorem{remark}[theorem]{Remark}
\newtheorem{assumption}[theorem]{Assumption}

\newtheorem{innercustomthm}{Theorem}
\newenvironment{customthm}[1]
{\renewcommand\theinnercustomthm{#1}\innercustomthm}
{\endinnercustomthm}

\newtheorem{innercustomass}{Assumption}
\newenvironment{customass}[1]
{\renewcommand\theinnercustomass{#1}\innercustomass}
{\endinnercustomass}

\counterwithin{figure}{section}

\newcommand{\eps}{\varepsilon}
\let\textdotlessi\i
\renewcommand{\i}{\ifmmode{\mathrm i}\else\textdotlessi\fi}

\newcommand\norm[1]{\left\lVert#1\right\rVert}
\newcommand\normm[1]{\big\lVert#1\big\rVert}
\newcommand\normk[1]{\lVert#1\rVert}
\newcommand\abs[1]{\left\lvert#1\right\rvert}
\newcommand{\R}{\mathbb{R}}

\newcommand\blfootnote[1]{%
	\begingroup
	\renewcommand\thefootnote{}\footnote{#1}%
	\addtocounter{footnote}{-1}%
	\endgroup
}

\usepackage{fancyhdr}
\title{\textbf{The Influence of the Memory Capacity of \\ Neural DDEs on the \\ Universal Approximation Property}}

\author{Christian Kuehn \orcidlink{0000-0002-7063-6173}$^{1,2,3}$ 
	\& Sara-Viola Kuntz \orcidlink{0009-0000-4611-9742}$^{1,2,3}$
}

\date{
	\small{$^1$\textit{Technical University of Munich, School of Computation, Information and Technology, \\ Department of Mathematics, Boltzmannstraße 3, 85748 Garching, Germany} \\
		$^2$\textit{Munich Data Science Institute (MDSI), Garching, Germany } \\
		$^3$\textit{Munich Center for Machine Learning (MCML), München, Germany }}\\[5mm]
	\large{November 26, 2026}
}
\begin{document}
	\maketitle

	\begin{abstract}
		Neural Ordinary Differential Equations (Neural ODEs), which are the continuous-time analog of Residual Neural Networks (ResNets), have gained significant attention in recent years. Similarly, Neural Delay Differential Equations (Neural DDEs) can be interpreted as an infinite-depth limit of Densely Connected Residual Neural Networks (DenseResNets). In contrast to traditional ResNet architectures, DenseResNets are feed-forward neural networks that allow for shortcut connections across all layers. These additional inter-layer connections introduce memory in the network architecture, as is typical in many modern architectures, such as U-Nets and LSTMs. In this work, we explore how the memory capacity in neural DDEs influences the universal approximation property. The key parameter for studying the memory capacity is the product $K\tau$ of the Lipschitz constant of the vector field and the delay of the DDE. In the case of non-augmented architectures, where the network width is not larger than the input dimension, neural ODEs and classical feed-forward neural networks cannot have the universal approximation property. We show that if the memory capacity $K\tau$ is sufficiently small, the dynamics of the neural DDE can be approximated by a neural ODE. Consequently, under bounds on the weights and the vector field, non-augmented neural DDEs with a small memory capacity cannot approximate all maps with a prescribed Lipschitz constant, however small. If the network width is smaller than the input or output dimension, i.e., if the architecture has a bottleneck, no universal approximation is possible for any memory capacity. In contrast, if the network width is at least as large as both the input and the output dimension and the memory capacity $K\tau$ is sufficiently large, we can establish the universal approximation property of neural DDEs with respect to the space of Lipschitz continuous functions. Moreover, if the neural DDE architecture is augmented, we can expand the parameter regions in which universal approximation is possible. Overall, our results show that by increasing the memory capacity $K\tau$, the infinite-dimensional phase space of DDEs with positive delay $\tau>0$ is not sufficient to guarantee a direct jump transition to universal approximation, but only after a certain memory threshold, universal approximation holds. 
	\end{abstract}
	
	\vspace{1mm}
	
	
	\vspace{1mm}
	
	
	\vspace{-2mm}
	
	\blfootnote{\hspace{-5.4mm}\Letter \; ckuehn@ma.tum.de (Christian Kuehn)  \\
		\Letter\; saraviola.kuntz@tum.de (Sara-Viola Kuntz, corresponding author)}
	
    \newpage
	\tableofcontents

	\vspace{8mm}
	
	\section{Introduction}
	
	An important concept of the theoretical and practical study of neural networks is expressivity. Expressivity is often characterized by universal approximation theorems, which assure that for any element of a given function space, there exists a choice of the network parameters such that the considered function can be approximated by the neural network with arbitrary precision \cite{Kratsios2021}. The universal approximation property guarantees that a given neural network is powerful enough to complete a given task up to a pre-defined error.	Expressivity via universal approximation focuses on the properties of the input-output map of a neural network with fixed parameters and does not answer the question of how and if a learning algorithm reaches optimal parameters. Nevertheless, as outlined in this work, expressivity depends strongly on the structure of the neural network, such that it is advantageous to characterize powerful and efficient architectures before the neural network is trained. 
	
	Probably among the first universal approximation-type theorems are results by Kolmogorov and Arnold in the 1950s~\cite{Arnold2009, Kolmogorov1956}. In the early 1960s, Sharkovskii showed that iterations of even relatively simple nonlinear maps can be incredibly rich and expressive~\cite{Sharkovskii1995}. Later on, these ideas were also embedded into other architectures. A well-known universal approximation theorem was proven by Cybenko in 1989 for single-layer neural networks with sigmoidal activation functions \cite{Cybenko1989} and then extended by Hornik to multilayer perceptrons (MLPs) with more general activation functions \cite{Hornik1991,Hornik1989}. MLPs are feed-forward neural networks (FNNs) structured in layers $h_l \in \mathbb{R}^{n_l}$, $l \in \{0,1,\ldots,L\}$ consisting of  $n_l$ neurons, where $h_0$ is the input layer, $h_L$ is the output layer, and $h_1,h_2,\ldots,h_{L-1}$ are the hidden layers. The layers are for $l\in\{0,\ldots,L-1\}$ iteratively updated by
	\begin{equation} \label{eq:MLP}
		h_{l+1} = f_l(h_{l},\theta_l)
	\end{equation}
	with parameters $\theta_l\in\R^{p_l}$ and layer map $f_l:\R^{n_l}\times \R^{p_l}\rightarrow \R^{n_{l+1}}$, typically based on a componentwise applied activation function such as $\tanh$, soft-plus, sigmoid or (normal, leaky or parametric) $\textup{ReLU}$. The resulting neural network is a map $\Phi_\theta: \R^{n_0}\rightarrow \R^{n_L}$, $h_0\mapsto h_L$ with parameters $\theta = (\theta_0,\ldots,\theta_{L-1})$, depth $L$ and width $n_\text{max} = \max\{n_0,\ldots,n_L\}.$
	
	Most universal approximation theorems for FNNs use that either the depth $L$ or the width $n_\text{max}$, and the number of parameters of the neural network need to be sufficiently large \cite{Kidger2022, Lu2017}. As larger networks are harder to train, it is of interest to define powerful but still efficient architectures. Hence, recent studies also analyzed the minimal width necessary for MLPs to have the universal approximation property \cite{Hanin2018,Johnson2018,Park2020}. As a result, these works show that an augmented architecture is necessary for universal approximation, i.e., $n_\text{max}>n_0$, which means that at least one layer has to be wider than the input layer. For non-augmented networks with $n_\text{max}\leq n_0$, where all layers $h_l$ have a width equal to or smaller than the input layer, no universal approximation is possible.
	
	Besides increasing the width or the depth of FNNs,  other modifications of the network architecture can increase the approximation capability. In this work, we mainly focus on additional inter-layer connections, which can skip one or more layers. One famous architecture using identity shortcut connections is the residual neural network, also called ResNets~\cite{He2016}. In ResNets, all layers consist of the same number of neurons $m = n_l$ for all $l \in\{0,\ldots,L\}$, and the layers are iteratively updated by
	\begin{equation} \label{eq:ResNet}
		h_{l+1} = h_{l} + f_{\text{ResNet},l}(h_{l}, \theta_l)
	\end{equation} 
	for $l \in \{0,\ldots,L-1\}$ with layer map $f_{\text{ResNet},l}:\R^{m}\times \R^{p_l}\rightarrow \R^{m}$ and parameters \mbox{$\theta_l\in \R^{p_l}$.} In particular, for deep networks with a large number of layers $L \gg 1$, it is advantageous to use residual architectures, as they mitigate the vanishing/exploding gradient problem during the training process~\cite{Glorot2010,He2016}. Furthermore, the additional identity shortcut connection in \eqref{eq:ResNet} increases the approximation capability of ResNets compared with MLPs \cite{Lin2018}.
	
	More advanced feed-forward neural networks allow for more general and longer inter-layer connections. In that way, information or features of previous layers, which might be forgotten during the iteration steps of deep neural networks, can be reintroduced in future layers, building a memory. Examples of networks with a few inter-layer connections are Highway Networks \cite{Srivastava2015}, Residual Memory Networks~(RMNs) \cite{Baskar2017}, and U-Nets \cite{Ronneberger2015}. In the case of Densely Connected Convolutional Neural Networks (DenseNets)~\cite{Huang2017}, Mixed Link Networks (MixNet) \cite{Wang2018}, and Dense Shortcut Networks (DSNets)~\cite{Zhang2020b}, the network is structured in layers, but all possible feed-forward inter-layer connections are allowed. The cited works show numerically that the performance of the neural network increases with the number and the length of the inter-layer connections. In the context of recurrent neural networks (RNNs), advanced architectures also introduce memory in terms of shortcut connections. Examples are Long Short-Term Memory Networks (LSTMs) \cite{Hochreiter1997} and Memory-Augmented Neural Networks (MANNs) \cite{Santoro2016}. Also, in neuroscience, the positive influence of memory terms, such as the delay of spiking neurons in Spiking Neural Networks (SNNs), is studied \cite{DAgostino2023,Masquelier2023, Bohte2016}. 
	
	In this work, we aim to study theoretically how the memory introduced by additional inter-layer connections influences the universal approximation capability of FNNs. Depending on the length of the inter-layer connections, we aim to distinguish whether universal approximation is possible or not. To include as many architectures as possible but to stay mathematically treatable, we consider densely connected feed-forward neural networks with a ResNet structure called DenseResNets. In analogy to ResNets, all layers have the same width $m$, and the update rule includes an identity shortcut connection. DenseResNets are defined by 
	\begin{equation}\label{eq:DenseResNet}
		h_{l+1} = h_{l} + f_{\textup{Dense},l}(h_{l}, h_{l-1}, \ldots, h_0, \theta_{\textup{Dense},l}), 
	\end{equation}
	for $l \in \{0,\ldots,L-1\}$ with $f_{\textup{Dense},l}: \mathbb{R}^{m} \times \ldots \times \mathbb{R}^m \times \mathbb{R}^{p_l} \rightarrow \mathbb{R}^m$, where $\theta_{\textup{Dense},l} \in \mathbb{R}^{p_l}$ stores all weights between the layers $h_{l+1}$ and $h_{l+1-i}$ for all $i \in \{1,\ldots,l+1\}$. As in MLPs and ResNets, the DenseResNet~\eqref{eq:DenseResNet} defines an input-output map $\Phi_\theta: \R^m\rightarrow \R^m: h_0 \mapsto h_L$. If inter-layer connections of the update rule~\eqref{eq:DenseResNet} are not used, the corresponding parameters are zero. We say that the memory capacity of the DenseResNet \eqref{eq:DenseResNet} increases if longer inter-layer connections are allowed. The memory capacity of the DenseResNet not only depends on the inter-layer connections but also on the choice of the layer maps $f_{\text{Dense},l}$.
	
	\begin{figure}[b!]
		\centering
		\begin{subfigure}{0.9\textwidth}
			\centering
			\vspace{5mm}
			\begin{tikzpicture}[lg]
			      \node[lglayerin]                 (h0)  {};
			      \node[lglayer, right=of h0]      (h1)  {};
			      \node[lglayer, right=of h1]      (h2)  {};
			      \node[lglayer, right=of h2]      (h3)  {};
			      \node[lglayer, right=24mm of h3] (hL1) {};
			      \node[lglayerout, right=of hL1]  (hL)  {};
			      \node[lglbl, below=2.5mm of h0]  {$h_0$};
			      \node[lglbl, below=2.5mm of h1]  {$h_1$};
			      \node[lglbl, below=2.5mm of h2]  {$h_2$};
			      \node[lglbl, below=2.5mm of h3]  {$h_3$};
			      \node[lglbl, below=2.5mm of hL1] {$h_{L-1}$};
			      \node[lglbl, below=2.5mm of hL]  {$h_L$};
			      \draw[->] (h0) -- (h1);  \draw[->] (h1) -- (h2);
			      \draw[->] (h2) -- (h3);  \draw[->, dotted, ddeGrayDark] (h3) -- (hL1);
			      \draw[->] (hL1) -- (hL);
			    \end{tikzpicture}
				\vspace{1mm}
			\caption{Layer structure of classical FNNs, such as MLPs and ResNets. Every layer $h_l$ is represented as a node of the graph.}
			\label{fig:DenseResNet_a}
		\end{subfigure}
		\begin{subfigure}{0.9\textwidth}
			\centering
			\vspace{2mm}
			\begin{tikzpicture}[lg]
			      \node[lglayerin]                 (h0)  {};
			      \node[lglayer, right=of h0]      (h1)  {};
			      \node[lglayer, right=of h1]      (h2)  {};
			      \node[lglayer, right=of h2]      (h3)  {};
			      \node[lglayer, right=24mm of h3] (hL1) {};
			      \node[lglayerout, right=of hL1]  (hL)  {};
			      \node[lglbl, below=2.5mm of h0]  {$h_0$};
			      \node[lglbl, below=2.5mm of h1]  {$h_1$};
			      \node[lglbl, below=2.5mm of h2]  {$h_2$};
			      \node[lglbl, below=2.5mm of h3]  {$h_3$};
			      \node[lglbl, below=2.5mm of hL1] {$h_{L-1}$};
			      \node[lglbl, below=2.5mm of hL]  {$h_L$};
			      \draw[->] (h0) -- (h1);  \draw[->] (h1) -- (h2);
			      \draw[->] (h2) -- (h3);  \draw[->, dotted, ddeGrayDark] (h3) -- (hL1);
			      \draw[->] (hL1) -- (hL);
			      \draw[->] (h0) to[bend left=28] (h2);
			      \draw[->] (h1) to[bend left=28] (h3);
			      \draw[->] (h0) to[bend left=31] (h3);
			      \draw[->, dotted, ddeGrayDark] (h3) to[bend left=38] (hL);
			      \draw[->, dotted, ddeGrayDark] (h2) to[bend left=41] (hL1);
			      \draw[->, dotted, ddeGrayDark] (h1) to[bend left=41] (hL1);
			      \draw[->, dotted, ddeGrayDark] (h0) to[out=35, in=145] (hL);
			    \end{tikzpicture}
				\vspace{1mm}
			\caption{Additional inter-layer connections of densely connected FNNs, such as DenseResNets. Every layer $h_l$ is represented as a node of the graph.}
			\label{fig:DenseResNet_b}
			\vspace{4mm}
		\end{subfigure}
		\begin{subfigure}{0.4\textwidth}
			\centering
			\vspace{2mm}
			\resizebox{0.9\linewidth}{!}{%
			\begin{tikzpicture}[lg,
			            se/.style={->, black},
			            de/.style={->, dotted, ddeGrayDark}]
			            \node[lglayerin,  label={135:$h_0$}]    (h0)   at (-1.905, 1.1){};
			            \node[lglayer,    label={ 90:$h_1$}]    (h1)   at ( 0.0,   2.2){};
			            \node[lglayer,    label={ 45:$h_2$}]    (h2)   at ( 1.905, 1.1){};
			            \node[text=ddeGrayDark, font=\boldmath]                                  (mid)  at ( 1.905,-1.1){$\dots$};
			            \node[lglayer,    label={270:$h_{L-1}$}](hLm1) at ( 0.0,  -2.2){};
			            \node[lglayerout, label={225:$h_L$}]    (hL)   at (-1.905,-1.1){};
			            \draw[se](h0)--(h1); \draw[se](h0)--(h2); \draw[se](h0)--(hLm1); \draw[se](h0)--(hL);
			            \draw[se](h1)--(h2); \draw[se](h1)--(hLm1); \draw[se](h1)--(hL);
			            \draw[se](h2)--(hLm1); \draw[se](h2)--(hL);  \draw[se](hLm1)--(hL);
			            \draw[de](h0)--(mid); \draw[de](h1)--(mid); \draw[de](h2)--(mid);
			            \draw[de](mid)--(hLm1); \draw[de](mid)--(hL);
			        \end{tikzpicture}%
			}
			\vspace{1mm}
			\caption{Representation of the graph of the densely connected FNN in (b) as a directed complete graph.}
			\label{fig:DenseResNet_c}
		\end{subfigure}
		\begin{subfigure}{0.1\textwidth}
			\textcolor{white}{.}
		\end{subfigure}
		\begin{subfigure}{0.4\textwidth}
			\centering
			\vspace{2mm}
			\resizebox{0.6\linewidth}{!}{%
			\begin{tikzpicture}[nn]
			            \foreach \i in {1,...,4} {\node[nnmain] (L\i) at (0, {2.25 - (\i-1)*1.5}) {};}
			            \foreach \j in {1,...,3} {\node[nnmain] (R\j) at (3, {1.5  - (\j-1)*1.5}) {};}
			            \foreach \i in {1,...,4}{\foreach \j in {1,...,3}{\draw[nnedge] (L\i) -- (R\j);}}
			            \node at (0,-3) {$h_l$};   \node at (3,-3) {$h_{l+1}$};
			        \end{tikzpicture}%
			}
			\vspace{1mm}
			\caption{The connections between all neurons of the layers $h_l$, $h_{l+1}$ of a fully connected FNN form a complete bipartite directed graph.}
			\label{fig:DenseResNet_d}
		\end{subfigure}
		\caption{Graphical representation of classical FNNs and densely connected FNNs. Layers are represented as squares and neurons as circles.}
		\label{fig:DenseResNet}
	\end{figure}

	In Figure \ref{fig:DenseResNet}\subref{fig:DenseResNet_a}, the layer structure of classical FNNs, such as MLPs and ResNets, is visualized. As there exist only connections between layer $h_{l-1}$ and $h_l$ for $l\in\{1,\ldots,L\}$, the underlying graph with nodes $h_l$, $l\in\{0,\ldots,L\}$ has $L$ edges. In Figure \ref{fig:DenseResNet}\subref{fig:DenseResNet_b}, the additional shortcut connections of densely connected FNNs are shown. As there exist  connections between layer $h_{l-i}$ and $h_l$ for every $i \in \{1,\ldots,l\}$ and $l\in\{1,\ldots,L\}$, the underlying graph with nodes $h_l$, $l\in\{0,\ldots,L\}$ has $\sum_{l = 1}^L l = \frac{L(L+1)}{2}$ edges. A densely connected FNN has as an underlying graph a complete graph, i.e., a graph where every node is connected to every node via an edge (cf.\ \cite{Diestel2025}). As the network has a feed-forward structure, every edge of the complete graph has a pre-defined orientation, as visualized in Figure \ref{fig:DenseResNet}\subref{fig:DenseResNet_c}. To avoid confusion between terminologies of different research areas, we clarify that in the context of FNNs, we refer to a neural network, where every layer can be connected with every layer in the feed-forward sense as in \eqref{eq:DenseResNet}, as a densely connected network. The underlying graph, which consists of nodes $h_l$, $l\in\{0,\ldots,L\}$, is a complete graph, i.e., the graph is fully connected. In contrast, a fully connected FNN is characterized by the property that every neuron of layer $h_l$ is connected to every neuron of layer $h_{l+1}$. In that case, the graph consisting of all neurons of layers $h_l$ and $h_{l+1}$ as nodes is a complete bipartite graph (cf.\ \cite{Diestel2025}), as shown in  Figure \ref{fig:DenseResNet}\subref{fig:DenseResNet_d}. \medskip

	A common approach to studying the network dynamics of large networks is to consider the infinite network or mean-field limit. In the case of ResNets and DenseResNets, the width $m$ of the network is fixed, and in the context of deep learning, it is of interest to study the infinite-depth limit $L \rightarrow \infty$. The goal of this work is to study the influence of the memory capacity of the DenseResNet \eqref{eq:DenseResNet} in the infinite-depth limit, which is called a neural delay differential equation (neural DDE). We have chosen a ResNet structure with the identity shortcut connection in \eqref{eq:DenseResNet} to define a similar network limit to that of ResNets, which is called a neural ordinary differential equation (neural ODE) \cite{Chen2018,Ruthotto2018,Weinan2017}. To illustrate the idea of the infinite-depth limit of residual architectures, we first introduce neural ODEs.
	In the case that $f_\textup{ResNet} = f_{\textup{ResNet},l}$ with parameter dimension $p =p_l$ for all $l \in \{0,\ldots,L-1\}$, the iterative updates~\eqref{eq:ResNet} of ResNets can be obtained as an Euler discretization of the neural ODE
	\begin{equation}\label{eq:neuralODE}
		\frac{\dd h}{\dd t} = f_\text{ODE}(h(t),\theta(t)), \qquad h(0) = h_0,
	\end{equation}
	on the time interval $[0,T]$ with step size $\delta = \frac{T}{L}$ and rescaled vector field
	\begin{equation*}
		f_\textup{ODE}(\cdot,\cdot) = \frac{1}{\delta} f_\textup{ResNet}(\cdot,\cdot):\R^m\times\R^p\rightarrow \R^m.
	\end{equation*}
	The function $h:[0,T]\rightarrow  \mathbb{R}^m$ represents the layers and the function $\theta:[0,T]\rightarrow  \mathbb{R}^{p}$ the parameters. The initial condition is the first layer $h_0$, and the output is the \mbox{time-$T$} map $h(T)$ (cf.~\cite{Guckenheimer2002}), corresponding to the last layer $h_L$. Hence, neural ODEs can be interpreted as an infinite-depth limit of ResNets. Neural ODEs are used for various machine learning tasks and are an active field of research~\cite{Cipriani2024, Kidger2022, Kuehn2023, Kuehn2024, RuizBalet2023}. 
	
	To be flexible with respect to the input and output dimension of the neural ODE, two affine linear layers $\lambda: \R^n \rightarrow \R^m$ and $\tilde \lambda: \R^m \rightarrow \R^q$ can be added before and after the solution of the initial value problem (IVP) \eqref{eq:neuralODE}, leading to an input-output map $\Phi_\theta: \R^n \rightarrow \R^q$. The neural ODE is called augmented if the dimension $m$ of the IVP is larger than the input dimension~$n$, otherwise non-augmented \cite{Dupont2019, Kuehn2023}. Depending on the relationship between the dimensions $n$, $m$, and $q$, neural ODEs show similar behavior to MLPs with respect to the approximation capability \cite{Kuehn2024}. As for most universal approximation theorems for FNNs, no universal approximation is possible if the architecture is non-augmented. In the case of an augmented architecture with $m\geq n+q$, the vector field of the neural ODE \eqref{eq:neuralODE} can be chosen in such a way that every continuous function can be represented exactly \cite{Kuehn2023}. If an exact representation is established, we refer to it as the universal embedding property, which is stronger than the universal approximation property. In this work, we aim to answer the question of whether additional inter-layer connections in DenseResNets can overcome the restrictions that FNNs and neural ODEs have in the non-augmented case. It is natural to expect that the approximation capability of DenseResNets increases with the length of the allowed inter-layer connections. 
	
	In analogy to neural ODEs, we study in this work the infinite-depth limit of the DenseResNet~\eqref{eq:DenseResNet}, given by neural DDEs. The transition to a continuous-time model gives us the opportunity to use the well-developed theory of delay differential equations and dynamical systems. In Section \ref{sec:modeling}, we explain how the DenseResNet \eqref{eq:DenseResNet} can be obtained as an Euler discretization of the DDE
	\begin{align}
		\label{eq:intro_DDE} \frac{\dd y}{\dd t} &= f_\textup{DDE}(y_t,\theta(t)) \qquad &&\text{for } t\in[0,T], \\  
		\label{eq:intro_DDE_initialdata} y(t) &= u(t) &&\text{for } t \in [-\tau,0],
	\end{align}
	with delay $\tau \geq 0$,  parameter function $\theta:[0,T]\rightarrow  \mathbb{R}^{ p}$, initial function $u:[-\tau,0]\rightarrow \R^m$ and $f_\textup{DDE}: \mathcal{C}\times \mathbb{R}^{p} \rightarrow \mathbb{R}^m$, where $\mathcal{C}\coloneqq C^0([-\tau,0],\mathbb{R}^m)$ is the space of continuous functions mapping from the interval $[-\tau,0]$ to $\mathbb{R}^m$. The delayed function $y_t \in \mathcal{C}$ is  defined by $y_t(s) = y(t+s)$ for $s\in [-\tau,0]$ and all $t\in\mathbb{R}$ such that $y(t+s)$ exists. In contrast to IVPs such as \eqref{eq:neuralODE}, it is in general necessary to define an initial function $u \in \mathcal{C}$ over the time interval $[-\tau,0]$ in~\eqref{eq:intro_DDE_initialdata} for the DDE \eqref{eq:intro_DDE}, as the phase space $\mathcal{C}$ is infinite-dimensional \cite{Hale1993}. The state $y(0)$ corresponds to the first layer $h_0$ and the time-$T$ map $y(T)$ to the last layer $h_L$ of \eqref{eq:DenseResNet}. Multiple choices exist for the function $f_\text{DDE}$, motivating neural DDEs as an infinite-depth limit of DenseResNets with the update rule \eqref{eq:DenseResNet}. Different modeling approaches are discussed in Section \ref{sec:modeling_discretization}, as multiple choices of the delayed function $y_t$ and the initial data $u$ are possible. However, a common property is that the length of the delay~$\tau$ corresponds to the length of the longest inter-layer connection of the DenseResNet. A typical choice for the initial function $u$ is the constant initial data $c_{y_0}:[-\tau,0] \rightarrow \mathbb{R}^m$, $c_{y_0}(t) = y_0$.
	
	In \cite{Zhu2021}, neural DDEs with a single, constant delay $\tau$ are studied, i.e., $y_t = y(t-\tau)$. The adjoint sensitivity method is used to optimize the parameters $\theta$ of the neural DDE \eqref{eq:intro_DDE}. It is shown theoretically that neural DDEs have the universal approximation property if $\tau = T$, and the performance is validated on real-world image data sets. The concept of neural DDEs with a single delay has been extended to Neural Piecewise-Constant Delay Differential Equations (PCDDEs) with multiple but piece-wise constant delays to increase the modeling capability of neural DDEs \cite{Zhu2022}. In \cite{Monsel2024}, it is shown that State Dependent Neural Delay Differential Equations (SDDDEs), which allow for delay terms $y_t = y(t-\tau(t,y(t)))$ depending on the time $t$ and the state $y(t)$, numerically outperform neural DDEs with constant delays. In the literature, other concepts also exist for introducing memory in neural ODE models, like neural ODEs based on second-order ODEs, which can be interpreted as the continuous-time analog of Momentum ResNets \cite{RuizBalet2022,Sander2021}.
	
	To treat all models simultaneously, which are either discussed in the literature or introduced in Section \ref{sec:modeling_discretization}, we drop the dependency on the parameter function $\theta$ and study neural DDEs based on general non-autonomous vector fields with general delayed functions $y_t \in\mathcal{C}$. Hence, the system of interest is of the form  
	\begin{align}
		\label{eq:intro_DDE_2} \frac{\dd y}{\dd t} &= F(t,y_t) \qquad &&\text{for } t\geq 0, \\  
		\label{eq:intro_DDE_initialdata_2} y(t) &= c_{y_0}(t) &&\text{for } t \in [-\tau,0],
	\end{align}
	with vector field $F:\mathbb{R}\times \mathcal{C}\rightarrow \mathbb{R}^m$ and delay $\tau \geq 0$. In analogy to neural ODEs, we can add two affine linear transformations  $\lambda:\mathbb{R}^n\rightarrow \mathbb{R}^m$ and $\tilde \lambda:\mathbb{R}^m\rightarrow \mathbb{R}^q$ before and after the solution of the DDE~\eqref{eq:intro_DDE_2}-\eqref{eq:intro_DDE_initialdata_2}, resulting in the general neural DDE architecture 
	\begin{equation} \label{eq:neuralDDE}
		\Phi: \mathbb{R}^n \rightarrow \mathbb{R}^q, \quad x \mapsto \tilde{\lambda}(y(0,c_{\lambda(x)})(T)),
	\end{equation}
	where $y(0,c_{\lambda(x)})(T)$ denotes the time-$T$ map of the solution of \eqref{eq:intro_DDE_2} with in time constant initial data $c_{y_0} = c_{\lambda(x)}$, $c_{\lambda(x)} =  \lambda(x)$ for $t\in [-\tau,0]$. Assumptions guaranteeing the well-posedness of the neural DDE \eqref{eq:neuralDDE} are discussed in Section \ref{sec:architecture}. In Section~\ref{sec:parameterized}, we show that every neural DDE whose vector field is parameterized as in~\eqref{eq:intro_DDE} is also a neural DDE with general vector field as in~\eqref{eq:intro_DDE_2}. Additionally, we show that the input-output maps of both architectures are uniformly close if their vector fields are uniformly close. \medskip
	
	Our original question, how the length of the inter-layer connections in the DenseResNet \eqref{eq:DenseResNet} influences the approximation capability, transfers in the continuous-time model~\eqref{eq:intro_DDE_2}-\eqref{eq:intro_DDE_initialdata_2} to the question, how the length of the delay $\tau$ influences the approximation capability. In the case that $\tau = 0$, the neural DDE architecture becomes a neural ODE. As non-augmented neural ODEs cannot have the universal approximation property, also non-augmented neural DDEs with $\tau = 0$ cannot have the universal approximation property. As already shown in a special case in \cite{Zhu2021}, a large delay $\tau = T$  leads to the universal approximation property for non-augmented neural DDEs. In this work, we analyze the transition from $\tau = 0$ to $\tau = T$ and characterize parameter regions with universal and no universal approximation. It turns out that the approximation capability of neural DDEs does not only depend on the delay $\tau$ but also on the Lipschitz constant $K$ of the vector field in~\eqref{eq:intro_DDE_2} with respect to the second variable. The delay $\tau$ corresponds to the length of the inter-layer connections of the DenseResNet, and the Lipschitz constant $K$ characterizes the chosen layer maps. As our main results depend on the product $K \tau$, we call $K\tau$ the memory capacity of a neural DDE.
	
	In Section \ref{sec:overview}, we give an overview of the results proven in this work. First, we introduce in Section~\ref{sec:universal} the universal approximation and the universal embedding property. Afterward, we state in Section~\ref{sec:nonaugmentedbasic} the theoretical results obtained by Zhu et al.\ regarding the universal approximation property of non-augmented neural DDEs without affine linear layers $\lambda$ and $\tilde \lambda$, in the case that $\tau = T$~\cite{Zhu2021}. In Section~\ref{sec:nonaugmentedgeneral}, we collect our negative results for general neural DDEs with the architecture~\eqref{eq:neuralDDE}. We state Theorem~\ref{th:tau0}, which shows that if the memory capacity $K\tau$ is sufficiently small, non-augmented neural DDEs with bounded weights and vector field cannot approximate all maps with a prescribed Lipschitz constant, however small. Furthermore, Theorem~\ref{th:bottleneck} and Theorem~\ref{th:outputbottleneck} show that if the dimension $m$ of the DDE is smaller than the input dimension $n$ or the output dimension~$q$, no universal approximation is possible for any memory capacity. In Section~\ref{sec:augmented}, we prove as a positive result Theorem \ref{th:universal_embedding}, showing via an explicit construction that if the memory capacity $K\tau$ is sufficiently large, any Lipschitz continuous function can be represented exactly on bounded open sets by a neural DDE with Lipschitz constant $K$ and delay~$\tau$, if $m\geq\max\{n,q\}$. Hence, we establish in a specific parameter regime the universal embedding property for non-augmented neural DDEs with $m = n\geq q$, showing that a large memory capacity increases the approximation capability. Finally, Theorem~\ref{th:universalaugmented} shows that if the dimension $m$ is at least as large as the sum $n+q$ of input and output dimensions, the universal embedding property holds for every delay $\tau\in[0,T]$, provided that the Lipschitz constant $K$ is sufficiently large. The proof only uses the additional dimensions of the vector field and holds analogously for neural ODEs. Hence, in the augmented case, no delay is needed to establish the universal embedding property. Overall, we characterize in the three-dimensional parameter space of Lipschitz constant~$K$, delay $\tau$, and dimension $m$ of the vector field, parameter regions of universal embedding and no universal approximation. We are able to establish universal embedding results due to the freedom of choice of the vector field $F$ in \eqref{eq:intro_DDE_2}. If a parameterized vector field $f_\text{DDE}$ in \eqref{eq:intro_DDE} is used, the universal embedding results become universal approximation theorems if the parameterization of~$f_\text{DDE}$ is rich enough to approximate the general vector field $F$.
	
	Our most important tool to prove the negative result for non-augmented neural DDEs is the theory of DDEs with small delay \cite{Chicone2003,Chicone2004a,Driver1968,Driver1976,Jarnik1975}. The main properties of DDEs with small delays are introduced in Section \ref{sec:dde_small_delay}. If the product $K \tau$ of the Lipschitz constant and delay is sufficiently small, every solution of a DDE is exponentially attracted towards a special solution, which is the solution of an ODE. In that way, the long-term behavior of DDEs in the infinite-dimensional phase space can be characterized by special solutions on a finite-dimensional inertial manifold. Hence, if~$K\tau$ is sufficiently small, the dynamics of a DDE resembles the dynamics of an ODE. As the proof of Theorem \ref{th:tau0} includes a careful tracking of intermediate error terms, the proof is postponed to Section \ref{sec:sectionproof}. The main idea is to explicitly construct a function with a local extreme point, which cannot be approximated by a non-augmented neural ODE. Then, the exponential attraction of neural DDE solutions towards special solutions is used to show that the same function also cannot be approximated by a non-augmented neural DDE. As the proof needs to combine results of different mathematical areas, we include a nomenclature at the end of this work.
	
	The main insight of this work is that by increasing the memory capacity $K\tau$, there exists a transition in the approximation capability of non-augmented neural DDEs with $m = n\geq q$. The existence of a positive delay $\tau >0$ makes the state space infinite-dimensional, so that the topological argument for finite-dimensional ODE flows no longer applies directly. Nevertheless, the infinite-dimensional state space is not sufficient to establish the universal approximation property for non-augmented neural DDEs. Transferred to DenseResNets, we expect that non-augmented architectures with a small memory capacity show the same restrictions as non-augmented FNNs, but large inter-layer connections can resolve this problem, provided that the width is not smaller than the input and output dimensions. This is a rigorous mathematical pointer as to why architectures with longer shortcut connections, such as DenseNets and U-Nets, show better practical performance than those without shortcut connections.

	\section{Modeling of Neural Delay Differential Equations}
	\label{sec:modeling}
	
	In this section, we study different modeling approaches for neural DDEs. The goal is to define neural DDEs in such a way that an Euler discretization leads to the iterative update rule of a DenseResNet. To that purpose, we establish conditions on the initial data and the concrete form of the delay functions. Finally, we define the general neural DDE model we study in this work, including two additional affine linear layers before and after the DDE to be flexible with respect to the input and output dimensions. Depending on the dimension of the input, the output, and the dimension of the DDE, we distinguish augmented and non-augmented models. Furthermore, we relate the general model to neural DDEs with a parameterized vector field, which are mainly used in practice. Readers who are completely familiar with neural DDE modeling and the relation between discrete models and their continuous limits, or those who want to start directly from the main results for the DDE limit, can skip ahead to Section~\ref{sec:overview} and return to Section~\ref{sec:modeling}.  
	
	\subsection{Euler Discretization of Differential-Difference Equations}
	\label{sec:modeling_discretization}
	
	As explained in the introduction, an Euler discretization of the neural ODE based on the initial value problem \eqref{eq:neuralODE} results in the update rule \eqref{eq:ResNet} of a residual neural network. In the case of a delay differential equation, additional conditions on the initial data and the delay functions are required to obtain a DenseResNet with the update rule
	\begin{equation}\label{eq:DenseResNet2}
		h_{l+1} = h_{l} + f_{\textup{Dense},l}(h_{l}, h_{l-1}, \ldots, h_0, \theta_{\textup{Dense},l}), 
	\end{equation}
	for $l \in \{0,\ldots,L-1\}$ and  $f_{\textup{Dense},l}: \mathbb{R}^{m} \times \cdots \times \mathbb{R}^m \times \mathbb{R}^{p_l} \rightarrow \mathbb{R}^m$, as an Euler discretization of a DDE. The general layer structure of a DenseResNet \eqref{eq:DenseResNet2} is visualized in Figure \ref{fig:DenseResNet}\subref{fig:DenseResNet_b}. The parameterized DDE we consider in the following is given by
	\begin{equation} \label{eq:DDE}
		\begin{aligned}
			\frac{\dd y}{\dd t} &= f_\textup{DDE}(y_t,\theta(t)) \qquad &&\text{for } t\in[0,T], \\  
			y(t) &= u(t) &&\text{for } t \in [-\tau,0],
		\end{aligned}
	\end{equation} 
	with delay $\tau \geq 0$, parameter function $\theta:[0,T]\rightarrow \mathbb{R}^{p}$, vector field $f_\textup{DDE}: \mathcal{C} \times \mathbb{R}^{p}\rightarrow \mathbb{R}^m$ and initial data $u \in \mathcal{C}$, where we abbreviate by $\mathcal{C}\coloneqq C^0([-\tau,0],\mathbb{R}^m)$ the space of continuous functions mapping from the interval $[-\tau,0]$ to $\mathbb{R}^m$. The delayed function $y_t \in \mathcal{C}$ is defined by $y_t(s) = y(t+s)$ for $s\in [-\tau,0]$ and all $t\in \mathbb{R}$ such that $y(t+s)$ exists. Hence, $y_t$ includes the current value $y(t)$, but also the history until $y(t-\tau)$. We call $y$ a solution of the DDE~\eqref{eq:DDE} if there exists $c\in(0,\infty]$, such that $y:[-\tau,c)\rightarrow \mathbb{R}^m$ is continuous and fulfills \eqref{eq:DDE} for $t\in[-\tau,c)$.  In the following, we always assume that the solution of \eqref{eq:DDE} exists for all $t\in [0,T]$ such that the time-$T$ map $y(T)$ is well-defined. The local existence of solutions of DDEs is discussed in Section~\ref{sec:architecture} and the global existence in Section~\ref{sec:globalexistence}.

	Special cases of the DDE \eqref{eq:DDE} are, for example, ordinary differential equations for $\tau = 0$ or integro-differential equations with continuous delay \cite{Hale1993}. As we study feed-forward neural networks structured in layers with fixed inter-layer connections, we restrict our modeling to DDEs with discrete delay. A general form of DDEs with discrete delay are differential-difference equations defined by
	\begin{equation} \label{eq:DDEdifference}
		\begin{aligned}
			\frac{\dd y}{\dd t} &= f_\textup{DDE}(y(t),y(t-\tau_1(t)),\ldots,y(t-\tau_J(t)),\theta(t)) \qquad &&\text{for } t \in [0,T], \\  
			y(t) &= u(t) &&\text{for } t \in [-\tau,0],
		\end{aligned}
	\end{equation}
	with $0 \leq \tau_j(t) \leq \tau$ for $j\in \{1,\dots,J\}$, $J \geq 0$ \cite{Hale1993}. In the following, we subdivide the modeling of neural DDEs as differential-difference equations into multiple steps. First, we discuss the discretization, then the choice of the initial data $u$, and finally the form of the delay functions $\tau_j$, $j\in \{1,\ldots,J\}$.

	\subsubsection*{Euler Discretization}
	
	In order to discretize the DDE \eqref{eq:DDEdifference}, we define the step size $\delta\coloneqq \frac{T}{L}$, where $T>0$ is the time up to which we solve the DDE, and $L\in\mathbb{N}$, such that $L+1$ is the number of layers the DenseResNet should have. Furthermore, we require that there exists a number $R\in\{0,1,2,\ldots\}$, such that $\tau = R\delta$, and we denote the grid times by $t_i\coloneqq i\delta$ with $i\in\{-R,\ldots,L\}$. Euler's explicit method computes iterates $y_i\in\mathbb{R}^m$ that approximate the solution values $y(t_i)$.

	The discretization can only correspond to the update rule~\eqref{eq:DenseResNet2} if the delayed functions $y(t-\tau_j(t))$, $j \in \{1,\ldots,J\}$, evaluated at the discrete time points $t_l = l \delta$, $l\in\{0,\ldots,L-1\}$, correspond to grid times $t_i$, $i\in\{-R,\ldots,l\}$, such that they can be replaced by iterates $y_i$, and if also the iterates with negative indices $y_i$, $i \in \{-R,\ldots,-1\}$ correspond to layers $h_l$, $l \in \{0,\ldots,L\}$. Hence, we make the following assumption on the discretization:
	\begin{assumption}[Discretization]\label{ass:discretization}
		There exist numbers $\alpha_{j,l} \in \{-R,\ldots,l\}$ such that
		\begin{equation*}
			t_l - \tau_j(t_l) = l\delta-\tau_j(l \delta) = \alpha_{j,l}\delta {}= t_{\alpha_{j,l}} \qquad \text{for all } j \in \{1,\ldots,J\} \text{ and } l \in \{0,\ldots,L-1\}.
		\end{equation*}
	\end{assumption}
	Under Assumption \ref{ass:discretization}, Euler's explicit method applied to the DDE \eqref{eq:DDEdifference} on the time interval $[-\tau,T]$ yields
	\begin{equation} \label{eq:DDEdifference1}
		\begin{aligned}
			y_{l+1} &= y_{l} + \delta f_\textup{DDE}(y_{l},y_{\alpha_{1,l}},\ldots,y_{\alpha_{J,l}},\theta_l), && l\in \{0,\ldots,L-1\}, \\
			y_{-r} &= u(t_{-r}), && r\in \{0,\ldots,R\},
		\end{aligned}
	\end{equation}
	where the iterate $y_{\alpha_{j,l}}$ takes the place of the delayed value $y(t_l-\tau_j(t_l)) = y(t_{\alpha_{j,l}})$ and $\theta_l \coloneqq \theta(t_l)\in\mathbb{R}^{p}$. To match the update rule \eqref{eq:DenseResNet2}, we define $h_l \coloneqq y_l \in \mathbb{R}^m$ for $l \in \{0,\ldots,L\}$ and use the samples $\theta_{\text{Dense},l} \coloneqq \theta_l \in \mathbb{R}^{p}$ of the parameter function as layer parameters for $l \in \{0,\ldots,L-1\}$, so that $p_l = p$.
	In Figure \ref{fig:discretization}, the discretization \eqref{eq:DDEdifference1} over the time interval $[-\tau,T]$ and  Assumption \ref{ass:discretization} are visualized. If Assumption \ref{ass:discretization} holds, we can define
	\begin{equation} \label{eq:vectorfield}
		f_{\text{Dense},l}(h_l,h_{l-1},\ldots,h_0,\theta_{\text{Dense},l}) \coloneqq \delta f_\text{DDE}(h_l,h_{\alpha_{1,l}},\ldots,h_{\alpha_{J,l}},\theta_{\text{Dense},l}),
	\end{equation}
	where coinciding entries are joined and reordered, and iterates with negative indices are identified with layers according to the choice of the initial data discussed below. In that way, the update rule \eqref{eq:DenseResNet2} of the DenseResNet is re-obtained from the discretization \eqref{eq:DDEdifference1} of the differential-difference equation~\eqref{eq:DDEdifference}. The procedure of equation~\eqref{eq:vectorfield} is illustrated in Example \ref{ex:discretization} at the end of this section. 
	
        \begin{figure}[h!]
        \vspace{2mm}
        \centering
        \resizebox{0.95\textwidth}{!}{%
        \begin{tikzpicture}[x=0.8cm, y=1.1cm, >={Stealth[length=2.0mm]}]
        \draw[->, line width=0.9pt] (-4,-1) -- (14.5,-1) node[right]{$t$};
        \draw[->, line width=0.9pt] (0,-1) -- (0,4.2) node[above]{$\;\;\R^{m}$};
        \foreach \x/\lbl in {-3/{t_{-R}}, -2/{t_{-2}}, -1/{t_{-1}},
                             1/{t_1}, 2/{t_2}, 3/{t_3}, 4/{t_4}}{
          \draw[line width=0.9pt] (\x,-0.9) -- (\x,-1.1);
          \node[below=2pt] at (\x,-1.1) {$\lbl$};}
        \node[below=2pt] at (6.5,-1.1) {$\dots$};
        \foreach \x/\lbl in {10/{t_{L-2}}, 11/{t_{L-1}}, 12/{t_L}}{
          \draw[line width=0.9pt] (\x,-0.9) -- (\x,-1.1);
          \node[below=2pt] at (\x,-1.1) {$\lbl$};}
        \node[above=2pt] at (-3,-0.95) {$-\tau=-3\delta$};
        \node[above=2pt] at (12,-0.95) {$T = L\delta$};
        \begin{scope}
          \clip (-3.5,-0.9) rectangle (0,4.2);
          \draw[ddeCyan, line width=1.6pt]
            plot[smooth, tension=0.7] coordinates
            {(-3,1.5)(-2,2.2)(-1,2.6)(0,2.4)(1,1.8)(2,1.2)(3,0.9)(4,1.0)(5,1.6)(6,2.5)(7,3.2)(8,3.4)(9,3.1)(10,2.5)(11,1.9)(12,1.6)};
        \end{scope}
        \begin{scope}
          \clip (0,-0.9) rectangle (12.5,4.2);
          \draw[ddeIndigo, line width=1.6pt]
            plot[smooth, tension=0.7] coordinates
            {(-3,1.5)(-2,2.2)(-1,2.6)(0,2.4)(1,1.8)(2,1.2)(3,0.9)(4,1.0)(5,1.6)(6,2.5)(7,3.2)(8,3.4)(9,3.1)(10,2.5)(11,1.9)(12,1.6)};
        \end{scope}
        \begin{scope}
          \clip (0,-0.9) rectangle (14,4.2);
          \draw[ddeIndigo, dotted, line width=1.6pt]
            plot[smooth, tension=0.7] coordinates
            {(0,1.5)(1,2.2)(2,2.6)(3,2.4)(4,1.8)(5,1.2)(6,0.9)(7,1.0)(8,1.6)(9,2.5)(10,3.2)(11,3.4)(12,3.1)(13,2.5)(14,1.9)(15,1.6)};
        \end{scope}
        \node[ddeCyan]   at (-2.3,2.75) {$u(t)$};
        \node[ddeIndigo] at (8,3.75)    {$y(t)$};
        \node[ddeIndigo] at (7.6,0.45)  {$y(t-\tau)$};
        \foreach \x/\y in {-3/1.5, 0/2.4, 1/1.8, 2/1.2, 3/0.9, 4/1.0, 10/2.5, 11/1.9, 12/1.6, 4/1.8}
          \fill[ddeGray] (\x,\y) circle (1.7pt);
        \node[ddeGray, anchor=east]  at (-3.15,1.5) {$y(t_{-R})$};
        \node[ddeGray, anchor=west]  at (0.05,2.75) {$y(t_0)$};
        \node[ddeGray]               at (0.9,1.2)   {$y(t_1)$};
        \node[ddeGray]               at (1.9,0.7)   {$y(t_2)$};
        \node[ddeGray]               at (3,0.45)    {$y(t_3)$};
        \node[ddeGray]               at (4.4,0.55)  {$y(t_4)$};
        \node[ddeGray, anchor=west]  at (10.45,2.75) {$y(t_{L-2})$};
        \node[ddeGray]               at (10.9,1.4)  {$y(t_{L-1})$};
        \node[ddeGray, anchor=west]  at (12.2,1.5)  {$y(t_L)$};
        \draw[<->, ddePurple, line width=0.8pt] (1,1.8) -- (4,1.8) node[midway, above]{$\tau$};
        \node[ddePurple] at (3.8,3.4) {$y(t_4-\tau_1(t_4))=y(t_1)$};
        \draw[->, ddePurple, line width=0.8pt] (3.95,3.15) -- (4.0,1.95);
        \draw[decorate, decoration={brace, amplitude=4pt}, ddePurple, line width=0.8pt]
          (-1,-0.75) -- (0,-0.75) node[midway, above=4pt, ddePurple]{$\delta$};
        \draw[decorate, decoration={brace, amplitude=4pt}, ddePurple, line width=0.8pt]
          (0,-0.75) -- (1,-0.75) node[midway, above=4pt, ddePurple]{$\delta$};
        \end{tikzpicture}%
        }
        \vspace{1mm}
        \caption{Discretization of the solution $y(t)$ of the neural DDE \eqref{eq:DDEdifference}, shown in light and dark blue, over the time interval $[-\tau,T]$. For a constant delay $\tau_1(t) = \tau = 3 \delta$, the delayed function $y(t-\tau_1(t))$ is shown in dotted blue. Assumption \ref{ass:discretization} is visualized for $l = 4$ in violet. It holds $\alpha_{1,4} = 1$, i.e., $y(t_4-\tau_1(t_4)) = y(\alpha_{1,4} \delta) = y(t_1)$. The solution values $y(t_i)$ at the grid times, which are approximated by the iterates $y_i$, are marked in gray.}\label{fig:discretization}          
    \end{figure}
	
	\subsubsection*{Choice of the Initial Data}
	
	A crucial property of the DenseResNet  \eqref{eq:DenseResNet2} is that the initial data is, independently of the length and the number of inter-layer connections, only the value of the first layer $h_0 \in \mathbb{R}^m$, i.e., no layers with negative indices exist. Transferred to the DDE \eqref{eq:DDEdifference}, we require that the initial data is again entirely given by the value $y(0) = u(0) = h_0$, corresponding to the first layer. In the following, we present two ways to model this property.
	
	A first option is to use the feature of DenseResNets that inter-layer connections of length~$i$ between the layers $h_l$ and $h_{l-i}$ can only occur for $i \leq l$. Transferred to the discretization~\eqref{eq:DDEdifference1}, we obtain the following assumption.
	\begin{customass}{2.2(a)}[One-Point Initial Data]\label{ass:initialdata}
		It holds 
		\begin{equation*}
			t_l - \tau_j(t_l) \geq 0 \qquad \text{for all } j \in \{1,\ldots,J\} \text{ and } l \in \{0,\ldots,L-1\},
		\end{equation*}
		such that every $y(t_l - \tau_j(t_l))$ corresponds to a layer $y_{\alpha_{j,l}} = h_{\alpha_{j,l}}$ with $\alpha_{j,l}\in\{0,\ldots,l\}$.
	\end{customass}
	The property of Assumption \ref{ass:initialdata} implies that iterates $y_i$ with negative indices $i \in \{-R,\ldots,-1\}$ are never used. For the discretization~\eqref{eq:DDEdifference1}, Assumption \ref{ass:initialdata} implies that only the initial value $u(0) = h_0$ is used, instead of the initial function $u$ on the interval $[-\tau,0]$. A second option, as proposed in \cite{Zhu2021}, is to use a constant initial function as initial data $u$.
	\begin{customass}{2.2(b)}[Constant Initial Data] \label{ass:constantInitialData}
		The DDE \eqref{eq:DDEdifference} uses the constant initial data 
		\begin{equation*}
			c_{h_0}:[-\tau,0] \rightarrow \mathbb{R}^m, \quad c_{h_0}(t) = h_0,
		\end{equation*}
		which is for every fixed $h_0\in \R^m$ a constant function in time.
	\end{customass}
	In that way, again, only the value $u(0) = h_0$ is used as an initial condition, but the initial function is defined on the whole interval $[-\tau,0]$ instead of at the point $t = 0$ only. For the discretization~\eqref{eq:DDEdifference1}, the constant initial function implies that $y_{-r} = h_0$ for $r\in\{0,\ldots,R\}$. In contrast to Assumption~\ref{ass:initialdata}, Assumption~\ref{ass:constantInitialData} allows for the usage of delay functions $\tau_j$ with $t_l - \tau_j(t_l) < 0$ for some $j \in \{1,\ldots,J\}$ and $l \in \{0,\ldots,L-1\}$, as all iterates $y_{\alpha_{j,l}}$ with $\alpha_{j,l}<0$ are equal to $h_0$.

	In both cases of Assumption~\ref{ass:initialdata} and Assumption~\ref{ass:constantInitialData}, the iterates $y_{\alpha_{j,l}}$ used for the delayed values are layers $h_0,\ldots,h_l$ for $j \in \{1,\ldots,J\}$ and $ l \in \{0,\ldots,L-1\}$, such that the discretization~\eqref{eq:DDEdifference1} has after the transformation~\eqref{eq:vectorfield} the same update rule as the DenseResNet \eqref{eq:DenseResNet2}. Moreover, for every discretization that satisfies Assumption~\ref{ass:initialdata}, one can, without loss of generality, choose constant initial data, as iterates with negative indices are never used. Hence, we always assume in the following that Assumption \ref{ass:constantInitialData} holds. Assumption \ref{ass:initialdata} is optional but supports the interpretation of DenseResNets as discretized neural DDEs. The results proven in this work also hold for fixed, continuous, non-constant initial data. The proofs then require slightly more work, but the main idea stays the same. To maintain the analogy to DenseResNets, we stick to the assumption that the initial data is constant in time.

	\subsubsection*{Choice of the Delay Functions}
	
	In the last modeling step, we now show constructively the existence of delay functions $\tau_j$, $j\in \{1,\ldots,J\}$, such that Assumption \ref{ass:discretization} is fulfilled. Furthermore, we give examples of delay functions, such that additionally, the optional Assumption \ref{ass:initialdata} is fulfilled. The three types of delay functions introduced in the following are shown in Figure \ref{fig:delayfunctions}.
	
	The first type of delay function $\tau_{j,A}$ we present corresponds to the idea that the distance of inter-layer connections is fixed. Every inter-layer connection of length $i$ between two layers $h_l$ and $h_{l-i}$ should correspond to a fixed delay of length $i\delta$. We define for $j \in \{1,\ldots,R\}$:
	\begin{equation*}
		\tau_{j,A}\in C^\infty([0,T],[0,\tau]), \qquad \tau_{j,A}(t) = j \delta.
	\end{equation*}
	As $j \leq R$, every delay function $\tau_{j,A}$ fulfills $0\leq \tau_{j,A}(t) \leq \tau = R\delta$, and Assumption \ref{ass:discretization} holds. The delay function $\tau_{j,A}$ can be modified to additionally fulfill Assumption \ref{ass:initialdata}. For that purpose, we define for $j \in \{1,\ldots,R\}$ the second type of delay function $\tau_{j,B}$:
	\begin{equation*}
		\tau_{j,B} \in C^\infty([0,T],[0,\tau]), \qquad \tau_{j,B}(t) = \gamma_{(j-1)\delta,j\delta}(t) \cdot j \delta =\begin{cases}
			0  &\text{ if } t \in [0,(j-1)\delta], \\
			j\delta  &\text{ if } t \in [j\delta,T].
		\end{cases}
	\end{equation*}
	Here, $\gamma_{r_1,r_2}:\mathbb{R}\rightarrow [0,1]$ is a smooth transition function with
	\begin{equation*}
		\gamma_{r_1,r_2} \in C^\infty(\mathbb{R},[0,1]), \qquad \gamma_{r_1,r_2}(t)  =\begin{cases}
			0  &\text{ if } t \in (-\infty,r_1], \\
			1  &\text{ if } t \in [r_2,\infty),
		\end{cases}
	\end{equation*}
	and $0<\gamma_{r_1,r_2}(t)<1$ for $t\in(r_1,r_2)$. The existence of such a smooth transition function is guaranteed by Lemma 2.21~in~\cite{Lee2013} via the partition of unity method. As $j \leq R$, the delay function $\tau_{j,B}$ fulfills $0\leq \tau_{j,B}(t) \leq \tau = R\delta$ and by construction also Assumption~\ref{ass:discretization} and Assumption~\ref{ass:initialdata} are fulfilled as $\tau_{j,B}(t_l)\leq t_l$ for all $l \in \{0,\ldots,L-1\}$, see also Figure \ref{fig:delayfunctions}.

	\begin{figure}[t]
		\centering
		\vspace{-2mm}
		\resizebox{0.9\textwidth}{!}{%
		\begin{tikzpicture}[x=0.7cm, y=0.7cm, >={Stealth[length=2.0mm]},
		    sparsedots/.style={dash pattern=on 0pt off 5pt, line cap=round}]
		\foreach \x in {1,2,4.3,5.3,6.3,7.8,8.8,13.1,14.1,17.1,18.1}
		  \draw[ddeGrayLight, line width=0.8pt] (\x,0) -- (\x,9.4);
		\foreach \y in {1,2,4.3,5.3,6.3,7.8,8.8}
		  \draw[ddeGrayLight, line width=0.8pt] (0,\y) -- (18.5,\y);
		\draw[->, black, line width=0.9pt] (0,0) -- (0,10.0) node[above]{$\tau(t)$};
		\draw[->, black, line width=0.9pt] (0,0) -- (18.8,0) node[right]{$t$};
		\draw[ddeGray, line width=0.8pt] (0,0) -- (8.8,8.8);
		\node[ddeGray] at (4.6,6.0) {$\textup{id}_{\R}$};
		\draw[ddeTeal, line width=1.6pt] (0,5.33) -- (18.1,5.33);
		\node[ddeTeal] at (2,5.9) {$\tau_{j,A}$};
		\draw[ddePurple, line width=1.6pt] (0,0) -- (4.3,0);
		\draw[ddePurple, line width=1.6pt, sparsedots]
		  (4.3,0) .. controls (4.6,0) and (5.0,5.25) .. (5.3,5.25);
		\draw[ddePurple, line width=1.6pt] (5.3,5.25) -- (18.1,5.25);
		\node[ddePurple] at (12.5,4.5) {$\tau_{j,B}$};
		\draw[ddeOchre, line width=1.6pt] (0,0) -- (5.3,0);
		\draw[ddeOchre, line width=1.6pt, sparsedots]
		  (5.3,0) .. controls (5.7,0) and (5.9,0.6) .. (6.3,1);
		\draw[ddeOchre, line width=1.6pt] (6.3,1) -- (13.1,7.8);
		\draw[ddeOchre, line width=1.6pt, sparsedots]
		  (13.1,7.8) .. controls (13.6,8.3) and (13.7,8.8) .. (14.1,8.8);
		\draw[ddeOchre, line width=1.6pt] (14.1,8.8) -- (18.1,8.8);
		\node[ddeOchre] at (16.5,9.25) {$\tau_{j,C}$};
		\foreach \x/\lbl in {1/{\delta}, 2/{2\delta}, 4.3/{(j-1)\delta}, 5.3/{j\delta},
		                     6.3/{(j+1)\delta}, 7.8/{\tau-\delta}, 8.8/{\tau},
		                     13.1/{(j-1)\delta+\tau}, 14.1/{j\delta+\tau}, 17.1/{T-\delta}, 18.1/{T}}
		  \node[anchor=north east, rotate=45] at (\x,-0.1) {$\lbl$};
		\foreach \y/\lbl in {1/{\delta}, 2/{2\delta}, 4.3/{(j-1)\delta}, 5.3/{j\delta},
		                     6.3/{(j+1)\delta}, 7.8/{\tau-\delta}, 8.8/{\tau=R\delta}}
		  \node[left] at (0,\y) {$\lbl$};
		\end{tikzpicture}%
		}
		\vspace{-2mm}
		\caption{Visualization of the three types of delay functions $\tau_{j,A}$, $\tau_{j,B},$ $\tau_{j,C}:[0,T]\rightarrow [0,\tau]$. $\tau_{j,A}$ and~$\tau_{j,B}$ are defined for $j \in \{1,\ldots,R\}$ and $\tau_{j,C}$ is defined for $j\in\{0,\ldots,L-1\}$. At every grid time~$l\delta$, all delay functions take values on the grid, as required by Assumption~\ref{ass:discretization}. As $\tau_{j,B}$ and $\tau_{j,C}$ satisfy Assumption~\ref{ass:initialdata}, they are at the grid times $l\delta$, $l\in\{0,\ldots,L\}$, not larger than the identity function shown in gray.} \label{fig:delayfunctions}		
	\end{figure}

	The third type of delay function corresponds to the idea that a DenseResNet with update rule~\eqref{eq:DenseResNet2} can refer at step $t_l$ to every previous layer $h_i$, $i\in \{0,\ldots,l\}$, i.e., to every previous iterate $y_i$ of the discretization~\eqref{eq:DDEdifference1}. If at step $t_l$, the function $f_{\text{Dense},l}$ accesses the layer $h_i$, the inter-layer connection has length $l-i \geq 0$, which corresponds to a delay of $(l-i)\delta$ in the DDE \eqref{eq:DDEdifference}. The length of the maximal inter-layer connection is $R$, corresponding to the maximal delay $\tau = R\delta$. Under the assumption that $R\geq 3$, we define for $j \in \{0,\ldots,L-1\}$ the following smooth delay functions:
	\begin{equation*}
		\tau_{j,C} \in C^\infty([0,T],[0,\tau]), \qquad \tau_{j,C}(t) =\begin{cases}
			\tau_{j,C1}(t)  &\text{ if } t \in [0,(j+1) \delta], \\
			\tau_{j,C2}(t) &\text{ if } t \in [(j+1) \delta,T],
		\end{cases}
	\end{equation*}
	with 
	\begin{align*}
		\tau_{j,C1}\in C^\infty(\mathbb{R},\mathbb{R}), \qquad \tau_{j,C1}(t) &= \gamma_{j\delta,(j+1)\delta}(t) \cdot (t-j\delta), \\
		\tau_{j,C2}\in C^\infty(\mathbb{R},\mathbb{R}), \qquad \tau_{j,C2}(t) &= (1-\gamma_{(j-1)\delta+\tau,j\delta+\tau}(t)) \cdot (t-j\delta-\tau) + \tau.
	\end{align*}
	The function $\tau_{j,C}$ is smooth, as $\tau_{j,C1}$ and $\tau_{j,C2}$ are both smooth functions, which coincide on an interval of positive length, as
	\begin{align*}
		\tau_{j,C1}(t) = t-j\delta = t-j\delta-\tau+\tau = \tau_{j,C2}(t) \qquad \text{for } t \in [(j+1)\delta,(j-1)\delta+\tau],
	\end{align*}
	with interval length $\tau-2\delta = (R-2)\delta \geq \delta > 0$. On the intervals, where the functions $\gamma_{r_1,r_2}$ are constant, we can simplify the functions $\tau_{j,C}$ and obtain
	\begin{align*}
		\tau_{j,C}(t)&=\begin{cases}
			0  &\text{ if } t \in [0,j\delta], \\
			t-j\delta  &\text{ if } t \in [(j+1)\delta,\min\{(j-1)\delta+\tau,T\}], \\
			\tau &\text{ if } t \in [j\delta+\tau,T].
		\end{cases}
	\end{align*}
	If $j > L-R$, the interval $[j\delta+\tau,T]$ is empty and the last case does not occur. As $0 \leq \gamma_{r_1,r_2}(t)\leq 1$ for all $t\in \mathbb{R}$, every function $\tau_{j,C}$ fulfills $0\leq \tau_{j,C}(t)\leq \tau$ for $t\in[0,T]$ and $j\in\{0,\ldots,L-1\}$, such that it is a delay function with maximal delay $\tau$. Furthermore, all $\tau_{j,C}$ fulfill Assumption \ref{ass:discretization} and Assumption~\ref{ass:initialdata}, see also Figure~\ref{fig:delayfunctions}. \medskip

	Altogether, we have shown that for constant initial data on the interval $[-\tau,0]$, there exist smooth delay functions, such that an Euler discretization of the differential-difference equation \eqref{eq:DDEdifference} results in a DenseResNet with update rule \eqref{eq:DenseResNet2}. To illustrate the discretization and the conditions stated, we discuss in the following a one-dimensional example with all three types of delay functions.
	
	\addtocounter{theorem}{1}

	\begin{figure}[b!]
		\centering
		\resizebox{0.9\textwidth}{!}{%
		\begin{tikzpicture}[x=1.15cm, y=1.15cm, >={Stealth[length=2.0mm]},
		    sparsedots/.style={dash pattern=on 0pt off 5pt, line cap=round}]
		\foreach \x in {1,...,8}
		  \draw[ddeGrayLight, line width=0.8pt] (\x,0) -- (\x,4.3);
		\foreach \y in {1,2,3}
		  \draw[ddeGrayLight, line width=0.8pt] (0,\y) -- (8.7,\y);
		\draw[->, black, line width=0.9pt] (0,0) -- (0,4.8) node[above]{$\tau(t)$};
		\draw[->, black, line width=0.9pt] (0,0) -- (9.4,0) node[right]{$t$};
		\draw[ddeGray, line width=0.8pt] (0,0) -- (4.3,4.3);
		\node[ddeGray] at (3.5,4.1) {$\textup{id}_{\R}$};
		\draw[ddeTeal, line width=1.6pt] (0,1) -- (8.7,1);
		\node[ddeTeal, right] at (8.7,1) {$\tau_{1,A}$};
		\draw[ddePurple, line width=1.6pt] (0,0) -- (1,0);
		\draw[ddePurple, line width=1.6pt, sparsedots]
		  (1,0) .. controls (1.3,0) and (1.8,2) .. (2,2);
		\draw[ddePurple, line width=1.6pt] (2,2) -- (8.7,2);
		\node[ddePurple, right] at (8.7,2) {$\tau_{2,B}$};
		\draw[ddeOchre, line width=1.6pt] (0,0) -- (1,0);
		\draw[ddeOchre, line width=1.6pt, sparsedots]
		  (1,0) .. controls (1.4,0) and (1.6,0.6) .. (2,1);
		\draw[ddeOchre, line width=1.6pt] (2,1) -- (3,2);
		\draw[ddeOchre, line width=1.6pt, sparsedots]
		  (3,2) .. controls (3.5,2.5) and (3.7,3) .. (4,3);
		\draw[ddeOchre, line width=1.6pt] (4,3) -- (8.7,3);
		\node[ddeOchre, right] at (8.7,3) {$\tau_{1,C}$};
		\foreach \x/\lbl in {1/{\delta},2/{2\delta},3/{3\delta},4/{4\delta},
		                     5/{5\delta},6/{6\delta},7/{7\delta},8/{8\delta}}
		  \node[below] at (\x,0) {$\lbl$};
		\node[left] at (0,1) {$\delta$};
		\node[left] at (0,2) {$2\delta$};
		\node[left] at (0,3) {$\tau=3\delta$};
		\node[anchor=east] at (-0.3,-1.0) {current state:};
		\node[anchor=east, ddeTeal] at (-0.3,-1.8) {delay $\tau_{1,A}$:};
		\node[anchor=east, ddePurple] at (-0.3,-2.6) {delay $\tau_{2,B}$:};
		\node[anchor=east, ddeOchre] at (-0.3,-3.4) {delay $\tau_{1,C}$:};
		\foreach \x/\val in {0/{y_0},1/{y_1},2/{y_2},3/{y_3},4/{y_4},
		                     5/{y_5},6/{y_6},7/{y_7},8/{y_8}}
		  \node at (\x,-1.0) {$\val$};
		\node at (9.4,-1.0) {$\dots$};
		\foreach \x/\val in {0/{y_{-1}},1/{y_0},2/{y_1},3/{y_2},4/{y_3},
		                     5/{y_4},6/{y_5},7/{y_6},8/{y_7}}
		  \node[ddeTeal] at (\x,-1.8) {$\val$};
		\node[ddeTeal] at (9.4,-1.8) {$\dots$};
		\foreach \x/\val in {0/{y_0},1/{y_1},2/{y_0},3/{y_1},4/{y_2},
		                     5/{y_3},6/{y_4},7/{y_5},8/{y_6}}
		  \node[ddePurple] at (\x,-2.6) {$\val$};
		\node[ddePurple] at (9.4,-2.6) {$\dots$};
		\foreach \x/\val in {0/{y_0},1/{y_1},2/{y_1},3/{y_1},4/{y_1},
		                     5/{y_2},6/{y_3},7/{y_4},8/{y_5}}
		  \node[ddeOchre] at (\x,-3.4) {$\val$};
		\node[ddeOchre] at (9.4,-3.4) {$\dots$};
		\end{tikzpicture}%
		}
		\vspace{3mm}
		\caption{Visualization of the delay functions $\tau_{1,A}$, $\tau_{2,B}$ and $\tau_{1,C}$ of Example \ref{ex:discretization}. Below the graphs, the iterates used for the current state and for the three delayed values in the Euler step at the grid time $t_l = l\delta$ are listed.}\label{fig:example_discretization}
	\end{figure}

	\begin{example}[DenseResNet as Discretized Neural DDE]\label{ex:discretization}
		\normalfont Let $\delta>0$, $L\geq 3$, and $T = L\delta$, and consider the one-dimensional differential-difference equation
		\begin{equation*}
			\begin{aligned}
				\frac{\dd y}{\dd t} &= f_\textup{DDE}(y(t),y(t-\tau_{1,A}(t)),y(t-\tau_{2,B}(t)),y(t-\tau_{1,C}(t)),\theta(t)) \qquad &&\text{for } t \in [0,T], \\  
				y(t) &= c_{h_0}(t) &&\text{for } t \in [-\tau,0],
			\end{aligned}
		\end{equation*}
		with delay $\tau = 3 \delta$, parameter function $\theta:[0,T]\rightarrow \mathbb{R}^{p}$, and constant initial data $c_{h_0}$ for some $h_0\in \mathbb{R}$. The three chosen delay functions are visualized in Figure \ref{fig:example_discretization} and satisfy Assumption~\ref{ass:discretization}. Euler's explicit method with step size $\delta$ on the time interval $[-\tau,T]$ yields
		\begin{equation*}
			\begin{aligned}
				y_{l+1} &= y_{l} + \delta f_\textup{DDE}(y_{l},y_{\alpha_{1,l}},y_{\alpha_{2,l}},y_{\alpha_{3,l}},\theta_l), && l\in \{0,\ldots,L-1\}, \\
				y_{-r} &= h_0, && r\in \{0,\ldots,3\},
			\end{aligned}
		\end{equation*}
		where $t_l = l\delta$, the iterates $y_{\alpha_{1,l}}$, $y_{\alpha_{2,l}}$, and $y_{\alpha_{3,l}}$ take the place of the delayed values at the times $t_l-\tau_{1,A}(t_l)$, $t_l-\tau_{2,B}(t_l)$, and $t_l-\tau_{1,C}(t_l)$, and $\theta_l \coloneqq \theta(t_l){}\in\mathbb{R}^p$. With the layers $h_l \coloneqq y_l$ for $l \in \{0,\ldots,L\}$, the discretization becomes the update rule of the DenseResNet
		\begin{equation*}
			h_{l+1} = h_{l} + f_{\textup{Dense},l}(h_{l}, h_{l-1}, \ldots, h_0, \theta_{\textup{Dense},l}),
		\end{equation*}
		for $l \in \{0,\ldots,L-1\}$ with $\theta_{\textup{Dense},l} \coloneqq \theta_l \in \mathbb{R}^p$. The functions $f_{\textup{Dense},l}$ arise from $\delta f_\textup{DDE}$ by reordering and merging coinciding entries, cf. equation~\eqref{eq:vectorfield}. For the given delay functions, it holds explicitly
		\begin{align*}
			f_{\text{Dense},0}(h_0,\theta_{\text{Dense},0}) &= \delta f_\text{DDE}(h_0,h_0,h_0,h_0,\theta_0), \\
			f_{\text{Dense},1}(h_1,h_0,\theta_{\text{Dense},1}) &= \delta f_\text{DDE}(h_1,h_0,h_1,h_1,\theta_1), \\
			f_{\text{Dense},2}(h_2,h_1,h_0,\theta_{\text{Dense},2}) &= \delta f_\text{DDE}(h_2,h_1,h_0,h_1,\theta_2),\\
			f_{\text{Dense},3}(h_3,\ldots,h_0,\theta_{\text{Dense},3}) &= \delta f_\text{DDE}(h_3,h_2,h_1,h_1,\theta_3), \\
			f_{\text{Dense},l}(h_l,\ldots,h_0,\theta_{\text{Dense},l}) &= \delta f_\text{DDE}(h_l,h_{l-1},h_{l-2},h_{l-3},\theta_l),
		\end{align*}
		for $l\geq 4$. The iterates used for the three delayed values at the grid times $t_l = l\delta$ are listed in Figure~\ref{fig:example_discretization}.
	\end{example}

	\subsection{Neural DDE Architectures}
	\label{sec:architecture}
	
	In this section, we introduce the general neural DDE architecture we study in this work and explain the relationship to the differential-difference equation \eqref{eq:DDEdifference} of the last section. We have shown that there exist choices of the vector field $f_\textup{DDE}: \mathcal{C} \times \mathbb{R}^{p}\rightarrow \mathbb{R}^m$, such that an Euler discretization of the~DDE
	\begin{equation} \label{eq:DDEparametrized}
		\begin{aligned}
			\frac{\dd y}{\dd t} &= f_\textup{DDE}(y_t,\theta(t)) \qquad &&\text{for } t\geq 0, \\  
			y(t) &= c_{y_0}(t) &&\text{for } t \in [-\tau,0],
		\end{aligned}
	\end{equation} 
	over the time interval $[0,T]$ with delay $\tau \geq 0$ and constant initial data $c_{y_0}:[-\tau,0]\rightarrow \mathbb{R}^m$, $c_{y_0}(t) = y_0$ leads to the update rule of a DenseResNet. The parameter function $\theta:[0,T]\rightarrow \mathbb{R}^{p}$ depends explicitly on the time $t$, making the vector field  $f_\text{DDE}$ non-autonomous. As $\theta$ is only defined on $[0,T]$, we extend it constantly by $\theta(t)\coloneqq\theta(0)$ for $t<0$ and $\theta(t)\coloneqq\theta(T)$ for $t>T$, so that $f_\textup{DDE}(y_t,\theta(t))$ is defined for all $t\in\mathbb{R}$. The extension does not change the solution on $[0,T]$. To treat all possible parameterizations of the vector field $f_\text{DDE}$ at once, we study in this work neural DDE architectures, which are based on the  general non-autonomous DDE
	\begin{equation} \label{eq:DDEconstantIC}
		\begin{aligned}
			\frac{\dd y}{\dd t} &= F(t,y_t) \qquad &&\text{for } t\geq 0, \\  
			y(t) &= c_{y_0}(t) &&\text{for } t \in [-\tau,0],
		\end{aligned}
	\end{equation}
	with  $F:\Omega\rightarrow \mathbb{R}^m$, $\Omega = \Omega_t \times \Omega_y \subset \mathbb{R}\times \mathcal{C}$ open, $[0,T]\subset \Omega_t$, and constant initial data $c_{y_0}\in \Omega_y$. In the following Assumption \ref{ass:vectorfield}  and Definition \ref{def:NDDE}, additional requirements on the vector field $F$ and the subsets $\Omega_t$, $\Omega_y$ are made in order to have a well-defined neural DDE. The relationship between the parameterized DDE \eqref{eq:DDEparametrized} and the general DDE \eqref{eq:DDEconstantIC} is discussed in Section \ref{sec:parameterized}. In the general DDE~\eqref{eq:DDEconstantIC}, we have the flexibility to choose the vector field $F$, whereas in the parameterized DDE~\eqref{eq:DDEparametrized}, the vector field $f_\text{DDE}$ is fixed and the parameter function $\theta: [0,T] \rightarrow \R^p$ can be chosen. 
	
	If there exists a constant $c\in(0,\infty]$, such that $y:[-\tau,c)\rightarrow \mathbb{R}^m$ is continuous and fulfills~\eqref{eq:DDEconstantIC} for $t\in[-\tau,c)$, then $y$ is called a solution of the DDE~\eqref{eq:DDEconstantIC}. The maximal time interval of existence $[-\tau,c)$ for the DDE~\eqref{eq:DDEconstantIC} with constant initial data $c_{y_0}$ is denoted by $\mathcal{I}_{y_0}$. The function $y_t$, which corresponds to the solution $y$ of \eqref{eq:DDEconstantIC}, is defined by $y_t(s) \coloneqq y(t+s)$ for $s\in[-\tau,0]$ and is an element of the infinite-dimensional state space $\mathcal{C}$.  The norm of an element $y_t$ of the Banach space $\mathcal{C}$ is denoted~by 
	\begin{equation*}
		\norm{y_t}_\infty \coloneqq \sup_{s\in[-\tau,0]} \norm{y_t(s)}_\infty.
	\end{equation*}
	In general, we denote the state of the solution $y$ at time $t$ with initial function~$u$ on the time interval $[t_0-\tau,t_0]$ by $y(t_0,u)(t)$. Hence, the solution of the DDE \eqref{eq:DDEconstantIC} with initial data~$c_{y_0}$ on the time interval $[-\tau,0]$ is denoted by $y(0,c_{y_0})(t)$ for $t \in \mathcal{I}_{y_0}$.
	
	ResNets and DenseResNets introduced in \eqref{eq:ResNet} and \eqref{eq:DenseResNet} have the property that all layers have the same dimension $m$. General neural networks are not restricted to learning functions with the same input and output dimensions. To be flexible with respect to the input and output dimensions of the neural DDE, we add one affine linear layer before and one affine linear layer after the solution map of the DDE. Under the assumptions specified in the following, we define the general neural DDE architecture
	\begin{equation}\label{eq:DDEaffinelinear}
		\Phi: \mathcal{X} \rightarrow \mathbb{R}^q, \quad x \mapsto \tilde{\lambda}(y(0,c_{\lambda(x)})(T)) = \widetilde{W} \cdot y(0,c_{Wx+b})(T) + \tilde{b} ,
	\end{equation}
	with $\mathcal{X}\subset \mathbb{R}^n$ open, affine linear layers $\lambda:\mathbb{R}^n\rightarrow \mathbb{R}^m$ and $\tilde \lambda:\mathbb{R}^m\rightarrow \mathbb{R}^q$ and $y$ being the solution of the DDE \eqref{eq:DDEconstantIC} existing on the time interval $[0,T]$.  The map $\lambda: \mathbb{R}^n \rightarrow \mathbb{R}^m$ is represented by a matrix $W \in \mathbb{R}^{m \times n}$ and a vector $b \in \mathbb{R}^m$, such that $\lambda(x) = Wx+b$, and the map $\widetilde{\lambda}: \mathbb{R}^m\rightarrow \mathbb{R}^q$ is represented by a matrix $\widetilde{W} \in \mathbb{R}^{q\times m}$ and a vector $\tilde{b} \in \mathbb{R}^q$, such that $\tilde\lambda(x) = \widetilde{W}x+\tilde b$. In analogy to classical feed-forward neural networks such as MLPs, the matrices $W$ and $\widetilde{W}$ are called weight matrices, and the vectors $b$ and $\tilde{b}$ biases. 
	The neural DDE architecture uses the activated state $a \coloneqq \lambda(x)$, calculated as the affine linear transformation $\lambda$ of the initial value $x\in \mathcal{X}$, to define the initial data $c_a = c_{\lambda(x)}$ for the DDE \eqref{eq:DDEconstantIC}, which is constant in time. We abbreviate the set of all activated states $a$ by $\mathcal{A}\coloneqq \lambda(\mathcal{X})$. The output $\Phi(x)$ of the neural DDE architecture is calculated by applying the second affine linear transformation $\tilde \lambda$ to the time-$T$ map of the solution $y(0,c_{\lambda(x)})$, denoted by  $y(0,c_{\lambda(x)})(T)$. In principle, we can also use non-constant initial data $u:[-\tau,0]\rightarrow \mathbb{R}^m$ for the DDE \eqref{eq:DDEconstantIC} if it is uniquely specified by the value $\lambda(x)\in\mathcal{A}$ and it holds $u(0) = \lambda(x)$. To maintain the relationship to DenseResNets explained in the last section, we restrict the upcoming analysis to neural DDEs defined by \eqref{eq:DDEaffinelinear} with constant initial data. Nevertheless, our results can be extended to fixed, non-constant, continuous initial data with $u(0) = \lambda(x)$. 
	
	In the following, we specify the necessary assumptions, such that the neural DDE in~\eqref{eq:DDEaffinelinear} is well-defined.
	
	\begin{assumption}[Neural DDE Well-Definedness]\label{ass:vectorfield}
		For the neural DDE architecture \eqref{eq:DDEaffinelinear} based on the DDE \eqref{eq:DDEconstantIC} it holds that $F:\Omega\rightarrow\mathbb{R}^m$ is defined on $\Omega = \Omega_t \times \Omega_y \subset \mathbb{R}\times \mathcal{C}$ open, $[0,T]\subset \Omega_t $, and $\Omega_0 \coloneqq \left\{c_{y_0}\in \mathcal{C}: y_0 \in \mathcal{A}\right\} \subset \Omega_y$, where $\mathcal{A} = \lambda(\mathcal{X})$. Furthermore:
		\begin{enumerate}[label=(\alph*), font=\normalfont]
			\item \label{ass:vectorfield_a} \emph{Existence}: for every $y_0 \in \mathcal{A}$, there exists a solution $y(0,c_{y_0})$ of \eqref{eq:DDEconstantIC} on $[-\tau,T]$, i.e., $[-\tau,T]\subset\mathcal{I}_{y_0}$, with $(t,y_t)\in\Omega$ for all $t\in[0,T]$.
			\item \label{ass:vectorfield_b} \emph{Uniqueness}: for every $y_0 \in \mathcal{A}$, the solution $y(0,c_{y_0})$ on $[-\tau,T]$ is unique.
			\item \label{ass:vectorfield_c} \emph{Continuous dependence}: for each fixed $t \in [0,T]$, the map $y_0 \mapsto y(0,c_{y_0})(t)$ is continuous on $\mathcal{A}$.
		\end{enumerate}
	\end{assumption}
	
	The regularity of the neural DDE \eqref{eq:DDEaffinelinear} follows from the regularity of the underlying vector field~$F$. In the following, we state the classical result on local existence, uniqueness, and differentiability of solutions of the general DDE
	\begin{equation} \label{eq:DDEgeneral}
		\begin{aligned}
			\frac{\dd y}{\dd t} &= F(t,y_t) \qquad &&\text{for } t\geq t_0, \\  
			y_{t_0}(t) &= u(t) &&\text{for } t \in [-\tau,0],
		\end{aligned}
	\end{equation}
	with vector field $F:\Omega\rightarrow \mathbb{R}^m$, $\Omega \subset \mathbb{R}\times \mathcal{C}$ open and initial data $u \in \mathcal{C}$. In contrast to previously defined DDEs, we state the results for an arbitrary initial time $t_0$, such that the initial data $u \in \mathcal{C}$ defines the state of the function $y_{t_0}(t)$ for $t\in[-\tau,0]$, i.e., the state of $y$ on the time interval $[t_0-\tau,t_0]$. To state the following result, we define the space $C^{0,k}_b(\Omega,\mathbb{R}^m)$ of functions \mbox{$F:\Omega\rightarrow \mathbb{R}^m$}, $\Omega \subset \mathbb{R}\times \mathcal{C}$, such that~$F$ and its Fréchet derivatives with respect to the second argument up to order $k$ exist, are jointly continuous on $\Omega$, and are bounded on $\Omega$. The space $C^{0,k}_b(\Omega,\mathbb{R}^m)$, equipped with the norm~$\norm{\cdot}_k$, which takes the supremum over all derivatives up to order~$k$, becomes a Banach space for every $k \geq 0$~\cite{Hale1993}.

	\begin{lemma}[DDE Local Existence, Uniqueness, Differentiability \cite{Hale1993}] \label{lem:dde_regularity}
		Let $F\in C^{0,k}_b(\Omega,\mathbb{R}^m)$ with $\Omega \subset \mathbb{R}\times \mathcal{C}$ open and $k \geq 1$. Then for each initial data $(t_0,u)\in\Omega$, there exists a unique solution $y(t_0,u): [t_0-\tau,t_0+c)\rightarrow \mathbb{R}^m$ of \eqref{eq:DDEgeneral} for some $c\in(0,\infty]$. Furthermore, $y(t_0,u)(t)$ is $k$-times continuously differentiable with respect to $u$ for $t$ in any compact set in the domain of definition of $y(t_0,u)$.
	\end{lemma}
	
	Under Assumption \ref{ass:vectorfield}, we can deduce from the last lemma that the neural DDE \eqref{eq:DDEaffinelinear} is $k$-times continuously differentiable, i.e., an element of the space $C^k(\mathcal{X},\mathbb{R}^q)$, if the vector field of the DDE is $k$-times continuously differentiable with respect to its second variable. 
	
	\begin{lemma}[Neural DDE Regularity]\label{lem:ndde_regularity}
		Let $F\in C^{0,k}_b(\Omega,\mathbb{R}^m)$ with $\Omega = \Omega_t \times \Omega_y \subset \mathbb{R}\times \mathcal{C}$ open and $k \geq 1$, and let Assumption~\ref{ass:vectorfield}\ref{ass:vectorfield_a} hold. Then Assumption~\ref{ass:vectorfield}\ref{ass:vectorfield_b} and Assumption~\ref{ass:vectorfield}\ref{ass:vectorfield_c} follow, and it holds for the neural DDE $\Phi$ defined by \eqref{eq:DDEaffinelinear} that $\Phi\in C^k(\mathcal{X},\mathbb{R}^q)$.
	\end{lemma}

	\begin{proof}
		By Lemma~\ref{lem:dde_regularity}, the solution $y(0,c_{y_0})$ is unique for every $y_0\in\mathcal{A}$, which shows Assumption~\ref{ass:vectorfield}\ref{ass:vectorfield_b}. Fix $t\in[0,T]$ and $y_0\in\mathcal{A}$. By Assumption~\ref{ass:vectorfield}\ref{ass:vectorfield_a}, the time $t$ lies in the domain of definition of $y(0,c_{y_0})$, so Lemma~\ref{lem:dde_regularity} implies that the solution map $u \mapsto y(0,u)(t)$ is defined and $k$-times continuously differentiable on a neighborhood of $c_{y_0}$ in $\Omega_y$. The map
		\begin{equation*}
			\iota: \mathbb{R}^m \rightarrow \mathcal{C}, \quad y_0 \mapsto c_{y_0},
		\end{equation*}
		is linear and bounded, as $\norm{c_{y_0}}_\infty = \norm{y_0}_\infty$, and therefore smooth. By Assumption~\ref{ass:vectorfield}, it holds $\iota(\mathcal{A}) = \Omega_0 \subset \Omega_y$, so by the chain rule the map $y_0 \mapsto y(0,c_{y_0})(t)$ is $k$-times continuously differentiable on a neighborhood of each point of $\mathcal{A}$ in $\mathbb{R}^m$, in particular continuous on $\mathcal{A}$, which shows Assumption~\ref{ass:vectorfield}\ref{ass:vectorfield_c}. For $t = T$, as both affine linear layers $\lambda$ and $\tilde \lambda$ are smooth, the chain rule implies that $\Phi$ is $k$-times continuously differentiable with respect to the initial data $x \in \mathcal{X}$.
	\end{proof}
	
	After having analyzed the regularity of the neural DDE architecture \eqref{eq:DDEaffinelinear}, we can define the classes $\textup{NDDE}_\tau^k(\mathcal{X},\mathbb{R}^q)$ of neural DDEs we study in this work. We subdivide these classes into non-augmented and augmented architectures, which are visualized in Figure \ref{fig:ndde_architectures} and defined in the following.
	
	\begin{figure}[b!]
		\centering
		\vspace{3mm}
		\begin{subfigure}{0.9\textwidth}
			\centering
			\begin{tikzpicture}[nn,
			                nnell/.style={ellipse, fill=ddeGrayLight, draw=ddeGray, line width=0.8pt, minimum width=9mm, minimum height=5mm, inner sep=0pt},
			                oe/.style={->, dotted, very thick, ddeGrayDark, shorten >=1.5pt, shorten <=1.5pt}, scale = 0.9]
			            \foreach \i/\y in {1/1.8, 2/0.6, 3/-0.6, 4/-1.8} \node[nnin]   (I\i) at (0,\y)   {};
			            \foreach \i/\y in {1/1.2, 2/0, 3/-1.2}           \node[nnmain] (A\i) at (1.8,\y) {};
			            \foreach \i/\y in {1/1.2, 2/0, 3/-1.2}           \node[nnell]  (C\i) at (3.6,\y) {};
			            \foreach \i/\y in {1/1.2, 2/0, 3/-1.2}           \node[nnmain] (Y\i) at (9.0,\y) {};
			            \node[nnout] (O1) at (10.8,0) {};
			            \foreach \i in {1,...,4}\foreach \j in {1,2,3} \draw[nnedge] (I\i) -- (A\j);
			            \foreach \i in {1,2,3} \draw[nnedge] (Y\i) -- (O1);
			            \draw[oe] (C1) -- (Y1);
			            \draw[oe] (C3) -- (Y3);
			            \node at (0.9,2.15) {$\lambda$};
			            \node at (9.9,1.55) {$\tilde{\lambda}$};
			            \node at (6.3,0.55) {\textbf{DDE}};
			            \node[nnhdr, draw=ddeGray, rounded corners=2pt, inner sep=5pt, font=\normalsize] at (6.3,-0.55) {$n \geq m$};
			            \node at (0,-2.7)    {$x \in \mathcal{X}$};
			            \node at (1.8,-2.7)  {$a \in \mathcal{A}$};
			            \node at (3.6,-2.7)  {$c_a \in \Omega_0$};
			            \node at (9.0,-2.7)  {$y(0,c_a)(T)$};
			            \node at (10.8,-2.7) {$\Phi(x)$};
			        \end{tikzpicture}
			\caption{Structure of a non-augmented neural DDE $\Phi\in \textup{NDDE}^k_{\tau,\textup{N}}(\mathcal{X},\mathbb{R}^q)$. The phase space of the DDE is non-augmented, i.e., it holds $m\leq n$. The visualized example has $n = 4$, $m = 3$, and $q = 1$.}
			\label{fig:ndde_nonaugmented}
		\end{subfigure}
		\begin{subfigure}{0.9\textwidth}
			\centering
			\vspace{5mm}
			\begin{tikzpicture}[nn,
			                nnell/.style={ellipse, fill=ddeGrayLight, draw=ddeGray, line width=0.8pt, minimum width=9mm, minimum height=5mm, inner sep=0pt},
			                oe/.style={->, dotted, very thick, ddeGrayDark, shorten >=1.5pt, shorten <=1.5pt}, scale = 0.9]
			            \foreach \i/\y in {1/1.2, 2/0, 3/-1.2}           \node[nnin]   (I\i) at (0,\y)   {};
			            \foreach \i/\y in {1/1.8, 2/0.6, 3/-0.6, 4/-1.8} \node[nnmain] (A\i) at (1.8,\y) {};
			            \foreach \i/\y in {1/1.8, 2/0.6, 3/-0.6, 4/-1.8} \node[nnell]  (C\i) at (3.6,\y) {};
			            \foreach \i/\y in {1/1.8, 2/0.6, 3/-0.6, 4/-1.8} \node[nnmain] (Y\i) at (9.0,\y) {};
			            \foreach \i/\y in {1/0.6, 2/-0.6}                \node[nnout]  (O\i) at (10.8,\y) {};
			            \foreach \i in {1,2,3}\foreach \j in {1,...,4} \draw[nnedge] (I\i) -- (A\j);
			            \foreach \i in {1,...,4}\foreach \j in {1,2} \draw[nnedge] (Y\i) -- (O\j);
			            \draw[oe] (C1) -- (Y1);
			            \draw[oe] (C4) -- (Y4);
			            \node at (0.9,2.15) {$\lambda$};
			            \node at (9.9,2.15) {$\tilde{\lambda}$};
			            \node at (6.3,0.55) {\textbf{DDE}};
			            \node[nnhdr, draw=ddeGray, rounded corners=2pt, inner sep=5pt, font=\normalsize] at (6.3,-0.55) {$n < m$};
			            \node at (0,-2.7)    {$x \in \mathcal{X}$};
			            \node at (1.8,-2.7)  {$a \in \mathcal{A}$};
			            \node at (3.6,-2.7)  {$c_a \in \Omega_0$};
			            \node at (9.0,-2.7)  {$y(0,c_a)(T)$};
			            \node at (10.8,-2.7) {$\Phi(x)$};
			        \end{tikzpicture}
			\caption{Structure of an augmented neural DDE $\Phi\in \textup{NDDE}^k_{\tau,\textup{A}}(\mathcal{X},\mathbb{R}^q)$. The phase space of the DDE is augmented, i.e., it holds $m>n$. The visualized example has $n = 3$, $m = 4$, and $q = 2$.}
			\label{fig:ndde_augmented}
		\end{subfigure}
		\caption{Visualization of the two different types of neural DDE architectures non-augmented $\textup{NDDE}^k_{\tau,\textup{N}}(\mathcal{X},\mathbb{R}^q)$ and augmented $\textup{NDDE}^k_{\tau,\textup{A}}(\mathcal{X},\mathbb{R}^q)$. The initial data of the DDE is the constant function $c_a\in\Omega_0 \subset \mathcal{C}$, based on the output $a = \lambda(x)$ of the first affine linear layer. The output of the neural DDE is the second affine linear layer $\tilde \lambda$ applied to the time-$T$ map $y(0,c_a)(T)$ of the solution of the DDE.}
		\label{fig:ndde_architectures}
	\end{figure}

	\begin{definition}[General Neural DDE] \label{def:NDDE}
		The set of all neural DDE architectures $\Phi: \mathcal{X} \rightarrow \mathbb{R}^q$, $\mathcal{X}\subset \mathbb{R}^n$ open, defined by~\eqref{eq:DDEaffinelinear} with vector field $F: \Omega \rightarrow \mathbb{R}^m$, $\Omega \subset \mathbb{R}\times \mathcal{C}$, satisfying Assumption \ref{ass:vectorfield} and delay $\tau\geq 0$, is denoted by
		\begin{enumerate}[label=(\alph*), font=\normalfont]
			\item \label{def:NDDE_a} $\textup{NDDE}^k_\tau(\mathcal{X},\mathbb{R}^q)$, if  $F\in C^{0,k}_b(\Omega,\mathbb{R}^m)$ with $k\geq 1$.
			\item \label{def:NDDE_b} $\textup{NDDE}_{\tau}^0(\mathcal{X},\mathbb{R}^q)$, without further assumptions on $F$.
		\end{enumerate}
		If $m \leq n$, the neural DDE architecture is called non-augmented and denoted by $\textup{NDDE}^k_{\tau,\textup{N}}(\mathcal{X},\mathbb{R}^q)$ (see Figure \ref{fig:ndde_architectures}\subref{fig:ndde_nonaugmented}), and if $m > n$, the neural DDE architecture is called augmented and denoted by $\textup{NDDE}^k_{\tau,\textup{A}}(\mathcal{X},\mathbb{R}^q)$ (see Figure \ref{fig:ndde_architectures}\subref{fig:ndde_augmented}).
	\end{definition}
	
	For $k\geq 1$, the definition of the class $\textup{NDDE}^k_\tau(\mathcal{X},\mathbb{R}^q)$, $\mathcal{X}\subset \mathbb{R}^n$ open, includes regularity assumptions on the vector field of the underlying DDE. For $k = 0$, all continuous and well-defined neural DDEs are collected in $\textup{NDDE}^0_\tau(\mathcal{X},\mathbb{R}^q)$, $\mathcal{X}\subset \mathbb{R}^n$ open, without any regularity assumptions on the vector field. By Lemma~\ref{lem:ndde_regularity} for $k\geq 1$, and as every $\Phi\in\textup{NDDE}^0_\tau(\mathcal{X},\mathbb{R}^q)$ is continuous by Assumption~\ref{ass:vectorfield}\ref{ass:vectorfield_c} as a composition of continuous maps, it holds $\textup{NDDE}^k_\tau(\mathcal{X},\mathbb{R}^q)\subset C^k(\mathcal{X},\mathbb{R}^q)$ for all $k\geq 0$. Continuous and unique solutions of the DDE \eqref{eq:DDEconstantIC} can, for example, be guaranteed if~$F$ satisfies a Lipschitz condition for the second variable and is either continuous in $t$ or fulfills the Carathéodory conditions, which allow for discontinuities in the time variable  (cf.\ \cite{Hale1993}).

	In the following, we define the parameter space $\mathbb{V}$ of weights and biases of neural DDE architectures $\Phi \in \textup{NDDE}^k_\tau(\mathcal{X},\mathbb{R}^q)$. The parameter space only includes the weight matrices and biases of the two affine linear transformations; the parameterization of the vector field is not considered.

	\begin{definition}[Parameter Space]\label{def:DDEweightspace}
		Let $\Phi\in \textup{NDDE}^k_\tau(\mathcal{X},\mathbb{R}^q)$, $\mathcal{X}\subset \mathbb{R}^n$ open, $k \geq 0$, be a neural DDE with weight matrices $W\in \mathbb{R}^{m \times n}$, $\widetilde{W}\in\mathbb{R}^{q \times m}$ and bias vectors $b \in \mathbb{R}^m$, $\tilde{b}\in \mathbb{R}^q$. The parameter space~$\mathbb{V}$ of all possible weights and biases is defined as
		\begin{equation*}
			\mathbb{V}\coloneqq \mathbb{R}^{m \times n} \times \mathbb{R}^{q \times m} \times \mathbb{R}^m \times \mathbb{R}^q.
		\end{equation*}
	\end{definition}

	\subsection{General and Parameterized Neural DDEs}\label{sec:parameterized}

	In practice, the vector field of a neural DDE is usually parameterized as in \eqref{eq:DDEparametrized}. In this section, we discuss the relationship between the neural DDE architecture $\textup{NDDE}^k_{\tau}(\mathcal{X},\mathbb{R}^q)$, which is based on the solution map of the general non-autonomous DDE \eqref{eq:DDEconstantIC}, and the neural DDE architecture, for which the underlying vector field $f_\text{DDE}$ is parameterized by a function $\theta: [0,T]\rightarrow \mathbb{R}^{p}$ as in \eqref{eq:DDEparametrized}, defined in the following.
	
	\begin{definition}[Parameterized Neural DDE] \label{def:NDDEparameterized}
		The set of all neural DDE architectures $\Phi_\theta: \mathcal{X} \rightarrow \mathbb{R}^q$, $\mathcal{X}\subset \mathbb{R}^n$ open, defined by~\eqref{eq:DDEaffinelinear}, but with underlying vector field $f_\textup{DDE}:\mathcal{C}\times \mathbb{R}^{p}\rightarrow \mathbb{R}^m$, delay $\tau\geq 0$ and parameter function $\theta:[0,T]\rightarrow \mathbb{R}^{p}$, is denoted by
		\begin{enumerate}[label=(\alph*), font=\normalfont]
			\item \label{def:NDDEparameterized_a} $\textup{NDDE}_{\tau,\theta}^k(\mathcal{X},\mathbb{R}^q)$, if $F(t,y_t) \coloneqq f_\textup{DDE}(y_t,\theta(t))$,  $F\in C_b^{0,k}(\Omega,\mathbb{R}^m)$ with $k \geq 1$ and $\Omega\subset \mathbb{R}\times \mathcal{C}$ open satisfies Assumption \ref{ass:vectorfield}.
			\item \label{def:NDDEparameterized_b} $\textup{NDDE}_{\tau,\theta}^0(\mathcal{X},\mathbb{R}^q)$, if $F: \Omega \rightarrow \mathbb{R}^m$, $F(t,y_t) \coloneqq f_\textup{DDE}(y_t,\theta(t))$ with $\Omega\subset \mathbb{R}\times \mathcal{C}$ open satisfies Assumption~\ref{ass:vectorfield}.
		\end{enumerate}
		For a fixed vector field $f_\textup{DDE}:\mathcal{C}\times \mathbb{R}^{p}\rightarrow \mathbb{R}^m$, the set of all parameter functions $\theta:[0,T]\rightarrow \mathbb{R}^{p}$, such that the corresponding neural DDE defined by \eqref{eq:DDEaffinelinear} is an element of $\textup{NDDE}_{\tau,\theta}^k(\mathcal{X},\mathbb{R}^q)$, $k \geq 0$, is denoted by $\Theta^k([0,T],\mathbb{R}^{p})$. The non-augmented and augmented classes are denoted analogously by $\textup{NDDE}^k_{\tau,\theta,\textup{N}}(\mathcal{X},\mathbb{R}^q)$ and $\textup{NDDE}^k_{\tau,\theta,\textup{A}}(\mathcal{X},\mathbb{R}^q)$.
	\end{definition}
	
	In the case $k\geq 1$, we trace the regularity of the parameterized neural DDE back to the regularity of the general neural DDE architecture $\textup{NDDE}_{\tau}^k(\mathcal{X},\mathbb{R}^q)$. In order to be $k$-times continuously differentiable, both the vector field $f_\text{DDE}$ and the parameter function $\theta$ have to be sufficiently regular. In the case $k = 0$, we collect in $\textup{NDDE}_{\tau,\theta}^0(\mathcal{X},\mathbb{R}^q)\subset C^0(\mathcal{X},\mathbb{R}^q)$ all neural DDE architectures~$\Phi_\theta$ which are continuous and well-defined. We want to emphasize that the parameter $k$ in the set of parameter functions $\Theta^k([0,T],\mathbb{R}^{p})$ does not describe the regularity of $\theta \in \Theta^k([0,T],\mathbb{R}^{p})$, but the regularity of the corresponding neural DDE defined by \eqref{eq:DDEaffinelinear} with underlying vector field $f_\text{DDE}(y_t,\theta(t))$: if $\theta \in \Theta^k([0,T],\mathbb{R}^{p})$, then $\Phi \in C^k(\mathcal{X},\mathbb{R}^q)$. Using the theory of Carathéodory DDEs, existence, uniqueness, and differentiable dependence on initial data can also be proven for vector fields $f_\text{DDE}$, which depend, for example, on functions $\theta$, which are only piece-wise continuous \cite{Hale1993}.
	
	In the following, we relate the general neural DDE architecture $\textup{NDDE}_{\tau}^k(\mathcal{X},\mathbb{R}^q)$ of Definition \ref{def:NDDE} studied in this work to the parameterized neural DDE architecture $\textup{NDDE}_{\tau,\theta}^k(\mathcal{X},\mathbb{R}^q)$ of Definition \ref{def:NDDEparameterized}, which is mainly used in practice. First, we show that every parameterized neural DDE is a general neural DDE.

	\begin{lemma}[Parameterized Neural DDEs Are General Neural DDEs]\label{lem:ndde_containment}
		Let $f_\textup{DDE}:\mathcal{C}\times\mathbb{R}^p\rightarrow\mathbb{R}^m$ be a fixed parameterized vector field and consider a parameter function $\theta\in\Theta^k([0,T],\mathbb{R}^p)$ with $k\geq 0$. Then the vector field $F(t,y_t)\coloneqq f_\textup{DDE}(y_t,\theta(t))$ defines a general neural DDE $\Phi\in\textup{NDDE}^k_\tau(\mathcal{X},\mathbb{R}^q)$ with $\Phi=\Phi_\theta$, where $\Phi_\theta\in\textup{NDDE}^k_{\tau,\theta}(\mathcal{X},\mathbb{R}^q)$ denotes the corresponding parameterized neural DDE with the same delay $\tau$ and the same affine linear layers $\lambda$ and $\tilde\lambda$. In particular, it holds
		\begin{equation*}
			\textup{NDDE}^k_{\tau,\theta}(\mathcal{X},\mathbb{R}^q) \subset \textup{NDDE}^k_\tau(\mathcal{X},\mathbb{R}^q).
		\end{equation*}
	\end{lemma}

	\begin{proof}
		By Definition~\ref{def:NDDEparameterized}, the parameterized neural DDE $\Phi_\theta$ is based on the DDE~\eqref{eq:DDEparametrized} with vector field $f_\textup{DDE}(y_t,\theta(t))$, where $\theta$ is extended constantly outside $[0,T]$. Setting $F(t,y_t)\coloneqq f_\textup{DDE}(y_t,\theta(t))$ turns this into the DDE~\eqref{eq:DDEconstantIC} of a general neural DDE with the same delay, the same constant initial data, and the same solution, hence $\Phi(x)=\Phi_\theta(x)$ for all $x\in\mathcal{X}$. As $\Phi_\theta\in\textup{NDDE}^k_{\tau,\theta}(\mathcal{X},\mathbb{R}^q)$, the vector field $F$ fulfills Assumption~\ref{ass:vectorfield}, and for $k\geq 1$ it holds $F\in C^{0,k}_b(\Omega,\mathbb{R}^m)$, hence $\Phi\in\textup{NDDE}^k_\tau(\mathcal{X},\mathbb{R}^q)$.
	\end{proof}

	In the following, we show that if the parameterized vector field $f_\textup{DDE}$ is uniformly close to a general vector field $F$, the corresponding input-output maps are uniformly close.
	For the latter, we need the sup-norm of a matrix, defined by
	\begin{equation}\label{eq:matrixnorm}
		\norm{A}_\infty = \max_{\norm{x}_\infty = 1} \norm{Ax}_\infty
	\end{equation}
	for $A \in \mathbb{R}^{r \times s}$ and $x \in \mathbb{R}^s$.

	\begin{proposition}[Distance between General and Parameterized Neural DDEs]\label{prop:ndde_distance}
		Let $\mathcal{X}\subset \mathbb{R}^n$ be open, $\Phi\in\textup{NDDE}^k_\tau(\mathcal{X},\mathbb{R}^q)$, $k\geq 0$, be a general neural DDE with vector field~$F$, and $\Phi_\theta\in\textup{NDDE}^{i}_{\tau,\theta}(\mathcal{X},\mathbb{R}^q)$, $i\geq 0$, be a parameterized neural DDE with vector field $f_\textup{DDE}$ and parameter function $\theta$. Both neural DDEs are defined on the time interval $[0,T]$ and based on the same affine linear layers $\lambda$ and~$\tilde\lambda$, where $\tilde\lambda(z) = \widetilde{W}z+\tilde b$. Assume that:
		\begin{itemize}
			\item The general vector field $F$ has global Lipschitz constant $K_F>0$ with respect to its second variable, i.e.,
			\begin{equation*}
				\norm{F(t,y_t)-F(t,z_t)}_\infty\leq K_F\norm{y_t-z_t}_\infty \quad \textup{for all } t\in[0,T],\ y_t,z_t\in\Omega_y.
			\end{equation*}
			\item For some $\eta>0$, the two vector fields satisfy
			\begin{equation*}
				\norm{F(t,y_t)-f_\textup{DDE}(y_t,\theta(t))}_\infty \leq \eta \quad \textup{ for all } (t,y_t)\in[0,T]\times\Omega_y,
			\end{equation*}
			where $\Omega_y$ is as in Assumption~\ref{ass:vectorfield} and contains the delayed functions of the solutions of both neural DDEs for all inputs $x\in\mathcal{X}$.
		\end{itemize}
		Then the input-output maps are uniformly close, and it holds
		\begin{equation}\label{eq:ndde_distance}
			\sup_{x\in\mathcal{X}}\norm{\Phi_\theta(x)-\Phi(x)}_\infty \leq \normm{\widetilde{W}}_\infty\cdot\frac{\eta}{K_F}\left(e^{K_F T}-1\right).
		\end{equation}
	\end{proposition}

	\begin{proof}
		Fix $x\in\mathcal{X}$ and set $y_0\coloneqq\lambda(x)$. As $\Phi$ and $\Phi_\theta$ use the same affine linear layer $\lambda$, both solutions start from the same constant initial data $c_{y_0}$. We write $y(t)\coloneqq y(0,c_{y_0})(t)$ for the solution of the general DDE~\eqref{eq:DDEconstantIC} and $y_\theta(t)\coloneqq y_\theta(0,c_{y_0})(t)$ for the solution of the parameterized DDE~\eqref{eq:DDEparametrized}, and set $\Delta(t)\coloneqq y_\theta(t)-y(t)$ with delayed function $\Delta_t\in\mathcal{C}$. On $[-\tau,0]$, both solutions are equal to $c_{y_0}$, hence, $\Delta_0\equiv 0$ and $\Delta(0)=0$. For $t\in[0,T]$, the delayed functions $y_t$ and $(y_\theta)_t$ lie by assumption in $\Omega_y$. Adding and subtracting $F(t,(y_\theta)_t)$ yields
		\begin{equation*}
			\frac{\dd \Delta}{\dd t}(t) = \bigl[f_\textup{DDE}((y_\theta)_t,\theta(t)) - F(t,(y_\theta)_t)\bigr] + \bigl[F(t,(y_\theta)_t) - F(t,y_t)\bigr].
		\end{equation*}
		The first bracket is bounded by assumption by $\norm{f_\textup{DDE}((y_\theta)_t,\theta(t))-F(t,(y_\theta)_t)}_\infty\leq\eta$, and the second bracket by the Lipschitz continuity of $F$ in its second variable, i.e.,
		\begin{equation*}
			\norm{F(t,(y_\theta)_t)-F(t,y_t)}_\infty\leq K_F\norm{(y_\theta)_t-y_t}_\infty=K_F\norm{\Delta_t}_\infty.
		\end{equation*}
		Hence, it follows that
		\begin{equation*}
			\norm{\frac{\dd \Delta}{\dd t}(t)}_\infty \leq \eta + K_F\norm{\Delta_t}_\infty \qquad \textup{for all } t\in[0,T].
		\end{equation*}
		Let $\Delta^{\max}(t)\coloneqq\max_{s\in[0,t]}\norm{\Delta(s)}_\infty$ denote the maximal error on the time interval $[0,t]$. As $\Delta$ vanishes on $[-\tau,0]$, it holds $\Delta^{\max}(0)=0$ and $\norm{\Delta_t}_\infty\leq\Delta^{\max}(t)$. As $\Delta(0)=0$, integrating the bound on the derivative of $\Delta$ gives for every $t\in[0,T]$
		\begin{equation*}
			\norm{\Delta(t)}_\infty \leq \int_0^t\bigl(\eta+K_F\norm{\Delta_s}_\infty\bigr)\,\dd s \leq \eta t + K_F\int_0^t \Delta^{\max}(s)\,\dd s.
		\end{equation*}
		As the right-hand side is non-decreasing in $t$, the same bound holds for $\Delta^{\max}(t)$, and Grönwall's inequality yields $\Delta^{\max}(t)\leq\frac{\eta}{K_F}(e^{K_F t}-1)$. At $t=T$, it follows
		\begin{equation*}
			\norm{y_\theta(0,c_{y_0})(T)-y(0,c_{y_0})(T)}_\infty = \norm{\Delta(T)}_\infty \leq \Delta^{\max}(T) \leq \frac{\eta}{K_F}\left(e^{K_F T}-1\right),
		\end{equation*}
		uniformly in $x\in\mathcal{X}$. As $\tilde\lambda(z)=\widetilde{W}z+\tilde b$, we obtain with the matrix norm~\eqref{eq:matrixnorm}
		\begin{equation*}
			\sup_{x\in\mathcal{X}}\norm{\Phi_\theta(x)-\Phi(x)}_\infty = \sup_{x\in\mathcal{X}}\norm{\widetilde{W}\left(y_\theta(0,c_{\lambda(x)})(T)-y(0,c_{\lambda(x)})(T)\right)}_\infty \leq \normm{\widetilde{W}}_\infty\cdot\frac{\eta}{K_F}\left(e^{K_F T}-1\right),
		\end{equation*}
		which is the claimed estimate~\eqref{eq:ndde_distance}.
	\end{proof}

	\section{Overview and Results}
	\label{sec:overview}
	
	In this section, we collect our main results regarding the universal embedding and approximation property of neural DDEs. Furthermore, we connect our established theorems to already existing results about basic neural DDEs without affine linear layers and neural ODEs, which are neural DDEs with delay $\tau = 0$. To that purpose, we define in Section \ref{sec:universal} the properties of universal embedding and universal approximation of $k$-times continuously differentiable functions. Additionally, we discuss the relationship between neural DDEs with general or parameterized vector fields with respect to the properties of universal embedding and universal approximation. Afterward, we collect in Section~\ref{sec:nonaugmentedbasic} existing results for basic non-augmented neural DDEs without affine linear transformations before and after the DDE provided by \cite{Zhu2021}. In Section \ref{sec:nonaugmentedgeneral}, we state our main result about the non-universal approximation of Lipschitz continuous maps by non-augmented neural DDEs with bounded weights and vector field if the memory capacity $K\tau$ is sufficiently small, which is proven in Section~\ref{sec:sectionproof}. Additionally, we show that no universal approximation is possible if $m<n$ or $m<q$. In Section~\ref{sec:augmented}, we prove for neural DDEs with $m\geq\max\{n,q\}$ the universal embedding property for Lipschitz continuous maps on bounded open sets if the memory capacity $K\tau$ of the neural DDE is sufficiently large. Furthermore, we show the universal embedding property for Lipschitz continuous functions independently of the delay if the dimension of the DDE is as large as the sum of input and output dimensions. 
	
	\subsection{Universal Embedding and Universal Approximation}
	\label{sec:universal}
	
	An important property of neural network models is universal approximation, i.e., the property to approximate any function in a given function space to an arbitrary level of precision. This means that the set of functions, which a neural network can represent, is dense in the underlying function space. In the following, we state the abstract definition of universal approximation of $k$-times continuously differentiable functions, $k \geq 0$. For a set $\mathcal{U}\subset\mathcal{X}$, $\mathcal{X}\subset \R^n$ open, we define the sup-norm of a function $\psi\in C^k(\mathcal{X},\R^q)$~by
	\begin{equation*}
		\norm{\psi}_{\infty,\mathcal{U}}\coloneqq\sup_{x\in\mathcal{U}}\norm{\psi(x)}_\infty.
	\end{equation*}
	
	\begin{definition}[Universal Approximation~\cite{Kuehn2026}] \label{def:universalapproximation}
		A neural network family $\mathcal{N} = \{\Phi_\theta: \mathcal{X} \rightarrow \mathbb{R}^q\}_{\theta \in \mathcal{P}}$ with $\mathcal{X}\subset \mathbb{R}^n$ open and parameters $\theta\in \mathcal{P}$ has the \emph{universal approximation property} with respect to the space $C^k(\mathcal{X},\mathbb{R}^q)$, $k \geq 0$, if for every $\varepsilon >0$, every compact subset $\mathcal{K}\subset \mathcal{X}$, and every function $\Psi \in C^k(\mathcal{X},\mathbb{R}^q)$, there exists $\theta\in \mathcal{P}$ such that $\norm{\Phi_\theta-\Psi}_{\infty,\mathcal{K}} < \varepsilon$.
	\end{definition}
	
	In the case of the general neural DDE architecture $\textup{NDDE}_\tau^k(\mathcal{X},\mathbb{R}^q)$, $\mathcal{X}\subset \mathbb{R}^n$ open, the parameters are the weights and biases of the parameter space $\mathbb{V}$ and the vector field $F$. If the underlying vector field of the neural DDE \eqref{eq:DDEaffinelinear} is parameterized as in \eqref{eq:DDEparametrized}, then the vector field $f_\text{DDE}$ is fixed and only the parameter function $\theta:[0,T]\rightarrow \mathbb{R}^{p}$ and the weights and biases in $\mathbb{V}$ can be chosen. 
	
	Classical feed-forward neural networks can obtain the universal approximation property by increasing the network width, depth, and the number of parameters~\cite{Hornik1989, Lin2018, Pinkus1999, Schaefer2006}. For neural ODEs, the approximation capability is increased by augmenting the phase space~\cite{Dupont2019,Kidger2022,Kuehn2023,Zhang2020a}. For the analysis of neural ODEs, the property of an exact representation, called the universal embedding property, is also used in the literature.	
	
	\begin{definition}[Universal Embedding~\cite{Kuehn2026}]\label{def:universalembedding}
		A neural network family $\mathcal{N} = \{\Phi_\theta: \mathcal{X} \rightarrow \mathbb{R}^q\}_{\theta \in \mathcal{P}}$ with $\mathcal{X}\subset \mathbb{R}^n$ open and parameters $\theta\in \mathcal{P}$ has the \emph{universal embedding property} with respect to the space $C^k(\mathcal{X},\mathbb{R}^q)$, $k \geq 0$, if for every function $\Psi \in C^k(\mathcal{X},\mathbb{R}^q)$ there exists $\theta\in \mathcal{P}$ such that $\Phi_\theta(x) = \Psi(x)$ for all $x \in \mathcal{X}$.
	\end{definition}
	
	In the case of neural ODEs it is shown that an augmented neural ODE in an $(n+q)$-dimensional phase space with general non-autonomous vector field is sufficient to represent any continuous map $\Psi\in C^0(\mathbb{R}^n,\mathbb{R}^q)$ exactly \cite{Kuehn2023}. This property can be used to show that if there exists a choice of parameters, such that a fixed parameterized vector field is arbitrarily close to the general non-autonomous vector field, then the given map $\Psi$ can be approximated by a parameterized neural ODE in terms of Definition \ref{def:universalapproximation}. 
	
	Lemma~\ref{lem:ndde_containment} and Proposition~\ref{prop:ndde_distance} allow us to transfer this idea to neural DDEs. More generally, the results of Section~\ref{sec:parameterized} allow us to transfer expressivity properties between general and parameterized architectures. If the general architecture $\textup{NDDE}^k_{\tau}(\mathcal{X},\mathbb{R}^q)$ does not have the universal approximation (embedding) property, then by Lemma~\ref{lem:ndde_containment} also the parameterized architecture $\textup{NDDE}^k_{\tau,\theta}(\mathcal{X},\mathbb{R}^q)$ does not have it, independently of the choice of $f_\textup{DDE}$. Conversely, assume that the general architecture has the universal approximation or embedding property with vector fields that are globally Lipschitz continuous in their second variable, and that for every such vector field $F$ and every $\eta>0$ there exists a parameter function $\theta$, such that $f_\textup{DDE}(\cdot,\theta(\cdot))$ is $\eta$-close to $F$ as in Proposition~\ref{prop:ndde_distance}, then the parameterized architecture has the universal approximation property.

	In this work, we restrict our analysis to neural DDEs with the architecture $\textup{NDDE}^k_\tau(\mathcal{X},\mathbb{R}^q)$, $k\geq 0$, defined by \eqref{eq:DDEaffinelinear} with general non-autonomous vector field \eqref{eq:DDEconstantIC}. If a vector field $f_\text{DDE}$ is fixed and parameterized by a function $\theta \in \Theta^i([0,T],\mathbb{R}^{p})$ as in \eqref{eq:DDEparametrized}, all our results can be transferred from the general architecture $\textup{NDDE}^k_\tau(\mathcal{X},\mathbb{R}^q)$ to the parameterized architecture $\textup{NDDE}^i_{\tau,\theta}(\mathcal{X},\mathbb{R}^q)$ using Lemma~\ref{lem:ndde_containment} and Proposition~\ref{prop:ndde_distance}.

	\subsection{Basic Non-Augmented Neural DDEs}
	\label{sec:nonaugmentedbasic}
	
	To study the universal embedding and approximation property of neural DDEs, we begin with non-augmented architectures $\Phi\in \textup{NDDE}_{\tau,\textup{N}}^k(\mathcal{X},\mathbb{R}^q)$, $\mathcal{X}\subset \mathbb{R}^n$ open, $k\geq 0$, where the dimension $m$ of the DDE is smaller than or equal to the input dimension~$n$. In this section, we restrict ourselves to the case that both affine linear layers are identity transformations with $n = m = q$. In particular, it holds $W = \widetilde{W} = \text{Id}_n \in \mathbb{R}^{n \times n}$, which is the identity matrix and $b = \tilde{b} = 0 \in \mathbb{R}^n$, such that $\lambda = \tilde{\lambda}= \text{id}_{\mathbb{R}^n}$ is the identity transformation on $\mathbb{R}^n$. We call the resulting architecture, already introduced in \cite{Zhu2021}, a basic non-augmented neural DDE. The case of general non-augmented neural DDEs is treated in the upcoming Section~\ref{sec:nonaugmentedgeneral}.
	
	A basic non-augmented neural DDE $\Phi\in \textup{NDDE}_{\tau,\textup{N}}^k(\mathcal{X},\mathbb{R}^n)$, $\mathcal{X}\subset \mathbb{R}^n$ open, maps the initial data $x \in \mathcal{X}$ to the time-$T$ map of the DDE with constant initial function $c_x \in \mathcal{C}$. First, we show with a one-dimensional example that introducing a large delay term $\tau>0$ increases the embedding capability of neural DDEs compared to neural ODEs. In the one-dimensional case, a typical example to show that basic neural ODEs, i.e., neural DDEs $\Phi \in \textup{NDDE}_{0,\textup{N}}^k(\mathbb{R},\mathbb{R})$ with $k \geq 0$, $\tau = 0$ and $\lambda = \tilde\lambda = \text{id}_\mathbb{R}$ do not have the universal embedding property with respect to the space $C^j(\mathbb{R},\mathbb{R})$ for any $j \geq 0$, is the following.
	
	\begin{example} \label{ex:minusxembedding_a}
		\normalfont
		Let $\Psi\in C^\infty(\mathbb{R},\mathbb{R})$, $x \mapsto -x$, $\Phi \in \textup{NDDE}_{0,\textup{N}}^k(\mathbb{R},\mathbb{R})$ with  $\lambda = \tilde\lambda = \textup{id}_\mathbb{R}$ and $k \geq 0$. As the solution of the ODE underlying the neural ODE $\Phi$ is, by definition of the architecture, unique and exists for $t \in [0,T]$, the solution curves of the ODE cannot cross in the phase space $\mathbb{R}$. Hence, the map $\Psi$ cannot be embedded as a time-$T$ map in the neural ODE $\Phi$, as solution curves would need to cross, see Figure~\ref{fig:example_minusx}\subref{fig:example_minusx_a}. As $\Psi\in C^j(\mathbb{R},\mathbb{R})$ for every $j \geq 0$, the architecture $ \textup{NDDE}_{0,\textup{N}}^k(\mathbb{R},\mathbb{R})$ does not have the universal embedding property with respect to the space $C^j(\mathbb{R},\mathbb{R})$ for any $j \geq 0$.
	\end{example}
	
	As also discussed after Theorem~\ref{th:tau0}, the architecture $ \textup{NDDE}_{0,\textup{N}}^k(\mathbb{R},\mathbb{R})$, $k\geq 1$, which allows for two additional affine linear transformations, does not have the universal approximation property with respect to the space $C^j(\mathbb{R},\mathbb{R})$ for any $j \geq 0$~\cite{Kuehn2024}. In comparison to neural ODEs, the phase space of neural DDEs is infinite-dimensional, such that we expect that the embedding capability of neural DDEs is larger than the embedding capability of neural ODEs. 
	
	Based on \cite{Zhu2021}, we show in the following that for $\tau = T = 1$ the restriction of the map $\Psi$ of Example~\ref{ex:minusxembedding_a} to $(-r,r)$, for every $r>0$, can be embedded in the neural DDE architecture $\textup{NDDE}_{1,\textup{N}}^\infty((-r,r),\mathbb{R})$ with identity input and output transformations $\lambda = \tilde\lambda = \text{id}_\mathbb{R}$. As required by Definition~\ref{def:NDDE}, the vector field of a neural DDE has to be bounded for $k\geq 1$. Hence, we compose the delayed state with a bounded map that coincides with the identity map on a bounded set containing the initial data. 

	\begin{definition}[Saturation Function, cf.~\cite{Hormander2003}]\label{def:saturation}
		Let $n\in\mathbb{N}$, $r>0$, and let $\zeta\in C^\infty(\mathbb{R},[0,1])$ be a compactly supported cutoff function with $\zeta(s) = 1$ for all $s\in[-r,r]$. The map $\textup{sat}_r:\mathbb{R}^n\rightarrow\mathbb{R}^n$ defined componentwise by
		\begin{equation*}
			[\textup{sat}_r(y)]_i \coloneqq \int_0^{[y]_i}\zeta(s)\,\dd s, \qquad i = 1,\ldots,n,
		\end{equation*}
		is called a saturation function with radius $r$.
	\end{definition}

	Cutoff functions $\zeta$ as in Definition~\ref{def:saturation} exist for every $r>0$ (cf.~\cite{Hormander2003}). In the following, $C^k_b(\mathbb{R}^n,\mathbb{R}^q)$, $k\geq 0$, denotes the space of $k$-times continuously differentiable functions $\Psi:\mathbb{R}^n\rightarrow\mathbb{R}^q$, which are bounded and have bounded derivatives up to order~$k$. In the upcoming proofs, we use the following properties of saturation functions.

	\begin{lemma}[Properties of Saturation Functions]\label{lem:saturation}
		Let $\textup{sat}_r:\mathbb{R}^n\rightarrow\mathbb{R}^n$ be a saturation function with $r>0$. Then it holds $\textup{sat}_r\in C^\infty_b(\mathbb{R}^n,\mathbb{R}^n)$, $\textup{sat}_r(y) = y$ for all $y\in\mathbb{R}^n$ with $\norm{y}_\infty\leq r$, and
		\begin{equation*}
			\norm{\textup{sat}_r(y_1)-\textup{sat}_r(y_2)}_\infty\leq\norm{y_1-y_2}_\infty \quad \text{for all } y_1,y_2\in\mathbb{R}^n.
		\end{equation*}
	\end{lemma}

	\begin{proof}
		As the $i$-th component of $\textup{sat}_r$ depends only on $[y]_i$, its derivative with respect to $[y]_i$ is given by $\zeta([y]_i)$. Hence, the only partial derivatives of order $j\geq 1$ that can be non-zero are the $j$-th derivatives with respect to $[y]_i$, given by $\zeta^{(j-1)}([y]_i)$, which are bounded as $\zeta^{(j-1)}$ is continuous with compact support. For $\norm{y}_\infty\leq r$, the function $\zeta$ is equal to $1$ on the interval between $0$ and $[y]_i$, hence $[\textup{sat}_r(y)]_i = [y]_i$ for $i = 1,\ldots,n$. Let $s_0>0$ with $\zeta(s) = 0$ for all $\abs{s}\geq s_0$. As $0\leq\zeta\leq 1$, it holds $\abs{[\textup{sat}_r(y)]_i}\leq s_0$ for all $y\in\mathbb{R}^n$. Hence, $\textup{sat}_r\in C^\infty_b(\mathbb{R}^n,\mathbb{R}^n)$. Finally, as $\zeta$ takes values in $[0,1]$, it holds
		\begin{equation*}
			\abs{[\textup{sat}_r(y_1)]_i-[\textup{sat}_r(y_2)]_i} = \abs{\int_{[y_2]_i}^{[y_1]_i}\zeta(s)\,\dd s}\leq\abs{[y_1]_i-[y_2]_i}
		\end{equation*}
		for $i = 1,\ldots,n$. Taking the maximum over $i$ yields the last property.
	\end{proof}

	\begin{example} \label{ex:minusxembedding_b}
		\normalfont
		Let $r>0$, $\Psi\in C^\infty(\mathbb{R},\mathbb{R})$, $x \mapsto -x$, let $\textup{sat}_r:\mathbb{R}\rightarrow\mathbb{R}$ be a saturation function with radius $r$, and let $\Phi \in \textup{NDDE}_{1,\textup{N}}^\infty((-r,r),\mathbb{R})$ with  $\lambda = \tilde\lambda = \textup{id}_\mathbb{R}$, defined by the one-dimensional DDE
		\begin{equation*}
			\begin{aligned}
				\frac{\dd y}{\dd t} &= -2 \,\textup{sat}_r(y(t-1)), \qquad&& t \in [0,1], \\
				y(t) &= c_x(t), &&t \in [-1,0],
			\end{aligned}
		\end{equation*}
		with constant in time initial function $c_x:[-1,0]\rightarrow \mathbb{R}$, $c_x(t) = x$ for $x\in(-r,r)$. By Lemma~\ref{lem:saturation}, the vector field is an element of $C^{0,k}_b(\mathbb{R}\times\mathcal{C},\mathbb{R})$ for every $k\geq 0$, and on the time interval $[0,1]$ it holds $\textup{sat}_r(y(t-1)) = \textup{sat}_r(x) = x$. Hence, on the time interval $[0,1]$, the DDE is uniquely solved by $y(0,c_x)(t) = (1-2t)\cdot x$, such that $y(1) = -x$ for all $x \in (-r,r)$. Hence, the restriction $\Psi|_{(-r,r)}$ is embedded in the basic neural DDE architecture $\textup{NDDE}_{1,\textup{N}}^\infty((-r,r),\mathbb{R})$, see Figure \ref{fig:example_minusx}\subref{fig:example_minusx_b}.
	\end{example}

	In Example \ref{ex:minusxembedding_b}, the restriction to a bounded interval $(-r,r)$ cannot be dropped, as the vector field of a neural DDE has to be bounded. The idea used in the construction of the vector field in Example \ref{ex:minusxembedding_b} is that for the large delay $\tau = T$, the  DDE is uniquely solvable via direct integration on the whole interval $[0,T]$. The following theorem generalizes this idea to $n$-dimensional maps and shows that every $\Psi\in C^k_b(\mathbb{R}^n, \mathbb{R}^n)$, $k\geq 0$, can be embedded on every bounded open set $\mathcal{X}\subset\mathbb{R}^n$ in the neural DDE architecture $\textup{NDDE}_{T,\textup{N}}^k(\mathcal{X},\mathbb{R}^n)$ with $\lambda = \tilde\lambda = \text{id}_{\mathbb{R}^n}$ and $\tau = T$.
	
	\begin{figure}[t]
		\centering
		\begin{subfigure}[t]{0.45\textwidth}
			\centering
			\resizebox{\linewidth}{!}{%
			\begin{tikzpicture}[
					oe/.style={dotted, line width=1.6pt, ddeGrayDark},
					axb/.style={-{Stealth[length=2.0mm]}, line width=1pt},
					axg/.style={-{Stealth[length=2.0mm]}, line width=1pt, ddeGray}]
					\useasboundingbox (-2.4,-2.65) rectangle (7.2,2.75);
					\def\Tend{5.5}
					\def\xs{1.6}
					\draw[axb] (0,-2.0) -- (0,2.5);
					\draw[axb] (\Tend,-2.0) -- (\Tend,2.5);
					\draw[axg] (0,0) -- (\Tend+0.9,0) node[right, ddeGrayDark] {$t$};
					\draw[oe] (0,0) -- (\Tend,0);
					\draw[oe] (0,\xs) -- (\Tend,-\xs);
					\draw[ddeNeg, line width=1.2pt] (\Tend/2,0) circle (6pt);
					\fill[ddeCyan] (0,0) circle (2.5pt);
					\fill[ddeCyan] (0,\xs) circle (2.5pt);
					\fill[ddeIndigo] (\Tend,0) circle (2.5pt);
					\fill[ddeIndigo] (\Tend,-\xs) circle (2.5pt);
					\node[anchor=east] at (-0.1,\xs) {\textcolor{ddeCyan}{$x^\ast$}};
					\node[anchor=east] at (-0.1,0.18) {\textcolor{ddeCyan}{$0$}};
					\node[ddeGrayDark, anchor=north] at (0.22,-0.12) {$0$};
					\node[ddeGrayDark, anchor=north] at (\Tend-0.18,-0.12) {$T$};
					\node[ddeGrayDark, anchor=north] at (\Tend,-2.15) {$\R$};
					\node[ddeGrayDark, anchor=north] at (0,-2.15) {$\R$};
					\node[anchor=west] at (\Tend+0.15,-\xs) {$-x^\ast$};
				\end{tikzpicture}%
			}
			\caption{Any two dotted trajectories from $y(0) = 0$ to $y(T) = 0$ and from $y(0) = x^*$ to $y(T) = -x^*$ for some $x^* \neq 0$ have to intersect, which is not possible for ODE solution curves.}
			\label{fig:example_minusx_a}
		\end{subfigure}
		\begin{subfigure}[t]{0.05\textwidth}
			\textcolor{white}{.}
		\end{subfigure}
		\begin{subfigure}[t]{0.45\textwidth}
			\centering
			\resizebox{\linewidth}{!}{%
			\begin{tikzpicture}[
					oe/.style={dotted, line width=1.6pt, ddeGrayDark},
					axb/.style={-{Stealth[length=2.0mm]}, line width=1pt},
					axg/.style={-{Stealth[length=2.0mm]}, line width=1pt, ddeGray},
					hist/.style={line width=1.6pt, ddeCyan, line cap=round},
					arr/.style={-{Stealth[length=2.0mm]}, line width=1pt, ddeGrayDark}]
					\useasboundingbox (-2.4,-2.65) rectangle (7.2,2.75);
					\def\Tend{5.5}
					\def\xs{1.6}
					\def\dedelay{1.5}
					\fill[ddeGray!30] (-\dedelay,-2.0) rectangle (0,2.1);
					\draw[line width=1pt] (-\dedelay,-2.0) -- (-\dedelay,2.1);
					\draw[line width=1pt] (0,-2.0) -- (0,2.1);
					\draw[axb] (\Tend,-2.0) -- (\Tend,2.5);
					\draw[axg] (0,0) -- (\Tend+0.9,0) node[right, ddeGrayDark] {$t$};
					\draw[hist] (-\dedelay,\xs) -- (0,\xs);
					\draw[hist] (-\dedelay,0) -- (0,0);
					\draw[oe] (0,0) -- (\Tend,0);
					\draw[oe] (0,\xs) -- (\Tend,-\xs);
					\draw[ddePos!70!black, line width=1.2pt] (\Tend/2,0) circle (6pt);
					\fill[ddeIndigo] (\Tend,0) circle (2.5pt);
					\fill[ddeIndigo] (\Tend,-\xs) circle (2.5pt);
					\node[anchor=east] at (-\dedelay-0.1,\xs) {\textcolor{ddeCyan}{$c_{x^\ast}$}};
					\node[anchor=east] at (-\dedelay-0.1,0) {\textcolor{ddeCyan}{$c_0$}};
					\node[ddeGrayDark] (mt) at (-1.05,-0.62) {$-\tau$};
					\draw[arr] (mt) -- (-\dedelay,-0.12);
					\node[ddeGrayDark, anchor=north] at (0.22,-0.12) {$0$};
					\node[ddeGrayDark, anchor=north] at (\Tend-0.18,-0.12) {$T$};
					\node[ddeGrayDark, anchor=north] at (\Tend,-2.15) {$\R$};
					\node[ddeGrayDark, anchor=north] at (-\dedelay/2,-2.15) {$\mathcal{C}$};
					\node[anchor=west] at (\Tend+0.15,-\xs) {$-x^\ast$};
				\end{tikzpicture}%
			}
			\caption{The trajectories of DDEs can intersect and there exists a DDE with constant initial data and solution curves going from $y(0) = 0$ to $y(T) = 0$ and from $y(0) = x^*$ to $y(T) = -x^*$.}
			\label{fig:example_minusx_b}
		\end{subfigure}
		\caption{The map $\Psi\in C^\infty(\mathbb{R},\mathbb{R})$, $x \mapsto -x$  cannot be embedded in the basic neural ODE architecture $\textup{NDDE}_{0,\textup{N}}^k(\mathbb{R},\mathbb{R})$, $k \geq 0$, but its restriction to $(-r,r)$, $r>0$, can be embedded in the basic neural DDE architecture $\textup{NDDE}_{1,\textup{N}}^\infty((-r,r),\mathbb{R})$.}
		\label{fig:example_minusx}
	\end{figure}	

	\begin{theorem}[Universal Embedding for Basic Neural DDEs with $\tau = T$, cf.~\cite{Zhu2021}] \label{th:universal_zhu}
		Let $k\geq 0$, $\mathcal{X}\subset\mathbb{R}^n$ open and bounded, and $\Psi\in C^k_b(\mathbb{R}^n,\mathbb{R}^n)$. Then the restriction $\Psi|_{\mathcal{X}}$ can be embedded in the  neural DDE architecture $\textup{NDDE}_{T,\textup{N}}^k(\mathcal{X},\mathbb{R}^n)$ with $\lambda = \tilde\lambda = \textup{id}_{\mathbb{R}^n}$ and delay $\tau = T$.
	\end{theorem}

	\begin{proof}
		As $\mathcal{X}$ is bounded, there exists $r>0$ with $\norm{x}_\infty\leq r$ for all $x\in\mathcal{X}$. Let $\textup{sat}_r:\mathbb{R}^n\rightarrow\mathbb{R}^n$ be a saturation function with radius $r$ as in Definition~\ref{def:saturation}, and let $\Phi \in \textup{NDDE}_{T,\textup{N}}^k(\mathcal{X},\mathbb{R}^n)$ with  $\lambda = \tilde\lambda = \text{id}_{\mathbb{R}^n}$ be defined by the $n$-dimensional DDE
		\begin{align} \label{eq:vectorfield_zhu}
			\begin{aligned}
				\frac{\dd y}{\dd t} &= \frac{1}{T} (\Psi(y(t-T))-\textup{sat}_r(y(t-T))), \qquad&& t \in [0,T], \\
				y(t) &= c_x(t), &&t \in [-T,0],
			\end{aligned}
		\end{align}
		with constant initial data $c_x: [-T,0] \rightarrow \mathbb{R}^n$, $c_x(t) = x$ for $x \in \mathcal{X} \subset \mathbb{R}^n$. As $\Psi\in C^k_b(\mathbb{R}^n,\mathbb{R}^n)$ and $\textup{sat}_r\in C^\infty_b(\mathbb{R}^n,\mathbb{R}^n)$ by Lemma~\ref{lem:saturation}, the constructed vector field is an element of $C^{0,k}_b(\mathbb{R}\times\mathcal{C},\mathbb{R}^n)$. On the time interval $[0,T]$, it holds $y(t-T) = x$ and, by Lemma~\ref{lem:saturation}, $\textup{sat}_r(x) = x$. Hence, on the time interval~$[0,T]$ the DDE is uniquely solved by $y(0,c_x)(t) = x + \frac{t}{T}(\Psi(x)-x)$, which is $k$-times continuously differentiable with respect to $x$, resulting in the time-$T$ map $y(T) = \Psi(x)$. As $\lambda = \tilde\lambda = \text{id}_{\mathbb{R}^n}$, the restriction $\Psi|_{\mathcal{X}}$ is embedded in the basic neural DDE architecture $\textup{NDDE}_{T,\text{N}}^k(\mathcal{X},\mathbb{R}^n)$.
	\end{proof}

	Theorem \ref{th:universal_zhu} is restricted to basic non-augmented neural DDEs with identity transformations as affine linear layers and a large delay $\tau = T$. Here, a constant discrete delay of type $A$, $\tau_{L,A}(t) = L\delta =T$ for $t\in[0,T]$ is used, cf.\ Section \ref{sec:modeling_discretization}.  A delay of length $\tau = T$ corresponds to the existence of an inter-layer connection of length $L$, such that the last layer $h_L$ can be directly computed from the initial data~$h_0$, which is a strong assumption on the underlying neural network architecture. We are interested in knowing if shorter inter-layer connections can also lead to an improvement of the expressivity of a neural network. In the following sections, we study the influence of parameters on the universal embedding and approximation capability of neural DDEs, including the length of the delay $\tau$, the Lipschitz constant~$K$, and the dimension $m$ of the vector field. Depending on specific parameter regions, we show universal embedding and non-universal approximation results.
	
	\subsection{Non-Universal Approximation with Small Memory Capacity or Bottlenecks}
	\label{sec:nonaugmentedgeneral}

	In this section, we study the universal approximation capability of the general neural DDE architecture $\textup{NDDE}_{\tau}^k(\mathcal{X},\mathbb{R}^q)$, $\mathcal{X}\subset \mathbb{R}^n$ open, $k\geq 0$, and show three negative results. In contrast to the setting of basic neural DDEs in Section~\ref{sec:nonaugmentedbasic}, general affine linear transformations $\lambda$, $\tilde \lambda$, are used, and the input and output dimensions $n$ and $q$ do not need to be the same. Theorem~\ref{th:tau0} and Theorem~\ref{th:bottleneck} concern the non-augmented architecture $\textup{NDDE}^k_{\tau,\textup{N}}(\mathcal{X},\mathbb{R}^q)$ with $m\leq n$, while Theorem~\ref{th:outputbottleneck} holds for all $m<q$, including augmented architectures.

	Theorem \ref{th:universal_zhu} of the last section has shown that even without the affine linear layers $\lambda$, $\tilde{\lambda}$, it is possible to embed any map $\Psi\in C^k_b(\mathbb{R}^n, \mathbb{R}^n)$ on bounded open sets as a time-$T$ map of a DDE with constant initial data if the delay $\tau$ coincides with the final time $T$. Naturally, the question arises if universal approximation is also possible if $0 \leq \tau < T$ and if the input and output dimensions of the map $\Psi$ are not the same. In Theorem~\ref{th:tau0}, we show that below a threshold $\tau_0(K)$ of the delay, neither the general nor the parameterized non-augmented architecture can approximate every map with a prescribed global Lipschitz constant $K_\Psi>0$. For $m<n$ or $m<q$, Theorem~\ref{th:bottleneck} and Theorem~\ref{th:outputbottleneck} exclude universal approximation for every delay. The complementary positive results are established in Section~\ref{sec:augmented}. For the upcoming theorems, we need the property of a vector field to be globally Lipschitz continuous with respect to the second variable.

	\begin{definition}[Lipschitz Continuity in the Second Variable]
		A vector field $F: \Omega_t \times \Omega_y \rightarrow \mathbb{R}^m$, $\Omega_t \times \Omega_y \subset \mathbb{R}\times \mathcal{C}$ is globally Lipschitz continuous in the second variable on $\Omega_t \times \Omega_y'$, $\Omega_y' \subset \Omega_y$, with Lipschitz constant $K$ if 
		\begin{equation*}
			\norm{F(t,y_t)-F(t,z_t)}_\infty \leq K \norm{y_t-z_t}_\infty \qquad \text{ for all } t \in \Omega_t \text{ and } y_t,z_t \in \Omega_y'.
		\end{equation*}
	\end{definition}

	To analyze the expressivity of neural DDEs when the memory capacity $K\tau$ is sufficiently small, we use the theory of DDEs with small delay, which is introduced in Section \ref{sec:dde_small_delay}. In Section~\ref{sec:sectionproof} we use that theory to prove Theorem~\ref{th:noapproximation} and Corollary~\ref{cor:noapproximation}. Under bounds on the weights and the vector field, non-augmented neural DDEs with a small memory capacity cannot approximate maps with a non-degenerate local extreme point arbitrarily well. A direct consequence is the following theorem. Here and in the upcoming theorems, a subscript $K$ denotes the subclass of neural DDEs whose vector field satisfies the conditions stated in the respective theorem, in particular global Lipschitz continuity with constant $K$ in its second variable.

	\begin{theorem}[No Universal Approximation of Lipschitz Maps for Non-Augmented Neural DDEs with Small Memory Capacity]\label{th:tau0}
		Let $k\geq 0$, $\mathcal{X}\subset\mathbb{R}^n$ open, and fix constants $w,\tilde w,\sigma_0,K_\Psi\in(0,\infty)$ and $A,b_\infty\in[0,\infty)$. Then there exist a map $\Psi\in C^\infty_b(\mathbb{R}^n,\mathbb{R}^q)$ with global Lipschitz constant $K_\Psi$ with respect to the max-norm, a compact set $\mathcal{K}\subset\mathcal{X}$, a constant $\eps>0$, and a continuous function $\tau_0\in C^0((0,\infty),(0,T])$ with $K\tau_0(K)e<1$ for all $K\in(0,\infty)$, such that for all $K\in(0,\infty)$ and $\tau\in[0,\tau_0(K))$, every neural DDE $\Phi\in\textup{NDDE}^k_{\tau,\textup{N},K}(\mathcal{X},\mathbb{R}^q)$ with
		\begin{itemize}
			\item  vector field $F:\Omega_t \times \mathcal{C} \rightarrow \mathbb{R}^m$ with global Lipschitz constant~$K$ in its second variable on $\Omega_t \times \mathcal{C}$ and $\norm{F(t,0)}_\infty \leq A$ for all $t\in [0,T]$, which is additionally continuous if $k = 0$,
			\item weight matrices $W, \widetilde{W}$ with $\normm{W}_\infty \leq w$, $\normm{\widetilde{W}}_\infty \leq\tilde{w}$, smallest singular value $\sigma_{\min}(W)\geq\sigma_0$, and bias vector $b$ with $\norm{b}_\infty\leq b_\infty$,
		\end{itemize}
		satisfies $\norm{\Phi-\Psi}_{\infty,\mathcal{K}}\geq\eps$. The same statement holds verbatim if the general class is replaced by the parameterized class $\textup{NDDE}_{\tau,\theta,\textup{N},K}^k(\mathcal{X},\mathbb{R}^q)$.
	\end{theorem}

	\begin{proof}
		The theorem directly follows from Corollary \ref{cor:noapproximation} proven in Section \ref{sec:sectionproof}. The statement for the parameterized class follows by Lemma~\ref{lem:ndde_containment}, as every parameterized neural DDE satisfying the bounds of the theorem is a general neural DDE satisfying the same bounds.
	\end{proof}

	\begin{remark}[Singular Value and Bias Bounds]\label{rem:singularvalue}
		The bounds $\sigma_{\min}(W)\geq\sigma_0$ and $\norm{b}_\infty\leq b_\infty$ on the affine linear layer $\lambda(x) = Wx+b$ are technical conditions that guarantee uniform control over all considered neural DDEs. The bias bound prevents the image $\lambda(\mathcal{X})$ from being translated arbitrarily far within the phase space $\mathbb{R}^m$, while the singular value bound prevents $W$ from collapsing directions in which the obstruction is constructed. Without such bounds, the constants entering the threshold $\tau_0$ cannot be uniform over the admissible parameters. For $W\in\mathbb{R}^{m\times n}$, the condition $\sigma_{\min}(W)\geq\sigma_0>0$ on the smallest of the $m$ singular values implies that $W$ has full row rank $m$. This requires $n\geq m$, which holds throughout the non-augmented setting.
	\end{remark}

	\begin{remark}[Geometry of the Obstruction]\label{rem:geometric}
		For fixed $K$, $w$, and $\tilde w$, every neural DDE in Theorem~\ref{th:tau0} is, for every $\tau\in[0,T]$, by Grönwall's inequality, globally Lipschitz continuous with respect to the max-norm with Lipschitz constant $w\tilde w e^{KT}$. Hence, the class cannot approximate a map which has a larger Lipschitz constant with respect to the max-norm. In particular, the class of neural DDEs does not have the universal approximation property with respect to $C^j(\mathcal{X},\mathbb{R}^q)$, $j\geq 0$, for any $\tau\in[0,T]$. The obstruction of Theorem~\ref{th:tau0} is of a different, geometric nature, as it holds for every Lipschitz constant $K_\Psi>0$, including all $K_\Psi\leq w\tilde w e^{KT}$. By Theorem~\ref{th:noapproximation}, for every map with a non-degenerate local extreme point, there exists a threshold $\tau_0(K)$, which depends on the map, such that for $\tau<\tau_0(K)$ no such neural DDE approximates the map with the accuracy of Theorem~\ref{th:noapproximation}. In particular, the obstruction also applies to maps with an arbitrarily small Lipschitz constant.
	\end{remark}

	The area of non-universal approximation defined by the parameter constraint of Theorem \ref{th:tau0} is given by
	\begin{equation*}
		A_\text{nUA} \coloneqq \left\{(K,\tau) \in (0,\infty) \times [0,T]: \tau < \tau_0(K)\right\}
	\end{equation*}
	and visualized in Figure \ref{fig:regions_nua}\subref{fig:region_nua}. For every global Lipschitz constant $K$ of the vector field with respect to the second variable, no universal approximation of all maps with Lipschitz constant $K_\Psi$ is possible if $\tau \in [0,\tau_0(K))$. A crucial assumption of DDEs with small delay is that $K\tau e<1$. Hence, it holds for the upper bound $\tau_0$ that $K \tau_0(K) e<1$, and consequently
	\begin{equation*}
		A_\text{nUA} \subset \left\{(K,\tau) \in (0,\infty) \times [0,T]: K \tau < \frac{1}{e}\right\}.
	\end{equation*}
	Neural DDEs without delay, also called neural ODEs, are a special case of Theorem~\ref{th:tau0}. In this case, the parameter bounds are not needed, as, independently of the Lipschitz constant $K$ of the vector field, no universal approximation or universal embedding results can exist for the non-augmented neural ODE architecture $\textup{NDDE}^k_{0,\text{N}}(\mathcal{X},\mathbb{R}^q)$, $k \geq 1$~\cite{Kuehn2024}. The non-universal approximation property of non-augmented neural ODEs is also discussed in \cite{Kidger2022,Kuehn2023}. If $(K,\tau)\in A_\text{nUA}$, the corresponding neural DDE shows the same properties regarding universal approximation as neural ODEs. Theorem~\ref{th:tau0} shows that the infinite-dimensional state space of DDEs for $\tau>0$ does not by itself remove the geometric obstruction of non-augmented neural ODEs. The expressivity of neural DDEs depends non-trivially on the delay $\tau$ and the Lipschitz constant $K$.

	\begin{figure}
		\centering
		\begin{subfigure}[t]{0.45\textwidth}
			\centering
			\resizebox{\linewidth}{!}{%
			\begin{tikzpicture}[validity/.style={line width=0.8pt, ddeGray, dashed}]
						\path[use as bounding box] (-0.3,-1.0) rectangle (\ddeXmax+0.3,\ddeYmax+0.6);
						\def\Kone{1.0}
						\pgfmathsetmacro{\Dgrey}{1.35*\ddeTval*\Kone}
						\pgfmathsetmacro{\Kgreystart}{\Dgrey/\ddeTval}
						\pgfmathsetmacro{\Dorange}{0.65*\Dgrey}
						\pgfmathsetmacro{\Korangestart}{\Dorange/\ddeTval}
						\draw[axisarrow] (0,0) -- (\ddeXmax,0) node[right, ddeGrayDark] {$K$};
						\draw[axisarrow] (0,0) -- (0,\ddeYmax) node[above, ddeGrayDark] {$\tau$};
						\fill[fillorange] (0,0) rectangle (\Korangestart,\ddeTval);
						\fill[fillorange, domain=\Korangestart:\ddeXmax-0.2, samples=150]
							plot (\x,{\Dorange/\x}) -- (\ddeXmax-0.2,0) -- (\Korangestart,0) -- cycle;
						\draw[lineorange, domain=\Korangestart:\ddeXmax-0.2, samples=150, smooth]
							plot (\x,{\Dorange/\x});
						\node[ddeNeg, above] at (\Korangestart,\ddeTval) {$\tau_0(K)$};
						\draw[validity, domain=\Kgreystart:\ddeXmax-0.2, samples=100, smooth]
							plot (\x,{\Dgrey/\x});
						\pgfmathsetmacro{\Kgreymid}{(\Kgreystart+\ddeXmax-0.2)/2}
						\node[ddeGray, above, yshift=1mm] at (\Kgreymid,{\Dgrey/\Kgreymid}) {$K\tau=\frac{1}{e}$};
						\draw[guide] (0,\ddeTval) -- (\ddeXmax-0.2,\ddeTval);
						\node[left] at (0,\ddeTval) {$T$};
					\end{tikzpicture}%
			}
			\caption{Parameter region $A_\text{nUA}$ in red of non-universal approximation for fixed $K_\Psi$, $A$, $b_\infty$, $\sigma_0$, $w$ and $\tilde w$ for the neural DDE architecture $\textup{NDDE}_{\tau,\textup{N},K}^k(\mathcal{X},\mathbb{R}^q)$, $k \geq 0$, with $m = n$, cf.\ Theorem \ref{th:tau0}.}
			\label{fig:region_nua}
		\end{subfigure}
		\begin{subfigure}[t]{0.05\textwidth}
			\textcolor{white}{.}
		\end{subfigure}
		\begin{subfigure}[t]{0.45\textwidth}
			\centering
			\resizebox{\linewidth}{!}{%
			\begin{tikzpicture}
						\path[use as bounding box] (-0.3,-1.0) rectangle (\ddeXmax+0.3,\ddeYmax+0.6);
						\draw[axisarrow] (0,0) -- (\ddeXmax,0) node[right, ddeGrayDark] {$K$};
						\draw[axisarrow] (0,0) -- (0,\ddeYmax) node[above, ddeGrayDark] {$\tau$};
						\fill[fillorange] (0,0) rectangle (\ddeXmax-0.2,\ddeTval);
						\draw[lineorange] (0,\ddeTval) -- (\ddeXmax-0.2,\ddeTval);
						\node[left] at (0,\ddeTval) {$T$};
					\end{tikzpicture}%
			}
			\caption{Parameter region in red of non-universal approximation for the neural DDE architecture $\textup{NDDE}^k_{\tau}(\mathcal{X},\mathbb{R}^q)$, $k\geq 0$, with $m<n$ or $m<q$, which covers all $(K,\tau)\in[0,\infty)\times[0,T]$, cf.\ Theorem~\ref{th:bottleneck} and Theorem~\ref{th:outputbottleneck}.}
			\label{fig:region_nua_bottleneck}
		\end{subfigure}
		\caption{Regions of non-universal approximation for neural DDEs with $\mathcal{X}\subset\mathbb{R}^n$ open.}
		\label{fig:regions_nua}
	\end{figure}

	\medskip

	In the following, we discuss two results in which neural DDEs fail to possess the universal approximation property, regardless of the delay. If $m<n$, the phase space forms a bottleneck, as the affine linear layer $\lambda$ then has a non-trivial kernel.

	\begin{theorem}[No Universal Approximation for Neural DDEs with $m<n$]\label{th:bottleneck}
		Let $\mathcal{X}\subset\mathbb{R}^n$ be open. Then for all $k\geq 0$, $\tau\in[0,T]$, and $m<n$, the class of neural DDEs $\textup{NDDE}^k_{\tau,\textup{N}}(\mathcal{X},\mathbb{R}^q)$ does not have the universal approximation property with respect to $C^j(\mathcal{X},\mathbb{R}^q)$ for any $j\geq 0$.
	\end{theorem}

	\begin{proof}
		Fix a point $x^\ast\in\mathcal{X}$, a radius $r>0$ such that the closed Euclidean ball 
		\begin{equation*}
			K_r(x^\ast) \coloneqq \left\{ x \in \mathbb{R}^n: \; \sum_{i =1}^n (x_i-x_i^\ast)^2 \leq r^2 \right\}
		\end{equation*}
		satisfies $K_r(x^\ast)\subset\mathcal{X}$, a constant $\eps>0$ with $2\eps<r^2$, and a target function $\Psi\in C^\infty(\mathcal{X},\mathbb{R}^q)$ with $\Psi_i(x) = \norm{x-x^\ast}_2^2$ for some $i\in\{1,\ldots,q\}$ and all other components equal to zero. By \eqref{eq:DDEaffinelinear}, every neural DDE $\Phi\in\textup{NDDE}^k_{\tau,\textup{N}}(\mathcal{X},\mathbb{R}^q)$ depends on $x$ only through $\lambda(x) = Wx+b$. As $\textup{rank}(W)\leq m<n$, there exists $v\in\ker W$ with $\norm{v}_2 = r$. For the point $s\coloneqq x^\ast+v\in\partial K_r(x^\ast)$ it holds $\lambda(s) = \lambda(x^\ast)$ and $\Phi_i(s) = \Phi_i(x^\ast)$. Using the explicit form of the target function $\Psi$, we obtain
		\begin{equation*}
			r^2 = \abs{\Psi_i(s)-\Psi_i(x^\ast)} \leq \abs{\Psi_i(s)-\Phi_i(s)} + \abs{\Phi_i(x^\ast)-\Psi_i(x^\ast)} \leq 2\norm{\Phi-\Psi}_{\infty,K_r(x^\ast)}.
		\end{equation*}
		Hence, $\norm{\Phi-\Psi}_{\infty,K_r(x^\ast)}\geq\frac{r^2}{2}>\eps$ for every considered neural DDE $\Phi$, where the compact set $K_r(x^\ast)$ does not depend on $\Phi$. As $\Psi\in C^j(\mathcal{X},\mathbb{R}^q)$ for every $j\geq 0$, the claim follows.
	\end{proof}

	The parameter region of non-universal approximation for the class $\textup{NDDE}^k_{\tau}(\mathcal{X},\mathbb{R}^q)$, $k\geq 0$, with $m<n$ is visualized in Figure~\ref{fig:regions_nua}\subref{fig:region_nua_bottleneck}. In this case, the obstruction is independent of the delay $\tau$ and the Lipschitz constant $K$. Analogously to Theorem~\ref{th:bottleneck}, where the affine linear layer $\lambda$ obstructs universal approximation for $m<n$, the affine linear layer $\tilde\lambda$ obstructs it for $m<q$. Again, the obstruction is independent of the delay, the Lipschitz constant, and the augmentation.

	\begin{theorem}[No Universal Approximation for Neural DDEs with $m<q$]\label{th:outputbottleneck}
		Let $\mathcal{X}\subset\mathbb{R}^n$ be open. Then for all $k\geq 0$, $\tau\in[0,T]$, and $m<q$, the class of neural DDEs $\textup{NDDE}^k_{\tau}(\mathcal{X},\mathbb{R}^q)$ does not have the universal approximation property with respect to $C^j(\mathcal{X},\mathbb{R}^q)$ for any $j\geq 0$.
	\end{theorem}

	\begin{proof}
		Fix $m+2$ pairwise distinct points $x^{(0)},\ldots,x^{(m+1)}\in\mathcal{X}$, the compact set $\mathcal{K}\coloneqq\{x^{(0)},\ldots,x^{(m+1)}\}$, and the minimal distance $d\coloneqq\min_{i\neq l}\norm{x^{(i)}-x^{(l)}}_2>0$. Consider a cutoff function $\zeta\in C^\infty(\mathbb{R}^n,\mathbb{R})$ with $\zeta(0) = 1$ and $\zeta(z) = 0$ for all $z\in\mathbb{R}^n$ with $\norm{z}_2\geq d$, which exists by \cite{Hormander2003}. Define the target function $\Psi\in C^\infty(\mathcal{X},\mathbb{R}^q)$ componentwise by
		\begin{equation*}
			[\Psi(x)]_i\coloneqq\begin{cases}\zeta(x-x^{(i)}), & i = 1,\ldots,m+1,\\ 0, & i = m+2,\ldots,q.\end{cases}
		\end{equation*}
		As $\norm{x^{(l)}-x^{(i)}}_2\geq d$ for $l\neq i$, it holds $\Psi(x^{(0)}) = 0$, and $\Psi(x^{(i)})$ is the $i$-th standard unit vector $\mathrm{e}_i$ of~$\mathbb{R}^q$ for $i = 1,\ldots,m+1$.

		By \eqref{eq:DDEaffinelinear}, every neural DDE $\Phi\in\textup{NDDE}^k_{\tau}(\mathcal{X},\mathbb{R}^q)$ satisfies $\Phi(x) = \widetilde{W}y(0,c_{\lambda(x)})(T)+\tilde b$, so the $m+1$ vectors defined by
		\begin{equation*}
			\Phi(x^{(i)})-\Phi(x^{(0)}) = \widetilde{W}\left(y(0,c_{\lambda(x^{(i)})})(T)-y(0,c_{\lambda(x^{(0)})})(T)\right), \qquad i = 1,\ldots,m+1,
		\end{equation*}
		lie in the image of $\widetilde{W}\in\mathbb{R}^{q\times m}$, which has dimension at most $m$. Hence, these vectors are linearly dependent, and there exists a vector $\nu\in\mathbb{R}^{m+1}\setminus\{0\}$ with
		\begin{equation*}
			\sum_{i=1}^{m+1}\nu_i\left(\Phi(x^{(i)})-\Phi(x^{(0)})\right) = 0.
		\end{equation*}
		As the vectors $\Psi(x^{(i)})-\Psi(x^{(0)})$, $i = 1,\ldots,m+1$, are the first $m+1$ standard unit vectors of $\mathbb{R}^q$, the vector $\sum_{i=1}^{m+1}\nu_i\left(\Psi(x^{(i)})-\Psi(x^{(0)})\right)$ has the components $\nu_1,\ldots,\nu_{m+1}$ followed by zeros. Consequently, it holds
		\begin{align*}
			\norm{\nu}_\infty &= \norm{\sum_{i=1}^{m+1}\nu_i\left(\Psi(x^{(i)})-\Psi(x^{(0)})\right)}_\infty \\
			&= \norm{\sum_{i=1}^{m+1}\nu_i\left(\left(\Psi(x^{(i)})-\Phi(x^{(i)})\right)-\left(\Psi(x^{(0)})-\Phi(x^{(0)})\right)\right)}_\infty \\
			&\leq \sum_{i=1}^{m+1}\abs{\nu_i}\cdot 2 \cdot \norm{\Phi-\Psi}_{\infty,\mathcal{K}} \leq 2(m+1)\norm{\nu}_\infty\norm{\Phi-\Psi}_{\infty,\mathcal{K}}.
		\end{align*}
		As $\nu\neq 0$, dividing by $\norm{\nu}_\infty$ yields $\norm{\Phi-\Psi}_{\infty,\mathcal{K}}\geq\frac{1}{2(m+1)}$ for every considered neural DDE $\Phi$, where the compact set $\mathcal{K}$ does not depend on $\Phi$. As $\Psi\in C^j(\mathcal{X},\mathbb{R}^q)$ for every $j\geq 0$, the claim follows.
	\end{proof}

	Together with Theorem~\ref{th:bottleneck}, the parameter region of non-universal approximation for the class $\textup{NDDE}^k_{\tau}(\mathcal{X},\mathbb{R}^q)$, $k\geq 0$, with $m<q$ is visualized in Figure~\ref{fig:regions_nua}\subref{fig:region_nua_bottleneck}.

	\subsection{Universal Embedding with Large Memory Capacity or Augmentation}
	\label{sec:augmented}

	The previous section showed in Theorem~\ref{th:tau0} that non-augmented neural DDEs with small memory capacity cannot approximate all maps with a prescribed global Lipschitz constant. We now prove the positive counterpart in two complementary forms. Theorem~\ref{th:universal_embedding} shows that for $m\geq\max\{n,q\}$, any globally Lipschitz continuous map $\Psi$ can be embedded exactly on bounded open sets if $K\tau$ lies above an explicit threshold depending on the Lipschitz constant $K_\Psi$ of $\Psi$. Theorem~\ref{th:universalaugmented} shows that if instead $m\geq n+q$, universal embedding holds for every $\tau\in[0,T]$, using only the additional dimensions of the vector field and not the delay.

	In the following theorem we show that if the memory capacity $K\tau$ is sufficiently large, the neural DDE architecture $\textup{NDDE}^k_{\tau,K}(\mathcal{X},\mathbb{R}^q)$ with $m\geq\max\{n,q\}$, $k \geq 0$, has the universal embedding property with respect to globally Lipschitz continuous functions in $C^0(\mathbb{R}^n,\mathbb{R}^q)$ for $k = 0$ or $C^k_b(\mathbb{R}^n,\mathbb{R}^q)$ for $k \geq 1$, restricted to bounded open sets $\mathcal{X}\subset\mathbb{R}^n$. Hence, the universal embedding property is shown for every $\tau\in(0,T]$, but with the drawback that the global Lipschitz constant of the underlying vector field increases if the delay becomes smaller.

	\begin{theorem}[Universal Embedding for Neural DDEs with Large Memory Capacity] \label{th:universal_embedding}
		Let $k \geq 0$, $\mathcal{X}\subset\mathbb{R}^n$ open and bounded, and $\Psi\in C^0(\mathbb{R}^n, \mathbb{R}^q)$ if $k = 0$ or $\Psi\in C^k_b(\mathbb{R}^n,\mathbb{R}^q)$ if $k\geq 1$, with global Lipschitz constant $K_\Psi\in[0,\infty)$ with respect to the max-norm. Fix constants $K \in [0,\infty)$, $\tau\in(0,T]$, $w,\tilde{w}\in(0,\infty)$, and $m\geq\max\{n,q\}$, such that
		\begin{equation*}
			K\tau  \geq2\left(1+\frac{K_\Psi}{w \tilde{w}}\right).
		\end{equation*}
		Then there exists a neural DDE $\Phi \in \textup{NDDE}_{\tau,K}^k(\mathcal{X},\mathbb{R}^q)$ embedding the restriction $\Psi|_{\mathcal{X}}$ with the following properties:
		\begin{itemize}
			\item The underlying vector field $F: \mathbb{R}\times \mathcal{C}\rightarrow \mathbb{R}^m$ is globally Lipschitz continuous in the second variable on $\mathbb{R}\times \mathcal{C}$  with Lipschitz constant $K\in[0,\infty)$ and satisfies $F(t,0) = 0$ for all $t\in\mathbb{R}$.
			\item The neural DDE $\Phi$ has weight matrices $W$, $\widetilde{W}$ with $\normm{W}_\infty = w$ and $\normm{\widetilde{W}}_\infty = \tilde w$, and bias vector $b = 0$.
		\end{itemize}
	\end{theorem}

	\begin{proof}
		Fix $\tau\in(0,T]$, $K \in[0,\infty)$, $w,\tilde{w}\in(0,\infty)$, and $m\geq\max\{n,q\}$ fulfilling the assumptions of the theorem. For the neural DDE $\Phi$ we choose the weights and biases $W = w \cdot \text{Id}_{m,n} \in \mathbb{R}^{m \times n}$, $b = 0 \in \mathbb{R}^m$, $\widetilde{W} = \tilde{w} \cdot \text{Id}_{q,m} \in \mathbb{R}^{q \times m}$, $\tilde{b} = \Psi(0) \in \mathbb{R}^q$, where $\text{Id}_{m,n}$ denotes the matrix in $\mathbb{R}^{m\times n}$, which has ones on the diagonal and zeros everywhere else. In order to embed the restriction $\Psi|_{\mathcal{X}}$, we define the DDE
		\begin{equation}\label{eq:DDE_universal}
			\begin{aligned}
				\frac{\dd y}{\dd t} &= F(t,y_t) \qquad &&\text{for } t\geq 0, \\
				y(t) &= c_{\lambda(x)}(t) &&\text{for } t \in [-\tau,0],
			\end{aligned}
		\end{equation}
		with constant initial data $c_{\lambda(x)}:[-\tau,0]\rightarrow \mathbb{R}^m$, $\lambda(x) = Wx$, $x \in \mathcal{X}$ and continuous vector field $F: \mathbb{R}\times \mathcal{C}\rightarrow \mathbb{R}^m$. As $\mathcal{X}$ is bounded, there exists $r>0$ with $w\norm{x}_\infty\leq r$ for all $x\in\mathcal{X}$. Let $\textup{sat}_r:\mathbb{R}^m\rightarrow\mathbb{R}^m$ be a saturation function with radius $r$ as in Definition~\ref{def:saturation}. As $\norm{\lambda(x)}_\infty = w\norm{x}_\infty$, it holds $\textup{sat}_r(\lambda(x)) = \lambda(x)$ for all $x\in\mathcal{X}$ by Lemma~\ref{lem:saturation}. The vector field $F$ is defined as
		\begin{equation}\label{eq:vectorfield_universal}
			F(t,y_t) = \begin{cases} F(0,y_t), & t\in(-\infty,0), \\[1mm] \frac{2(\tau-t)}{\tau^2} \biggl(\frac{1}{\tilde w}\cdot \textup{Id}_{m,q}\cdot \Bigl(\Psi\Bigl(\frac{1}{w}\cdot \textup{Id}_{n,m}\cdot \textup{sat}_r\bigl(y(t-\tau_{R,A}(t))\bigr)\Bigr)-\Psi(0)\Bigr) & \\ \hspace{15mm} -\textup{sat}_r\bigl(y(t-\tau_{R,A}(t))\bigr)\biggr), \qquad &t \in [0,\tau], \\[1mm] 0, &t \in (\tau,\infty),
			\end{cases}
		\end{equation}
		where $\tau_{R,A}:[0,T]\rightarrow [0,\tau]$, $\tau_{R,A}(t) = R\delta =  \tau$ is a delay function of type A defined in Section \ref{sec:modeling_discretization}. As $\Psi\in C^0(\mathbb{R}^n,\mathbb{R}^q)$ if $k = 0$ or $\Psi\in C^k_b(\mathbb{R}^n,\mathbb{R}^q)$ if $k\geq 1$, and $\textup{sat}_r\in C^\infty_b(\mathbb{R}^m,\mathbb{R}^m)$ is bounded, it follows $F\in C^{0,0}(\mathbb{R}\times\mathcal{C},\mathbb{R}^m)$ if $k = 0$ or $F\in C_b^{0,k}(\mathbb{R}\times\mathcal{C},\mathbb{R}^m)$ if $k\geq 1$. Moreover, it holds $F(t,0) = 0$ for all $t\in\mathbb{R}$, as $\textup{sat}_r(0) = 0$. For the initial function $u = c_{\lambda(x)}$ it holds $\frac{1}{w}\cdot \textup{Id}_{n,m} \cdot \textup{sat}_r(\lambda(x)) = x$ as $m\geq n$. Hence, for $t\in[0,\tau]$, the DDE \eqref{eq:DDE_universal} is solved by
		\begin{align*}
			y(t) &= \lambda(x)+ \left(\frac{1}{\tilde w}\cdot \textup{Id}_{m,q}\cdot (\Psi(x)-\Psi(0))-\lambda(x)\right) \int_0^t \frac{2(\tau-s)}{\tau^2}  \, \dd s \\
			&= Wx + \left(\frac{1}{\tilde w}\cdot\textup{Id}_{m,q}\cdot(\Psi(x)-\Psi(0))-Wx\right)\left(\frac{t(2\tau-t)}{\tau^2}\right) \\
			\Rightarrow \qquad  y(\tau) &= Wx+ \frac{1}{\tilde w}\cdot \textup{Id}_{m,q}\cdot (\Psi(x)-\Psi(0))-Wx \\
			&= \frac{1}{\tilde w}\cdot \textup{Id}_{m,q}\cdot (\Psi(x)-\Psi(0)).
		\end{align*}
		As for $t \in (\tau,T]$ the vector field is zero, it follows $y(t) = y(\tau)$ for all $t\in(\tau,T]$, such that
		\begin{equation*}
			\Phi(x) = \tilde \lambda (y(T)) = \widetilde{W}\cdot \frac{1}{\tilde w}\cdot \textup{Id}_{m,q}\cdot (\Psi(x)-\Psi(0))+\Psi(0) = \Psi(x) \qquad \text{ for all } x \in \mathcal{X}
		\end{equation*}
		as $m\geq q$. It remains to show that $F$ satisfies Assumption~\ref{ass:vectorfield} for the set of initial values $\mathcal{A} = \lambda(\mathcal{X})$. The conditions on the domain are fulfilled, as $F$ is defined on $\mathbb{R}\times\mathcal{C}$. Let $g:[0,\tau]\times\mathbb{R}^m\rightarrow\mathbb{R}^m$ be the continuous map with $F(t,y_t) = g(t,y(t-\tau))$ for $t\in[0,\tau]$. We use the method of steps (cf.~\cite{Hale1993}). For $t\in[0,\tau]$, the delayed argument $t-\tau$ lies in the initial interval $[-\tau,0]$, so every solution $y$ with initial function $c_{y_0}$, $y_0\in\mathcal{A}$, satisfies $\frac{\dd y}{\dd t}(t) = g(t,y_0)$ for $t\in[0,\tau]$ and $\frac{\dd y}{\dd t}(t) = 0$ for $t\in(\tau,T]$. Hence, every such solution is given by
		\begin{equation*}
			y(0,c_{y_0})(t) = y_0+\int_0^{\min\{t,\tau\}}g(s,y_0)\,\dd s, \qquad t\in[0,T],
		\end{equation*}
		which is the solution computed above for $y_0 = \lambda(x)$. This shows uniqueness, and as $g$ is continuous, the map $y_0\mapsto y(0,c_{y_0})(t)$ is continuous on $\mathcal{A}$ for each fixed $t\in[0,T]$. $\Phi$ has by construction a delay $\tau\in(0,T]$, weight matrices $W$, $\widetilde{W}$ with $\normm{W}_\infty = w$ and $\normm{\widetilde{W}}_\infty = \tilde w$, and bias vector $b = 0$. The chosen vector field \eqref{eq:vectorfield_universal} is globally Lipschitz continuous in the second variable with Lipschitz constant~$K$, as for all $t \in [0,\tau]$ and $y_t,z_t \in \mathcal{C}$ it follows
		\begin{align*}
			&\hspace{4mm}\norm{F(t,y_t)-F(t,z_t)}_\infty \\
			& \leq \frac{2}{\tau}\Bigl(\frac{1}{\tilde w}\norm{\Psi\Bigl(\frac{1}{w}\cdot\textup{Id}_{n,m}\cdot \textup{sat}_r(y(t-\tau))\Bigr)-\Psi\Bigl(\frac{1}{w}\cdot\textup{Id}_{n,m}\cdot \textup{sat}_r(z(t-\tau))\Bigr)}_\infty + \norm{y(t-\tau)-z(t-\tau)}_\infty \Bigr) \\
			& \leq \frac{2}{\tau}\left(\frac{K_\Psi}{\tilde w}\norm{\frac{1}{w}\cdot \textup{Id}_{n,m}\cdot\bigl(\textup{sat}_r(y(t-\tau))-\textup{sat}_r(z(t-\tau))\bigr)}_\infty + \norm{y(t-\tau)-z(t-\tau)}_\infty \right) \\
			& \leq \frac{2}{\tau}\left(1+\frac{K_\Psi}{w \tilde{w}}\right)\norm{y(t-\tau)-z(t-\tau)}_\infty \leq \frac{2}{\tau}\left(1+\frac{K_\Psi}{w \tilde{w}}\right)\norm{y_t-z_t}_\infty \leq K\norm{y_t-z_t}_\infty.
		\end{align*}
		In the first step of the proof we used that for $t\in[0,\tau]$, it holds $\frac{2(\tau-t)}{\tau^2}\leq\frac{2}{\tau}$ and 
		\begin{equation*}
			\norm{\textup{sat}_r(y_1)-\textup{sat}_r(y_2)}_\infty\leq\norm{y_1-y_2}_\infty
		\end{equation*}
		by Lemma~\ref{lem:saturation}. In the second step, the global Lipschitz continuity of $\Psi$ is used. The last step follows from the given parameter inequality. For $t>\tau$, the Lipschitz condition follows trivially as the vector field is zero, and for $t<0$ the estimate at $t = 0$ applies. Hence, for every choice of $\tau$, $w$, $\tilde{w}$ and $K$ fulfilling the assumptions, we constructed a neural DDE $\Phi \in \textup{NDDE}_{\tau,K}^k(\mathcal{X},\mathbb{R}^q)$ with delay $\tau$, vector field with global Lipschitz constant $K$ in the second variable on $\mathbb{R}\times \mathcal{C}$ and weight matrices with   $\normm{W}_\infty = w$ and $\normm{\widetilde{W}}_\infty = \tilde w$, embedding $\Psi|_{\mathcal{X}}$.
	\end{proof}

	\begin{remark}
		If $\Psi\in C^k_b(\mathbb{R}^n,\mathbb{R}^q)$, $k\geq 1$, then the norm of the first derivative of $\Psi$ is bounded on $\mathbb{R}^n$, defining by the mean value theorem a global Lipschitz constant $K_\Psi$. If $\Psi \in C^0(\mathbb{R}^n,\mathbb{R}^q)$, Theorem~\ref{th:universal_embedding} only treats the case that $\Psi$ is globally Lipschitz continuous with Lipschitz constant $K_\Psi \in [0,\infty)$, which can be arbitrarily large.
	\end{remark}

	\begin{remark} \label{rem:universal}
		The vector field~\eqref{eq:vectorfield_universal} of Theorem \ref{th:universal_embedding} has a similar construction as the vector field~\eqref{eq:vectorfield_zhu} of Theorem \ref{th:universal_zhu} for basic non-augmented neural DDEs. The pre-factor $\frac{2(\tau-t)}{\tau^2}$ in \eqref{eq:vectorfield_universal} instead of $\frac{1}{T}$ in \eqref{eq:vectorfield_zhu} is necessary to guarantee continuity in the case that $\tau <T$ and the definition of the vector field is different on the time intervals $[0,\tau]$ and $(\tau,T]$.
	\end{remark}

	The area of universal embedding defined by the parameter constraint of Theorem \ref{th:universal_embedding} is given by
	\begin{equation*}
		A_\text{UE} \coloneqq \left\{(K,\tau) \in [0,\infty) \times [0,T]: K \tau \geq2\left(1+\frac{K_\Psi}{w \tilde{w}}\right)\right\}
	\end{equation*}
	and visualized in Figure \ref{fig:regions_ue}\subref{fig:region_ue}. For a fixed choice of the Lipschitz constant $K_\Psi\in[0,\infty)$ of the map $\Psi$ and fixed values $w,\tilde w$ of the norms of the weight matrices $W,\widetilde{W}$, the region of universal embedding is highlighted in the $K$-$\tau$ parameter plane. In the non-colored region, Theorem \ref{th:universal_embedding} makes no statement about whether universal embedding and approximation are possible or not. This includes the case of neural ODEs with $\tau = 0$, which is a special case of the architectures treated in Theorem \ref{th:tau0}. Independently of the choice of the weights and biases in $\mathbb{V}$, and the chosen map $\Psi$, the parameter regions $A_\text{UE}$ of Theorem \ref{th:universal_embedding} and $A_\text{nUA}$ of Theorem \ref{th:tau0} never overlap as
	\begin{align*}
		A_\text{UE} &\subset \left\{(K,\tau) \in [0,\infty) \times [0,T]: K \tau \geq 2\right\}, \\
		A_\text{nUA} & \subset \left\{(K,\tau) \in (0,\infty) \times [0,T]: K \tau < \frac{1}{e}\right\},
	\end{align*}
	which are disjoint as $\frac{1}{e}<2$.

	\begin{remark}[Boundary Case $m = n\geq q$]\label{rem:boundary}
		The boundary case $m = n$ with $q\leq n$ satisfies simultaneously $n\geq m$ and $m\geq\max\{n,q\}$, and is therefore covered by both Theorem~\ref{th:tau0} and Theorem~\ref{th:universal_embedding}. It illustrates the role of the memory capacity. Let $\mathcal{X}$ be bounded, and fix $K_\Psi>0$ and the constants of Theorem~\ref{th:tau0} with $\sigma_0\leq w$. The maps in $C^k_b(\mathbb{R}^n,\mathbb{R}^q)$ with global Lipschitz constant $K_\Psi$ include the map $\Psi$ of Theorem~\ref{th:tau0}. For $(K,\tau)\in A_\text{UE}$, every such map is embedded by a neural DDE that satisfies all constraints of Theorem~\ref{th:tau0}, as the construction of Theorem~\ref{th:universal_embedding} uses $W = w\cdot\textup{Id}_n$ with $\sigma_{\min}(W) = w\geq\sigma_0$, $b = 0$, and a continuous vector field with $F(t,0) = 0$. For $(K,\tau)\in A_\text{nUA}$, no such neural DDE approximates $\Psi$ with accuracy $\eps$. As $A_\text{UE}$ and $A_\text{nUA}$ are disjoint and together cover large parts of the parameter plane, the boundary case $m = n$ shows how the memory capacity governs the transition between the two regimes for a fixed class of target maps.
	\end{remark}

	The memory capacity $K\tau$ can be interpreted as a parameter describing the transition from no universal approximation to universal embedding. Theorem~\ref{th:universal_embedding} and Theorem~\ref{th:tau0} characterize the expressivity for $(K,\tau) \in 	A_\text{UE} \cup A_\text{nUA}$. For the remaining parameter regions near the critical transition, we cannot make any statements yet and leave the question open for future work. \medskip

	\begin{figure}
		\centering
		\begin{subfigure}[t]{0.45\textwidth}
			\centering
			\resizebox{\linewidth}{!}{%
			\begin{tikzpicture}
						\path[use as bounding box] (-0.3,-1.0) rectangle (\ddeXmax+0.3,\ddeYmax+0.6);
						\def\Cval{4.6}
						\pgfmathsetmacro{\Kmin}{\Cval/\ddeTval}
						\draw[axisarrow] (0,0) -- (\ddeXmax,0) node[right, ddeGrayDark] {$K$};
						\draw[axisarrow] (0,0) -- (0,\ddeYmax) node[above, ddeGrayDark] {$\tau$};
						\fill[fillforest]
							(\Kmin,\ddeTval) --
							plot[domain=\Kmin:\ddeXmax-0.2,samples=100] (\x,{\Cval/\x}) --
							(\ddeXmax-0.2,\ddeTval) -- cycle;
						\draw[lineforest, domain=\Kmin:\ddeXmax-0.2, samples=100, smooth]
							plot (\x,{\Cval/\x});
						\draw[guide] (0,\ddeTval) -- (\Kmin,\ddeTval);
						\node[left] at (0,\ddeTval) {$T$};
						\node[ddePos!70!black, anchor=north west] at (\Kmin+0.3,\ddeYmax)
							{$K\tau \ge 2\left(1+\frac{K_\Psi}{w\tilde w}\right)$};
					\end{tikzpicture}%
			}
			\caption{Parameter region $A_\text{UE}$ in green of universal embedding for fixed $m\geq\max\{n,q\}$, $K_\Psi$, $w$ and $\tilde w$ for the neural DDE architecture $\textup{NDDE}^k_{\tau,K}(\mathcal{X},\mathbb{R}^q)$,  $k \geq 0$, cf.\ Theorem~\ref{th:universal_embedding}.}
			\label{fig:region_ue}
		\end{subfigure}
		\begin{subfigure}[t]{0.05\textwidth}
			\textcolor{white}{.}
		\end{subfigure}
		\begin{subfigure}[t]{0.45\textwidth}
			\centering
			\resizebox{\linewidth}{!}{%
			\begin{tikzpicture}
						\path[use as bounding box] (-0.3,-1.0) rectangle (\ddeXmax+0.3,\ddeYmax+0.6);
						\def\Kstar{1.4}
						\draw[axisarrow] (0,0) -- (\ddeXmax,0) node[right, ddeGrayDark] {$K$};
						\draw[axisarrow] (0,0) -- (0,\ddeYmax) node[above, ddeGrayDark] {$\tau$};
						\fill[fillforest] (\Kstar,0) rectangle (\ddeXmax-0.2,\ddeTval);
						\draw[lineforest] (\Kstar,\ddeTval) -- (\ddeXmax-0.2,\ddeTval);
						\draw[lineforest] (\Kstar,0) -- (\Kstar,\ddeTval);
						\draw[guide] (0,\ddeTval) -- (\Kstar,\ddeTval);
						\node[left] at (0,\ddeTval) {$T$};
						\draw[line width=0.9pt] (\Kstar,0.07) -- (\Kstar,-0.07);
						\node[below] at (\Kstar,-0.25) {$\frac{K_\Psi}{w\tilde w T}$};
					\end{tikzpicture}%
			}
			\caption{Parameter region $\widetilde{A}_\text{UE}$ in green of universal embedding for fixed $m\geq n+q$, $K_\Psi$, $w$ and $\tilde w$ for the neural DDE architecture $\textup{NDDE}_{\tau,\textup{A},K}^k(\mathcal{X},\mathbb{R}^q)$, $k\geq 0$, cf.\ Theorem \ref{th:universalaugmented}.}
			\label{fig:region_ue_aug}
		\end{subfigure}
		\caption{Regions $A_\text{UE}$ and $\widetilde{A}_\text{UE}$ of universal embedding for a fixed Lipschitz constant $K_\Psi$ of the map $\Psi$ and neural DDEs with $\mathcal{X}\subset\mathbb{R}^n$ open, which is bounded in the case of $A_\text{UE}$.}
		\label{fig:regions_ue}
	\end{figure}

	So far, we have analyzed how the Lipschitz constant $K$ and the delay $\tau$ influence the expressivity of neural DDEs with $m\geq\max\{n,q\}$. Another parameter, which can influence the universal embedding and approximation property, is the dimension $m$ of the underlying DDE. For augmented architectures $\textup{NDDE}_{\tau,\textup{A}}^k(\mathcal{X},\mathbb{R}^q)$, $\mathcal{X}\subset \mathbb{R}^n$ open, $k\geq 0$, it holds $m>n$. In the case of augmented neural ODEs, i.e., neural DDEs with $\tau = 0$, the assumption $m \geq n+q$ is sufficient to show that any map $\Psi \in C^k(\mathcal{X},\mathbb{R}^{q})$, $k\geq 0$, $\mathcal{X} \subset \mathbb{R}^{n}$ open, can be embedded in an augmented neural ODE  \cite{Kuehn2023}. As every ODE can be interpreted as a DDE, the same idea applies to augmented neural DDEs if $m \geq n+q$. In contrast to the neural ODE result in~\cite{Kuehn2023}, the following theorem also controls the Lipschitz constant $K$ of the vector field and the norms of the weight matrices. The theorem requires $\Psi$ to be globally Lipschitz continuous and to lie in the space $C^k_b(\mathbb{R}^n,\mathbb{R}^q)$, to guarantee a bounded vector field of the DDE.

	\begin{theorem}[Universal Embedding for Augmented Neural DDEs] \label{th:universalaugmented}
		Let $k \geq 0$, $\mathcal{X}\subset\mathbb{R}^n$ open, and $\Psi\in C^k_b(\mathbb{R}^n,\mathbb{R}^q)$ with global Lipschitz constant $K_\Psi\in[0,\infty)$ with respect to the max-norm. Fix constants $m \geq n+q$, $\tau \in [0,T]$, $w,\tilde w\in (0,\infty)$ and $K \in [0,\infty)$, such that
		\begin{equation*}
			K T \geq \frac{K_\Psi}{w \tilde{w}}.
		\end{equation*}
		Then there exists an augmented neural DDE $\Phi \in \textup{NDDE}_{\tau,\textup{A},K}^k(\mathcal{X},\mathbb{R}^q)$  embedding the restriction $\Psi|_{\mathcal{X}}$ with the following properties:
		\begin{itemize}
			\item The underlying vector field $F: \mathbb{R}\times \mathcal{C}\rightarrow \mathbb{R}^m$  is globally Lipschitz continuous in the second variable on $\mathbb{R}\times \mathcal{C}$  with Lipschitz constant $K\in[0,\infty)$.
			\item The neural DDE $\Phi$ has weight matrices $W$, $\widetilde{W}$ with $\normm{W}_\infty = w$ and $\normm{\widetilde{W}}_\infty = \tilde w$.
		\end{itemize}
	\end{theorem}

	\begin{proof}
		Fix $m \geq n+q$, $\tau\in[0,T]$, $w, \tilde w \in (0,\infty)$, $K \geq \frac{K_\Psi}{w \tilde w T}$, and define an augmented neural DDE $\Phi \in \textup{NDDE}_{\tau,\textup{A},K}^k(\mathcal{X},\mathbb{R}^q)$ with weights and biases $W = w \cdot  \textup{Id}_{m,n}\in\mathbb{R}^{m \times n}$, $b = 0\in\mathbb{R}^m$, $\tilde{b} = 0 \in \mathbb{R}^q$ and
		\begin{equation*}
			\widetilde{W} = \begin{pmatrix}
				0_{q,n} && \hspace{-4mm}\tilde w \cdot \textup{Id}_{q,q} && \hspace{-4mm} 0_{q,m-n-q}
			\end{pmatrix} \in \mathbb{R}^{q \times m}.
		\end{equation*}
		In order to embed the restriction $\Psi|_{\mathcal{X}}$, we define the DDE
		\begin{equation}\label{eq:DDE_universal_aug}
			\begin{aligned}
				\frac{\dd y}{\dd t} &= F(t,y_t) \qquad &&\text{for } t\geq 0, \\
				y(t) &= c_{\lambda(x)}(t) &&\text{for } t \in [-\tau,0],
			\end{aligned}
		\end{equation}
		with constant initial data $c_{\lambda(x)}:[-\tau,0]\rightarrow \mathbb{R}^m$, $\lambda(x) = Wx$, $x \in \mathcal{X}$ and continuous vector field $	F: \mathbb{R}\times \mathcal{C}\rightarrow \mathbb{R}^m$. The vector field $F$ is defined as
		\begin{equation} \label{eq:vectorfield_aug}
			F(t,y_t) =\begin{pmatrix}
				\big[\frac{\dd y}{\dd t}\big]_{1,\ldots,n} \textcolor{white}{\ldots ....} \\
				\big[\frac{\dd y}{\dd t}	\big]_{n+1,\ldots,n+q}\textcolor{white}{.} \\
				\big[\frac{\dd y}{\dd t}\big]_{n+q+1,\ldots,m}
			\end{pmatrix}
			=
			\begin{pmatrix}
				0 \textcolor{white}{\big[\frac{1}{1}\big]_1}\\
				\frac{1}{\tilde w T} \cdot \Psi\left(\frac{1}{w}\cdot y_{1,\ldots,n}(t)\right)\\
				0 \textcolor{white}{\big[\frac{1}{1}\big]_1}
			\end{pmatrix}, \qquad t\in\mathbb{R},
		\end{equation}
		which is the vector field of an ODE independently of the delay $\tau\in[0,T]$. The initial data is only used at the point $t = 0$ and given by
		\begin{equation*}
			\begin{pmatrix}
				y_{1,\ldots,n}(0) \\
				y_{n+1,\ldots,n+q}(0) \\
				y_{n+q+1,\ldots,m}(0)
			\end{pmatrix}
			= \lambda(x) = Wx+b = \begin{pmatrix}
				wx \\ 0 \\ 0
			\end{pmatrix}.
		\end{equation*}
		With the constant initial function $c_{\lambda(x)}$ on the time interval $[-\tau,0]$, the specified initial value problem can be interpreted as a DDE. As $\Psi\in C^k_b(\mathbb{R}^n,\mathbb{R}^q)$, it follows $F\in C^{0,0}(\mathbb{R}\times\mathcal{C},\mathbb{R}^m)$ if $k = 0$ or $F\in C_b^{0,k}(\mathbb{R}\times\mathcal{C},\mathbb{R}^m)$ if $k\geq 1$. The unique, continuous solution of \eqref{eq:DDE_universal_aug} is given by
		\begin{equation*}
			\begin{pmatrix}
				y_{1,\ldots,n}(t) \\
				y_{n+1,\ldots,n+q}(t) \\
				y_{n+q+1,\ldots,m}(t)
			\end{pmatrix}
			=  \begin{pmatrix}
				wx \\ \frac{t}{\tilde w T} \cdot \Psi(x) \\ 0
			\end{pmatrix},
		\end{equation*}
		such that it follows with the affine linear layer $\tilde \lambda: y \mapsto \widetilde{W}y+\tilde b$ that
		\begin{equation*}
			\Phi(x) = \tilde{\lambda}(y(T))  = \widetilde{W}\cdot \begin{pmatrix}
				wx \\ \frac{1}{\tilde w}\cdot \Psi(x) \\ 0
			\end{pmatrix} = \Psi(x) \qquad \text{for all } x \in \mathcal{X}.
		\end{equation*}
		It remains to show that $F$ satisfies Assumption~\ref{ass:vectorfield} for the set of initial values $\mathcal{A} = \lambda(\mathcal{X})$. As the first $n$ components of $F$ vanish, every solution $y$ with initial function $c_{y_0}$, $y_0\in\mathcal{A}$, satisfies $y_{1,\ldots,n}(t) = [y_0]_{1,\ldots,n}$ for all $t\in[0,T]$. As $F(t,y_t)$ does not depend on $t$ and depends on $y_t$ only through $y_{1,\ldots,n}(t)$, it follows $\frac{\dd y}{\dd t}(t) = F(0,c_{y_0})$ for all $t\in[0,T]$. Hence, every such solution is given by $y(0,c_{y_0})(t) = y_0+tF(0,c_{y_0})$, which is the solution computed above for $y_0 = \lambda(x)$. This shows uniqueness, and as $F(0,c_{y_0})$ depends continuously on $y_0$, the map $y_0\mapsto y(0,c_{y_0})(t)$ is continuous on~$\mathcal{A}$ for each fixed $t\in[0,T]$. Hence, the defined neural DDE $\Phi$ embeds the restriction $\Psi|_{\mathcal{X}}$ and is an element of $\textup{NDDE}^k_{\tau,\textup{A},K}(\mathcal{X},\mathbb{R}^q)$. The neural DDE~$\Phi$ has by construction a delay $\tau\in[0,T]$ and the defined vector field~\eqref{eq:vectorfield_aug} is globally Lipschitz continuous in the second variable on $\mathbb{R}\times \mathcal{C}$ with Lipschitz constant $K$, as for all $t \in \mathbb{R}$ and $y_t,z_t \in \mathcal{C}$ it follows
		\begin{align*}
			\norm{F(t,y_t)-F(t,z_t)}_\infty  & \leq \norm{\frac{1}{\tilde w T}\Psi\left(\frac{1}{w}\cdot y_{1,\ldots,n}(t)\right)- \frac{1}{\tilde w T}\Psi\left(\frac{1}{w}\cdot z_{1,\ldots,n}(t)\right)}_\infty \\
			& = \frac{1}{\tilde wT}\norm{\Psi\left(\frac{1}{w}\cdot y_{1,\ldots,n}(t)\right)- \Psi\left(\frac{1}{w}\cdot z_{1,\ldots,n}(t)\right)}_\infty \\
			& \leq \frac{K_\Psi}{w \tilde w T} \norm{y_{1,\ldots,n}(t)-z_{1,\ldots,n}(t)}_\infty  \\
			& \leq K \norm{y(t)-z(t)}_\infty \leq K\norm{y_t-z_t}_\infty.
		\end{align*}
		In the proof we used in the third line that the map $\Psi$ is globally Lipschitz continuous on $\mathbb{R}^n$ with Lipschitz constant $K_\Psi\in[0,\infty)$ and the last line follows from the assumption that $\frac{K_\Psi}{w \tilde w T} \leq K$.
	\end{proof}

	The area of universal embedding of the augmented neural DDE architecture $\textup{NDDE}_{\tau,\textup{A}}^k(\mathcal{X},\mathbb{R}^q)$, $\mathcal{X}\subset \mathbb{R}^n$, $k \geq 0$ in dimension $m \geq n+q$ is given by
	\begin{equation*}
		\widetilde{A}_\text{UE} \coloneqq \left\{(K,\tau) \in [0,\infty) \times [0,T]: KT \geq \frac{K_\Psi}{w \tilde w}\right\},
	\end{equation*}
	as visualized in Figure \ref{fig:regions_ue}\subref{fig:region_ue_aug}. The corresponding result for neural ODEs is contained in Theorem~\ref{th:universalaugmented} as the special case $\tau = 0$. If $K<\frac{K_\Psi}{w \tilde w T}$, Theorem \ref{th:universalaugmented} does not make any statement about whether universal embedding or approximation is possible or not.

	\begin{figure}[b!]
		\centering
		\begin{overpic}[width=0.9\textwidth,tics=10]
			{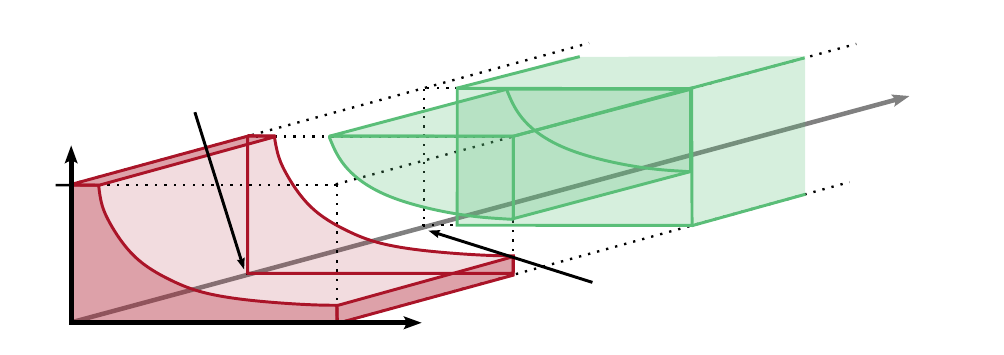}
			\put(91.5,25.3){$m$}
			\put(43,2){$K$}
			\put(2.5,16){$T$}
			\put(6.5,21.7){$\tau$}
			\put(11,26){$n = \max\{n,q\}$}
			\put(60,6){$n+q$}
		\end{overpic}
		\caption{Common visualization for the results of Theorem~\ref{th:tau0}, Theorem~\ref{th:universal_embedding} and Theorem~\ref{th:universalaugmented}  for maps $\Psi \in C^k_b(\mathbb{R}^n,\mathbb{R}^q)$, $k \geq 1$, with fixed global Lipschitz constant $K_\Psi$, restricted to a bounded open set $\mathcal{X}\subset\mathbb{R}^n$, and the neural DDE architecture $\textup{NDDE}_{\tau,K}^k(\mathcal{X},\mathbb{R}^q)$ in the three-dimensional parameter space $(K,\tau,m)\in[0,\infty) \times [0,T]\times \mathbb{N}$ for $q\leq n$. Parameter regions, where maps with global Lipschitz constant $K_\Psi$ cannot all be approximated under the bounds of Theorem~\ref{th:tau0}, are visualized in red, and parameter regions with the universal embedding property are visualized in green. For $m = n$, both regions occur, which shows the expressivity transition in the memory capacity $K\tau$, see Remark~\ref{rem:boundary}. For clarity of illustration, the figure does not show that universal approximation fails for all $(K,\tau)$ if $m<n$, cf.\ Theorem~\ref{th:bottleneck}. For $q>n$, universal approximation additionally fails for all $(K,\tau)$ if $m<q$ by Theorem~\ref{th:outputbottleneck}.}
		\label{fig:parameter4}
	\end{figure}

	In the case of scalar neural DDEs used, for example, for classification tasks, i.e., $q = 1$, all augmented architectures fulfill $m \geq n+1$, such that Theorem~\ref{th:universal_embedding} and Theorem~\ref{th:universalaugmented} give universal embedding results for all possible dimensions $m\geq n$. In the case that $\min\{n,q\} \geq 2$, our theorems do not characterize the expressivity of augmented neural DDEs if the dimension $m$ fulfills $\max\{n,q\}<m<n+q$ and the memory capacity is below the threshold of Theorem~\ref{th:universal_embedding}. The regime $\max\{n,q\}<m<n+q$ is non-empty if and only if $\min\{n,q\}\geq 2$. From the proof of Theorem \ref{th:universalaugmented} we observe that $n+q$ dimensions are necessary for the simple construction of the vector field. Hence, we expect that more advanced methods are required to prove a result for $\max\{n,q\}<m<n+q$.

	In Figure \ref{fig:parameter4} we visualize the results of Theorem~\ref{th:tau0}, Theorem~\ref{th:universal_embedding} and Theorem~\ref{th:universalaugmented} for maps $\Psi \in C^k_b(\mathbb{R}^n,\mathbb{R}^q)$, $k \geq 1$, with fixed global Lipschitz constant $K_\Psi$, restricted to a bounded open set $\mathcal{X}\subset\mathbb{R}^n$, for the case $q\leq n$, where $n = \max\{n,q\}$. Depending on the dimension $m$, the Lipschitz constant $K$ and the delay $\tau$ of the neural DDE architecture $\textup{NDDE}_{\tau,K}^k(\mathcal{X},\mathbb{R}^q)$, the areas of universal embedding $A_\text{UE}$, $\widetilde{A}_\text{UE}$ and no universal approximation $A_\text{nUA}$ are shown. The universal embedding Theorem \ref{th:universal_embedding} holds for all $m \geq \max\{n,q\}$, and the universal embedding Theorem \ref{th:universalaugmented} holds for all $m \geq n+q$. For $m \geq n+q$, the parameter region $A_\text{UE}$ is contained in the parameter region $\widetilde{A}_\text{UE}$: if $(K,\tau)\in A_\text{UE}$, then it holds
	\begin{equation*}
		KT \geq K\tau \geq 2\left(1+\frac{K_\Psi}{w \tilde{w}}\right) \geq \frac{K_\Psi}{w \tilde w},
	\end{equation*}
	as $\tau \leq T$, so it follows $(K,\tau)\in \widetilde A_\text{UE}$. Theorem \ref{th:tau0} shows that under certain assumptions, no universal approximation of all maps with global Lipschitz constant $K_\Psi$ is possible for all non-augmented architectures, i.e., for all $m \leq n$. By Theorem~\ref{th:bottleneck} and Theorem~\ref{th:outputbottleneck}, universal approximation is impossible in the entire $(K,\tau)$-plane if $m<n$ or $m<q$. Hence, the condition $m\geq\max\{n,q\}$ of Theorem~\ref{th:universal_embedding} is necessary for universal approximation. In particular, if $q>n$, augmentation alone does not suffice, as universal approximation fails for every $(K,\tau)$ in the regime $n<m<q$. For clarity, these regions are not shown in Figure~\ref{fig:parameter4}, but in Figure~\ref{fig:regions_nua}\subref{fig:region_nua_bottleneck}. Our results already cover large areas of the three-dimensional parameter space $(K,\tau,m)\in[0,\infty) \times [0,T]\times \mathbb{N}$. For future work, it remains to study the parameter regions closer to the critical transition from universal embedding to no universal approximation.

	\section{Delay Differential Equations with Small Delay}
	\label{sec:dde_small_delay}
	
	In this section, we state a few results from Chicone \cite{Chicone2003}, Jarn\'{\i}k and Kurzweil \cite{Jarnik1975}, and Ryabov and Driver \cite{Driver1962,Driver1968} about delay differential equations with small delay. With the introduced results, we prove in Section \ref{sec:sectionproof} our main Theorem \ref{th:noapproximation}, which shows that if the memory capacity $K \tau$ is sufficiently small, non-augmented neural DDEs with bounded weights and vector field cannot approximate maps with a non-degenerate local extreme point arbitrarily well.
	
	First, we state in Section \ref{sec:globalexistence} conditions which guarantee the global existence of solutions of DDEs. In the context of neural DDEs, we aim for global existence to guarantee the existence of solutions on the whole time interval $[0,T]$. In Section \ref{sec:specialsolutions} we introduce the concept of special solutions of DDEs and show their existence if the delay $\tau$ is sufficiently small. In Section~\ref{sec:expattraction}, we show that under certain conditions, all solutions of DDEs with small delay are exponentially attracted towards the finite-dimensional inertial manifold of special solutions. Afterward, we show in Section \ref{sec:specialODE} that the inertial manifold of special solutions is generated by an ordinary differential equation. Due to the exponential attraction towards solution curves of ODEs, we can use properties of ODEs to approximate the dynamics of DDEs with small delay. Finally, we connect in Section \ref{sec:NDDEsmalldelay} the results stated for DDEs with small delay to the neural DDE architectures defined in Section \ref{sec:architecture}. 
	
	\subsection{Conditions for Global Existence}
	\label{sec:globalexistence}
	
	As already introduced in Section \ref{sec:modeling}, we abbreviate by $\mathcal{C}\coloneqq C^0([-\tau,0],\mathbb{R}^m)$ the Banach space of continuous functions equipped with the sup-norm $\norm{y_t}_\infty \coloneqq \sup_{s\in[-\tau,0]} \norm{y_t(s)}_\infty$. An element $y_t \in \mathcal{C}$ is defined by $y_t(s)\coloneqq y(t+s)$ for $s\in[-\tau,0]$ and all $t\in \mathbb{R}$ such that $y(t+s)$ exists. In this section, we state results for the general non-autonomous DDE
	\begin{equation}\label{eq:DDEnonautonomous}
		\begin{aligned}
			\frac{\dd y}{\dd t} &= F(t,y_t) \qquad &&\text{for } t\geq t_0, \\  
			y_{t_0}(t) &= u(t) &&\text{for } t \in [-\tau,0],
		\end{aligned}
	\end{equation}
	with globally defined vector field $F:\mathbb{R}\times \mathcal{C}\rightarrow \mathbb{R}^m$ and initial data $u\in \mathcal{C}$. If there exists $c\in(0,\infty]$, such that $y:[t_0-\tau,t_0+c)\rightarrow \mathbb{R}^m$ is continuous and fulfills \eqref{eq:DDEnonautonomous} for $t \in [t_0-\tau, t_0+c)$, then $y(t_0,u)(t)$ is called a solution of the DDE \eqref{eq:DDEnonautonomous} with initial data $y(t_0,u)_{t_0} = u$.
	
	In Section \ref{sec:architecture}, we already discussed Lemma \ref{lem:dde_regularity} about the local existence, uniqueness and differentiable dependence of solutions of the DDE \eqref{eq:DDEnonautonomous} for vector fields $F\in C^{0,k}_b(\Omega,\mathbb{R}^m)$ with bounded derivatives up to order $k$, $k \geq 1$, in the second variable, and $\Omega \subset \mathbb{R}\times \mathcal{C}$ open. To be consistent with the literature, we state in this section the results for DDEs with small delay for the more general class of continuous vector fields, which are Lipschitz continuous with respect to the second variable. Note that $F\in C^{0,k}_b(\Omega,\mathbb{R}^m)$ with $\Omega \subset \mathbb{R}\times \mathcal{C}$ open and $k \geq 1$ implies that $F\in C^{0,0}(\Omega,\mathbb{R}^m)$ is Lipschitz continuous with respect to the second variable on every compact set $\Omega'\subset\Omega$. If in addition every set $\{\varphi\in\mathcal{C} : (t,\varphi)\in\Omega\}$, $t\in\mathbb{R}$, is convex, for example if $\Omega=\Omega_t\times\mathcal{C}$ with $\Omega_t\subset\mathbb{R}$ open, then $F$ is globally Lipschitz continuous with respect to the second variable by the mean value theorem, and the global bound on the first derivative with respect to the second variable can be taken as a global Lipschitz constant. The relevant result on local existence, uniqueness, and continuous dependence for Lipschitz continuous DDEs is stated below.
	
	\begin{theorem}[DDE Local Existence, Uniqueness, Continuous Dependence \cite{Hale1993}] \label{th:DDEexistenceuniqueness}
		Let $F\in C^{0,0}(\Omega, \mathbb{R}^m)$ with $\Omega \subset \mathbb{R}\times \mathcal{C}$ open be Lipschitz continuous with respect to the second variable on compact sets, i.e., for each compact set $\Omega' \subset \Omega$ there is a constant $K_{\Omega'}>0$ such that
		\begin{equation*}
			\norm{F(t,y_t)-F(t,z_t)}_\infty \leq K_{\Omega'} \norm{y_t-z_t}_\infty \qquad \text{for all } (t,y_t), (t,z_t) \in \Omega'.
		\end{equation*}
		Then for each initial data $(t_0,u) \in \Omega$ there exists a unique continuous solution  of \eqref{eq:DDEnonautonomous} on its maximal interval of existence, denoted by $y(t_0,u):[t_0-\tau,t_0+c)$ $\rightarrow  \mathbb{R}^m$  for some $c\in(0,\infty]$. Furthermore, the solution $y(t_0,u)$ depends continuously on the initial data $u$.
	\end{theorem}
	
	In the case that $F\in C^{0,k}_b(\Omega,\mathbb{R}^m)$ with $\Omega \subset \mathbb{R}\times \mathcal{C}$ open and $k \geq 1$, the differentiable dependence on initial conditions of Lemma~\ref{lem:dde_regularity} holds in addition to Theorem~\ref{th:DDEexistenceuniqueness}, i.e., that $y(t_0,u)(t)$ is $k$-times continuously differentiable with respect to $u$ for $t$ in any compact set in the domain of definition of $y(t_0,u)$. 
	
	For all upcoming results, we assume that the vector field $F$ of \eqref{eq:DDEnonautonomous} is only weakly nonlinear, which means $F$ fulfills the following properties: $F$ is continuous and globally defined,  globally Lipschitz continuous with respect to the second variable and there exists an upper bound for $\norm{F(t,0)}_\infty $ for all $t \in \mathbb{R}$. In Section \ref{sec:NDDEsmalldelay} we argue why it is in the context of neural DDEs no restriction to assume that the vector field $F$ is only weakly nonlinear, provided that the vector field is globally defined in its second variable, i.e., $\Omega_y = \mathcal{C}$, and, for $k = 0$, continuous and globally Lipschitz continuous with respect to its second variable.
	
	\begin{assumption}[Weakly Nonlinear Vector Field \cite{Driver1968}] \label{ass:weaklynonlinear}
		The vector field $F$ is weakly nonlinear, i.e., it fulfills the following properties:
		\begin{enumerate}[label=(\alph*), font=\normalfont]
			\item $F$ is continuous and globally defined, i.e., $F\in C^{0,0}(\mathbb{R}\times \mathcal{C}, \mathbb{R}^m)$.
			\item There exists $A \geq 0$ such that $\norm{F(t,0)}_\infty \leq A $ for all $t \in \mathbb{R}$.
			\item $F$ is globally Lipschitz continuous with respect to the second variable, i.e., there exists a constant $K>0$ such that
			\begin{equation*}
				\norm{F(t,y_t)-F(t,z_t)}_\infty \leq K \norm{y_t-z_t}_\infty \qquad \text{for all } (t,y_t), (t,z_t) \in \mathbb{R} \times \mathcal{C}.
			\end{equation*}
		\end{enumerate}
	\end{assumption}

	Theorem \ref{th:DDEexistenceuniqueness} guarantees only the existence of local unique solutions of \eqref{eq:DDEnonautonomous}. In order to define a neural DDE as in Definition \ref{def:NDDE}, we need to guarantee that the solution of the DDE exists globally on the time interval $[0,T]$. Under Assumption~\ref{ass:weaklynonlinear} that the vector field $F$ is only weakly nonlinear, global existence of unique solutions of \eqref{eq:DDEnonautonomous} can be proven for all $t \geq t_0$, as the following theorem shows.
	
	\begin{theorem}[Global Existence for DDEs with Weakly Nonlinear Vector Field \cite{Driver1962,Driver1968}] \label{th:globalexistence}
		Let $F\in C^{0,0}(\mathbb{R}\times \mathcal{C}, \mathbb{R}^m)$ fulfill Assumption \ref{ass:weaklynonlinear}. Then for each initial data $(t_0,u) \in \mathbb{R}\times \mathcal{C}$, there exists a unique continuous and globally defined solution $y(t_0,u): [t_0-\tau,\infty)\rightarrow \mathbb{R}^m$ of the DDE \eqref{eq:DDEnonautonomous}. For $u,v\in \mathcal{C}$ the following estimates hold for $t \geq t_0$:
		\begin{enumerate}[label=(\alph*), font=\normalfont]
			\item \quad $\norm{y(t_0,u)(t)}_\infty \leq \norm{u}_\infty e^{K(t-t_0)} + \frac{A}{K}\left(e^{K(t-t_0)}-1\right)$,
			\item \quad $\norm{y(t_0,u)(t)-y(t_0,v)(t)}_\infty \leq \norm{u-v}_\infty e^{K(t-t_0)}$.
		\end{enumerate}
	\end{theorem}

	\subsection{Existence of Special Solutions}
	\label{sec:specialsolutions}
	
	An important concept of DDEs with small delay $\tau>0$ are special solutions, which are characterized by global existence on all of $\mathbb{R}$ and an exponential growth bound. The general definition of a special solution of the DDE \eqref{eq:DDEnonautonomous} is given in the following. 
	
	\begin{definition}[Special Solution \cite{Chicone2003}] \label{def:specialsolution}
		A continuously differentiable function $\bar{y}:\mathbb{R}\rightarrow \mathbb{R}^m$ with $\frac{\dd \bar{y}}{\dd t}(t) = F(t,\bar{y}_t)$ for all $t\in\mathbb{R}$ is called a special solution of the DDE \eqref{eq:DDEnonautonomous} if 
		\begin{equation}\label{eq:growthcondition}
			\sup_{t \in \mathbb{R}} e^{-\frac{\abs{t}}{\tau} } \norm{\bar{y}(t)}_\infty < \infty.
		\end{equation}
	\end{definition} 
	
	For notational convenience, we always denote special solutions with an overline. If the vector field~$F$ of the DDE \eqref{eq:DDEnonautonomous} fulfills Assumption \ref{ass:weaklynonlinear} and the product of the Lipschitz constant~$K$ and the delay $\tau$ satisfies $K\tau e < 1$, then the upcoming theorem shows that for every $(t_0,y_0)\in\mathbb{R}\times \mathbb{R}^m$, a unique special solution $\bar{y}$ with $\bar{y}(t_0) = y_0$ exists. In contrast to general solutions of the DDE \eqref{eq:DDEnonautonomous}, special solutions are uniquely determined by the finite-dimensional initial data $y_0 \in \mathbb{R}^m$ at a time point $t_0 \in \mathbb{R}$. In the general case, it is necessary to define an initial function $u\in \mathcal{C}$ on the time interval $[t_0-\tau,t_0]$ to specify a unique solution. The next theorem states additionally that the unique special solution through a given point $(t_0,y_0)\in\mathbb{R}\times \mathbb{R}^m$ is the only solution with $\bar{y}(t_0) = y_0$ satisfying the growth condition \eqref{eq:growthcondition}.

	\begin{theorem}[Existence of Special Solutions \cite{Driver1968}]\label{th:specialsolutions_uniqueness} 
		Let  $F\in C^{0,0}(\mathbb{R}\times \mathcal{C}, \mathbb{R}^m)$ fulfill Assumption \ref{ass:weaklynonlinear} and assume $K \tau e < 1$. Then for every $(t_0,y_0)\in\mathbb{R}\times \mathbb{R}^m$ there exists a special solution $\bar{y}:\mathbb{R}\rightarrow \mathbb{R}^m$ of \eqref{eq:DDEnonautonomous} which satisfies $\bar{y}(t_0) = y_0$. The special solution $\bar{y}$ is the only solution of \eqref{eq:DDEnonautonomous} satisfying the growth condition \eqref{eq:growthcondition} and fulfilling $\bar{y}(t_0) = y_0$.
	\end{theorem}
	
	Hence, special solutions are characterized by an initial condition at one time point, like solutions of initial value problems of ordinary differential equations. As there exists a one-to-one correspondence between special solutions and initial conditions in $\mathbb{R}^m$, the manifold of special solutions is finite-dimensional, as discussed in the upcoming section. In the following example, we illustrate Definition~\ref{def:specialsolution} and Theorem~\ref{th:specialsolutions_uniqueness} for a one-dimensional linear DDE with discrete delay.

	\begin{example}[One-Dimensional Linear DDE] \label{ex:linear1} \normalfont
		Consider the one-dimensional linear DDE with discrete delay
		\begin{equation} \label{eq:DDEexample1}
			\frac{\dd y}{\dd t} = F(t,y_t) = K_0 y(t-\tau)
		\end{equation}
		with $K_0 \in \mathbb{R}\setminus\{0\}$ and delay $\tau>0$. By linearity, the DDE \eqref{eq:DDEexample1} is Lipschitz continuous in $y_t$ with Lipschitz constant $K = \abs{K_0}$ and Assumption \ref{ass:weaklynonlinear} is fulfilled with $A = 0$.
		
		If  $K\tau e \geq 1$, the smallness condition of Theorem \ref{th:specialsolutions_uniqueness} is not fulfilled, hence uniqueness of special solutions of the DDE \eqref{eq:DDEexample1} through a given initial condition $(t_0,y_0)\in\mathbb{R}\times \mathbb{R}$ cannot be guaranteed. If for example $K_0\tau = - \frac{\pi}{2}$, then $K \tau e = \frac{e \pi}{2} \geq 1$ and the DDE~\eqref{eq:DDEexample1} is solved by
		\begin{equation*}
			\tilde y_c(t) = y_0 \cos(K_0(t-t_0))+ c \sin(K_0(t-t_0))
		\end{equation*}
		for any constant $c \in \mathbb{R}$. Every solution $ \tilde y_c$ fulfills the initial condition $\tilde y_c(t_0) = y_0$ and the growth condition \eqref{eq:growthcondition}, as $\tilde y_c$ is a bounded function. Consequently, for  $K_0\tau = - \frac{\pi}{2}$ infinitely many special solutions of the DDE \eqref{eq:DDEexample1} exist. The solutions $\tilde y_{0}$ and $\tilde y_{1}$ are visualized in Figure \ref{fig:DDEsolutions}.

			\begin{figure}[b!]
			\centering
			\begin{tikzpicture}[x=0.58cm, y=0.95cm,
			    ax/.style={-{Stealth[length=2.0mm]}, line width=0.9pt},
			    sol/.style={line width=1.6pt},
			    lead/.style={-{Stealth[length=2.0mm]}, line width=0.8pt, ddeGrayDark},
				scale = 1.2,]
			  \pgfmathsetmacro{\Kv}{0.698132}
			  \pgfmathsetmacro{\av}{1.4}
			  \pgfmathsetmacro{\tauv}{2.25}
			  \pgfmathsetmacro{\muone}{-0.217515}
			  \pgfmathsetmacro{\mutwo}{-0.791700}
			  \fill[ddeGray!20] (-\tauv,-2.05) rectangle (0,3.4);
			  \draw[ax] (-3.6,0) -- (15.8,0) node[right, ddeGrayDark] {$t$};
			  \draw[ax] (0,-2.05) -- (0,3.4) node[above, ddeGrayDark] {$y(t)$};
			  \foreach \x in {-2.25,4.5,9}{\draw[line width=0.9pt] (\x,-0.12) -- (\x,0.12);}
			  \node[ddeCyan, below] at (-1.7,-0.18) {$t_0-\tau$};
			  \node[below] at (0.4,-0.1) {$t_0$};
			  \node[below] at (4.48,-0.1) {$t_0+\frac{\pi}{K}$};
			  \node[below] at (9,-0.1) {$t_0+\frac{2\pi}{K}$};
			  \draw[sol, ddeCyan, line cap=round] (-\tauv,\av) -- (0,\av);
			  \node[ddeCyan, anchor=east] at (-2.45,\av) {$c_{y_0}$};
			  \draw[sol, ddeSpecial, domain=-3.504:14, samples=200, smooth]
			    plot (\x,{\av*exp(\muone*\x)});
			  \draw[sol, ddeSpecial, densely dashed, domain=-0.963:14, samples=200, smooth]
			    plot (\x,{\av*exp(\mutwo*\x)});
			  \draw[sol, ddePurple, domain=-3.6:14, samples=300, smooth]
			    plot (\x,{\av*cos((\Kv*\x) r)});
			  \draw[sol, ddePurple, densely dashed, domain=-3.6:14, samples=300, smooth]
			    plot (\x,{\av*cos((\Kv*\x) r)-sin((\Kv*\x) r)});
			  \draw[sol, ddeIndigo, smooth, tension=0.7] plot coordinates
			    {(0,1.400)(0.5625,1.050)(1.125,0.700)(1.6875,0.350)(2.25,0.000)
			     (2.8125,-0.306)(3.375,-0.525)(3.9375,-0.656)(4.5,-0.700)(5.0625,-0.660)
			     (5.625,-0.554)(6.1875,-0.405)(6.75,-0.233)(7.3125,-0.062)(7.875,0.091)
			     (8.4375,0.212)(9,0.292)(9.5625,0.328)(10.125,0.324)(10.6875,0.286)
			     (11.25,0.222)(11.8125,0.143)(12.375,0.061)(12.9375,-0.016)(13.5,-0.080)};
			  \fill (0,\av) circle (1.7pt);
			  \draw[lead] (1.45,2.30) -- (0.10,1.47);
			  \node[anchor=west] at (1.40,2.42) {$y_0$};
			  \node[ddeSpecial, anchor=east] at (-1.15,3.05) {$y_{\lambda_2}$};
			  \node[ddeSpecial, anchor=south] at (5.50,0.48) {$y_{\lambda_1}$};
			  \node[ddePurple, anchor=south] at (8.11,1.78) {$\tilde y_1$};
			  \draw[lead] (14.30,-1.58) -- (14.02,-1.34);
			  \node[ddePurple, anchor=west] at (14.35,-1.62) {$\tilde y_0$};
			  \draw[lead] (11.50,0.50) -- (10.80,0.32);
			  \node[ddeIndigo, anchor=west] at (11.55,0.55) {$y(t_0,c_{y_0})$};
			\end{tikzpicture}
			\vspace{1mm}
			\caption{Solutions of the DDE \eqref{eq:DDEexample1} through the point $(t_0,y_0)$. If $K_0 \tau = -\frac{\pi}{2}$, $\tilde y_{0}$ (solid violet) and~$\tilde y_{1}$ (dashed violet) are two of infinitely many special solutions. If $-1<K_0\tau e<0$, then there exists a special solution $y_{\lambda_1}$ (solid orange) and a second globally defined solution $y_{\lambda_2}$ (dashed orange) through the point $(t_0,y_0)$, but $y_{\lambda_2}$ does not satisfy the growth condition \eqref{eq:growthcondition}. If the constant function~$c_{y_0}$ is chosen as initial data, the solution exists on the time interval $[t_0-\tau,\infty)$. The solution $y(t_0,c_{y_0})$ is plotted for some $K_0<0$ in blue and shows damped oscillatory behavior.}
			\label{fig:DDEsolutions}			
		\end{figure}

		If the smallness condition $K\tau e<1$ of Theorem \ref{th:specialsolutions_uniqueness} is fulfilled, it holds $-\frac{1}{e} < K_0 \tau<\frac{1}{e} $. The unique solution of \eqref{eq:DDEexample1} with initial condition $y(t_0) = y_0$ can be found with the exponential ansatz
		\begin{equation*}
			y_\lambda(t_0,y_0)(t) = y_0 e^{\lambda (t-t_0)}.
		\end{equation*}
		To be a solution of the DDE \eqref{eq:DDEexample1}, $y_\lambda$ has to fulfill the `characteristic equation' $\lambda = K_0e^{-\lambda \tau}$ \cite{Driver1976}. 
		This equation has at most two real solutions, determined by Lambert's W-function:
		\begin{equation} \label{eq:characteristic}
			\lambda = K_0e^{-\lambda \tau} \quad \Leftrightarrow \quad \lambda \tau e^{\lambda \tau} = K_0 \tau \quad \Rightarrow \quad \lambda_1 = \frac{W_0(K_0\tau)}{\tau}, \; \lambda_2 = \frac{W_{-1}(K_0\tau)}{\tau}.
		\end{equation}
		The solution $W_0(K_0\tau)$ exists if $-\frac{1}{e} \leq K_0 \tau $ and fulfills $-1<W_0(K_0\tau)<1$ for $-\frac{1}{e} <K_0\tau<e$. The solution $W_{-1}(K_0\tau)$ exists if $-\frac{1}{e} \leq K_0\tau < 0$ and fulfills $W_{-1}(K_0\tau) < -1$ if $-\frac{1}{e}<K_0 \tau$ \cite{Corless1996}. If the smallness condition $K\tau e<1$ is fulfilled, it holds $-1<K_0\tau e< 1$, such that the solution $\lambda_1$ of the characteristic equation, corresponding to the branch $W_0$,  exists and fulfills $-\frac{1}{\tau}<\lambda_1<\frac{1}{\tau}$. If $K\tau e<1$ and  $K_0 \tau < 0$,  additionally the solution $\lambda_2$ corresponding to the branch $W_{-1}$ exists and fulfills $\lambda_2 < -\frac{1}{\tau}$ \cite{Driver1968}. The two branches $W_0$ and $W_{-1}$ of Lambert's W-function and the solutions $\lambda_1$ and $\lambda_2$ are  visualized in Figure~\ref{fig:lambertsW}.

		\begin{figure}[h!]
			\vspace{3mm}
		\centering
		\begin{tikzpicture}
				\pgfmathsetmacro{\tauval}{0.45}
				\pgfmathsetmacro{\invtauval}{1/\tauval}
				\pgfmathsetmacro{\tendL}{-1.7}
				\pgfmathsetmacro{\yend}{\tendL/\tauval}
				\pgfmathsetmacro{\tendW}{\yend}
				\begin{axis}[
					width=9.5cm, height=8cm,
					axis lines=middle,
					axis line style={-{Stealth[length=2.0mm]}, line width=0.9pt},
					xmin=-0.95, xmax=3.05,
					ymin=-4.1, ymax=2.35,
					xtick={-0.367879441,0.367879441,2.718281828},
					xticklabels={$-\frac{1}{e}$,$\frac{1}{e}$,$e$},
					ytick={-\invtauval,-1,1,\invtauval},
					yticklabels={$-\frac{1}{\tau}$,$-1$,$1$,$\frac{1}{\tau}$},
					tick align=outside,
					tick style={thin, black},
					tick label style={font=\small},
					xmajorgrids=true, ymajorgrids=true,
					grid style={dotted, ddeGray!70},
					xlabel={$K_0\tau$},
					xlabel style={at={(current axis.right of origin)}, anchor=west, xshift=2pt, text=ddeGrayDark},
					clip=true,
					axis on top,
					scale = 0.9
					]
					\addplot[ddeSpecial!45, line width=1.6pt, domain=-1:1, samples=120, variable=t]
						({t*exp(t)},{t});
					\addplot[ddeSpecial!45, line width=1.6pt, dotted, domain=\tendW:-1, samples=120, variable=t]
						({t*exp(t)},{t});
					\addplot[ddeSpecial, line width=1.6pt, domain=-1:1, samples=120, variable=t]
						({t*exp(t)},{t/\tauval});
					\addplot[ddeSpecial, line width=1.6pt, dotted, domain=\tendL:-1, samples=120, variable=t]
						({t*exp(t)},{t/\tauval});
					\node[ddeSpecial!45, anchor=west] at (axis cs:2.0,0.63) {$W_0$};
					\node[ddeSpecial!45, anchor=east] at (axis cs:-0.42,-1.6) {$W_{-1}$};
					\node[ddeSpecial, anchor=west] at (axis cs:1.45,1.95) {$\lambda_1$};
					\node[ddeSpecial, anchor=east] at (axis cs:-0.42,-3.35) {$\lambda_2$};
				\end{axis}
			\end{tikzpicture}
		\vspace{0mm}
		\caption{The two branches $W_0$ and $W_{-1}$ of Lambert's W-function and the solutions $\lambda_1$ and $\lambda_2$ of the characteristic equation of the DDE \eqref{eq:DDEexample1}.}
		\label{fig:lambertsW}			
	\end{figure}
		
		The solution $y_{\lambda_1}$ of the DDE \eqref{eq:DDEexample1} is a special solution, as it is defined on all of $\mathbb{R}$ and it holds
		\begin{equation*}
			\sup_{t \in \mathbb{R}} e^{-\frac{\abs{t}}{\tau} }  \norm{e^{\lambda_1t}} \leq \sup_{t \in \mathbb{R}} e^{-\frac{\abs{t}}{\tau} }  e^{\abs{\lambda_1} \abs{t}}= \sup_{t \in \mathbb{R}}  e^{\left(\abs{\lambda_1}-\frac{1}{\tau}\right) \abs{t}} < \infty,
		\end{equation*}
		as $\abs{\lambda_1}-\frac{1}{\tau}<0$ because $-\frac{1}{\tau} < \lambda_1  < \frac{1}{\tau}$. By Theorem \ref{th:specialsolutions_uniqueness}, the special solution $y_{\lambda_1}$ is unique and it holds $\bar{y}(t_0,y_0) = y_{\lambda_1}(t_0,y_0)$. Indeed, if additionally the solution $y_{\lambda_2}$ exists, then $y_{\lambda_2}$ is no special solution: it holds
		\begin{equation*}
			\sup_{t \in \mathbb{R}} e^{-\frac{ \abs{t}}{\tau} }  \norm{e^{\lambda_2t}} \geq \lim_{t \rightarrow - \infty} e^{\left(\abs{\lambda_2}-\frac{1}{\tau}\right) \abs{ t} } = \infty, 
		\end{equation*}
		as $\abs{\lambda_2}-\frac{1}{\tau}>0$ because $\lambda_2 < -\frac{1}{\tau}$. For the case $-1< K_0 \tau e < 0$, the special solution $y_{\lambda_1}$ and the second solution $y_{\lambda_2}$ of the DDE \eqref{eq:DDEexample1} are shown in Figure \ref{fig:DDEsolutions}. 
		
		The initial condition $y(t_0) = y_0$ uniquely determines the initial function on the time interval $[t_0-\tau,t_0]$, given by the special solution $y_{\lambda_1}$ restricted to the time interval $[t_0-\tau,t_0]$. Other initial functions, such as the constant function $c_{y_0}:[t_0-\tau,t_0]\rightarrow \mathbb{R}$, $c_{y_0}(t) = y_0$, with $y_0\neq 0$, can lead to solutions which fulfill the  growth condition \eqref{eq:growthcondition} with the supremum taken over $t\geq t_0$, but which cannot be extended to a solution on all of $\mathbb{R}$. An extension of $y(t_0,c_{y_0})$ to the interval $[t_0-2\tau,\infty)$ would satisfy 
		\begin{equation*}
			0 = \frac{\dd}{\dd t}y(t_0,c_{y_0})(t) = K_0\, y(t_0,c_{y_0})(t-\tau) \quad \text{for } t\in(t_0-\tau,t_0),
		\end{equation*}
		hence $y(t_0,c_{y_0})(t) = 0$ for $t\in(t_0-2\tau,t_0-\tau)$ as $K_0\neq 0$. It follows $y(t_0,c_{y_0})(t_0-\tau) = 0$ by continuity, contradicting $y(t_0,c_{y_0})(t_0-\tau) = y_0$. The solution $y(t_0,c_{y_0})$ of the DDE \eqref{eq:DDEexample1} with constant initial data $c_{y_0}$ is given~by
		\begin{equation*}
			y(t_0,c_{y_0})(t) = y_0 \left[ 1 + \sum_{j = 1}^{l} K_0^j \frac{(t-t_0-(j-1)\tau)^j}{j!}\right]
		\end{equation*}
		for $t\in[t_0+(l-1)\tau,t_0+l\tau]$, $l\in\mathbb{N}$, and plotted on the time interval $[t_0-\tau,\infty)$ in Figure \ref{fig:DDEsolutions}. Depending on the sign and the magnitude of $K_0$, the solution  $y(t_0,c_{y_0})$ shows oscillatory behavior~\cite{Smith2011}.
	\end{example}

	\subsection{Exponential Attraction of Special Solutions}
	\label{sec:expattraction}
	
	Special solutions are an important concept of DDEs with small delay $\tau$, as they also characterize the behavior of general solutions of the DDE \eqref{eq:DDEnonautonomous}. Recall that a general solution $y:[t_0-\tau,t_0+c)\rightarrow\mathbb{R}^m$ of the DDE \eqref{eq:DDEnonautonomous} with $c \in(0,\infty]$ and initial data $y_{t_0}(t) = u(t)$ for $t\in[-\tau,0]$ is denoted by $y(t_0,u)$ for $(t_0,u) \in \mathbb{R}\times \mathcal{C}$. If Assumption \ref{ass:weaklynonlinear} is fulfilled, then by Theorem \ref{th:globalexistence} every solution of the DDE~\eqref{eq:DDEnonautonomous} exists on the time interval $[t_0-\tau,\infty)$. The following theorem shows that every solution of the DDE~\eqref{eq:DDEnonautonomous} is exponentially fast attracted to some special solution.
	
	\begin{theorem}[Exponential Attraction \cite{Jarnik1975}] \label{th:specialsolutions_expattraction}
		Let  $F\in C^{0,0}(\mathbb{R}\times \mathcal{C}, \mathbb{R}^m)$ fulfill Assumption \ref{ass:weaklynonlinear} and assume $K \tau e < 1$. Then for each initial data $(t_0,u)\in\mathbb{R}\times \mathcal{C}$ with solution $y(t_0,u):[t_0-\tau,\infty)\rightarrow \mathbb{R}^m$ of the DDE \eqref{eq:DDEnonautonomous}, there exists the limit
		\begin{equation}\label{eq:ICspecialsolution}
			\bar{y}_0 \coloneqq \lim_{s \rightarrow \infty} \bar{y}(s,y(t_0,u)(s))(t_0) \in \mathbb{R}^m.
		\end{equation}
		The special solution $\bar{y}(t_0,\bar{y}_0)$ of \eqref{eq:DDEnonautonomous} through $(t_0,\bar{y}_0)\in \mathbb{R}\times \mathbb{R}^m$, uniquely determined by the limit $\bar{y}_0\in \mathbb{R}^m$, corresponds to the given solution $y(t_0,u)$ and is denoted by 
		\begin{equation*}
			S(y(t_0,u)) = \bar{y}(t_0,\bar{y}_0).
		\end{equation*} 
		For the given solution and the corresponding special solution, it holds
		\begin{equation*}
			C_u \coloneqq \sup_{t \geq t_0} e^{\frac{t-t_0}{\tau}} \norm{y(t_0,u)(t)-S(y(t_0,u))(t)}_\infty < \infty.
		\end{equation*}
	\end{theorem}
	
	The limit in Theorem \ref{th:specialsolutions_expattraction} determines the initial condition $\bar{y}_0 \in\mathbb{R}^m$ at time $t_0 \in \mathbb{R}$ of the special solution $\bar{y}$, so that the special solution $\bar{y}(t_0,\bar{y}_0)$ and the given solution $y(t_0,u)$ agree asymptotically as $t\rightarrow\infty$, i.e., their difference converges to zero. The estimate of Theorem \ref{th:specialsolutions_expattraction} guarantees that every solution is exponentially fast attracted towards a special solution $\bar{y}$.

	A direct consequence of Theorem \ref{th:specialsolutions_expattraction} is that all solutions are exponentially fast attracted to the invariant and finite-dimensional manifold of special solutions of the DDE~\eqref{eq:DDEnonautonomous}. A manifold with this property is called an inertial manifold.

	\begin{corollary}[Inertial Manifold \cite{Chicone2003}] \label{cor:inertialmanifold}
		Let $F\in C^{0,0}(\mathbb{R}\times \mathcal{C}, \mathbb{R}^m)$ fulfill Assumption~\ref{ass:weaklynonlinear} and assume $K\tau e<1$. Then for every $t\in\mathbb{R}$, the set of delayed functions of special solutions at time $t$
		\begin{equation*}
			\mathcal{M}(t) \coloneqq \big\{ \bar{y}_t : \bar{y} \textup{ is a special solution of the DDE \eqref{eq:DDEnonautonomous}} \big\} = \big\{ [\bar{y}(t,y_0)]_t : y_0 \in \mathbb{R}^m \big\} \subset \mathcal{C}
		\end{equation*}
		is a Lipschitz manifold homeomorphic to $\mathbb{R}^m$. The family $\{\mathcal{M}(t)\}_{t\in\mathbb{R}}$ is an inertial manifold, i.e., it is invariant in the sense that $u\in\mathcal{M}(t_0)$ implies $y(t_0,u)_t\in\mathcal{M}(t)$ for all $t\geq t_0$, it is finite-dimensional, and it exponentially attracts all solutions of \eqref{eq:DDEnonautonomous}.
	\end{corollary}
	
	The Lipschitz-manifold property follows from the following lemma, which estimates the distance between special solutions. The lemma uses the solution $\lambda_1$ of the characteristic equation $\lambda = K_0 e^{-\lambda \tau}$ of the linear DDE of Example~\ref{ex:linear1} in the case $K_0<0$, so that $K = -K_0$.

	\begin{lemma}[Exponential Growth Bound \cite{Driver1968}] \label{lem:specialsolutions}
		Let $F\in C^{0,0}(\mathbb{R}\times \mathcal{C}, \mathbb{R}^m)$ fulfill Assumption \ref{ass:weaklynonlinear} and assume $K \tau e < 1$. Let $\lambda_1$ be the unique real solution of the characteristic equation $\lambda = -K e^{-\lambda \tau}$ with $-\frac{1}{\tau}<\lambda_1 < 0 $. Then it holds $K<\vert \lambda_1 \vert <Ke$  and it follows for the special solutions $\bar{y}(t_0,y_0)$, $\bar{y}(t_0,z_0)$ of the DDE \eqref{eq:DDEnonautonomous} with $(t_0,y_0),(t_0,z_0)\in\mathbb{R}\times \mathbb{R}^m$ that
		\begin{equation*}
			\norm{\bar{y}(t_0,y_0)(t) - \bar{y}(t_0,z_0)(t)}_\infty \leq \begin{cases}
				\norm{y_0-z_0}_\infty e^{\vert \lambda_1\vert (t_0-t)} & \text{if } t \leq t_0,\\
				\norm{y_0-z_0}_\infty e^{\vert\lambda_1\vert \tau + K(t-t_0)} & \text{if } t>t_0.
			\end{cases}
		\end{equation*} 
	\end{lemma}

	The lemma implies that for fixed $t\in\mathbb{R}$, the map $y_0\mapsto[\bar{y}(t,y_0)]_t$ from $\mathbb{R}^m$ onto $\mathcal{M}(t)$ is Lipschitz continuous, and its inverse $\varphi\mapsto\varphi(0)$ is Lipschitz continuous with Lipschitz constant $1$. The invariance of ${\{{\mathcal{M}}(t)\}}_{t\in\R}$ follows from the uniqueness of solutions in Theorem~\ref{th:globalexistence}, and the exponential attraction from Theorem~\ref{th:specialsolutions_expattraction}. \medskip
	
	In the following, we continue with the one-dimensional linear DDE with discrete delay of Example~\ref{ex:linear1}, illustrate the exponential convergence of Theorem \ref{th:specialsolutions_expattraction}, and explicitly determine the inertial manifold ${\{{\mathcal{M}}(t)\}}_{t\in\R}$ of Corollary \ref{cor:inertialmanifold}.
	
	\begin{example}[One-Dimensional Linear DDE Continued]  \label{ex:linear2} \normalfont
		Consider the one-dimensional linear DDE with discrete delay
		\begin{equation} \label{eq:DDEexample2}
			\frac{\dd y}{\dd t} = F(t,y_t) = K_0 y(t-\tau)
		\end{equation}
		with $K_0 \in \mathbb{R}\setminus\{0\}$, delay $\tau>0$ and Lipschitz constant $K = \abs{K_0}$, which fulfills Assumption \ref{ass:weaklynonlinear}.  If the smallness assumption $K\tau e<1$ is fulfilled, the unique special solution with initial condition $y(t_0) = y_0$ is by Example \ref{ex:linear1} given by
		\begin{equation*}
			y_{\lambda_1}(t_0,y_0)(t) = y_0e^{\lambda_1(t-t_0)},
		\end{equation*}
		where $\lambda_1$ is the unique real solution of the characteristic equation $\lambda = K_0 e^{-\lambda \tau}$ with $\lambda_1 > - \frac{1}{\tau}$. Consequently, for every $t\in\mathbb{R}$ the manifold of special solutions at time $t$ is given by
		\begin{equation*}
			\mathcal{M}(t) = \big\{ [y_{\lambda_1}(t_0,y_0)]_t : y_0 \in \mathbb{R} \big\} = \big\{ s\mapsto y_0e^{\lambda_1 s} : y_0 \in \mathbb{R} \big\} \subset \mathcal{C}.
		\end{equation*}
		$\mathcal{M}(t)$ is a one-dimensional linear subspace of $\mathcal{C}$ that does not depend on $t$, as the DDE \eqref{eq:DDEexample2} is autonomous. To illustrate the exponential attraction of  Theorem \ref{th:specialsolutions_expattraction}, we consider as in Example \ref{ex:linear1} the solution $y(t_0,c_{y_0}):[t_0-\tau,\infty) \rightarrow \mathbb{R}$ of \eqref{eq:DDEexample2} with constant initial data $c_{y_0} \in \mathcal{C}$. If $K\tau e <1$, then Theorem \ref{th:specialsolutions_expattraction} implies that the limit
		\begin{equation*}
			\bar{y}_0 \coloneqq \lim_{s \rightarrow \infty} \bar{y}(s,y(t_0,c_{y_0})(s))(t_0) = \lim_{s \rightarrow \infty} 	y_{\lambda_1}(s,y(t_0,c_{y_0})(s))(t_0) = \lim_{s \rightarrow \infty} y(t_0,c_{y_0})(s) e^{\lambda_1 (t_0-s)} \in \mathbb{R}
		\end{equation*}
		exists. Hence, the special solution corresponding to $y(t_0,c_{y_0})$ is given by $S(y(t_0,c_{y_0})) = \bar{y}(t_0,\bar{y}_0)$. Furthermore, there exists a tube of size $\pm C_{y_0}e^{-\frac{t-t_0}{\tau}}$ around the special solution $\bar{y}(t_0,\bar{y}_0)$, in which the solution $y(t_0,c_{y_0})$ is contained for $t\geq t_0$, as shown in Figure \ref{fig:expattraction}.
	\end{example}
	
		\begin{figure}[b!]
		\centering
		\begin{tikzpicture}[x=0.66cm, y=1.15cm,
		    ax/.style={-{Stealth[length=2.0mm]}, line width=0.9pt},
		    sol/.style={line width=1.6pt},
		    lead/.style={-{Stealth[length=2.0mm]}, line width=0.8pt, ddeGrayDark},
			scale = 1.2,]
		  \pgfmathsetmacro{\tauv}{2.25} 
		  \pgfmathsetmacro{\av}{1.4} 
		  \pgfmathsetmacro{\abarv}{1.0}
		  \pgfmathsetmacro{\muv}{-0.30}
		  \pgfmathsetmacro{\nuv}{-0.5}
		  \pgfmathsetmacro{\omv}{1.30}
		  \pgfmathsetmacro{\kappav}{1.5}
		
		  \pgfmathsetmacro{\ymin}{-0.45}
		  \pgfmathsetmacro{\ymax}{3.6}
		  \pgfmathsetmacro{\xmin}{-3.3}
		  \pgfmathsetmacro{\xmax}{10.6}
		  \pgfmathsetmacro{\xend}{9}
		
		  \fill[ddeGray!20] (-\tauv,\ymin) rectangle (0,\ymax);
		
		  \fill[ddeSpecial!12]
		    plot[domain=-\tauv:\xend, samples=150, smooth]
		      (\x,{\abarv*exp(\muv*\x)+\kappav*(\av-\abarv)*exp(\nuv*\x)})
		    -- plot[domain=\xend:-\tauv, samples=150, smooth]
		      (\x,{\abarv*exp(\muv*\x)-\kappav*(\av-\abarv)*exp(\nuv*\x)})
		    -- cycle;
		  \draw[ddeGrayDark, line width=0.8pt, densely dashed, domain=-\tauv:\xend, samples=150, smooth]
		    plot (\x,{\abarv*exp(\muv*\x)+\kappav*(\av-\abarv)*exp(\nuv*\x)});
		  \draw[ddeGrayDark, line width=0.8pt, densely dashed, domain=-\tauv:\xend, samples=150, smooth]
		    plot (\x,{\abarv*exp(\muv*\x)-\kappav*(\av-\abarv)*exp(\nuv*\x)});
		
		  \draw[ax] (\xmin,0) -- (\xmax,0) node[right, ddeGrayDark] {$t$};
		  \draw[ax] (0,\ymin) -- (0,\ymax) node[above, ddeGrayDark] {$y(t)$};
		
		  \draw[line width=0.9pt] (-\tauv,-0.06) -- (-\tauv,0.06);
		  \node[ddeCyan, below] at (-\tauv,-0.1) {$t_0-\tau$};
		  \node[below] at (0.4,-0.1) {$t_0$};
		
		  \draw[sol, ddeCyan, line cap=round] (-\tauv,\av) -- (0,\av);
		  \node[ddeCyan, anchor=east] at (-\tauv-0.2,\av) {$c_{y_0}$};
		
		  \draw[sol, ddeSpecial, domain=-\tauv:\xend, samples=150, smooth]
		    plot (\x,{\abarv*exp(\muv*\x)});
		
		\draw[sol, ddeIndigo, domain=0:\xend, samples=200, smooth]
		    plot (\x,{\abarv*exp(\muv*\x)+(\av-\abarv)*exp(\nuv*\x)*cos((\omv*\x) r)});
		
		  \fill (0,\av) circle (1.7pt);
		  \fill[ddeSpecial] (0,\abarv) circle (1.7pt);
		  \draw[lead] (1.55,2.15) -- (0.10,1.45);
		  \node[anchor=west] at (1.60,2.27) {$y_0$};
		  \draw[lead] (-1.15,0.45) -- (0,1.0);
		  \node[ddeSpecial, anchor=east] at (-1.20,0.45) {$\bar y_0$};
		
		  \node[ddeSpecial, anchor=south] at (3.9,1.25) {$\bar y(t_0,\bar y_0)$};
		  \draw[lead] (3.7,1.25) -- (2.55,0.51);
		  \node[ddeIndigo, anchor=south] at (2.05,1.0) {$y(t_0,c_{y_0})$};
		\end{tikzpicture}
		\vspace{1mm}
		\caption{Solution $y(t_0,c_{y_0})$ of the linear one-dimensional DDE \eqref{eq:DDEexample2} with constant initial data~$c_{y_0}$ and special solution $\bar{y}(t_0,\bar{y}_0)$ through the point $(t_0,\bar{y}_0)$ towards which the solution $y(t_0,c_{y_0})$ is exponentially attracted. The exponential tube around the special solution $\bar{y}(t_0,\bar{y}_0)$, which contains the general solution $y(t_0,c_{y_0})$, is highlighted in orange.}
		\label{fig:expattraction}		
	\end{figure}
	
	\subsection{Special Solutions as Solutions of ODEs}
	\label{sec:specialODE}
	
	In Section \ref{sec:specialsolutions}, we have seen that it is sufficient to specify one initial condition $y_0\in \mathbb{R}^m$ at time $t_0\in \mathbb{R}$ instead of an initial function $u \in \mathcal{C}$ to define a unique special solution of the DDE \eqref{eq:DDEnonautonomous}. Consequently, the family ${\{\mathcal{M}(t)\}}_{t\in\R}$ of special solutions is, by Corollary \ref{cor:inertialmanifold}, a finite-dimensional inertial manifold. In the following theorem, we show that special solutions are not only solutions of DDEs, but there exists for every special solution $\bar{y}$ a finite-dimensional initial value problem, which is uniquely solved by $\bar{y}$. Here, the vector field of the defined initial value problem only depends on the vector field~$F$ of the DDE \eqref{eq:DDEnonautonomous}, not on the specific special solution solving the initial value problem.
	
	\begin{theorem}[ODE for Special Solutions] \label{th:specialsolutionsODEs}
		Let  $F\in C^{0,0}(\mathbb{R}\times \mathcal{C}, \mathbb{R}^m)$ fulfill Assumption \ref{ass:weaklynonlinear} and assume $K \tau e < 1$. Then for every $(t_0,y_0)\in \mathbb{R}\times \mathbb{R}^m$, the initial value problem 
		\begin{equation} \label{eq:specialODE}
			\frac{\dd z(t)}{\dd t} = \widetilde{F}(t,z(t)) \coloneqq F(t,[\bar{y}(t,z(t))]_t), \qquad z(t_0) = y_0,
		\end{equation}
		with $\widetilde{F}\in C^{0,0}(\mathbb{R} \times \mathbb{R}^m,\mathbb{R}^m)$, is uniquely solved by the special solution $\bar{y}(t_0,y_0)$ of the DDE \eqref{eq:DDEnonautonomous}. The vector field $\widetilde{F}$ is globally Lipschitz continuous with respect to the second variable with Lipschitz constant~$\vert \lambda_1 \vert < Ke $, where $\lambda_1$ is the unique real solution of the characteristic equation $\lambda = -K e^{-\lambda \tau}$ with $-\frac{1}{\tau}<\lambda_1 < 0 $.
	\end{theorem}
	
	\begin{proof}
		The idea of the construction of the vector field $\widetilde{F}$ is stated in \cite{Chicone2003} without proof. The given initial value problem is solved by the special solution $\bar{y}(t_0,y_0)$, as
		\begin{align*}
			\frac{\dd}{\dd t} \; \bar{y}(t_0,y_0)(t) & = F(t,[\bar{y}(t_0,y_0)]_t) \\
			& =  F(t,[\bar{y}(t,\bar{y}(t_0,y_0)(t))]_t) =  \widetilde F(t,\bar{y}(t_0,y_0)(t)),
		\end{align*}
		where we used in the first step that $\bar{y}(t_0,y_0)$ is a solution of the DDE \eqref{eq:DDEnonautonomous} and in the second step that $\bar{y}(t_0,y_0) = \bar{y}(t,\bar{y}(t_0,y_0)(t))$ because the unique special solution through the point $(t_0,y_0)$ also passes through the point $(t,\bar{y}(t_0,y_0)(t))$. The last step follows from the definition of the vector field~$\widetilde{F}$. The initial condition $z(t_0) = y_0$ is fulfilled, as $\bar{y}(t_0,y_0)(t_0) = y_0$.
		
		By Assumption \ref{ass:weaklynonlinear}, the vector field $F$ is globally Lipschitz continuous with respect to the second variable with Lipschitz constant $K$. Hence, it follows for the vector field $\widetilde{F}\in C^{0,0}(\mathbb{R} \times \mathbb{R}^m,\mathbb{R}^m)$ that
		\begin{align*}
			\norm{\widetilde{F}(t,z(t))-\widetilde{F}(t,y(t))}_\infty& =  \norm{F(t,[\bar{y}(t,z(t))]_t)-F(t,[\bar{y}(t,y(t))]_t)}_\infty \\
			&\leq   K \norm{[\bar{y}(t,z(t))]_t - [\bar{y}(t,y(t))]_t}_\infty  \\
			&= K \cdot \sup_{s \in [-\tau,0]} \norm{\bar{y}(t,z(t))(t+s) - \bar{y}(t,y(t))(t+s)}_\infty \\
			& \leq  K \cdot \sup_{s \in [-\tau,0]} \norm{z(t)-y(t)}_\infty e^{-\vert \lambda_1\vert  s} \\
			& =  K e^{\vert\lambda_1 \vert \tau}\norm{z(t)-y(t)}_\infty \\
			& = \vert\lambda_1 \vert\norm{z(t)-y(t)}_\infty
		\end{align*}
		for all $(t,y(t)),(t,z(t))\in \mathbb{R}\times \mathbb{R}^m$. Consequently, $\widetilde{F}$ is globally Lipschitz continuous with respect to the second variable with Lipschitz constant $\vert \lambda_1 \vert $. In the second line, we used the global Lipschitz continuity of $F$ with respect to the second variable. In the fourth line, Lemma \ref{lem:specialsolutions} is used to estimate the difference between the two special solutions. The last step follows from the property that $\lambda_1<0$ solves $\lambda_1 = -K e^{-\lambda_1 \tau}$, so $\abs{\lambda_1} = Ke^{\abs{\lambda_1}\tau}$.   
		
		It remains to show that $\widetilde{F}$ is continuous. As $F$ is continuous, it suffices to show that the map $(t,z)\mapsto[\bar{y}(t,z)]_t$ from $\mathbb{R}\times\mathbb{R}^m$ to $\mathcal{C}$ is continuous. Let $(t,z),(t',z')\in\mathbb{R}\times\mathbb{R}^m$. The special solution $\bar{y}(t,z)$ passes through the point $(t',\bar{y}(t,z)(t'))$, hence $\bar{y}(t,z) = \bar{y}(t',\bar{y}(t,z)(t'))$. Using Lemma~\ref{lem:specialsolutions} with initial time $t'$ together with $z = \bar{y}(t,z)(t)$, we obtain
		\begin{align*}
			\norm{[\bar{y}(t',z')]_{t'} - [\bar{y}(t,z)]_t}_\infty & \leq \norm{[\bar{y}(t',z')]_{t'} - [\bar{y}(t',z)]_{t'}}_\infty \\
			& \quad + \norm{[\bar{y}(t',z)]_{t'} - [\bar{y}(t',\bar{y}(t,z)(t'))]_{t'}}_\infty \\
			& \quad + \norm{[\bar{y}(t,z)]_{t'} - [\bar{y}(t,z)]_t}_\infty \\
			& \leq e^{\vert\lambda_1\vert\tau}\norm{z'-z}_\infty + e^{\vert\lambda_1\vert\tau}\norm{\bar{y}(t,z)(t)-\bar{y}(t,z)(t')}_\infty \\
			& \quad + \norm{[\bar{y}(t,z)]_{t'} - [\bar{y}(t,z)]_t}_\infty.
		\end{align*}
		For $(t',z')\rightarrow(t,z)$, the first term converges to zero, the second term converges to zero by continuity of $\bar{y}(t,z)$, and the third term converges to zero by uniform continuity of $\bar{y}(t,z)$ on compact intervals. Hence, it follows $\widetilde{F}\in C^{0,0}(\mathbb{R}\times\mathbb{R}^m,\mathbb{R}^m)$.

		As the vector field $\widetilde{F}$ is continuous and globally Lipschitz continuous with respect to the second variable, the Picard-Lindelöf theorem (cf.\ \cite{Hartman2002}) implies that $\bar{y}(t_0,y_0)$ is the unique solution of the given initial value problem.
	\end{proof}
	
	The last theorem implies that the dynamical system on the inertial manifold ${\{{\mathcal{M}}(t)\}}_{t\in\R}$ of special solutions generated by the DDE  \eqref{eq:DDEnonautonomous} agrees with the dynamical system generated by the ODE \eqref{eq:specialODE}. For a fixed initial condition $(t_0,y_0) \in \mathbb{R}\times \mathbb{R}^m$, the special solution $\bar{y}(t_0,y_0)$ of the DDE \eqref{eq:DDEnonautonomous} agrees with the unique solution of the initial value problem \eqref{eq:specialODE}. \medskip
	
	In the following, we study again the one-dimensional linear DDE with discrete delay of Example~\ref{ex:linear1} and Example~\ref{ex:linear2} and determine the ODE \eqref{eq:specialODE} of Theorem \ref{th:specialsolutionsODEs} generating all special solutions.

	\begin{example}[One-Dimensional Linear DDE Continued]  \label{ex:linear3} \normalfont
		Consider the one-dimensional linear DDE with discrete delay
		\begin{equation} \label{eq:DDEexample3}
			\frac{\dd y}{\dd t} = F(t,y_t) = K_0 y(t-\tau)
		\end{equation}
		with $K_0 \in \mathbb{R}\setminus\{0\}$, delay $\tau>0$ and Lipschitz constant $K = \abs{K_0}$, which fulfills Assumption \ref{ass:weaklynonlinear}.  If the smallness assumption $K\tau e<1$ is fulfilled, the unique special solution with initial condition $y(t_0) = y_0$ is by Example \ref{ex:linear1} given by
		\begin{equation*}
			y_{\lambda_1}(t_0,y_0)(t) = y_0e^{\lambda_1(t-t_0)},
		\end{equation*}
		where $\lambda_1$ is the unique real solution of the characteristic equation $\lambda = K_0 e^{-\lambda \tau}$ with $\lambda_1 > - \frac{1}{\tau}$. Theorem~\ref{th:specialsolutionsODEs} states that the special solution of the DDE \eqref{eq:DDEexample3} with initial data $(t_0,y_0)\in\mathbb{R}\times\mathbb{R}$ is the unique solution of the initial value problem
		\begin{equation} \label{eq:examapleODE}
			\begin{aligned}
				\frac{\dd y}{\dd t} &= \widetilde{F}(t,y) \coloneqq F(t,[\bar{y}(t,y(t))]_t) = K_0 y(t) e^{-\lambda_1 \tau}  = \lambda_1 y, \\
				y(t_0) &= y_0,
			\end{aligned}
		\end{equation}
		where the characteristic equation $\lambda_1 = K_0 e^{-\lambda_1 \tau}$ was used.
		Indeed, the unique, $\lambda_1$-dependent solution of the linear initial value problem \eqref{eq:examapleODE} is the special solution $y_{\lambda_1}(t_0,y_0)$.
	\end{example}

	\subsection{Neural DDEs with Small Delay}
	\label{sec:NDDEsmalldelay}
	
	In this section, we aim to relate the results of Section~\ref{sec:globalexistence} to Section~\ref{sec:specialODE} to neural DDEs with the architecture of Definition \ref{def:NDDE} introduced in Section \ref{sec:architecture}. A neural DDE $\Phi \in\textup{NDDE}_\tau^k(\mathcal{X},\mathbb{R}^q)$, $\mathcal{X}\subset \mathbb{R}^n$ open, $0 \leq \tau\leq T$ and $k \geq 0$ is according to equation \eqref{eq:DDEaffinelinear} defined by
	\begin{equation} \label{eq:DDEaffinelinear2}
		\Phi: \mathcal{X} \rightarrow \mathbb{R}^q, \quad x \mapsto \tilde{\lambda}(y(0,c_{\lambda(x)})(T)) = \widetilde{W} \cdot y(0,c_{Wx+b})(T) + \tilde{b} ,
	\end{equation}
	with two affine linear layers $\lambda$, $\tilde{\lambda}$ and vector field $F: \Omega\rightarrow \mathbb{R}^m$, $\Omega \subset \mathbb{R} \times \mathcal{C}$ satisfying Assumption~\ref{ass:vectorfield}. The vector field $F$ is defined on an open set $\Omega = \Omega_t \times \Omega_y \subset \mathbb{R}\times \mathcal{C}$ with $[0,T]\subset \Omega_t$. For general neural DDEs, it is sufficient if the vector field is defined on an open set  $\Omega \subset \mathbb{R} \times \mathcal{C}$, but for the results of the last sections, it is necessary to study vector fields which are globally defined on $\mathbb{R}\times \mathcal{C}$. In the following, we introduce a notation for all neural DDE architectures $\Phi \in\textup{NDDE}_\tau^k(\mathcal{X},\mathbb{R}^q)$, $\mathcal{X}\subset \mathbb{R}^n$, where the underlying vector field is globally defined and fulfills Assumption \ref{ass:weaklynonlinear}.
	
	\begin{definition}[Weakly Nonlinear Neural DDE] \label{def:NDDEglobal}
		The subset of all neural DDE architectures $\Phi \in \textup{NDDE}_\tau^k(\mathcal{X},\mathbb{R}^q)$, $\mathcal{X}\subset \mathbb{R}^n$ open, $k \geq 0$, for which the underlying vector field $F$ is globally defined and weakly nonlinear, i.e., it fulfills Assumption~\ref{ass:weaklynonlinear}, is denoted by $\overline{\textup{NDDE}}^k_\tau(\mathcal{X},\mathbb{R}^q) \subset C^k(\mathcal{X},\mathbb{R}^q)$. 
	\end{definition}
	
	In the following we show that every neural DDE $\Phi \in \textup{NDDE}_\tau^k(\mathcal{X},\mathbb{R}^q)$, $\mathcal{X}\subset \mathbb{R}^n$ open, $k\geq 0$, with vector field $F:\Omega_t \times \mathcal{C}\rightarrow \mathbb{R}^m$, which is globally Lipschitz continuous with respect to its second variable and continuous if $k = 0$, and $ [0,T] \subset \Omega_t$, admits a globally defined, weakly nonlinear vector field $\overline{F}$ corresponding to a neural DDE $\overline{\Phi}\in\overline{\textup{NDDE}}^k_\tau(\mathcal{X},\mathbb{R}^q)$, such that $\Phi(x) = \overline{\Phi}(x)$ for all $x \in \mathcal{X}$. As discussed in Section~\ref{sec:globalexistence}, such a global Lipschitz constant always exists if $k\geq 1$. For general neural DDEs $\Phi \in \textup{NDDE}_\tau^k(\mathcal{X},\mathbb{R}^q)$, $\mathcal{X}\subset \mathbb{R}^n$ open, $k\geq1$, it holds for the underlying vector field that $F\in C_b^{0,k}(\Omega_t \times \Omega_y,\mathbb{R}^m)$, $k\geq 1$, where $\Omega_t\times \Omega_y \subset \mathbb{R} \times \mathcal{C}$ open. The following theorem makes the additional assumption on the vector field $F$ that $\Omega_y = \mathcal{C}$, i.e., the vector field is globally defined with respect to the second variable.

	\begin{theorem}[Equivalent Weakly Nonlinear Neural DDE] \label{th:NDDEextended}
		Let $\Phi \in \textup{NDDE}_\tau^k(\mathcal{X},\mathbb{R}^q)$, $\mathcal{X}\subset \mathbb{R}^n$ open, $k\geq 0$, with vector field $F:\Omega_t \times \mathcal{C}\rightarrow \mathbb{R}^m$ with global Lipschitz constant $K$ with respect to the second variable, where $F$ is continuous if $k = 0$. Then there exists a neural DDE $\overline{\Phi}\in \overline{\textup{NDDE}}_\tau^k(\mathcal{X},\mathbb{R}^q)$, such that  $\Phi(x) = \overline{\Phi}(x)$ for all $x \in \mathcal{X}$, which is based on a vector field $\overline{F}\in C_b^{0,k}(\mathbb{R}\times \mathcal{C},\mathbb{R}^m)$ if $k\geq 1$ and $\overline{F}\in C^{0,0}(\mathbb{R}\times\mathcal{C},\mathbb{R}^m)$ if $k = 0$, with global Lipschitz constant $K$ with respect to the second variable and
		\begin{equation*}
			\overline{F}\big\vert_{[0,T]\times \mathcal{C}}= F\big\vert_{[0,T]\times \mathcal{C}}.
		\end{equation*}
		
	\end{theorem}
	
	\begin{proof}
		Let  $\Phi \in \textup{NDDE}_\tau^k(\mathcal{X},\mathbb{R}^q)$, $\mathcal{X}\subset \mathbb{R}^n$ open, $k\geq 0$, with vector field $F:\Omega_t \times \mathcal{C}\rightarrow \mathbb{R}^m$. By Definition~\ref{def:NDDE}, it holds $F\in C_b^{0,k}(\Omega_t \times \mathcal{C},\mathbb{R}^m)$ if $k\geq 1$, and $F$ is continuous by assumption if $k = 0$, and Assumption \ref{ass:vectorfield} implies $[0,T]\subset \Omega_t$ with $\Omega_t \subset \mathbb{R}$ open. We define the globally extended vector field
		\begin{equation*}
			\overline{F}: \mathbb{R} \times \mathcal{C} \rightarrow \mathbb{R}^m, \qquad \overline{F}(t,y_t) \coloneqq F(\min\{\max\{t,0\},T\},y_t) = \begin{cases}
				F(0,y_t)  &\text{ if } t<0, \\
				F(t,y_t)  &\text{ if } 0\leq t \leq T, \\
				F(T,y_t) &\text{ if } t >T.
			\end{cases}
		\end{equation*}
		As $t\mapsto\min\{\max\{t,0\},T\}$ is continuous with values in $[0,T]\subset\Omega_t$, the properties of $F$ on $[0,T]\times\mathcal{C}$ carry over to $\overline{F}$ on $\mathbb{R}\times\mathcal{C}$. Consequently, it holds $\overline F\in C_b^{0,k}(\mathbb{R}\times\mathcal{C},\mathbb{R}^m)$ if $k\geq 1$ and $\overline{F}\in C^{0,0}(\mathbb{R}\times\mathcal{C},\mathbb{R}^m)$ if $k = 0$.
		As $F$ is globally Lipschitz continuous with respect to the second variable on $\Omega_t\times \mathcal{C}$, the vector field $\overline{F}$ is globally Lipschitz continuous on $\mathbb{R}\times \mathcal{C}$ with respect to the second variable with the same Lipschitz constant~$K$, as
		\begin{equation*}
			\sup_{t \in \mathbb{R}}\norm{\overline{F}(t,y_t)-\overline{F}(t,z_t)}_\infty = \sup_{t \in [0,T]}\norm{F(t,y_t)-F(t,z_t)}_\infty\leq K \norm{y_t-z_t}_\infty \qquad \text{for all } y_t,z_t \in \mathcal{C}.
		\end{equation*}
		The constructed vector field $\overline{F}$ fulfills Assumption \ref{ass:weaklynonlinear}: $\overline{F}$ is continuous and globally defined and we can estimate
		\begin{equation*}
			\sup_{t \in \mathbb{R}}\norm{\overline{F}(t,0)}_\infty = 	\sup_{t \in [0,T]}\norm{\overline{F}(t,0)}_\infty \eqqcolon A < \infty
		\end{equation*}
		for some constant $A \geq 0$, as $t \mapsto \norm{\overline{F}(t,0)}_\infty$ is a continuous function, which attains its maximum on the compact interval $[0,T]$. Furthermore, $\overline{F}$ is globally Lipschitz continuous with respect to the second variable with Lipschitz constant $K$. Consequently, the DDE 
		\begin{equation*}
			\begin{aligned}
				\frac{\dd y}{\dd t} &= \overline{F}(t,y_t) \qquad &&\text{for } t\geq t_0, \\  
				y_{t_0}(t) &= u(t) &&\text{for } t \in [-\tau,0],
			\end{aligned}
		\end{equation*}
		has by Theorem \ref{th:globalexistence} for every initial data $(t_0,u)\in\mathbb{R}\times \mathcal{C}$ a unique, continuous and globally defined solution $y(t_0,u):[t_0-\tau,\infty)\rightarrow \mathbb{R}^m$.
		
		The vector field $\overline{F}$ also fulfills Assumption \ref{ass:vectorfield}: it holds $[0,T]\subset \mathbb{R}$ and $[-\tau,T]\subset \mathcal{I}_{y_0}$, as all solutions of the DDE with initial data $(0,c_{y_0})\in\mathbb{R}\times\mathcal{C}$ exist globally on the time interval $[-\tau,\infty)$, they are unique by Theorem~\ref{th:globalexistence}, and they depend continuously on the initial data by Theorem~\ref{th:DDEexistenceuniqueness}. Furthermore, it holds in the notation of Assumption \ref{ass:vectorfield} that $\Omega_0 \subset \Omega_y = \mathcal{C}$. Hence,  $\overline{\Phi}:\mathcal{X}\rightarrow \mathbb{R}^q$, $\mathcal{X}\subset \mathbb{R}^n$ open, defined by \eqref{eq:DDEaffinelinear2} with vector field $\overline{F}$, $k\geq 0$, is by Definition \ref{def:NDDE} a neural DDE $\overline{\Phi}\in \textup{NDDE}_\tau^k(\mathcal{X},\mathbb{R}^q)$. Definition \ref{def:NDDEglobal} implies that $\overline{\Phi}\in \overline{\textup{NDDE}}_\tau^k(\mathcal{X},\mathbb{R}^q)$. As by construction of the vector field~$\overline{F}$ it holds 
		\begin{equation*}
			\overline{F}\big\vert_{[0,T]\times \mathcal{C}} = F\big\vert_{[0,T]\times \mathcal{C}},
		\end{equation*}
		it follows by the uniqueness in Theorem~\ref{th:DDEexistenceuniqueness} that for every $x\in\mathcal{X}$ the solutions of both DDEs with initial data $(0,c_{\lambda(x)})$ coincide on $[-\tau,T]$, hence $\Phi(x) = \overline{\Phi}(x)$ for all $x \in \mathcal{X}$.
	\end{proof}
	
	\begin{remark}
		The proof of Theorem \ref{th:NDDEextended} even shows a stronger statement: it is not necessary to assume $\Phi \in \textup{NDDE}_\tau^k(\mathcal{X},\mathbb{R}^q)$, $\mathcal{X}\subset \mathbb{R}^n$ open, $k\geq 0$, which includes checking that the underlying vector field fulfills Assumption \ref{ass:vectorfield}. If a map $\Phi$ is defined by \eqref{eq:DDEaffinelinear2} with vector field $F\in C_b^{0,k}([0,T] \times \mathcal{C},\mathbb{R}^m)$, $k\geq 1$, or $F\in C^{0,0}([0,T]\times\mathcal{C},\mathbb{R}^m)$, $k = 0$, which is globally Lipschitz continuous with respect to the second variable, then the extended vector field $\overline{F}$ automatically fulfills Assumption~\ref{ass:vectorfield} and Assumption~\ref{ass:weaklynonlinear}.
	\end{remark}

	\section{Approximation of Local Extreme Points}
	\label{sec:sectionproof}
	
	In this section, we formulate and prove our main result, Theorem \ref{th:noapproximation}, about the non-universal approximation property of non-augmented neural DDEs with sufficiently small memory capacity $K \tau$. Under certain assumptions on the Lipschitz constant $K$ and the delay $\tau$ of the neural DDE, we show that for functions $\Psi$, which have in one component a non-degenerate local extreme point, there exists no neural DDE satisfying the required constraints on weights and vector field that approximates $\Psi$ with arbitrary precision in the sup-norm. As $\Psi$ can be chosen smooth and bounded with an arbitrarily small global Lipschitz constant, it follows that under suitable conditions non-augmented neural DDEs cannot approximate all maps with a prescribed global Lipschitz constant, as stated in the upcoming Corollary \ref{cor:noapproximation}. In Section \ref{sec:nonaugmentedgeneral}, we already introduced Theorem \ref{th:tau0}, which is a direct consequence of the results shown in this section. 
	
	To prove our main result, we introduce in Section \ref{sec:nondegenerateCP} functions with local non-degenerate extreme points. These functions can be characterized via Morse functions, which are functions where every critical point is non-degenerate. For our main result, we are interested in those critical points that are either a local minimum or a local maximum. In Section \ref{sec:mainresult}, we state our main Theorem \ref{th:noapproximation} and its Corollary \ref{cor:noapproximation}. In preparation for proving our main result, we study in Section \ref{sec:initialbehavior} the initial behavior of the neural DDE, given by the first affine linear transformation and the dynamics of the underlying DDE itself under the assumption that the memory capacity $K\tau$ is sufficiently small. Afterward, we study in Section \ref{sec:levelsets} necessary conditions for an approximation of $\Psi$ by a neural DDE and geometrically interpret these conditions as a separation of the phase space by level sets. Finally, we are able to prove in Section \ref{sec:proof} our main Theorem \ref{th:noapproximation}.
	
	\subsection{Characterization of Non-Degenerate Extreme Points}
	\label{sec:nondegenerateCP}
	
	A twice continuously differentiable scalar function $\Psi \in C^2(\mathcal{X},\mathbb{R})$, $\mathcal{X}\subset \mathbb{R}^n$ open, has a critical point $p \in \mathcal{X}$ if its gradient vanishes, i.e., $\nabla \Psi(p) = 0$. Depending on the rank of the Hessian matrix of $\Psi$ at the point $p$, the critical point is called degenerate or non-degenerate. Functions, where every critical point is non-degenerate, are called Morse functions, as defined in the following.  
	
	\begin{definition}[Morse Function~\cite{Hirsch1976, Morse1934}] \label{def:morse}
		A map $\Psi \in C^2(\mathcal{X}, \mathbb{R})$ with $\mathcal{X} \subset \mathbb{R}^n$ open is called a Morse function if all critical points of $\Psi$ are non-degenerate, i.e., for every critical point $p \in \mathcal{X}$ defined by a zero gradient $\nabla \Psi(p) = 0 \in \mathbb{R}^n$, the Hessian matrix $H_{\Psi}(p)\in \mathbb{R}^{n \times n}$ is non-singular. A critical point $p \in \mathcal{X}$ of a Morse function has index $r$ if $r$ eigenvalues of $H_{\Psi}(p)$ are negative.
	\end{definition}
	
	It is of interest to study Morse functions, as they are dense in the Banach space of $k$-times continuously differentiable functions from $\mathbb{R}^n$ to $\mathbb{R}$, if $k \geq n+1 \geq 2$, cf.\ \cite{Kuehn2024}. Hence, it is a generic property of a sufficiently smooth function to be a Morse function. The following lemma shows that every Morse function can be locally transformed into a quadratic form. 
	
	\begin{lemma}[Morse-Palais Lemma~\cite{Hirsch1976,Palais1969}] \label{lem:morse}
		Let $\Psi \in C^{k+2}(\mathcal{X},\mathbb{R})$ with $\mathcal{X} \subset \mathbb{R}^n$ open, $k \geq 1$ be a Morse function with critical point $p \in \mathcal{X}$  with index $r$. Then there exists a neighborhood $\mathcal{U}$ of $0 \in \mathbb{R}^n$ and a $C^k$-diffeomorphism $\mu: \mathcal{U} \rightarrow \mu(\mathcal{U}) \subset \mathcal{X}$ with $\mu(0) = p$, such that for $(u_1,\ldots,u_n) \in \mathcal{U}$
		\begin{equation}\label{eq:morse_palais}
			\Psi(\mu(u_1,\ldots,u_n)) = \Psi(p) - \sum_{j = 1}^r u_j^2 + \sum_{j = r+1}^n u_j^2.
		\end{equation}
	\end{lemma}
	
	In the case that the index $r$ is zero, the map $\Psi\circ\mu:\mathcal{U}\rightarrow \mathbb{R}$ has a local minimum at the critical point $0\in\mathcal{U}$, in the case that $r = n$, the map $\Psi \circ \mu$ has a local maximum at $0\in\mathcal{U}$. Our main Theorem~\ref{th:noapproximation} is based on non-degenerate local extreme points, i.e., Morse functions with a critical point with index $0$ or $n$. Hence, our results apply to all functions $\Psi$, which have somewhere in their domain of definition a local non-degenerate extreme point.
	
	In order to restrict our analysis to the open neighborhood $\mathcal{U}$, we define in the following the radius of definition $r_0$ of a critical point of a map $\Psi$ of Lemma \ref{lem:morse}. To use the symmetry properties of \eqref{eq:morse_palais}, we first define the open and closed Euclidean balls, given for $s>0$ and $v \in \mathbb{R}^n$ by
	\begin{align*}
		B_{s}(v) &\coloneqq \left\{ x \in \mathbb{R}^n: \;  \sum_{i =1}^n (x_i-v_i)^2 < s^2 \right\} \subset \mathbb{R}^n, \\
		K_{s}(v) &\coloneqq \left\{ x \in \mathbb{R}^n: \; \sum_{i =1}^n (x_i-v_i)^2 \leq s^2 \right\} \subset \mathbb{R}^n. 
	\end{align*}
	By definition, the boundary and interior of the open and closed balls $B_{s}(v)$ and $K_{s}(v)$ are given by
	\begin{equation*}
		\partial K_{s}(v) = \partial B_{s}(v) = K_{s}(v) \setminus B_{s}(v), \qquad \text{int}(K_{s}(v)) = \text{int}(B_{s}(v)) = B_{s}(v).
	\end{equation*}
	In the following, we can now define the radius of definition $r_0>0$ of a critical point of a function $\Psi$ of Lemma \ref{lem:morse}.
	
	\begin{definition}[Radius of Definition]\label{def:radiusmorse}
		Let $\Psi\in C^3(\mathcal{X},\mathbb{R})$ with $\mathcal{X}\subset\mathbb{R}^n$ open have a non-degenerate critical point $p\in\mathcal{X}$ with index $r$. Then $r_0>0$ is called a radius of definition of the critical point $p$ of~$\Psi$ if there exist an open neighborhood $\mathcal{U}$ of $0\in\mathbb{R}^n$ and a $C^1$-diffeomorphism $\mu:\mathcal{U}\rightarrow\mu(\mathcal{U})\subset\mathcal{X}$ with $\mu(0) = p$ satisfying \eqref{eq:morse_palais} and $K_{r_0}(0)\subset\mathcal{U}$. In this case, we say that $\mu$ realizes the radius of definition~$r_0$.
	\end{definition}
	
	The existence of such a neighborhood $\mathcal{U}$ and diffeomorphism $\mu$ is guaranteed by Lemma~\ref{lem:morse} for every non-degenerate critical point $p$, as \eqref{eq:morse_palais} is a local result that only relies on the non-degeneracy of the Hessian matrix at $p$. For a local minimum, the set $\mu(K_{r_0}(0))$ does not depend on the choice of~$\mu$, as it is the connected component of the sublevel set $\{x\in\mathcal{X}:\Psi(x)\leq\Psi(p)+r_0^2\}$ containing $p$. For a local maximum, the analogous statement holds for the superlevel set $\{x\in\mathcal{X}:\Psi(x)\geq\Psi(p)-r_0^2\}$.
	
	Given a radius of definition $r_0>0$ of a critical point of a function $\Psi$, the map $\Psi \circ \mu: K_{r_0}(0)\rightarrow \mathbb{R}$ attains in the case of an index $r = 0$ ($r = n$) its global maximum (minimum) at the boundary $\partial  K_{r_0}(0)$ and its global minimum (maximum) at $0 \in K_{r_0}(0)$. This follows from the fact that \eqref{eq:morse_palais} is an exact equality: for $r = 0$ it holds 
	\begin{equation*}
		(\Psi\circ\mu)(u) = \Psi(p)+\norm{u}_2^2
	\end{equation*}
	on all of $K_{r_0}(0)$, so the global minimum $\Psi(p)$ is attained uniquely at $u = 0$ and the global maximum $\Psi(p)+r_0^2$ on $\partial K_{r_0}(0)$. The case $r = n$ is analogous. The main idea of the upcoming proof is the property that in order to approximate the map~$\Psi$, a neural DDE has to separate the image of the inner point $0 \in K_{r_0}(0)$ from the images of all boundary points of $\partial K_{r_0}(0)$ in the phase space by a level set of the output transformation composed with the DDE solution map. The DDE solution map can be approximated by the solution map of an ODE, so that this level set is a deformed hyperplane. If the index of the critical point is between $1$ and $n-1$,  the map $\Psi\circ\mu:\mathcal{U}\rightarrow \mathbb{R}$ has a saddle point at $0\in\mathcal{U}$. At a saddle point, $\Psi\circ\mu-\Psi(p)$ takes both signs on every sphere $\partial K_{s}(0)\subset\mathcal{U}$, so the separation argument of the upcoming proof does not apply. Hence, our results only apply to local non-degenerate extreme points and not to saddle points.
	
	\subsection{Statement of the Main Result}
	\label{sec:mainresult}
	
	In this section, we aim to study the universal approximation property of non-augmented neural DDEs in the sense of Definition \ref{def:universalapproximation}. A neural DDE as defined in \eqref{eq:DDEaffinelinear} is a map $\Phi: \mathcal{X}\rightarrow \mathbb{R}^q$, $\mathcal{X} \subset \mathbb{R}^n$ open, which should approximate the map $\Psi: \mathcal{X}\rightarrow \mathbb{R}^q$. For a given $\eps>0$ and a compact set $\mathcal{K}\subset\mathcal{X}$, the neural DDE $\Phi$ approximates a map $\Psi$ with accuracy $\eps$ on $\mathcal{K}$, if
	\begin{equation} \label{eq:approximationDDE}
		\norm{\Phi-\Psi}_{\infty,\mathcal{K}} = \sup_{x\in\mathcal{K}}\norm{\Phi(x)-\Psi(x)}_\infty<\eps.
	\end{equation}
	In the non-augmented case $m\leq n$, we show via functions $\Psi$, which have in at least one component map $\Psi_i$ a local non-degenerate extreme point, that for a sufficiently small memory capacity $K\tau$, such functions cannot be approximated to an accuracy determined by the radius of definition of the extreme point.

	\begin{theorem}[No Approximation of Non-Degenerate Local Extreme Points] \label{th:noapproximation}
		Let $\Psi:\mathcal{X}\rightarrow \mathbb{R}^q$,  $\mathcal{X}\subset \mathbb{R}^n$ open, have a component map $\Psi_i\in C^3(\mathcal{X}, \mathbb{R})$, which has a non-degenerate local extreme point with radius of definition $r_0>0$ and let $\eps>0$ with $2\eps<r_0^2$. Furthermore, fix constants $w,\tilde{w},\sigma_0 \in (0,\infty)$, $A,b_\infty\in[0,\infty)$, and $k\geq 0$.
		
		Then there exists a continuous function $\tau_0\in C^0((0,\infty),(0,T])$ with $K \tau_0(K) e<1$ for all $K\in(0,\infty)$ with the following property: for all $K\in(0,\infty)$ and $\tau \in [0,\tau_0(K))$, every neural DDE $\Phi \in \textup{NDDE}_{\tau,\textup{N},K}^k(\mathcal{X},\mathbb{R}^q)$ with
		\begin{itemize}
			\item  vector field $F:\Omega_t \times \mathcal{C} \rightarrow \mathbb{R}^m$ with global Lipschitz constant~$K$ in its second variable on $\Omega_t \times \mathcal{C}$ and $\norm{F(t,0)}_\infty \leq A$ for all $t\in [0,T]$, which is additionally continuous if $k = 0$,
			\item weight matrices $W, \widetilde{W}$ with $\normm{W}_\infty \leq w$, $\normm{\widetilde{W}}_\infty \leq\tilde{w}$ and smallest singular value $\sigma_{\min}(W)\geq\sigma_0$, as well as a bias vector $b$ with $\norm{b}_\infty\leq b_\infty$,
		\end{itemize}
		cannot approximate the map $\Psi$ with accuracy $\eps$ on the compact set $\mathcal{K}\coloneqq\mu(K_{r_0}(0))\subset\mathcal{X}$, where $\mu$ realizes the radius of definition $r_0$ as in Definition~\ref{def:radiusmorse}, i.e., $\norm{\Phi-\Psi}_{\infty,\mathcal{K}}\geq\eps$ for every such $\Phi$.
	\end{theorem}
	
	The theorem is proven in the upcoming Section \ref{sec:proof}. In the setting of Theorem \ref{th:noapproximation}, we assume that the underlying DDE of the considered neural DDE $\Phi \in \textup{NDDE}_{\tau,\textup{N},K}^k(\mathcal{X},\mathbb{R}^q)$ uses constant initial data. With slight modifications, the proof also works for fixed, continuous, non-constant initial data, but the main idea stays the same. To maintain the analogy to DenseResNets as introduced in Section~\ref{sec:modeling_discretization}, we focus on DDEs with constant initial data.
	
	It is crucial to study local extreme points, as they have other topological properties than saddle points. A direct consequence of Theorem \ref{th:noapproximation} is that for a sufficiently small memory capacity~$K\tau$, non-augmented neural DDEs, and as a special case also non-augmented neural ODEs with $\tau = 0$, with certain bounds on the parameters cannot approximate all maps with a prescribed global Lipschitz constant.
	
	\begin{corollary}[No Approximation of a Lipschitz Map for Non-Augmented Neural DDEs with Small Memory Capacity] \label{cor:noapproximation}
		Let $\mathcal{X}\subset\mathbb{R}^n$ be open and fix constants $w,\tilde w,\sigma_0,K_\Psi\in(0,\infty)$, $A,b_\infty\in[0,\infty)$, and $k\geq 0$. Then there exist a map $\Psi\in C^\infty_b(\mathbb{R}^n,\mathbb{R}^q)$ with global Lipschitz constant $K_\Psi$ with respect to the max-norm, a compact set $\mathcal{K}\subset\mathcal{X}$, a constant $\eps>0$, and a continuous function $\tau_0\in C^0((0,\infty),(0,T])$ with $K\tau_0(K)e<1$ for all $K\in(0,\infty)$ with the following property: for all $K\in(0,\infty)$ and $\tau\in[0,\tau_0(K))$, every neural DDE $\Phi\in\textup{NDDE}^k_{\tau,\textup{N},K}(\mathcal{X},\mathbb{R}^q)$ with
		\begin{itemize}
			\item  vector field $F:\Omega_t \times \mathcal{C} \rightarrow \mathbb{R}^m$ with global Lipschitz constant~$K$ in its second variable on $\Omega_t \times \mathcal{C}$ and $\norm{F(t,0)}_\infty \leq A$ for all $t\in [0,T]$, which is additionally continuous if $k = 0$,
			\item weight matrices $W, \widetilde{W}$ with $\normm{W}_\infty \leq w$, $\normm{\widetilde{W}}_\infty \leq\tilde{w}$ and smallest singular value $\sigma_{\min}(W)\geq\sigma_0$, as well as a bias vector $b$ with $\norm{b}_\infty\leq b_\infty$,
		\end{itemize}
		satisfies $\norm{\Phi-\Psi}_{\infty,\mathcal{K}}\geq\eps$.
	\end{corollary}
	
	\begin{proof}
		Fix $x^\ast\in\mathcal{X}$ and $i\in\{1,\ldots,q\}$. As $\mathcal{X}$ is open, there exists a radius $r>0$ with $K_{2r}(x^\ast)\subset\mathcal{X}$ and $6\sqrt{n}\,r\leq K_\Psi$. Set $\eps\coloneqq \frac{r^2}{4}>0$.
		
		Let $\textup{sat}_{4r^2}\in C^\infty_b(\mathbb{R},\mathbb{R})$ be a saturation function with radius $4r^2$ as in Definition~\ref{def:saturation}, whose cutoff function $\zeta$ vanishes outside $[-9r^2,9r^2]$. Define the map $\Psi\in C^\infty(\mathbb{R}^n,\mathbb{R}^q)$ by
		\begin{equation*}
			\Psi_i(x)\coloneqq\textup{sat}_{4r^2}\left(\norm{x-x^\ast}_2^2\right)
		\end{equation*}
		and all other components equal to zero. As $\textup{sat}_{4r^2}' = \zeta$ vanishes outside $[-9r^2,9r^2]$, all derivatives of $\Psi$ of order at least one vanish outside $K_{3r}(x^\ast)$. Additionally, $0\leq\zeta\leq 1$ implies $0\leq\Psi_i\leq 9r^2$, hence $\Psi\in C^\infty_b(\mathbb{R}^n,\mathbb{R}^q)$. To estimate the Lipschitz constant, we write $\Psi_i(x) = g(\norm{x-x^\ast}_2)$ with $g:[0,\infty)\rightarrow\mathbb{R}$, $g(s)\coloneqq\textup{sat}_{4r^2}(s^2)$. It holds
		\begin{equation*}
			\abs{g'(s)} = 2s\,\zeta(s^2)\leq 6r \qquad \text{for all } s\geq 0,
		\end{equation*}
		as $\zeta(s^2) = 0$ for $s>3r$. Hence, by the mean value theorem, the reverse triangle inequality, and the equivalence of norms with $\norm{\cdot}_2\leq\sqrt{n}\,\norm{\cdot}_\infty$, it holds
		\begin{equation*}
			\norm{\Psi(x)-\Psi(y)}_\infty = \abs{g(\norm{x-x^\ast}_2)-g(\norm{y-x^\ast}_2)} \leq 6r\,\norm{x-y}_2\leq 6\sqrt{n}\,r\,\norm{x-y}_\infty\leq K_\Psi\norm{x-y}_\infty
		\end{equation*}
		for all $x,y\in\mathbb{R}^n$.
		
		Let $\mathcal{U}\coloneqq B_{2r}(0)\subset\mathbb{R}^n$ and define the diffeomorphism $\mu:\mathcal{U}\rightarrow B_{2r}(x^\ast)\subset\mathcal{X}$, $\mu(u) = x^\ast+u$. For $u\in\mathcal{U}$, it holds $\norm{u}_2^2<4r^2$, so by Lemma~\ref{lem:saturation} it follows
		\begin{equation*}
			\Psi_i(\mu(u_1,\ldots,u_n)) = \sum_{j = 1}^{n}u_j^2 = \Psi_i(x^\ast)+\sum_{j = 1}^{n}u_j^2.
		\end{equation*}
		Hence, the component map $\Psi_i$ has at $x^\ast$ a non-degenerate local minimum with index $0$ and radius of definition $r$, realized by $\mu$, as $K_{r}(0)\subset\mathcal{U}$, and it holds $0<2\eps<r^2$. Theorem~\ref{th:noapproximation}, applied to the restriction of $\Psi$ to $\mathcal{X}$ with the constants fixed in the statement, yields a continuous function $\tau_0$, such that for all $K\in(0,\infty)$ and $\tau\in[0,\tau_0(K))$, every neural DDE $\Phi$ in the class satisfies $\norm{\Phi-\Psi}_{\infty,\mathcal{K}}\geq\eps$ with $\mathcal{K} = \mu(K_{r}(0)) = K_r(x^\ast)$.
	\end{proof}
	
	In the following Section~\ref{sec:initialbehavior} and Section~\ref{sec:levelsets}, we prove several preparatory lemmas to show Theorem~\ref{th:noapproximation} in Section \ref{sec:proof}. In order to use the results about neural DDEs with small delay introduced in Section~\ref{sec:dde_small_delay}, we study neural DDEs  $\overline{\Phi}\in \overline{\textup{NDDE}}_{\tau,\textup{N},K}^k(\mathcal{X},\mathbb{R}^q)$, which are based on globally defined vector fields $\overline{F}\in C_b^{0,k}(\mathbb{R}\times \mathcal{C},\mathbb{R}^m)$, or $\overline{F}\in C^{0,0}(\mathbb{R}\times\mathcal{C},\mathbb{R}^m)$ if $k = 0$, which fulfill Assumption \ref{ass:weaklynonlinear} with constants $A\geq 0$ and $K>0$. The dynamics of the neural DDE $\overline{\Phi}$ can be subdivided into three maps:
	\begin{equation*}
		\overline\Phi = \tilde\lambda \circ H  \circ \lambda,
	\end{equation*}
	where $\tilde \lambda: \mathbb{R}^m \rightarrow \mathbb{R}^q$ and $\lambda: \mathbb{R}^n \rightarrow \mathbb{R}^m$ are the affine linear maps before and after the DDE, as defined in \eqref{eq:DDEaffinelinear}. The solution map of the underlying DDE 
	\begin{equation}\label{eq:proof_vectorfield}
		\begin{aligned}
			\frac{\dd y}{\dd t} &= \overline  F(t,y_t) \qquad &&\text{for } t\geq 0, \\  
			y(t) &= c_{y_1}(t) &&\text{for } t \in [-\tau,0],
		\end{aligned}
	\end{equation}
	with constant initial data $c_{y_1}:[-\tau,0]\rightarrow \mathbb{R}^m$, $c_{y_1}(t) = y_1$ on the time interval $[0,T]$ is denoted by
	\begin{equation} \label{eq:mapH}
		H:  \mathbb{R}^m \rightarrow \mathbb{R}^m, \qquad H(y_1) =y(0,c_{y_1})(T).
	\end{equation}
	The dynamics of the DDE \eqref{eq:proof_vectorfield} is subdivided into two intervals $[0,t^\ast]$ and $[t^\ast,T]$ for some $t^\ast\in[0,T)$. The time $t^\ast$ is parameterized in terms of the delay $\tau\geq 0$ as $t^\ast = \beta \tau$ with $\beta>0$.  Given $\beta>0$, we need to choose $\tau<\frac{T}{\beta}$ to guarantee $t^\ast<T$. The map $G: \mathbb{R}^m \rightarrow \mathbb{R}^m$ is defined as the solution map of the DDE \eqref{eq:proof_vectorfield}  with constant initial data $(0,c_{y_1}) \in \mathbb{R} \times \mathcal{C}$ on the time interval $[0,t^\ast]$:
	\begin{equation} \label{eq:mapG}
		G: \mathbb{R}^m \rightarrow \mathbb{R}^m, \qquad G(y_1) = y(0,c_{y_1})(t^\ast).
	\end{equation}
	As the underlying DDE \eqref{eq:proof_vectorfield} is based on a vector field $\overline{F}$, which is globally Lipschitz continuous with respect to the second variable, the solution $y(0,c_{y_1})(t)$ depends by Theorem \ref{th:DDEexistenceuniqueness} continuously on the initial data $c_{y_1}\in\mathcal{C}$ and hence also continuously on $y_1\in \mathbb{R}^m$. Consequently, the solution maps $G$ and~$H$ are continuous. 
	
	In the upcoming analysis, we use the subscript $0$ for sets in $\mathbb{R}^n$, before the affine linear map $\lambda$ is applied. The main analysis is carried out on a compact set $\mathcal{K}_0\subset \mathbb{R}^n$. The subscript $1$ is used for all sets in~$\mathbb{R}^m$, after the application of the affine linear map $\lambda$ corresponding to the time $t = 0$ of the DDE \eqref{eq:proof_vectorfield}, such as the compact set $\mathcal{K}_1\coloneqq \lambda(\mathcal{K}_0)$. The subscript $2$ is used for all sets in~$\mathbb{R}^m$ after the application of the map $G$, corresponding to the time $t = t^\ast$, such as the compact set $\mathcal{K}_2 \coloneqq G(\mathcal{K}_1$). All sets in~$\mathbb{R}^m$, which correspond to the time $t = T$ after the map $H$ is applied, are denoted with a subscript $3$, such as the compact set $\mathcal{K}_3 \coloneqq H(\mathcal{K}_1)$. The output of the neural DDE after the application of the linear transformation $\tilde \lambda$  is denoted with a subscript 4. The subscript notation and the maps $H$ and $G$ introduced in \eqref{eq:mapH} and \eqref{eq:mapG} are visualized in Figure \ref{fig:notation}.
	
	\begin{figure}[h!]
		\vspace{3mm}
		\centering
		\resizebox{0.9\textwidth}{!}{%
		\begin{tikzpicture}[
			set/.style={draw=ddeGray, thick, fill=ddeGrayLight},
			dep/.style={-{Stealth[length=2.0mm]}, thick},
			labelnode/.style={text=ddeGray},
			timenode/.style={text=black}
		]
			\filldraw[set] plot[smooth cycle] coordinates {(0.639,0.000) (0.623,0.110) (0.588,0.214) (0.533,0.308) (0.460,0.386) (0.374,0.446) (0.282,0.489) (0.188,0.517) (0.095,0.536) (0.000,0.551) (-0.099,0.561) (-0.205,0.564) (-0.319,0.553) (-0.435,0.519) (-0.543,0.455) (-0.629,0.363) (-0.683,0.249) (-0.698,0.123) (-0.676,0.000) (-0.623,-0.110) (-0.553,-0.201) (-0.477,-0.275) (-0.403,-0.338) (-0.333,-0.397) (-0.264,-0.456) (-0.188,-0.517) (-0.101,-0.573) (-0.000,-0.615) (0.112,-0.634) (0.227,-0.625) (0.338,-0.585) (0.435,-0.519) (0.514,-0.432) (0.573,-0.331) (0.613,-0.223) (0.635,-0.112)};
			\filldraw[set] plot[smooth cycle] coordinates {(3.671,-0.160) (3.687,-0.051) (3.682,0.056) (3.652,0.159) (3.599,0.252) (3.527,0.330) (3.443,0.394) (3.353,0.444) (3.260,0.486) (3.165,0.523) (3.064,0.558) (2.954,0.588) (2.831,0.605) (2.699,0.601) (2.567,0.568) (2.448,0.502) (2.357,0.407) (2.304,0.292) (2.290,0.169) (2.313,0.051) (2.359,-0.053) (2.417,-0.142) (2.476,-0.220) (2.531,-0.294) (2.586,-0.368) (2.647,-0.444) (2.722,-0.519) (2.815,-0.584) (2.927,-0.631) (3.051,-0.651) (3.179,-0.640) (3.301,-0.601) (3.411,-0.539) (3.503,-0.458) (3.577,-0.365) (3.633,-0.265)};
			\filldraw[set] plot[smooth cycle] coordinates {(6.671,-0.160) (6.710,-0.053) (6.715,0.059) (6.680,0.166) (6.609,0.256) (6.518,0.325) (6.424,0.377) (6.335,0.422) (6.252,0.470) (6.165,0.523) (6.067,0.576) (5.951,0.617) (5.823,0.631) (5.694,0.612) (5.574,0.559) (5.472,0.481) (5.389,0.387) (5.326,0.282) (5.290,0.169) (5.290,0.053) (5.327,-0.056) (5.391,-0.148) (5.467,-0.224) (5.539,-0.289) (5.604,-0.352) (5.665,-0.422) (5.731,-0.502) (5.815,-0.584) (5.925,-0.651) (6.054,-0.683) (6.187,-0.668) (6.306,-0.612) (6.403,-0.529) (6.481,-0.438) (6.549,-0.347) (6.613,-0.257)};
			\filldraw[set] plot[smooth cycle] coordinates {(9.609,-0.020) (9.689,0.084) (9.738,0.196) (9.744,0.305) (9.705,0.396) (9.636,0.459) (9.556,0.499) (9.479,0.526) (9.406,0.552) (9.330,0.579) (9.238,0.599) (9.127,0.600) (9.000,0.572) (8.869,0.515) (8.744,0.436) (8.629,0.345) (8.524,0.247) (8.431,0.143) (8.356,0.032) (8.311,-0.084) (8.305,-0.195) (8.334,-0.288) (8.381,-0.360) (8.433,-0.416) (8.479,-0.469) (8.521,-0.526) (8.570,-0.589) (8.643,-0.647) (8.747,-0.677) (8.874,-0.664) (9.008,-0.605) (9.131,-0.515) (9.238,-0.411) (9.333,-0.307) (9.424,-0.210) (9.516,-0.117)};
			\filldraw[set] plot[smooth cycle] coordinates {(12.514,0.050) (12.598,0.115) (12.657,0.182) (12.678,0.242) (12.658,0.288) (12.609,0.316) (12.548,0.330) (12.486,0.337) (12.428,0.345) (12.367,0.352) (12.292,0.353) (12.198,0.342) (12.086,0.314) (11.966,0.270) (11.848,0.214) (11.736,0.153) (11.633,0.088) (11.538,0.022) (11.457,-0.047) (11.402,-0.115) (11.380,-0.177) (11.390,-0.225) (11.420,-0.260) (11.455,-0.286) (11.486,-0.310) (11.514,-0.337) (11.546,-0.367) (11.599,-0.391) (11.683,-0.398) (11.794,-0.378) (11.916,-0.332) (12.034,-0.270) (12.141,-0.202) (12.237,-0.136) (12.329,-0.073) (12.422,-0.013)};
			\draw[dep] (0.89,0) -- node[above] {$\lambda$} (2.04,0);
			\draw[dep] (3.94,0) -- node[above] {$G, \overline{G}$} (5.04,0);
			\draw[dep] (6.97,0) -- node[above] {$\overline{I}$} (8.06,0);
			\draw[dep] (9.99,0) -- node[above] {$\tilde{\lambda}$} (11.13,0);
			\draw[dep] (3,0.81) to[out=55, in=125] node[above] {$H, \overline{H}$} (9,0.80);
			\node[labelnode] at (0,-1.5) {$\mathcal{K}_0 \subset \R^{n}$};
			\node[labelnode] at (3,-1.5) {$\mathcal{K}_1 \subset \R^{m}$};
			\node[labelnode] at (6,-1.5) {$\mathcal{K}_2 \subset \R^{m}$};
			\node[labelnode] at (9,-1.5) {$\mathcal{K}_3 \subset \R^{m}$};
			\node[labelnode] at (12,-1.5) {$\mathcal{K}_4 \subset \R^{q}$};
			\node[timenode] at (3,-2.3) {$t = 0$};
			\node[timenode] at (6,-2.3) {$t = t^\ast$};
			\node[timenode] at (9,-2.3) {$t = T$};
		\end{tikzpicture}%
		}
		\caption{Visualization of the subscript notation and the maps $G$ and $H$ introduced in this section. Additionally, the corresponding time intervals of the maps $\overline{G}$, $\overline{H}$, and $\overline{I}$ introduced in the next section are shown.}
		\label{fig:notation}
	\end{figure}
	
	\subsection{Initial Behavior and DDE Dynamics}
	\label{sec:initialbehavior}
	
	In this section, we study the initial behavior and the DDE dynamics of non-augmented neural DDEs, i.e., the dynamics of the first affine linear transformation $\lambda: \mathbb{R}^n\rightarrow \mathbb{R}^m$ and the solution maps $H$ and~$G$ of the underlying DDE defined in \eqref{eq:mapH} and \eqref{eq:mapG}. 
	
	For non-augmented neural DDEs, it holds $n \geq m$, i.e., the input dimension $n$ is at least as large as the dimension $m$ of the phase space of the underlying DDE. In this section, we restrict our analysis to the initial behavior of the neural DDE in the case that $n = m$; the case $n>m$ is treated separately in the proof of Theorem \ref{th:noapproximation} in Section \ref{sec:proof}. An additional assumption we make on the map~$\Psi$ approximated by the neural DDE is that $\Psi$ has a component map $\Psi_i$ with a non-degenerate local minimum as an extreme point. The opposite case, that the non-degenerate local extreme point of $\Psi_i$ is a local maximum, is also postponed to the proof in Section \ref{sec:proof}.
	
	In the following lemma, we study the behavior of the map $\Psi$, which has a component map $\Psi_i\in C^3(\mathcal{X}, \mathbb{R})$ with a non-degenerate local minimum, under the first affine linear layer $\lambda$ and the assumption that $n = m$.
	
	\begin{lemma}[Initial Behavior under the Affine Linear Map $\lambda$] \label{lem:initial_behavior}
		Let $\Psi:\mathcal{X}\rightarrow \mathbb{R}^q$,  $\mathcal{X}\subset \mathbb{R}^n$ open, have a component map $\Psi_i\in C^3(\mathcal{X}, \mathbb{R})$, which has a non-degenerate local minimum $p \in \mathcal{X}$ with radius of definition $r_0>0$, i.e., $\Psi_i$ has on a neighborhood~$\mathcal{U}$ of $0\in \mathbb{R}^n$ under the $C^1$-diffeomorphism $\mu:\mathcal{U}\rightarrow \mu(\mathcal{U})$  with $\mu(0) = p$ the form \eqref{eq:morse_palais} and $K_{r_0}(0)\subset \mathcal{U}$. Let $\mathcal{K}_0 \coloneqq \mu(K_{r_0}(0))$ and $\lambda: \mathbb{R}^n \rightarrow \mathbb{R}^m$, $\lambda(x) = Wx+b$, $W\in \mathbb{R}^{m\times n}$, $b\in \mathbb{R}^m$ be an affine linear map with $m\leq n$. Then the following assertions hold:
		\begin{enumerate}[label=(\alph*), font=\normalfont]
			\item The map $ \Psi_i: \mathcal{K}_0 \rightarrow \mathbb{R}$ attains its global minimum at $p$ with value $\Psi_i(p)$ and its global maximum at every point of the boundary $\partial \mathcal{K}_0$ with value $\Psi_i(p)+r_0^2$.
			\item If $n = m$ and $\sigma_{\min}(W)\geq\sigma_0$ for some constant $\sigma_0>0$, where $\sigma_{\min}(W)$ denotes the smallest singular value of $W$, then $\textup{rank}(W) = n$ and there exists $r_1 > 0$, depending only on $\sigma_0$ and the fixed set $\mathcal{K}_0$, such that $K_{r_1}(\lambda(p)) \subset \mathcal{K}_1$, where $\mathcal{K}_1\coloneqq \lambda(\mathcal{K}_0)\subset \mathbb{R}^n$ is compact. The map $ \Psi_i \circ \lambda^{-1}: \mathcal{K}_1 \rightarrow \mathbb{R}$ attains its global minimum at $\lambda(p)$ with value $\Psi_i(p)$ and its global maximum at every point of the boundary $\partial \mathcal{K}_1$ with value $\Psi_i(p)+r_0^2$. 
			\item If $\textup{rank}(W)< n$, which always holds for $m<n$, then there exists $s \in \partial \mathcal{K}_0$ with $\lambda(s)= \lambda(p)$.
		\end{enumerate}
	\end{lemma}
	
	\begin{proof}
		As $\Psi:\mathcal{X}\rightarrow \mathbb{R}^q$,  $\mathcal{X}\subset \mathbb{R}^n$ open, has a component map $\Psi_i\in C^3(\mathcal{X}, \mathbb{R})$, which has a non-degenerate local minimum $p \in \mathcal{X}$ with radius of definition $r_0>0$, there exists a $C^1$-diffeomorphism $\mu: \mathcal{U} \rightarrow \mu(\mathcal{U})$ with $\mu(0) = p$ and $K_{r_0}(0)\subset \mathcal{U}$, such that for $(u_1,\ldots,u_n) \in \mathcal{U}$
		\begin{equation}\label{eq:psiform}
			\Psi_i(\mu(u_1,\ldots,u_n)) = \Psi_i(p) + \sum_{j =1}^n u_j^2,
		\end{equation}
		as a non-degenerate local minimum is a critical point with index $0$.

		As the diffeomorphism $\mu$ is continuous, the set $\mathcal{K}_0 \coloneqq \mu(K_{r_0}(0))$ is compact. It follows directly from \eqref{eq:psiform} that the component map $\Psi_i: \mathcal{K}_0 \rightarrow \mathbb{R}$ attains its global minimum at $p = \mu(0)$ with value $\Psi_i(p)$ and its global maximum on all points of the boundary $\partial \mathcal{K}_0$ with value $\Psi_i(p)+r_0^2$, as the boundary of $K_{r_0}$ is mapped to the boundary of $\mathcal{K}_0$ under the diffeomorphism $\mu$, which proves part~(a).
		
		If $n = m$ and $\sigma_{\min}(W)\geq\sigma_0>0$, then $\text{rank}(W) = n$ and the affine linear map $\lambda:\mathbb{R}^n\rightarrow \mathbb{R}^n$, $\lambda(x) = Wx+b$ is a diffeomorphism, which maps the compact set $\mathcal{K}_0$ to the compact set $\mathcal{K}_1 \coloneqq \lambda(\mathcal{K}_0)$ and the boundary $\partial \mathcal{K}_0$ to the boundary~$\partial \mathcal{K}_1$. Consequently, the map $\Psi_i \circ \lambda^{-1}: \mathcal{K}_1 \rightarrow \mathbb{R}$ attains its global minimum at $\lambda(p)$ with value $\Psi_i(p)$ and its global maximum at all points of the boundary~$\partial \mathcal{K}_1$ with value $\Psi_i(p)+r_0^2$. As $0\in \text{int}(K_{r_0}(0))$ and $\mu$ is a diffeomorphism, it holds $p = \mu(0) \in \mu(\text{int}(K_{r_0}(0))) = \text{int}(\mathcal{K}_0)$. Hence, there exists a radius $r'>0$, depending only on the fixed set $\mathcal{K}_0$, with $K_{r'}(p)\subset\mathcal{K}_0$. As $\sigma_{\min}(W)$ is the smallest singular value of $W$, it holds $\norm{W^{-1}u}_2\leq\sigma_{\min}(W)^{-1}\norm{u}_2$ for all $u\in\mathbb{R}^n$. Hence, for every $y\in\mathbb{R}^n$ with $\norm{y-\lambda(p)}_2\leq\sigma_0 r'$, we obtain
		\begin{equation*}
			\norm{\lambda^{-1}(y)-p}_2 = \norm{W^{-1}(y-\lambda(p))}_2\leq\sigma_{\min}(W)^{-1}\norm{y-\lambda(p)}_2\leq r',
		\end{equation*}
		so that $\lambda^{-1}(y)\in K_{r'}(p)\subset\mathcal{K}_0$ and therefore $y\in\mathcal{K}_1$. Consequently, $K_{r_1}(\lambda(p))\subset\mathcal{K}_1$ for $r_1\coloneqq\sigma_0 r'>0$, which depends only on $\sigma_0$ and the fixed radius $r'$ of $\mathcal{K}_0$ at $p$. In particular, $r_1$ is uniform over all weight matrices $W$ with $\sigma_{\min}(W)\geq\sigma_0$, which proves part~(b).
		
		If $\text{rank}(W) < n$, then the kernel of the matrix $W$
		\begin{equation*}
			\text{ker}(W) = \{x \in \mathbb{R}^n: Wx = 0\}
		\end{equation*}
		is non-trivial. Let $v\in\text{ker}(W)\setminus\{0\}$, then the line $\{p + \alpha v: \alpha \in \mathbb{R}\}$ intersects $\partial \mathcal{K}_0$, as $\mathcal{K}_0$ is a full-dimensional subset of $\mathbb{R}^n$ homeomorphic to the closed ball $K_{r_0}(0)\subset \mathbb{R}^n$ with $p \in \text{int}(\mathcal{K}_0)$. Let
		\begin{equation*}
			s\in \{p + \alpha v: \alpha \in \mathbb{R}\} \cap \partial \mathcal{K}_0
		\end{equation*}
		be such an intersecting point, so that $s \in\partial \mathcal{K}_0$ and $s - p\in\text{ker}(W)$. By construction, it holds
		\begin{equation*}
			\lambda(s) = Ws + b = Wp + W(s-p) + b = \lambda(p),
		\end{equation*}
		which proves part (c).
	\end{proof}
	
	In the following lemma, we study the initial behavior of the neural DDE $\overline{\Phi}$, i.e., the map $G$ defined in~\eqref{eq:mapG} as the solution map of the DDE~\eqref{eq:proof_vectorfield} on the time interval $[0,t^\ast]$. Here, we use the property that the time $t^\ast = \beta \tau$ is sufficiently small, if the parameter $\beta$ is fixed and the delay $\tau$ is sufficiently small.

	\begin{lemma}[Initial Behavior under the Solution Map $G$] \label{lem:mapG}
		Let the following assumptions be satisfied:
		\begin{itemize}
			\item $\mathcal{K}_1 \subset \mathbb{R}^m$ is a compact set with $M \coloneqq \sup_{y_1 \in \mathcal{K}_1} \norm{y_1}_\infty>0$ and $\lambda(p)\in \mathcal{K}_1$, where $p$ and $\lambda$ are as in Lemma~\ref{lem:initial_behavior}, with $K_{r_1}(\lambda(p))\subset \mathcal{K}_1$ for some $r_1 >0$.
			\item The underlying vector field $\overline{F}$ of the neural DDE $\overline{\Phi}\in\overline{\textup{NDDE}}^k_{\tau,\textup{N},K}(\mathcal{X},\mathbb{R}^q)$, $\mathcal{X}\subset \mathbb{R}^n$ open, $k \geq 0$, fulfills Assumption \ref{ass:weaklynonlinear} with $A\geq 0$ and $K>0$, which define the constant $C_1 \coloneqq \sqrt{m}\,(KM+A)e^{KT}$.
			\item The map $G: \mathbb{R}^m\rightarrow \mathbb{R}^m$ is the solution map of the DDE \eqref{eq:proof_vectorfield} on the time interval $[0,t^\ast]$ as defined in \eqref{eq:mapG}. It holds $t^\ast = \beta \tau\leq T$ with $\beta>0$ and  $\delta_{1,t^\ast} \coloneqq C_1 t^\ast$.
		\end{itemize}
		Then for every $\beta>0$ and every $0<\kappa\leq r_1$, there exists an upper bound $\tau_{1,\beta,\kappa}=\frac{\kappa}{2C_1 \beta}>0$, such that for all $\tau \in [0,\tau_{1,\beta,\kappa}]$ and every $y_1\in \mathcal{K}_1$ it holds
		\begin{equation*}
			\norm{G(y_1)-y_1}_{2} \leq \delta_{1,t^\ast} \leq \frac{\kappa}{2}.
		\end{equation*}
		In particular, for $\lambda(p)\in \mathcal{K}_1$ we have
		\begin{equation*}
			G(\lambda(p))\in K_{\delta_{1,t^\ast}}(\lambda(p)) \subset K_{\kappa/2}(\lambda(p)) \subset K_{r_1}(\lambda(p)) \subset \mathcal{K}_1.
		\end{equation*}
	\end{lemma}
	
	\begin{proof}
		For $\beta>0$ and  $0<\kappa\leq r_1$, let $\tau_{1,\beta,\kappa}\coloneqq \frac{\kappa}{2C_1 \beta}>0$, then it holds for all $\tau \in [0,\tau_{1,\beta,\kappa}]$ that
		\begin{equation*}
			2\delta_{1,t^\ast} = 2C_1 \beta \tau \leq 2 C_1 \beta\tau_{1,\beta,\kappa} = \kappa \leq   r_1.
		\end{equation*}
		For  $y_1 \in \mathcal{K}_1$ we estimate
		\begin{align*}
			&\norm{G(y_1)-y_1}_\infty = \norm{y(0,c_{y_1})(t^\ast)-y(0,c_{y_1})(0)}_\infty  \\ =\;&\norm{\int_0^{t^\ast}\overline F(t,y_t)\; \dd t }_\infty \leq \int_0^{t^\ast} \norm{\overline F(t,y_t)-\overline F(t,0) }_\infty +  \norm{\overline F(t,0)}_\infty \; \dd t \\
			\leq\;  & \int_0^{t^\ast} (K \norm{y_t}_\infty + A) \; \dd t \leq \int_0^{t^\ast} \left(K \left(Me^{Kt}+\frac{A}{K}\left(e^{Kt}-1\right)\right) + A \right) \; \dd t \\
			\leq \; & \int_0^{t^\ast} (KM+A)e^{Kt} \; \dd t  \leq (KM+A)e^{KT} t^\ast.
		\end{align*}
		Here, we first used the integral formulation of the DDE \eqref{eq:proof_vectorfield}, then the Lipschitz continuity of $\overline{F}$ and finally Theorem \ref{th:globalexistence}(a) to estimate $\norm{y_t}_\infty$ with initial data $c_{y_1}$, which fulfills $\norm{c_{y_1}}_\infty = \norm{y_1}_\infty \leq M$. By the equivalence of norms on $\mathbb{R}^m$, it holds $\norm{\cdot}_2\leq\sqrt{m}\,\norm{\cdot}_\infty$, and it follows
		\begin{equation*}
			\norm{G(y_1)-y_1}_2\leq\sqrt{m}\,(KM+A)e^{KT}t^\ast = C_1t^\ast = C_1\beta\tau = \delta_{1,t^\ast}.
		\end{equation*}
		Consequently, under the map $G$ points move at most the Euclidean distance $\delta_{1,t^\ast}$, such that
		\begin{equation*}
			G(\lambda(p))\in K_{\delta_{1,t^\ast}}(\lambda(p)) \subset K_{\kappa/2}(\lambda(p)) \subset K_{r_1}(\lambda(p)) \subset \mathcal{K}_1,
		\end{equation*}
		as $K_{r_1}(\lambda(p))\subset \mathcal{K}_1$. 
	\end{proof}

	The geometric objects and statements of Lemma \ref{lem:initial_behavior} and Lemma \ref{lem:mapG} are visualized in Figure \ref{fig:proof_lambda}. After having studied the initial behavior of the neural DDE $\overline{\Phi}$, we aim to track the behavior until time $T$, which is characterized by an exponential attraction of general solutions of the DDE~\eqref{eq:proof_vectorfield} to special solutions. By Theorem~\ref{th:specialsolutionsODEs}, there exists an initial value problem
	\begin{equation} \label{eq:proof_specialODE}
		\frac{\dd z(t)}{\dd t} = \widetilde{F}(t,z(t)) \coloneqq \overline{F}(t,[\bar{y}(t,z(t))]_t), \qquad z(0) = y_1,
	\end{equation}
	with $\widetilde{F}:\mathbb{R}\times \mathbb{R}^m\rightarrow \mathbb{R}^m$ generating all special solutions of the DDE \eqref{eq:proof_vectorfield}. In analogy to the maps~$H$ and $G$ defined in \eqref{eq:mapH} and \eqref{eq:mapG} characterizing the solution map of the DDE \eqref{eq:proof_vectorfield} over the time intervals $[0,T]$ and $[0,t^\ast]$, we define the solution maps of the ODE \eqref{eq:proof_specialODE} over the time intervals $[0,t^\ast]$ and $[0,T]$ as 
	\begin{align} \label{eq:mapGbar}
		\overline{G}:  \mathbb{R}^m \rightarrow \mathbb{R}^m,
		\qquad \overline{G}(y_1) =\bar{y}(0,y_1)(t^\ast),  \\
		\label{eq:mapHbar} \overline{H}:  \mathbb{R}^m \rightarrow \mathbb{R}^m,
		\qquad \overline{H}(y_1) =\bar{y}(0,y_1)(T).
	\end{align}
	Here, $\bar{y}(0,y_1)$ is the unique solution of the ODE \eqref{eq:proof_specialODE}, which is by Theorem \ref{th:specialsolutionsODEs}  the unique special solution of the DDE~\eqref{eq:proof_vectorfield} through the point $(0,y_1)\in\mathbb{R}\times \mathbb{R}^m$. For $\tau = 0$, the DDE~\eqref{eq:proof_vectorfield} is an ODE on $\mathcal{C}\cong\mathbb{R}^m$ and special solutions are not defined. In this case, we set $\bar{y}\coloneqq y$ and $\widetilde{F}\coloneqq\overline{F}$, so that~\eqref{eq:proof_specialODE} coincides with~\eqref{eq:proof_vectorfield}, $\overline{G} = G$, and $\overline{H} = H$. The following lemma characterizes the distance between the solutions of the DDE \eqref{eq:proof_vectorfield} and their corresponding special solutions in terms of the solution maps~$G$, $H$, $\overline{G}$, and $\overline{H}$.
	
	\begin{lemma}[Exponential Attraction] \label{lem:exponentialattraction}
		Let the following assumption be satisfied:
		\begin{itemize}
			\item The underlying vector field $\overline{F}$ of the neural DDE $\overline{\Phi}\in\overline{\textup{NDDE}}^k_{\tau,\textup{N},K}(\mathcal{X},\mathbb{R}^q)$, $\mathcal{X}\subset \mathbb{R}^n$ open, $k \geq 0$, fulfills Assumption~\ref{ass:weaklynonlinear} with constants $A\geq 0$ and $K>0$ and it holds $K\tau e <\frac{1}{2}$.
			\item The maps $H$ and $G$ are the solution maps of the DDE \eqref{eq:proof_vectorfield} as defined in \eqref{eq:mapH}, \eqref{eq:mapG} and the maps $\overline{G}$ and $\overline{H}$ are the solution maps of the ODE \eqref{eq:proof_specialODE} as defined in \eqref{eq:mapGbar} and \eqref{eq:mapHbar} with $t^\ast = \beta \tau$, $\beta>0$. 
			\item $\mathcal{K}_1 \subset \mathbb{R}^m$ is compact and $r_1>0$ is a constant.
		\end{itemize}
		Then there exists a constant $C_2>\frac{r_1}{2}$, depending continuously on $K$ and on $A$, $m$, $r_1$, and an upper bound for the constant $M = \sup_{y_1\in\mathcal{K}_1}\norm{y_1}_\infty$ of Lemma~\ref{lem:mapG}, and in particular independent of $\tau$ and of the specific vector field $\overline{F}$, such that for every $y_1 \in \mathcal{K}_1$, there exists $\bar{y}_1\in \mathbb{R}^m$, uniquely determined by the limit \eqref{eq:ICspecialsolution} for the initial data $(0,c_{y_1})\in\R\times \mathcal{C}$, such that
		\begin{equation*}
			\normk{G(y_1)-\overline{G}(\bar{y}_1)}_{2} < \delta_{2,\beta} \coloneqq C_2 e^{-\beta}, \qquad \normk{H(y_1)-\overline{H}(\bar{y}_1)}_{2} < \delta_{3,\tau} \coloneqq C_2 e^{-\frac{T}{\tau}}, 
		\end{equation*}
		if $\tau>0$. If $\tau = 0$, then $\bar{y}_1 = y_1$, $G(y_1) = \overline{G}(\bar{y}_1)$, and $H(y_1) = \overline{H}(\bar{y}_1)$, and we set $\delta_{3,0}\coloneqq 0$.
	\end{lemma}
	
		\begin{figure}
		\centering
		\resizebox{\textwidth}{!}{%
		\begin{tikzpicture}[>={Stealth[length=2.0mm]},x=1cm,y=1cm]
				\colorlet{cSetLine}{ddePurple}
				\colorlet{cSetFill}{ddePurple!15}
				\colorlet{cBallOuter}{ddeGrayLight!70}
				\colorlet{cBallInner}{ddeGrayLight}
				\colorlet{cNet}{ddeOchre}
				\colorlet{cNormLine}{ddeGrayDark}
				\colorlet{cNormFill}{ddeGrayLight}
				\colorlet{cAnnot}{ddeGrayDark}
				\useasboundingbox (0,0.6) rectangle (15.02,5.5);
				\fill[cNormFill] (1.50,3.00) circle (1.05);
				\draw[cNormLine,very thick] (1.50,3.00) circle (1.05);
				\fill[cNormLine] (1.50,3.00) circle (2.0pt);
				\node[cNormLine,anchor=south west,font=\small] at (1.62,3.02) {$0$};
				\node[cNormLine,font=\small] at (1.50,1.45) {$K_{r_0}(0)$};
				\draw[cAnnot,very thick,->] (2.85,3.35) -- (4.25,3.35);
				\draw[cAnnot,very thick,->] (4.25,2.65) -- (2.85,2.65);
				\node[cAnnot,font=\small] at (3.55,3.60) {$\mu$};
				\node[cAnnot,font=\small] at (3.55,2.24) {$\mu^{-1}$};
				\fill[cSetFill] plot[smooth cycle] coordinates {(6.834,3.050) (6.831,3.241) (6.803,3.433) (6.745,3.624) (6.649,3.804) (6.513,3.959) (6.341,4.074) (6.145,4.136) (5.942,4.142) (5.750,4.098) (5.579,4.021) (5.430,3.929) (5.296,3.836) (5.165,3.748) (5.026,3.658) (4.879,3.553) (4.735,3.419) (4.613,3.251) (4.535,3.050) (4.520,2.833) (4.574,2.622) (4.690,2.438) (4.851,2.296) (5.035,2.198) (5.224,2.139) (5.406,2.106) (5.580,2.087) (5.750,2.077) (5.921,2.079) (6.096,2.100) (6.270,2.150) (6.433,2.236) (6.574,2.358) (6.686,2.510) (6.765,2.681) (6.813,2.863)};
				\draw[cSetLine,very thick] plot[smooth cycle] coordinates {(6.834,3.050) (6.831,3.241) (6.803,3.433) (6.745,3.624) (6.649,3.804) (6.513,3.959) (6.341,4.074) (6.145,4.136) (5.942,4.142) (5.750,4.098) (5.579,4.021) (5.430,3.929) (5.296,3.836) (5.165,3.748) (5.026,3.658) (4.879,3.553) (4.735,3.419) (4.613,3.251) (4.535,3.050) (4.520,2.833) (4.574,2.622) (4.690,2.438) (4.851,2.296) (5.035,2.198) (5.224,2.139) (5.406,2.106) (5.580,2.087) (5.750,2.077) (5.921,2.079) (6.096,2.100) (6.270,2.150) (6.433,2.236) (6.574,2.358) (6.686,2.510) (6.765,2.681) (6.813,2.863)};
				\fill[cSetLine] (5.670,3.090) circle (2.0pt);
				\node[cSetLine,anchor=south west,font=\small] at (5.790,3.110) {$p$};
				\node[cSetLine,font=\small] at (5.75,1.45) {$\mathcal{K}_0 = \mu\left(K_{r_0}(0)\right)$};
				\draw[cAnnot,very thick,->] (7.03,3.35) -- (9.52,3.35);
				\draw[cAnnot,very thick,->] (9.52,2.65) -- (7.03,2.65);
				\node[cAnnot,font=\small] at (8.28,3.60) {$\lambda$};
				\node[cAnnot,font=\small] at (8.28,2.24) {$\lambda^{-1}$};
				\node[cAnnot,font=\small] at (8.28,4.35) {$\textup{rank}(W) = n$};
				\fill[cSetFill] plot[smooth cycle] coordinates {(12.364,2.942) (12.407,3.119) (12.422,3.301) (12.402,3.485) (12.338,3.662) (12.223,3.819) (12.058,3.943) (11.853,4.020) (11.628,4.046) (11.402,4.025) (11.191,3.971) (11.003,3.900) (10.830,3.827) (10.662,3.757) (10.485,3.687) (10.296,3.605) (10.102,3.495) (9.924,3.350) (9.789,3.171) (9.720,2.971) (9.730,2.769) (9.816,2.587) (9.962,2.438) (10.145,2.330) (10.343,2.256) (10.539,2.207) (10.729,2.172) (10.917,2.145) (11.109,2.129) (11.309,2.131) (11.516,2.161) (11.719,2.225) (11.907,2.324) (12.069,2.454) (12.198,2.605) (12.295,2.769)};
				\draw[cSetLine,very thick] plot[smooth cycle] coordinates {(12.364,2.942) (12.407,3.119) (12.422,3.301) (12.402,3.485) (12.338,3.662) (12.223,3.819) (12.058,3.943) (11.853,4.020) (11.628,4.046) (11.402,4.025) (11.191,3.971) (11.003,3.900) (10.830,3.827) (10.662,3.757) (10.485,3.687) (10.296,3.605) (10.102,3.495) (9.924,3.350) (9.789,3.171) (9.720,2.971) (9.730,2.769) (9.816,2.587) (9.962,2.438) (10.145,2.330) (10.343,2.256) (10.539,2.207) (10.729,2.172) (10.917,2.145) (11.109,2.129) (11.309,2.131) (11.516,2.161) (11.719,2.225) (11.907,2.324) (12.069,2.454) (12.198,2.605) (12.295,2.769)};
				\fill[cBallOuter] (11.070,3.095) circle (0.63);
				\draw[cNormLine,very thick] (11.070,3.095) circle (0.63);
				\fill[cBallInner] (11.070,3.095) circle (0.28);
				\draw[cNormLine,very thick] (11.070,3.095) circle (0.28);
				\fill[cSetLine] (11.070,3.095) circle (2.0pt);
				\fill[cNet] (11.196,3.184) circle (2.0pt);
				\node[cSetLine,font=\small] at (11.74,5.15) {$\mathcal{K}_1 = \lambda(\mathcal{K}_0)$};
				\draw[cSetLine,thick] (11.54,4.95) -- (11.191,4.001);
				\node[cNet,anchor=west,font=\small] at (12.57,4.25) {$G(\lambda(p))$};
				\draw[cNet,thick] (12.50,4.2) -- (11.256,3.234);
				\node[cNormLine,anchor=west,font=\small] at (12.32,1.35) {$K_{r_1}(\lambda(p))$};
				\draw[cSetLine,thick] (12.27,1.5) -- (11.553,2.690);
				\node[cSetLine,font=\small] at (10.87,1.05) {$\lambda(p)$};
				\draw[cSetLine,thick] (10.92,1.27) -- (11.050,2.995);
				\node[cNormLine,font=\small] at (9.17,1.35) {$K_{\kappa/2}(\lambda(p))$};
				\draw[cSetLine,thick] (9.62,1.55) -- (10.816,2.977);
			\end{tikzpicture}%
		}
		\caption{Visualization of the geometric objects and statements of Lemma \ref{lem:initial_behavior} and Lemma \ref{lem:mapG}: the $C^1$-diffeomorphism $\mu$ maps the set $K_{r_0}(0)$ onto the set $\mathcal{K}_0$ and it holds $p = \mu(0)$. The affine linear map $\lambda$ maps in the case that $\text{rank}(W) = n$ the set $\mathcal{K}_0$ onto the set $\mathcal{K}_1$, which contains the closed ball $K_{r_1}(\lambda(p))$ for some $r_1>0$. Furthermore, for $\tau$ sufficiently small, it holds $G(\lambda(p))\in K_{\kappa/2}(\lambda(p))$ for a given constant $0<\kappa\leq r_1$.}
		\label{fig:proof_lambda}
	\end{figure}

	\begin{proof}
		If $\tau = 0$, then $t^\ast = 0$, and by the convention $\bar{y} = y$ for $\tau = 0$ it holds $\bar{y}_1 = y_1$. Hence, $G(y_1) = \overline{G}(\bar{y}_1) = y_1$ and $H(y_1) = \overline{H}(\bar{y}_1)$, and the claims hold for every constant $C_2>\frac{r_1}{2}$. Let $\tau>0$ in the following.
		
		As the underlying vector field $\overline{F}$ of the neural DDE $\overline{\Phi}\in\overline{\textup{NDDE}}^k_{\tau,\textup{N},K}(\mathcal{X},\mathbb{R}^q)$ fulfills Assumption~\ref{ass:weaklynonlinear} and $K\tau e<1$, Theorem \ref{th:specialsolutions_expattraction} implies that for every initial data $(0,c_{y_1})\in\mathbb{R}\times \mathcal{C}$ specified by $y_1\in\mathcal{K}_1$, there exists
		\begin{equation*}
			\bar{y}_1 \coloneqq \lim_{s \rightarrow \infty} \bar{y}(s,y(0,c_{y_1})(s))(0) \in \mathbb{R}^m \quad \text{and} \quad C_{y_1} \coloneqq \sup_{t \geq 0} e^{\frac{t}{\tau}} \norm{y(0,c_{y_1})(t)-S(y(0,c_{y_1}))(t)}_\infty < \infty,
		\end{equation*}
		where $S(y(0,c_{y_1})) = \bar{y}(0,\bar{y}_1)$ is the special solution of the DDE \eqref{eq:proof_vectorfield} through the point $(0,\bar{y}_1)\in\mathbb{R}\times\mathbb{R}^m$. By the proof of Theorem~2.1 in \cite{Jarnik1975}, it holds
		\begin{equation*}
			C_{y_1}\leq C(K\tau)\,\norm{c_{y_1}-[\bar{y}(0,y_1)]_0}_\infty
		\end{equation*}
		with a constant $C(K\tau)\geq 0$ depending only on $K\tau$, where
		\begin{equation*}
		C_\ast\coloneqq\sup\left\{C(K\tau): 0<K\tau e\leq\frac{1}{2}\right\}<\infty.
		\end{equation*}
		To bound the right-hand side uniformly, let 
		\begin{equation*}
			\widehat{M}\coloneqq\sup_{t\in\mathbb{R}}e^{-\frac{\abs{t}}{\tau}}\norm{\bar{y}(0,y_1)(t)}_\infty,
		\end{equation*}
		which is finite by the growth condition \eqref{eq:growthcondition} of special solutions. As $\abs{\xi}\leq\abs{s}+\tau$ for $\xi\in[s-\tau,s]$, it holds
		\begin{equation}\label{eq:attraction_proof1}
			\norm{[\bar{y}(0,y_1)]_s}_\infty\leq e^{1+\frac{\abs{s}}{\tau}}\widehat{M} \qquad \text{for all } s\in\mathbb{R}.
		\end{equation}
		The special solution $\bar{y}(0,y_1)$ satisfies
		\begin{equation}\label{eq:attraction_proof0}
			\bar{y}(0,y_1)(t) = y_1+\int_0^t\overline{F}(s,[\bar{y}(0,y_1)]_s)\,\dd s \qquad \text{for all } t\in\mathbb{R},
		\end{equation}
		and by Assumption~\ref{ass:weaklynonlinear}, it holds $\norm{\overline{F}(s,u)}_\infty\leq A+K\norm{u}_\infty$ for all $(s,u)\in\mathbb{R}\times\mathcal{C}$. Inserting this bound and \eqref{eq:attraction_proof1} in \eqref{eq:attraction_proof0}, it follows for all $t\in\mathbb{R}$ that
		\begin{equation}\label{eq:attraction_proof}
			\norm{\bar{y}(0,y_1)(t)}_\infty\leq\norm{y_1}_\infty+A\abs{t}+Ke\widehat{M}\int_0^{\abs{t}}e^{\frac{s}{\tau}}\,\dd s\leq\norm{y_1}_\infty+A\abs{t}+K\tau e\widehat{M}e^{\frac{\abs{t}}{\tau}}.
		\end{equation}
		Multiplying \eqref{eq:attraction_proof} by $e^{-\frac{\abs{t}}{\tau}}$, using $\abs{t}e^{-\frac{\abs{t}}{\tau}}\leq\frac{\tau}{e}$ and $K\tau e<\frac{1}{2}$, and taking the supremum over $t\in\mathbb{R}$ yields
		\begin{equation*}
			\widehat{M}\leq\norm{y_1}_\infty+\frac{\tau A}{e}+\frac{\widehat{M}}{2} \qquad \Rightarrow \qquad \widehat{M}\leq 2\norm{y_1}_\infty+\frac{2\tau A}{e}.
		\end{equation*}
		Combining \eqref{eq:attraction_proof0} with the bound on $\overline{F}$, the estimate \eqref{eq:attraction_proof1}, and the bound on $\widehat{M}$, it follows for $s\in[-\tau,0]$ that
		\begin{align*}
			\norm{y_1-\bar{y}(0,y_1)(s)}_\infty &\leq\int_{-\tau}^0\norm{\overline{F}(\xi,[\bar{y}(0,y_1)]_\xi)}_\infty\,\dd\xi\leq\tau A+K\tau e^2\widehat{M} \\
			&\leq e\norm{y_1}_\infty+2\tau A\leq eM+\frac{A}{Ke},
		\end{align*}
		where $K\tau e<\frac{1}{2}$ is used in the last two steps. Hence, it follows $C_{y_1}\leq C_\ast\left(eM+\frac{A}{Ke}\right)$ for every $y_1\in\mathcal{K}_1$, every $\tau$ with $0<K\tau e<\frac{1}{2}$, and every vector field $\overline{F}$ satisfying Assumption~\ref{ass:weaklynonlinear} with the constants $K$ and $A$. The constant
		\begin{equation*}
			C_2\coloneqq\sqrt{m}\,C_\ast\left(eM+\frac{A}{Ke}\right)+\frac{r_1}{2}+1
		\end{equation*}
		therefore satisfies $C_2>\frac{r_1}{2}$ and $C_2>\sqrt{m}\,C_{y_1}$ for every $y_1\in\mathcal{K}_1$, depends continuously on $K$, and, being non-decreasing in $M$, remains valid with $M$ replaced by any upper bound. For every $y_1 \in \mathcal{K}_1$ it follows
		\begin{align*}
			\normk{G(y_1)-\overline{G}(\bar{y}_1)}_{2} &= \norm{y(0,c_{y_1})(t^\ast)-S(y(0,c_{y_1}))(t^\ast)}_{2} \leq\sqrt{m}\,C_{y_1}e^{-\beta} < C_2 e^{-\beta} \eqqcolon \delta_{2,\beta}, \\
			\normk{H(y_1)-\overline{H}(\bar{y}_1)}_{2} &= \norm{y(0,c_{y_1})(T)-S(y(0,c_{y_1}))(T)}_{2} \leq\sqrt{m}\,C_{y_1}e^{-\frac{T}{\tau}} < C_2 e^{-\frac{T}{\tau}}  \eqqcolon \delta_{3,\tau},
		\end{align*}
		which implies the result.
	\end{proof}
	
	The solution maps $\overline{G}$ and $\overline{H}$ characterize the solutions of the initial value problem \eqref{eq:proof_specialODE} over the time intervals $[0,t^\ast]$ and $[0,T]$. To additionally characterize the solution of the ODE \eqref{eq:proof_specialODE} over the time interval $[t^\ast,T]$, we introduce the map $\overline{I}$ defined by
	\begin{equation} \label{eq:mapIbar}
		\overline{I}: \mathbb{R}^m \rightarrow \mathbb{R}^m,
		\qquad \overline{I}(y_2) =\bar{y}(t^\ast,y_2)(T).
	\end{equation}
	By definition it holds $\overline{H} = \overline{I}\circ \overline{G}$.  In Figure \ref{fig:notation}, the corresponding time intervals of the solution maps~$\overline{G}$,~$\overline{H}$, and $\overline{I}$ are visualized. At time $t = 0$, our main analysis is restricted to the compact set~$\mathcal{K}_1$, which is mapped under the continuous map $G$ to the compact set $\mathcal{K}_2 = G(\mathcal{K}_1)$. By Lemma~\ref{lem:exponentialattraction}, all special solutions corresponding to solutions of the DDE \eqref{eq:proof_vectorfield} with initial data $(0,c_{y_1})\in \mathbb{R}\times \mathcal{C}$ with $y_1 \in \mathcal{K}_1$ have at time $t=t^\ast$ at most the distance $\delta_{2,\beta}$ to the general solutions. The following lemma shows that $\overline{I}$ is a homeomorphism of $\mathbb{R}^m$ onto itself.

	\begin{lemma}[Homeomorphic Transformation] \label{lem:diffeomorphisms}
		Let the underlying vector field $\overline{F}$ of the neural DDE $\overline{\Phi}\in\overline{\textup{NDDE}}^k_{\tau,\textup{N},K}(\mathcal{X},\mathbb{R}^q)$, $\mathcal{X}\subset \mathbb{R}^n$ open, $k \geq 0$, fulfill Assumption \ref{ass:weaklynonlinear} with global Lipschitz constant $K>0$. Let $K\tau e <1$ and the map $\overline{I}$ be defined by \eqref{eq:mapIbar}, which is the solution map over the time interval $[t^\ast,T]$ of the ODE~\eqref{eq:proof_specialODE} generating all special solutions. Then $\overline{I}:\mathbb{R}^m\rightarrow\mathbb{R}^m$ is a homeomorphism, whose inverse is the solution map of the ODE \eqref{eq:proof_specialODE} backward in time from $T$ to $t^\ast$.
	\end{lemma}
	
	\begin{proof}
		By Theorem~\ref{th:specialsolutionsODEs}, the vector field $\widetilde{F}$ of the ODE \eqref{eq:proof_specialODE} is continuous and globally Lipschitz continuous with respect to its second variable, with Lipschitz constant smaller than $Ke$. For $\tau = 0$, the same statement holds by the convention $\widetilde{F} = \overline{F}$ with Lipschitz constant $K<Ke$. Hence, for every initial condition $z(t_0) = y_0$ with $(t_0,y_0)\in\mathbb{R}\times\mathbb{R}^m$, the ODE $\frac{\dd z}{\dd t} = \widetilde{F}(t,z(t))$ has a unique solution, which exists globally forward and backward in time. Consequently, $\overline{I}$ is bijective, and its inverse maps $y_3\in\mathbb{R}^m$ to the value at time $t^\ast$ of the solution with $z(T) = y_3$. By Grönwall's inequality, $\overline{I}$ and $\overline{I}^{-1}$ are both Lipschitz continuous with Lipschitz constant $e^{Ke(T-t^\ast)}$, so $\overline{I}$ is a homeomorphism.
	\end{proof}
	
	\subsection{Separation of the Phase Space by Level Sets}
	\label{sec:levelsets}
	
	After having studied the first affine linear transformation $\lambda$ and the DDE dynamics of the neural DDE $\overline{\Phi}\in\overline{\textup{NDDE}}^k_{\tau,\textup{N},K}(\mathcal{X},\mathbb{R}^q)$ via the maps $G$, $H$, $\overline{G}$ and $\overline{H}$ in the last section, we continue with the second affine linear transformation $\tilde \lambda$ and the solution map $\overline{I}$ in this section. We derive conditions, which are necessary for a neural DDE $\overline{\Phi}\in\overline{\textup{NDDE}}^k_{\tau,\textup{N},K}(\mathcal{X},\mathbb{R}^q)$ in order to approximate in the $i$-th component a given map $\Psi_i$ with a non-degenerate local extreme point with accuracy $\eps>0$, as defined in Theorem \ref{th:noapproximation}, and interpret them geometrically as a separation of the phase space by level sets. To illustrate this geometry, we state the Jordan-Brouwer separation theorem, for which we need to define an orientable hypersurface.
	
	\begin{definition}[Orientable Hypersurface \cite{Lima1988}]\label{def:hypersurface}
		A subset $\mathcal{M} \subset \mathbb{R}^m$ is called a $C^k$-hypersurface, $k \geq 1$, if for every $x \in \mathcal{M}$ there exists a function $\varphi \in C^k(\mathcal{U},\mathbb{R})$ with $x \in \mathcal{U}$, $\mathcal{U} \subset \mathbb{R}^m$ open, such that $\nabla \varphi (x) \neq 0$ for all $x \in \mathcal{U}$ and $\varphi^{-1}(z) = \mathcal{M} \cap \mathcal{U}$ for some $z \in \mathbb{R}$. A $C^k$-hypersurface, $k \geq 1$, is called orientable if there exists a continuous vector field of normal unit vectors, i.e., a map $v\in C^0(\mathcal{M},\mathbb{R}^m)$ with $\norm{v(x)}_2 = 1$ and $v(x)$ is normal to $\mathcal{M}$ at $x$ for every $x \in \mathcal{M}$.
	\end{definition}
	
	For closed, orientable hypersurfaces $\mathcal{M}$, we introduce in the following the Jordan-Brouwer separation theorem, which states that the complement of $\mathcal{M}$ in $\mathbb{R}^m$ consists of exactly two components. If the hypersurface $\mathcal{M}$ is compact, the space $\mathbb{R}^m$ is separated by $\mathcal{M}$ into an interior and an exterior set.
	
	\begin{theorem}[Jordan-Brouwer Separation Theorem \cite{Lima1988}] \label{th:JordanBrouwer}
		Let $\mathcal{M}\subset \mathbb{R}^m$ be a connected, closed, orientable $C^k$-hypersurface, $k\geq 1$. Then its complement $\mathbb{R}^m\setminus\mathcal{M}$ has two connected open components~$\mathcal{M}_1$ and $\mathcal{M}_2$, each of which has $\mathcal{M}$ as its boundary. 
	\end{theorem}

	The original theorem in \cite{Lima1988} is given under the additional assumption that the hypersurface $\mathcal{M}$ is compact and smooth, but the author states in a remark that the same proof applies for closed and continuously differentiable hypersurfaces. In Figure \ref{fig:separationtheorem}, the Jordan-Brouwer separation theorem is visualized both for the case that the hypersurface $\mathcal{M}$ is compact and the case that $\mathcal{M}$ is non-compact.

		\begin{figure}
		\centering
		\vspace{-2mm}
		\begin{subfigure}{0.44\textwidth}
			\centering
			\begin{tikzpicture}[scale=1.3]
			    \fill[ddeGrayLight] \blobcurve;
			    \draw[ddeGrayDark, line width=1.1pt] \blobcurve;
			    \node[ddeGrayDark] at (0.25,1.4) {$\mathcal{M}$};
			    \node[ddeGrayDark] at (1.45,2.55) {$\mathcal{M}_1$};
			    \node[ddeGrayDark] at (3.45,1.0) {$\mathcal{M}_2$};
			\end{tikzpicture}
			\vspace{5mm}
			\caption{If the hypersurface $\mathcal{M}$ is compact,  $\mathbb{R}^m\setminus\mathcal{M}$ consists of one bounded component $\mathcal{M}_1$ and one unbounded component $\mathcal{M}_2$.}
			\label{fig:separationtheoremA}
		\end{subfigure}
		\begin{subfigure}{0.05\textwidth}
			\textcolor{white}{.}
		\end{subfigure}
		\begin{subfigure}{0.44\textwidth}
			\centering
			\begin{tikzpicture}[scale=1.3]
			    \useasboundingbox (0,0) rectangle (4.3,4.6);
			    \begin{scope}
			        \clip (0,0) rectangle (4.3,4.3);
			        \fill[ddeGrayLight] \noncompactcurve -- (4.3,4.6) -- (4.3,0) -- (0,0) -- cycle;
			    \end{scope}
			    \draw[ddeGrayDark, line width=1.1pt] \noncompactcurve;
			    \node[ddeGrayDark] at (2.5,4.0) {$\mathcal{M}$};
			    \node[ddeGrayDark] at (0.9,3.0) {$\mathcal{M}_1$};
			    \node[ddeGrayDark] at (3.6,1.0) {$\mathcal{M}_2$};
			\end{tikzpicture}
			\vspace{1mm}
			\caption{If the hypersurface $\mathcal{M}$ is non-compact, the complement $\mathbb{R}^m\setminus\mathcal{M}$ consists of two unbounded components $\mathcal{M}_1$ and $\mathcal{M}_2$.}
			\label{fig:separationtheoremB}
		\end{subfigure}
		\caption{Visualization of the Jordan-Brouwer Separation Theorem \ref{th:JordanBrouwer}.}
		\label{fig:separationtheorem}
	\end{figure}

	As our main Theorem \ref{th:noapproximation} treats the $i$-th component of a given map $\Psi_i$, we restrict our upcoming analysis to the $i$-th component of the neural DDE $\overline{\Phi}$, which is calculated as 
	\begin{equation*}
		\overline\Phi_i = \tilde\lambda_i \circ H  \circ \lambda.
	\end{equation*}
	The $i$-th component of the affine linear map $\tilde \lambda$ is the affine linear map $\tilde \lambda_i: \mathbb{R}^m \rightarrow \mathbb{R}$, $\tilde\lambda_i(y) = [\widetilde{W}y]_i + \tilde  b_i$. If we assume that the $i$-th row of the weight matrix $\widetilde{W}$ is non-zero, then the pre-image of $z\in \mathbb{R}$ under the map $\tilde \lambda_i$ 
	\begin{equation} \label{eq:lambdai_preimage1}
		\tilde{\lambda}_i^{-1}\left(z\right) \coloneqq \left\{y \in \mathbb{R}^m: \tilde\lambda_i(y) = z  \right\} \subset \mathbb{R}^m 
	\end{equation}
	is by definition a smooth hyperplane, which separates $\mathbb{R}^m$ into two open half-spaces. As the map $\overline I$ defined in \eqref{eq:mapIbar} is by Lemma \ref{lem:diffeomorphisms} a homeomorphism, the pre-image of $\tilde{\lambda}_i^{-1}(z)$ under the map $\overline{I}$ is a level set $\mathcal{H}_z$. Hence, the hyperplane $\tilde{\lambda}_i^{-1}(z)$ deformed by $\overline{I}^{-1}$ separates $\mathbb{R}^m$ into two unbounded components $\mathcal{H}_z^{<}$ and $\mathcal{H}_z^{>}$. Geometrically, this is the non-compact situation of the Jordan-Brouwer Separation Theorem~\ref{th:JordanBrouwer} illustrated in Figure~\ref{fig:separationtheorem}\subref{fig:separationtheoremB}. As $\mathcal{H}_z$ does not need to be differentiable, we derive the separation directly from the homeomorphism $\overline{I}$ instead of using the Jordan-Brouwer Separation Theorem~\ref{th:JordanBrouwer}.
	
	\begin{lemma}[Separating Level Sets] \label{lem:hypersurface}
		Let  $\overline{\Phi}\in\overline{\textup{NDDE}}^k_{\tau,\textup{N},K}(\mathcal{X},\mathbb{R}^q)$, $\mathcal{X}\subset \mathbb{R}^n$ open, $k \geq 0$, be a neural DDE with underlying vector field $\overline{F}$ satisfying Assumption \ref{ass:weaklynonlinear} and $K\tau e<1$. Let  $\tilde \lambda_i: \mathbb{R}^m \rightarrow \mathbb{R}$, $\tilde\lambda_i(y) = [\widetilde{W}y]_i + \tilde  b_i$ be an affine linear map with  $\widetilde{W}\in\mathbb{R}^{q \times m}$, $\tilde b_i \in \mathbb{R}$ and $[\widetilde{W}^\top]_i \neq 0$, let $\overline{I}$ be the map defined in \eqref{eq:mapIbar}, and let $\Lambda_i\coloneqq\tilde\lambda_i\circ\overline{I}:\mathbb{R}^m\rightarrow\mathbb{R}$. For $z\in\mathbb{R}$, define
		\begin{equation*}
			\mathcal{H}_z \coloneqq \Lambda_i^{-1}(z), \qquad \mathcal{H}_z^{<} \coloneqq \Lambda_i^{-1}((-\infty,z)), \qquad \mathcal{H}_z^{>} \coloneqq \Lambda_i^{-1}((z,\infty)).
		\end{equation*}
		Then $\mathcal{H}_z = \overline{I}^{-1}(\tilde{\lambda}_i^{-1}(z))$ is closed and connected, unbounded for $m\geq 2$, and $\mathcal{H}_z^{<}$ and $\mathcal{H}_z^{>}$ are the two connected components of $\mathbb{R}^m\setminus\mathcal{H}_z$, which are open, unbounded, and have $\mathcal{H}_z$ as their boundary. Moreover, $\Lambda_i$ has no local extremum. The closure of the open sets $\mathcal{H}_z^{<}$ and $\mathcal{H}_z^{>}$ is denoted by $\mathcal{H}_z^{\leq} \coloneqq  \mathcal{H}_z^{<} \cup \mathcal{H}_z $ and $\mathcal{H}_z^{\geq}\coloneqq  \mathcal{H}_z^{>} \cup \mathcal{H}_z $, respectively.
	\end{lemma}
	
	\begin{proof}
		By Lemma~\ref{lem:diffeomorphisms}, $\overline{I}$ is a homeomorphism of $\mathbb{R}^m$ onto itself, and $\mathcal{H}_z$, $\mathcal{H}_z^{<}$, and $\mathcal{H}_z^{>}$ are the pre-images under $\overline{I}$ of the hyperplane $\tilde{\lambda}_i^{-1}(z)$ and of the two open half-spaces on which $\tilde\lambda_i<z$ and $\tilde\lambda_i>z$, respectively. As $[\widetilde{W}^\top]_i\neq 0$, the hyperplane $\tilde{\lambda}_i^{-1}(z)$ is closed and connected, and unbounded for $m\geq 2$. Its complement consists of two half-spaces, which are open, connected, and unbounded, with the hyperplane as their common boundary. The homeomorphism $\overline{I}^{-1}$ preserves closedness, openness, connectedness, and boundaries, which yields the corresponding properties of $\mathcal{H}_z$, $\mathcal{H}_z^{<}$, and $\mathcal{H}_z^{>}$. If one of the sets $\mathcal{H}_z$, $\mathcal{H}_z^{<}$ or $\mathcal{H}_z^{>}$ were bounded, its closure would be compact, and so would be the image of this closure under the continuous map $\overline{I}$, which contains an unbounded half-space or, for $m\geq 2$, the unbounded hyperplane. Hence, also unboundedness is preserved.
		
		Assume that $\Lambda_i$ has a local minimum at $y_2\in\mathbb{R}^m$, i.e., $\Lambda_i(y)\geq\Lambda_i(y_2)$ for all $y\in\mathcal{U}_{y_2}$, where $\mathcal{U}_{y_2}$ is an open neighborhood of $y_2$. As $\overline{I}$ is a homeomorphism, it maps $\mathcal{U}_{y_2}$ onto the open neighborhood $\overline{I}(\mathcal{U}_{y_2})$ of $\overline{I}(y_2)$, which contains all points $\overline{I}(y_2)-s[\widetilde{W}^\top]_i$ for $s>0$ sufficiently small. For a fixed $s^\ast>0$ sufficiently small, it follows
		\begin{equation*}
			y_2^\ast\coloneqq\overline{I}^{-1}\left(\overline{I}(y_2)-s^\ast[\widetilde{W}^\top]_i\right)\in\mathcal{U}_{y_2}
		\end{equation*}
		and consequently
		\begin{equation*}
			\Lambda_i(y_2^\ast) = \tilde\lambda_i\left(\overline{I}(y_2)-s^\ast[\widetilde{W}^\top]_i\right) = \tilde\lambda_i\left(\overline{I}(y_2)\right)-s^\ast\norm{[\widetilde{W}^\top]_i}_2^2<\Lambda_i(y_2),
		\end{equation*}
		which is a contradiction to the assumption that $y_2$ is a local minimum of $\Lambda_i$. A local maximum is excluded analogously, using the point $\overline{I}(y_2)+s^\ast[\widetilde{W}^\top]_i$.
	\end{proof}
	
	In the following, we study the relationship between the level sets $\mathcal{H}_z$ of Lemma~\ref{lem:hypersurface} and the boundary of a compact set $\mathcal{K}_1$.
	
	\begin{lemma}[Existence of an Intersection Point of $\partial \mathcal{K}_1$ and $\mathcal{H}^{\leq}_{z}$] \label{lem:intersectionpoint}
		Let $\mathcal{K}_1\subset \mathbb{R}^m$ be a compact set. Let the assumptions of Lemma~\ref{lem:hypersurface} be fulfilled and let $\mathcal{H}_z$ be the level set defined therein for a given $z\in\mathbb{R}$. Assume there is a point $y_1\in\mathcal{K}_1$ with $y_1 \in \mathcal{H}^{<}_{z}$, then it follows that there exists $y_1^\ast \in \partial \mathcal{K}_1$ with $y_1^\ast \in \mathcal{H}^{\leq}_{z}$. 	
	\end{lemma}
	
	\begin{proof}
		Let $\Lambda_i = \tilde\lambda_i\circ\overline{I}$ be as in Lemma~\ref{lem:hypersurface}, which is continuous, has no local extremum, and satisfies $\mathcal{H}_z^{<} = \Lambda_i^{-1}((-\infty,z))$. As $\mathcal{K}_1$ is compact and $\Lambda_i$ is continuous, $\Lambda_i$ attains its minimum over $\mathcal{K}_1$ at some point $y_1^\ast\in\mathcal{K}_1$. For that point it holds $\Lambda_i(y_1^\ast)\leq\Lambda_i(y_1)<z$, i.e., $y_1^\ast\in\mathcal{H}_z^{<}\subset\mathcal{H}_z^{\leq}$. If $y_1^\ast\in\textup{int}(\mathcal{K}_1)$, then~$y_1^\ast$ is a local minimum of $\Lambda_i$ on $\mathbb{R}^m$, contradicting Lemma~\ref{lem:hypersurface}. Hence, $y_1^\ast\in\mathcal{K}_1\setminus\textup{int}(\mathcal{K}_1) = \partial\mathcal{K}_1$.
	\end{proof}
	
	Let $m\geq 2$ and let $\mathcal{K}_1$ be connected with a point $y_1\in\mathcal{K}_1\cap\mathcal{H}_z^{<}$. Then two cases occur, as visualized in Figure~\ref{fig:intersection}: either the level set $\mathcal{H}_z$ does not meet $\mathcal{K}_1$, in which case $\mathcal{K}_1\subset\mathcal{H}_z^{<}$ by connectedness of~$\mathcal{K}_1$, or $\mathcal{H}_z$ meets $\mathcal{K}_1$, in which case the connected and unbounded set $\mathcal{H}_z$ also intersects the boundary~$\partial\mathcal{K}_1$. In both cases, a point $y_1^\ast$ as in Lemma~\ref{lem:intersectionpoint} exists.

	\begin{figure}[b!]
		\vspace{-1mm}
		\centering
		\begin{subfigure}{0.44\textwidth}
			\centering
			\begin{tikzpicture}[scale=1.3]
				\useasboundingbox (0,0) rectangle (4.3,4.6);
				\begin{scope}
					\clip (0,0) rectangle (4.3,4.3);
					\fill[ddeGrayLight] \noncompactcurve -- (4.3,4.6) -- (4.3,0) -- (0,0) -- cycle;
				\end{scope}
				\draw[ddeGrayDark, line width=1.1pt] \noncompactcurve;
				\node[ddeGrayDark] at (2.5,4.0) {$\mathcal{H}_z$};
				\node[ddeGrayDark] at (3.6,1.0) {$\mathcal{H}_z^{>}$};
				\node[ddeGrayDark] at (0.55,1.2) {$\mathcal{H}_z^{<}$};
				\draw[ddePurple, line width=1.0pt] (1.25,2.6) circle (0.75);
				\fill[ddePurple] (1.25,2.6) circle (2.5pt);
				\node[ddePurple] at (0.95,2.6) {$y_1$};
				\node[ddePurple] at (0.5,3.45) {$\mathcal{K}_1$};
				\fill[ddePurple] (1.825,2.118) circle (2.5pt);
				\node[ddePurple] at (1.95,1.8) {$y_1^\ast$};
			\end{tikzpicture}
			\vspace{1mm}
			\caption{It holds $y_1 \in \mathcal{H}_z^<$ and $\mathcal{H}_{z} \subset \mathbb{R}^m\setminus\mathcal{K}_1$, hence every point $y^\ast_1\in\partial \mathcal{K}_1$ fulfills  $y_1^\ast \in \mathcal{H}^{\leq}_{z}$.}
			\label{fig:intersectionA}
		\end{subfigure}
		\begin{subfigure}{0.05\textwidth}
			\textcolor{white}{.}
		\end{subfigure}
		\begin{subfigure}{0.44\textwidth}
			\centering
			\begin{tikzpicture}[scale=1.3]
				\useasboundingbox (0,0) rectangle (4.3,4.6);
				\begin{scope}
					\clip (0,0) rectangle (4.3,4.3);
					\fill[ddeGrayLight] \noncompactcurve -- (4.3,4.6) -- (4.3,0) -- (0,0) -- cycle;
				\end{scope}
				\draw[ddeGrayDark, line width=1.1pt] \noncompactcurve;
				\node[ddeGrayDark] at (2.5,4.0) {$\mathcal{H}_z$};
				\node[ddeGrayDark] at (3.6,1.0) {$\mathcal{H}_z^{>}$};
				\node[ddeGrayDark] at (0.55,1.2) {$\mathcal{H}_z^{<}$};
				\draw[ddePurple, line width=1.0pt] (2.05,2.45) circle (0.75);
				\fill[ddePurple] (2.05,2.45) circle (2.5pt);
				\node[ddePurple] at (1.75,2.45) {$y_1$};
				\node[ddePurple] at (1.25,3.3) {$\mathcal{K}_1$};
				\fill[ddePurple] (2.801,2.470) circle (2.5pt);
				\node[ddePurple] at (3.2,2.3) {$y_1^\ast$};
			\end{tikzpicture}
			\vspace{1mm}
			\caption{It holds $\mathcal{H}_{z} \cap \partial\mathcal{K}_1 \neq \varnothing$, hence there exists  $y^\ast_1\in\mathcal{H}_{z} \cap \partial\mathcal{K}_1$ and it holds $y_1^\ast \in \mathcal{H}^{\leq}_{z}$.}
			\label{fig:intersectionB}
		\end{subfigure}
		\caption{Visualization of the two cases of Lemma \ref{lem:intersectionpoint}.}
		\label{fig:intersection}
	\end{figure}

	In addition to pre-images of points under the map $\tilde \lambda_i$ defined in \eqref{eq:lambdai_preimage1}, we also define the pre-image of an open interval
	\begin{equation} \label{eq:lambdai_preimage2}
		\tilde{\lambda}_i^{-1}\left((c,d)\right) \coloneqq \left\{y \in \mathbb{R}^m: \tilde\lambda_i(y) \in (c,d)  \right\} \subset \mathbb{R}^m,
	\end{equation}
	which is an open subset of $\mathbb{R}^m$.
	In our main Theorem \ref{th:noapproximation}, we study the approximation of a given component map $\Psi_i$ by the $i$-th component of the neural DDE, denoted by $\overline{\Phi}_i$. By Lemma \ref{lem:initial_behavior}, the map $\Psi_i:\mathcal{K}_0\rightarrow \mathbb{R}$ attains its global minimum at $p$ with value $\Psi_i(p)$ and its global maximum at the boundary $\partial \mathcal{K}_0$ with value $\Psi_i(p)+r_0^2$. To study whether $\Psi_i$ can be approximated by $\overline{\Phi}_i$ with accuracy~$\eps$, we define the following pre-images under the map $\tilde \lambda_i$:
	\begin{align}
		\label{eq:Y3}\mathcal{Y}_3 &\coloneqq \tilde{\lambda}^{-1}_i\left((\Psi_i(p)-\eps,\Psi_i(p)+\eps)\right) \subset \mathbb{R}^m, \\
		\label{eq:Z3}\mathcal{Z}_3 & \coloneqq \tilde{\lambda}^{-1}_i\left((\Psi_i(p)+r_0^2-\eps,\Psi_i(p)+r_0^2+\eps)\right) \subset \mathbb{R}^m.
	\end{align}
	In addition, we define for the $\tau$-dependent constant $\delta_{3,\tau}$ of Lemma \ref{lem:exponentialattraction} and the constant $\tilde w$ with $\norm{[\widetilde{W}^\top]_i}_2\leq\tilde w$ the extensions of the pre-images $\mathcal{Y}_3 $ and $\mathcal{Z}_3$, given by 
	\begin{align}
		\label{eq:Y3delta}\mathcal{Y}_{3,\delta_{3,\tau}} &\coloneqq \tilde{\lambda}_i^{-1}\left((\Psi_i(p)-\eps-\tilde w\delta_{3,\tau},\Psi_i(p)+\eps+\tilde w\delta_{3,\tau})\right) \\
		&\supseteq \left\{z+ \alpha:z \in \mathcal{Y}_3, \alpha \in B_{\delta_{3,\tau}}(0) \right\}, \nonumber\\
		\label{eq:Z3delta}\mathcal{Z}_{3,\delta_{3,\tau}} &\coloneqq \tilde{\lambda}^{-1}_i\left((\Psi_i(p)+r_0^2-\eps-\tilde w\delta_{3,\tau},\Psi_i(p)+r_0^2+\eps+\tilde w\delta_{3,\tau})\right) \\
		&\supseteq \left\{z+ \alpha:z \in \mathcal{Z}_3, \alpha \in B_{\delta_{3,\tau}}(0) \right\}. \nonumber
	\end{align}
	The inclusions hold as $\tilde\lambda_i$ has Euclidean operator norm $\norm{[\widetilde{W}^\top]_i}_2\leq\tilde w$, so a displacement of length at most $\delta_{3,\tau}$ changes the value of $\tilde\lambda_i$ by at most $\tilde w\delta_{3,\tau}$.
	In the following lemma, we show that under the assumption $0<2\eps<r_0^2$ of our main Theorem \ref{th:noapproximation}, there exists an upper bound $\tau_3$, such that for all $\tau\in[0,\tau_3]$, the sets $\mathcal{Y}_{3,\delta_{3,\tau}} $ and $\mathcal{Z}_{3,\delta_{3,\tau}}$ are disjoint. The sets $\mathcal{Y}_3$, $\mathcal{Z}_3$, $\mathcal{Y}_{3,\delta_{3,\tau_3}}$ and $\mathcal{Z}_{3,\delta_{3,\tau_3}}$ are visualized in Figure \ref{fig:setsYandZ}.
	
	\begin{lemma}[Positive Distance between $ \mathcal{Y}_{3,\delta_{3,\tau}}$ and $\mathcal{Z}_{3,\delta_{3,\tau}}$] \label{lem:distance}
		Let $0<2\eps<r_0^2$ and  $\tilde \lambda_i: \mathbb{R}^m \rightarrow \mathbb{R}$, $\tilde\lambda_i(y) = [\widetilde{W}y]_i + \tilde  b_i$ be an affine linear map with  $\widetilde{W}\in\mathbb{R}^{q \times m}$, $\tilde b_i \in \mathbb{R}$, $[\widetilde{W}^\top]_i \neq 0$ and $\normm{\widetilde{W}}_\infty\leq\tilde w$.  Let the assumptions of Lemma~\ref{lem:hypersurface} and Lemma~\ref{lem:exponentialattraction} be fulfilled, let $\mathcal{H}_z$ be the level set defined in Lemma~\ref{lem:hypersurface}, and let $C_2$ and $\delta_{3,\tau}$ be as in Lemma~\ref{lem:exponentialattraction}. Then there exists $\tau_3>0$, such that for all $\tau\in [0,\tau_3]$ it holds
		\begin{equation*}
			\mathcal{Y}_{3,\delta_{3,\tau}}\subset \mathcal{Y}_{3,\delta_{3,\tau_3}}, \qquad  \mathcal{Z}_{3,\delta_{3,\tau}}\subset\mathcal{Z}_{3,\delta_{3,\tau_3}}, \qquad \mathcal{Y}_{3,\delta_{3,\tau}} \cap	\mathcal{Z}_{3,\delta_{3,\tau}} = \varnothing
		\end{equation*} 
		and 
		\begin{equation*}
			\normm{\tilde{\lambda}_i(y_3)-\tilde{\lambda}_i(z_3)}_\infty \geq 	r_0^2-2\eps- 2\tilde w\delta_{3,\tau_3} \eqqcolon \delta^\ast > 0
		\end{equation*}
		for all $y_3 \in \mathcal{Y}_{3,\delta_{3,\tau}}$ and $z_3 \in \mathcal{Z}_{3,\delta_{3,\tau}}$. Furthermore, it holds
		\begin{equation*}
			\normm{\tilde{\lambda}_i(\overline{I}(y_2))-\tilde{\lambda}_i(\overline{I}(z_2))}_\infty \geq \delta^\ast > 0
		\end{equation*}
		for all $y_2\in	\mathcal{H}^{\leq}_{\Psi_i(p)+\eps+\tilde w\delta_{3,\tau_3}}$  and $z_2\in\mathcal{H}^{\geq}_{\Psi_i(p)+r_0^2-\eps-\tilde w\delta_{3,\tau_3}}$. 
	\end{lemma}
	
	\begin{proof}
		By Lemma \ref{lem:exponentialattraction} it holds $\delta_{3,\tau} = C_2 e^{-\frac{T}{\tau}}$ for some constant $C_2>\frac{r_1}{2}>0$. Let
		\begin{equation*}
			\delta^\ast\coloneqq\frac{1}{2}\left(\max\left\{0,r_0^2-2\eps-2\tilde wC_2\right\}+r_0^2-2\eps\right),
		\end{equation*}
		which satisfies $\max\{0,r_0^2-2\eps-2\tilde wC_2\}<\delta^\ast<r_0^2-2\eps$, as $0<2\eps<r_0^2$, and depends continuously on~$C_2$. Define 
		\begin{equation*}
			\tau_3 \coloneqq \frac{T}{\ln\left(\frac{2\tilde wC_2}{r_0^2-2\eps-\delta^\ast}\right)}>0.
		\end{equation*}
		Then it holds for all $\tau \in [0,\tau_3]$ that
		\begin{equation*}
			r_0^2-2\eps- 2\tilde w\delta_{3,\tau} \geq	r_0^2-2\eps- 2\tilde w\delta_{3,\tau_3} = r_0^2-2\eps-2\tilde wC_2\cdot \frac{r_0^2-2\eps-\delta^\ast}{2\tilde wC_2} = \delta^\ast> 0. 
		\end{equation*}
		As it holds by definitions \eqref{eq:Y3delta} and \eqref{eq:Z3delta} that
		\begin{equation*}
			\mathcal{Y}_{3,\delta_{3,\tau}}\subset \mathcal{Y}_{3,\delta_{3,\tau_3}} \quad \text{and} \quad \mathcal{Z}_{3,\delta_{3,\tau}}\subset\mathcal{Z}_{3,\delta_{3,\tau_3}},
		\end{equation*}
		it follows $\mathcal{Y}_{3,\delta_{3,\tau}} \cap \mathcal{Z}_{3,\delta_{3,\tau}} =  \varnothing$ for every $\tau\in [0,\tau_3]$ and
		\begin{equation*}
			\normm{\tilde{\lambda}_i(y_3)-\tilde{\lambda}_i(z_3)}_\infty \geq r_0^2-2\eps- 2\tilde w\delta_{3,\tau}  \geq r_0^2-2\eps- 2\tilde w\delta_{3,\tau_3} = \delta^\ast > 0
		\end{equation*}
		for all $y_3 \in \mathcal{Y}_{3,\delta_{3,\tau}}$ and $z_3 \in \mathcal{Z}_{3,\delta_{3,\tau}}$.	By the definition of the sets $\mathcal{H}_z^{\leq}$ and $\mathcal{H}_z^{\geq}$ of Lemma \ref{lem:hypersurface} it follows for $y_2\in\mathcal{H}^{\leq}_{\Psi_i(p)+\eps+\tilde w\delta_{3,\tau_3}}$ and $z_2\in\mathcal{H}^{\geq}_{\Psi_i(p)+r_0^2-\eps-\tilde w\delta_{3,\tau_3}}$ that
		\begin{equation*}
			\normm{\tilde{\lambda}_i(\overline{I}(y_2))-\tilde{\lambda}_i(\overline{I}(z_2))}_\infty \geq \delta^\ast > 0. \qedhere
		\end{equation*}
	\end{proof}

	\begin{figure}
		\centering
		\resizebox{0.9\textwidth}{!}{%
		\begin{tikzpicture}[>={Stealth[length=2.0mm]},x=0.92cm,y=1cm]
		  \colorlet{cSupFull}{ddeTeal!15}
		  \colorlet{cSupTight}{ddeTeal!40}
		  \colorlet{cSupLine}{ddeTeal}
		  \colorlet{cSupText}{ddeTeal!75!black}
		  \colorlet{cSubFull}{ddeIndigo!15}
		  \colorlet{cSubTight}{ddeIndigo!40}
		  \colorlet{cSubLine}{ddeIndigo}
		  \colorlet{cSubText}{ddeIndigo}
		  \colorlet{cDist}{ddeGray}
		  \fill[cSupFull]  (0,5.05) rectangle (8,7.55);
		  \fill[cSupTight] (0,5.60) rectangle (8,7.00);
		  \draw[cSupLine,very thick] (0,7.55) -- (8,7.55);
		  \draw[cSupLine,very thick] (0,7.00) -- (8,7.00);
		  \draw[cSupLine,very thick] (0,6.30) -- (8,6.30);
		  \draw[cSupLine,very thick] (0,5.60) -- (8,5.60);
		  \draw[cSupLine,very thick] (0,5.05) -- (8,5.05);
		  \fill[cSubFull]  (0,0.75) rectangle (8,3.25);
		  \fill[cSubTight] (0,1.30) rectangle (8,2.70);
		  \draw[cSubLine,very thick] (0,3.25) -- (8,3.25);
		  \draw[cSubLine,very thick] (0,2.70) -- (8,2.70);
		  \draw[cSubLine,very thick] (0,2.00) -- (8,2.00);
		  \draw[cSubLine,very thick] (0,1.30) -- (8,1.30);
		  \draw[cSubLine,very thick] (0,0.75) -- (8,0.75);
		  \draw[cSupLine,decorate,decoration={brace,amplitude=4pt}]
		    (-0.15,7.00) -- (-0.15,7.55) node[midway,anchor=east,xshift=-4pt,cSupText,font=\small] {$\tilde w\delta_{3,\tau_3}$};
		  \draw[cSupLine,decorate,decoration={brace,amplitude=4pt}]
		    (-0.15,6.30) -- (-0.15,7.00) node[midway,anchor=east,xshift=-4pt,cSupText,font=\small] {$\eps$};
		  \draw[cSupLine,decorate,decoration={brace,amplitude=4pt}]
		    (-0.15,5.60) -- (-0.15,6.30) node[midway,anchor=east,xshift=-4pt,cSupText,font=\small] {$\eps$};
		  \draw[cSupLine,decorate,decoration={brace,amplitude=4pt}]
		    (-0.15,5.05) -- (-0.15,5.60) node[midway,anchor=east,xshift=-4pt,cSupText,font=\small] {$\tilde w\delta_{3,\tau_3}$};
		
		  \draw[cSubLine,decorate,decoration={brace,amplitude=4pt}]
		    (-0.15,2.70) -- (-0.15,3.25) node[midway,anchor=east,xshift=-4pt,cSubText,font=\small] {$\tilde w\delta_{3,\tau_3}$};
		  \draw[cSubLine,decorate,decoration={brace,amplitude=4pt}]
		    (-0.15,2.00) -- (-0.15,2.70) node[midway,anchor=east,xshift=-4pt,cSubText,font=\small] {$\eps$};
		  \draw[cSubLine,decorate,decoration={brace,amplitude=4pt}]
		    (-0.15,1.30) -- (-0.15,2.00) node[midway,anchor=east,xshift=-4pt,cSubText,font=\small] {$\eps$};
		  \draw[cSubLine,decorate,decoration={brace,amplitude=4pt}]
		    (-0.15,0.75) -- (-0.15,1.30) node[midway,anchor=east,xshift=-4pt,cSubText,font=\small] {$\tilde w\delta_{3,\tau_3}$};
		  \node[cSupText,anchor=west,font=\small] at (8.2,7.55)
		    {$\tilde\lambda_i^{-1}\left(\Psi_i(p)+r_0^2+\eps+\tilde w\delta_{3,\tau_3}\right)$};
		  \node[cSupText,anchor=west,font=\small] at (8.2,6.30)
		    {$\tilde\lambda_i^{-1}\left(\Psi_i(p)+r_0^2\right)$};
		  \node[cSupText,anchor=west,font=\small] at (8.2,5.05)
		    {$\tilde\lambda_i^{-1}\left(\Psi_i(p)+r_0^2-\eps-\tilde w\delta_{3,\tau_3}\right)$};
		
		  \node[cSubText,anchor=west,font=\small] at (8.2,3.25)
		    {$\tilde\lambda_i^{-1}\left(\Psi_i(p)+\eps+\tilde w\delta_{3,\tau_3}\right)$};
		  \node[cSubText,anchor=west,font=\small] at (8.2,2.00)
		    {$\tilde\lambda_i^{-1}\left(\Psi_i(p)\right)$};
		  \node[cSubText,anchor=west,font=\small] at (8.2,0.75)
		    {$\tilde\lambda_i^{-1}\left(\Psi_i(p)-\eps-\tilde w\delta_{3,\tau_3}\right)$};
		  \node[cSupText,anchor=base,font=\small] at (3.0,4.55) {$\mathcal{Z}_3$};
		  \draw[cSupLine,thick,->] (3.0,4.80) -- (3.0,6.20);
		  \node[cSupText,anchor=base,font=\small] at (5.0,4.55) {$\mathcal{Z}_{3,\delta_{3,\tau_3}}$};
		  \draw[cSupLine,thick,->] (5.0,4.80) -- (5.0,5.28);
		
		  \node[cSubText,anchor=base,font=\small] at (3.0,3.60) {$\mathcal{Y}_3$};
		  \draw[cSubLine,thick,->] (3.0,3.46) -- (3.0,2.07);
		  \node[cSubText,anchor=base,font=\small] at (5.0,3.60) {$\mathcal{Y}_{3,\delta_{3,\tau_3}}$};
		  \draw[cSubLine,thick,->] (5.0,3.46) -- (5.0,3.02);
		  \draw[cDist,very thick,<->] (1.0,3.25) -- (1.0,5.05);
		  \node[cDist,anchor=west,font=\small] at (1.15,4.15) {positive distance};
		\end{tikzpicture}%
		}
		\caption{Visualization of the sets $\mathcal{Y}_3$, $\mathcal{Z}_3$, $\mathcal{Y}_{3,\delta_{3,\tau_3}}$ and $\mathcal{Z}_{3,\delta_{3,\tau_3}}$. By Lemma \ref{lem:distance}, the sets $\mathcal{Y}_{3,\delta_{3,\tau}}$ and $\mathcal{Z}_{3,\delta_{3,\tau}}$ have for all $\tau \in [0,\tau_3]$ a positive distance.}
		\label{fig:setsYandZ}
	\end{figure}

	\subsection{Proof of the Main Result}
	\label{sec:proof}

	In this section, we combine the results obtained in the previous Sections \ref{sec:initialbehavior} and \ref{sec:levelsets} to prove Theorem~\ref{th:noapproximation}, which is stated for better readability here again.
	
	\begin{customthm}{\ref{th:noapproximation}}[No Approximation of Non-Degenerate Local Extreme Points] 
		Let $\Psi:\mathcal{X}\rightarrow \mathbb{R}^q$,  $\mathcal{X}\subset \mathbb{R}^n$ open, have a component map $\Psi_i\in C^3(\mathcal{X}, \mathbb{R})$, which has a non-degenerate local extreme point with radius of definition $r_0>0$ and let $\eps>0$ with $2\eps<r_0^2$. Furthermore, fix constants $w,\tilde{w},\sigma_0 \in (0,\infty)$, $A,b_\infty\in[0,\infty)$, and $k\geq 0$.
		
		Then there exists a continuous function $\tau_0\in C^0((0,\infty),(0,T])$ with $K \tau_0(K) e<1$ for all $K\in(0,\infty)$ with the following property: for all $K\in(0,\infty)$ and $\tau \in [0,\tau_0(K))$, every neural DDE $\Phi \in \textup{NDDE}_{\tau,\textup{N},K}^k(\mathcal{X},\mathbb{R}^q)$ with
		\begin{itemize}
			\item  vector field $F:\Omega_t \times \mathcal{C} \rightarrow \mathbb{R}^m$ with global Lipschitz constant~$K$ in its second variable on $\Omega_t \times \mathcal{C}$ and $\norm{F(t,0)}_\infty \leq A$ for all $t\in [0,T]$, which is additionally continuous if $k = 0$,
			\item weight matrices $W, \widetilde{W}$ with $\normm{W}_\infty \leq w$, $\normm{\widetilde{W}}_\infty \leq\tilde{w}$ and smallest singular value $\sigma_{\min}(W)\geq\sigma_0$, as well as a bias vector $b$ with $\norm{b}_\infty\leq b_\infty$,
		\end{itemize}
		cannot approximate the map $\Psi$ with accuracy $\eps$ on the compact set $\mathcal{K}\coloneqq\mu(K_{r_0}(0))\subset\mathcal{X}$, where $\mu$ realizes the radius of definition $r_0$ as in Definition~\ref{def:radiusmorse}, i.e., $\norm{\Phi-\Psi}_{\infty,\mathcal{K}}\geq\eps$ for every such $\Phi$.
	\end{customthm}
	
	\begin{proof}
		We subdivide the proof into seven steps.
		\begin{enumerate}
			\item \textbf{Restriction to non-degenerate local minima:} Let $\Psi:\mathcal{X}\rightarrow \mathbb{R}^q$,  $\mathcal{X}\subset \mathbb{R}^n$ open, have a component map $\Psi_i\in C^3(\mathcal{X}, \mathbb{R})$, which has a non-degenerate local extreme point $p \in \mathcal{X}$ with radius of definition $r_0>0$. By Definitions \ref{def:morse} and \ref{def:radiusmorse}, there exists  a neighborhood $\mathcal{U}$ of $0 \in \mathbb{R}^n$ with  $K_{r_0}(0)\subset \mathcal{U}$ and a $C^1$-diffeomorphism $\mu: \mathcal{U} \rightarrow \mu(\mathcal{U}) \subset \mathcal{X}$ with $\mu(0) = p$, realizing the radius of definition $r_0$, such that for $(u_1,\ldots,u_n) \in \mathcal{U}$
			\begin{equation} \label{eq:psiform2}
				\Psi_i(\mu(u_1,\ldots,u_n)) = \Psi_i(p) \pm \sum_{j = 1}^n u_j^2,
			\end{equation}
			as the non-degenerate local extreme point $p \in \mathcal{X}$ is a critical point with index $r = 0$ in case of a local minimum and index $r = n$ in case of a local maximum. The case $r = 0$ corresponds to a plus sign in \eqref{eq:psiform2}, the case $r = n$ to a minus sign. Let $\mathcal{K}_0\coloneqq\mu(K_{r_0}(0))\subset\mathcal{X}$, which is the compact set $\mathcal{K}$ of Theorem~\ref{th:noapproximation}.
			
			Without loss of generality, we can assume that the extreme point $p\in \mathcal{X}$ is a local minimum. Otherwise, the proof is carried out for the approximation of the map $-\Psi$, whose component map $-\Psi_i$ satisfies \eqref{eq:psiform2} with a plus sign and the same diffeomorphism $\mu$. Hence, the map $-\Psi_i$ has a non-degenerate local minimum at $p$ with the same radius of definition $r_0$ and the same compact set $\mathcal{K}_0$. A map $\Psi$ can be approximated with accuracy $\eps$ on $\mathcal{K}_0$ by a neural DDE  $\Phi \in \textup{NDDE}_{\tau,\textup{N},K}^k(\mathcal{X},\mathbb{R}^q)$ satisfying the assumptions of Theorem~\ref{th:noapproximation} if and only if $-\Psi$ can be approximated with accuracy $\eps$ on $\mathcal{K}_0$ by some neural DDE  $\Phi_{-}$: given a neural DDE $\Phi$ approximating a given map $\Psi$ with accuracy $\eps$ on $\mathcal{K}_0$, i.e.,
			\begin{equation*}
				\norm{\Phi(x)-\Psi(x)}_\infty < \eps \qquad \text{for all } x \in \mathcal{K}_0,
			\end{equation*}
			then the neural DDE $\Phi_{-} \in \textup{NDDE}_{\tau,\textup{N},K}^k(\mathcal{X},\mathbb{R}^q)$, which is constructed by replacing the last affine linear layer  $\tilde{\lambda}$ by $-\tilde{\lambda}$, approximates $-\Psi$ with accuracy $\eps$ on $\mathcal{K}_0$ and vice versa. As $\normm{-\widetilde{W}}_\infty = \normm{\widetilde{W}}_\infty$, the neural DDE $\Phi_-$ satisfies the same parameter constraints as $\Phi$.
			
			\item \textbf{Restriction to the case $m = n$ and non-degenerate weight matrices:} In this theorem, we study non-augmented neural DDEs $\Phi\in \textup{NDDE}_{\tau,\textup{N},K}^k(\mathcal{X},\mathbb{R}^q)$, i.e., it holds $n \geq m$, where the first affine linear map is given by $\lambda:\mathbb{R}^n\rightarrow\mathbb{R}^m$, $\lambda(x) = Wx+b$ with $W\in\mathbb{R}^{m\times n}$.
			
			We first show that weight matrices $W$ with $\textup{rank}(W)<n$ cannot approximate the map $\Psi$. This includes every weight matrix $W$ in the case $n>m$, cf.\ Theorem~\ref{th:bottleneck}. By Lemma~\ref{lem:initial_behavior}(c), there exists a point $s \in \partial \mathcal{K}_0$ with $\lambda(s) = \lambda(p)$. For every neural DDE $\Phi\in\textup{NDDE}_{\tau,\textup{N},K}^k(\mathcal{X},\mathbb{R}^q)$ with such a weight matrix $W$ it follows $\Phi(p) = \Phi(s)$, as after the affine linear transformation $\lambda$ the same initial data is used for the underlying DDE as $\lambda(s) = \lambda(p)$. By Lemma \ref{lem:initial_behavior}(a) it holds $\Psi_i(s) = \Psi_i(p)+r_0^2$, which implies
			\begin{align*}
				\norm{\Phi-\Psi}_{\infty,\mathcal{K}_0} \geq \;	&\sup_{x\in \mathcal{K}_0}\abs{\Phi_i(x)-\Psi_i(x)}\geq \max_{x\in \{p,s\}}\abs{\Phi_i(x)-\Psi_i(x)} \\
				= \; & \max\left\{ \abs{\Phi_i(p)-\Psi_i(p)}, \abs{\Phi_i(p) - \Psi_i(p)-r_0^2} \right\} \geq \frac{r_0^2}{2}> \eps,
			\end{align*}
			as by assumption $2 \eps < r_0^2$. Hence, it suffices to restrict the upcoming analysis to $n = m$ and $\textup{rank}(W) = n$, which is in that case also implied by the assumption $\sigma_{\min}(W)\geq\sigma_0>0$.
			
			The second affine linear map of the neural DDE architecture is a map $\tilde{\lambda}: \mathbb{R}^m \rightarrow \mathbb{R}^q$, $\tilde \lambda (x) = \widetilde{W} x + \tilde b$. In the upcoming analysis, we assume that the $i$-th row of the weight matrix $\widetilde{W} \in \mathbb{R}^{q \times m}$, denoted by $[\widetilde{W}^\top]_i \in \mathbb{R}^m$ is not the zero vector. If $[\widetilde{W}^\top]_i = 0$, then the output of the $i$-th component $\Phi_i$ of every neural DDE $\Phi\in\textup{NDDE}_{\tau,\textup{N},K}^k(\mathcal{X},\mathbb{R}^q)$ with such a weight matrix $\widetilde{W}$ would be constant $\Phi_i(x) = \tilde b_i$ for all $x \in \mathcal{X}$, as $\tilde \lambda_i$ then becomes the constant function $x \mapsto \tilde b_i$. A constant component $\Phi_i$ of the neural DDE cannot approximate the given map $\Psi_i$ with accuracy~$\eps$, as
			\begin{align*}
				\norm{\Phi-\Psi}_{\infty,\mathcal{K}_0} \geq 	\;&\sup_{x\in \mathcal{K}_0}\abs{\Phi_i(x)-\Psi_i(x)} \geq \sup_{x\in\{p\}\cup\partial \mathcal{K}_0}\abs{\tilde b_i-\Psi_i(x)} \\
				\geq \; & \max\left\{ \abs{\tilde b_i-\Psi_i(p)}, \abs{\tilde b_i- \Psi_i(p)-r_0^2} \right\} \geq \frac{r_0^2}{2}> \eps,
			\end{align*} 
			where we used again Lemma \ref{lem:initial_behavior}(a) that $\Psi_i(x) = \Psi_i(p)+r_0^2$ for all $x \in \partial \mathcal{K}_0$ and $2 \eps < r_0^2$ by assumption.  Hence, Theorem \ref{th:noapproximation} holds for weight matrices $\widetilde{W}$ with $[\widetilde{W}^\top]_i = 0$, such that we restrict the upcoming analysis to the case that $[\widetilde{W}^\top]_i \neq 0$. In summary, it suffices to consider weight matrices with $\textup{rank}(W) = n$ and $[\widetilde{W}^\top]_i\neq 0$, with the additional constraint $\sigma_{\min}(W)\geq\sigma_0$ and the bias bound $\norm{b}_\infty\leq b_\infty$.
			
			\item \textbf{Change to a globally defined vector field:} In this theorem, we study neural DDEs  $\Phi \in \textup{NDDE}_{\tau,\textup{N},K}^k(\mathcal{X},\mathbb{R}^q)$, $k\geq 0$, with vector field $F:\Omega_t \times \mathcal{C} \rightarrow \mathbb{R}^m$ with Lipschitz constant~$K$ in the second variable on $\Omega_t \times \mathcal{C}$ and $\norm{F(t,0)}_\infty \leq A$ for all $t\in [0,T]$, which is continuous if $k = 0$. By Theorem \ref{th:NDDEextended}, there exists a neural DDE $\overline{\Phi}\in \overline{\textup{NDDE}}_{\tau,\textup{N},K}^k(\mathcal{X},\mathbb{R}^q)$ with the same weights and biases, such that  $\Phi(x) = \overline{\Phi}(x)$ for all $x \in \mathcal{X}$, which is based on a vector field $\overline{F}$ fulfilling Assumption \ref{ass:weaklynonlinear} with constants $K>0$, $A\geq 0$, and
			\begin{equation*}
				\overline{F}\big\vert_{[0,T]\times \mathcal{C}}= F\big\vert_{[0,T]\times \mathcal{C}}.
			\end{equation*}
			
			From now on, we study the neural DDE $\overline{\Phi}$ with globally defined vector field $\overline{F}$. By Theorem~\ref{th:NDDEextended}, the final result directly holds for all neural DDEs  $\Phi \in \textup{NDDE}_{\tau,\textup{N},K}^k(\mathcal{X},\mathbb{R}^q)$, $k\geq 0$, which are considered in this theorem. The dynamics of the neural DDE $\overline \Phi$ is subdivided into three maps
			\begin{equation*}
				\overline\Phi = \tilde\lambda \circ H  \circ \lambda,
			\end{equation*}
			with affine linear transformations $\lambda$, $\tilde \lambda$ and the solution map $H$  of the underlying DDE~\eqref{eq:proof_vectorfield} over the time interval $[0,T]$ as defined in  \eqref{eq:mapH}. The solution map of the DDE \eqref{eq:proof_vectorfield} over the time interval $[0,t^\ast]$ with $t^\ast \in [0,T)$ and $t^\ast = \beta \tau$ for some $\beta>0$ specified in the following is denoted by $G:\mathbb{R}^m\rightarrow \mathbb{R}^m$ as defined in \eqref{eq:mapG}. Given $\beta>0$, we need to choose $\tau<\frac{T}{\beta}$ to guarantee a well-defined solution map $G$ on the time interval $[0,t^\ast]\subset[0,T]$.
			\item \textbf{Analysis of the initial behavior via the map $G$: } As we assume by step 2 that $m = n$ and $\sigma_{\min}(W)\geq\sigma_0$, which implies $\text{rank}(W) = n$, Lemma \ref{lem:initial_behavior}(b) implies that  there exists $r_1 > 0$, depending only on $\sigma_0$ and the fixed set $\mathcal{K}_0$ and hence independent of the specific choice of~$W$, such that $K_{r_1}(\lambda(p)) \subset \mathcal{K}_1$, where $\mathcal{K}_1\coloneqq \lambda(\mathcal{K}_0)\subset \mathbb{R}^m$ is compact and $\mathcal{K}_0 \coloneqq \mu(K_{r_0}(0))$. Furthermore, the component map $ \Psi_i \circ \lambda^{-1}: \mathcal{K}_1 \rightarrow \mathbb{R}$ attains its global minimum at $\lambda(p)$ with value $\Psi_i(p)$ and its global maximum at every point of the boundary $\partial \mathcal{K}_1$ with value $\Psi_i(p)+r_0^2$. After the first affine linear layer $\lambda$, the initial behavior of the DDE solution can be characterized via the solution map $G$ over the time interval $[0,t^\ast]$, as defined in \eqref{eq:mapG}. In our setting, all assumptions of Lemma \ref{lem:mapG} are fulfilled. As $\mathcal{K}_1 = W\mathcal{K}_0+b$ with $\normm{W}_\infty\leq w$ and $\norm{b}_\infty\leq b_\infty$, the constant $M = \sup_{y_1\in\mathcal{K}_1}\norm{y_1}_\infty$ of Lemma~\ref{lem:mapG} satisfies $M\leq wM_0+b_\infty$ with the fixed constant $M_0\coloneqq\sup_{x\in\mathcal{K}_0}\norm{x}_\infty$. As the proof of Lemma~\ref{lem:mapG} uses $M$ only as an upper bound for $\norm{y_1}_\infty$, $y_1\in\mathcal{K}_1$, the lemma holds verbatim with $M$ replaced by $wM_0+b_\infty$, i.e., with
			\begin{equation*}
				C_1 = \sqrt{m}\,(K(wM_0+b_\infty)+A)e^{KT}>0.
			\end{equation*}
			Let $\delta_{1,t^\ast}\coloneqq C_1 t^\ast$ with this constant $C_1$, which only depends on the given parameters $K$, $A$, $T$, $w$, $b_\infty$, $m$, and the fixed set $\mathcal{K}_0$, hence $C_1$ is independent of the specific parameters of the neural DDE and of the choice of $\tau$ and $\beta$. From Lemma \ref{lem:mapG} it follows that for every $\beta>0$ and every $0<\kappa\leq r_1$, there exists an upper bound $\tau_{1,\beta,\kappa}=\frac{\kappa}{2C_1 \beta}>0$, such that for all $\tau \in [0,\tau_{1,\beta,\kappa}]$ and every $y_1\in \mathcal{K}_1$ it holds
			\begin{equation} \label{eq:Glemma}
				\norm{G(y_1)-y_1}_{2} \leq \delta_{1,t^\ast} \leq \frac{\kappa}{2}.
			\end{equation}
			In particular, for $\lambda(p)\in \mathcal{K}_1$ we have
			\begin{equation} \label{eq:GlambdaP}
				G(\lambda(p))\in K_{\delta_{1,t^\ast}}(\lambda(p)) \subset K_{r_1}(\lambda(p)) \subset \mathcal{K}_1.
			\end{equation}
			\item  \textbf{Exponential attraction towards special solutions:} After having analyzed the dynamics of the DDE \eqref{eq:proof_vectorfield} on the time interval $[0,t^\ast]$ via the map $G$, we study now the solution map of the DDE \eqref{eq:proof_vectorfield} on the time interval $[t^\ast,T]$, which is characterized by an exponential attraction of general solutions towards special solutions. Special solutions are characterized by an exponential growth bound and a single initial condition and can be obtained as solutions of the initial value problem \eqref{eq:proof_specialODE}. Assume $\tau < \frac{1}{2Ke}$, then $K\tau e<\frac{1}{2}$ and all assumptions of Lemma \ref{lem:exponentialattraction} are fulfilled. In the following we denote by  $\overline{G}$ and $\overline{H}$ the solution maps of the ODE \eqref{eq:proof_specialODE} over the time intervals $[0,t^\ast]$ and $[0,T]$ as defined in \eqref{eq:mapGbar} and \eqref{eq:mapHbar}. By Lemma \ref{lem:exponentialattraction}, applied with the bound $M\leq wM_0+b_\infty$ of step 4, there exists a constant $C_2>\frac{r_1}{2}>0$, depending continuously on $K$ and otherwise only on $A$, $m$, $r_1$, $w$, $b_\infty$, and $\mathcal{K}_0$, and in particular independent of $\tau$ and of the specific neural DDE, such that for every $y_1 \in \mathcal{K}_1$ there exists $\bar{y}_1\in \mathbb{R}^m$,  uniquely determined by the limit \eqref{eq:ICspecialsolution} for the initial data $(0,c_{y_1})\in\R\times \mathcal{C}$ with
			\begin{align}
				\normk{G(y_1)-\overline{G}(\bar{y}_1)}_{2} < \delta_{2,\beta} &\coloneqq C_2 e^{-\beta}, \label{eq:Gestimate} \\
				\normk{H(y_1)-\overline{H}(\bar{y}_1)}_{2} < \delta_{3,\tau} &\coloneqq C_2 e^{-\frac{T}{\tau}}. \label{eq:Hestimate}
			\end{align}
			For $\tau = 0$, Lemma~\ref{lem:exponentialattraction} instead yields $\bar{y}_1 = y_1$, $G(y_1) = \overline{G}(\bar{y}_1)$, and $H(y_1) = \overline{H}(\bar{y}_1)$ with $\delta_{3,0} = 0$, so all upcoming conclusions drawn from the estimates \eqref{eq:Gestimate} and \eqref{eq:Hestimate} also hold for $\tau = 0$. These inequalities estimate the difference between the considered solution of the DDE \eqref{eq:proof_vectorfield} and the corresponding special solution $S(y(0,c_{y_1})) = \bar{y}(0,\bar{y}_1)$ towards which it converges, via the maps $G$ and $\overline{G}$  at time $t = t^\ast$, and via the maps $H$ and $\overline{H}$ at time $t = T$. In the following, we aim to combine the estimate \eqref{eq:Gestimate} with the statements \eqref{eq:Glemma} and \eqref{eq:GlambdaP}. To that purpose, let
			\begin{equation} \label{eq:ass_beta_tau}
				0<\kappa \leq r_1, \quad \beta=\beta_\kappa \coloneqq \ln\left(\frac{2C_2}{\kappa}\right)> 0 \quad \text{and} \quad 0\leq \tau< \min\left\{\tau_{1,\beta,\kappa},\frac{1}{2Ke},\frac{T}{\beta}\right\}.
			\end{equation}
			Here, $\beta_\kappa$ depends on $\kappa$ and the previously defined constant $C_2$ with $2C_2 > r_1 \geq \kappa>0$ and $\tau_{1,\beta,\kappa}=\frac{\kappa}{2C_1 \beta}>0$ is the upper bound of Lemma \ref{lem:mapG} necessary for the estimates \eqref{eq:Glemma}, \eqref{eq:GlambdaP}. The upper bound $\frac{1}{2Ke}$ is necessary for the exponential estimates \eqref{eq:Gestimate} and \eqref{eq:Hestimate} and ensures that $C_2$ can be chosen uniformly in $\tau$, and the upper bound $\frac{T}{\beta}$ guarantees that $t^\ast = \beta \tau <T$. 
			Then it follows from \eqref{eq:Glemma} and \eqref{eq:Gestimate} that
			\begin{equation} \label{eq:Gdistance}
				\norm{\overline{G}(\bar{y}_1)-y_1}_{2} \leq \norm{\overline{G}(\bar{y}_1)-G(y_1)}_{2} + \norm{G(y_1)-y_1}_{2} \leq \delta_{1,t^\ast}+\delta_{2,\beta} < \kappa,
			\end{equation}
			as
			\begin{equation*}
				\delta_{1,t^\ast}+\delta_{2,\beta} = C_1 \beta \tau + C_2 e^{-\beta} < C_1 \beta \tau_{1,\beta,\kappa} + C_2 e^{-\beta_\kappa} = \frac{\kappa}{2} + \frac{\kappa}{2} = \kappa.
			\end{equation*}
			In particular, for $\lambda(p)\in \mathcal{K}_1$ of \eqref{eq:GlambdaP} it follows
			that there exists a point $\overline{\lambda(p)}\in\mathbb{R}^m$ specifying the corresponding special solution $S(y(0,c_{\lambda(p)})) = \bar{y}\big(0,\overline{\lambda(p)}\big)$, such that 
			\begin{equation} \label{eq:Gbarlambdap}
				\overline{G}\big(\overline{\lambda(p)}\big) \in B_{\delta_{1,t^\ast}+\delta_{2,\beta}}(\lambda(p)) \subset K_{\kappa}(\lambda(p))\subset K_{r_1}(\lambda(p))\subset \mathcal{K}_{1},
			\end{equation}
			as by Lemma \ref{lem:initial_behavior} it holds $K_{r_1}(\lambda(p))\subset \mathcal{K}_{1}$. 
			\item  \textbf{Necessary conditions for approximation:} We prove our main Theorem~\ref{th:noapproximation} by contradiction. To that purpose we assume that there exists a neural DDE $\Phi\in \textup{NDDE}^k_{\tau,\text{N},K}(\mathcal{X},\mathbb{R}^q)$ satisfying the assumptions of the theorem, for which the corresponding globally defined neural DDE is denoted by $\overline{\Phi}\in \overline{\textup{NDDE}}^k_{\tau,\text{N},K}(\mathcal{X},\mathbb{R}^q)$ with $\overline{\Phi}(x) = \Phi(x)$ for all $x \in \mathcal{X}$, which approximates the given map $\Psi$ with accuracy $\eps>0$ on $\mathcal{K}_0$. In particular this implies the component map $\Psi_i$ is approximated by $\overline{\Phi}_i$ with accuracy $\eps$ on $\mathcal{K}_0$, i.e.,
			\begin{equation} \label{eq:proof_approximation}
				\abs{\overline{\Phi}_i(x) - \Psi_i(x)} \leq \norm{\overline{\Phi}(x) - \Psi(x)}_\infty < \eps \qquad \text{for all } x\in\mathcal{K}_0. 
			\end{equation}
			Inserting the non-degenerate local minimum $p \in \mathcal{X}$ of $\Psi_i$ in \eqref{eq:proof_approximation} and using the estimate \eqref{eq:Hestimate}, it follows
			\begin{align}
				&\abs{\overline{\Phi}_i(p) - \Psi_i(p)} = \abs{\tilde\lambda_i(H(\lambda(p)))-\Psi_i(p)} < \eps, \nonumber\\
				\Rightarrow \quad & H(\lambda(p)) \in \mathcal{Y}_3 \coloneqq \tilde{\lambda}^{-1}_i\left((\Psi_i(p)-\eps,\Psi_i(p)+\eps)\right) \subset \mathbb{R}^m, \label{eq:Y3estimate0}\\	
				\Rightarrow \quad & \overline{H}\big(\overline{\lambda(p)}\big) \in \mathcal{Y}_{3,\delta_{3,\tau}} \coloneqq \tilde{\lambda}^{-1}_i\left((\Psi_i(p)-\eps-\tilde w\delta_{3,\tau},\Psi_i(p)+\eps+\tilde w\delta_{3,\tau})\right) \subset \mathbb{R}^m, \label{eq:Y3estimate}
			\end{align}
			with the open sets $\mathcal{Y}_3$ and $\mathcal{Y}_{3,\delta_{3,\tau}}$ defined in \eqref{eq:Y3} and \eqref{eq:Y3delta}. For any point $s \in \partial \mathcal{K}_0$ it holds by Lemma \ref{lem:initial_behavior}(a) that $\Psi_i(s) = \Psi_i(p) + r_0^2$. As we assume that $m = n$ and $\text{rank}(W) = n$, Lemma~\ref{lem:initial_behavior}(b) additionally states that $\Psi_i(\lambda^{-1}(y_1)) = \Psi_i(p) + r_0^2$ for every $y_1 \in \partial \mathcal{K}_1$. It follows for $y_1 \in \partial \mathcal{K}_1$ with the estimates \eqref{eq:Hestimate} and \eqref{eq:proof_approximation} that 
			\begin{align}
				&\abs{\overline{\Phi}_i(\lambda^{-1}(y_1)) - \Psi_i(\lambda^{-1}(y_1))} = \abs{\tilde\lambda_i(H(y_1))-\Psi_i(p)-r_0^2} < \eps, \nonumber\\
				\Rightarrow \quad & H(y_1) \in \mathcal{Z}_3 \coloneqq \tilde{\lambda}^{-1}_i\left((\Psi_i(p)+r_0^2-\eps,\Psi_i(p)+r_0^2+\eps)\right) \subset \mathbb{R}^m, \label{eq:Z3estimate0}\\	
				\Rightarrow \quad & \overline{H}(\bar{y}_1) \in \mathcal{Z}_{3,\delta_{3,\tau}} \coloneqq \tilde{\lambda}^{-1}_i\left((\Psi_i(p)+r_0^2-\eps-\tilde w\delta_{3,\tau},\Psi_i(p)+r_0^2+\eps+\tilde w\delta_{3,\tau})\right) \subset \mathbb{R}^m, \label{eq:Z3estimate}
			\end{align}
			with the open sets $\mathcal{Z}_3$ and $\mathcal{Z}_{3,\delta_{3,\tau}}$ defined in \eqref{eq:Z3} and \eqref{eq:Z3delta}. The necessary conditions for approximation  \eqref{eq:Y3estimate0}-\eqref{eq:Z3estimate} are visualized in Figure \ref{fig:necessaryconditions}\subref{fig:necessary_a}. As by assumption of Theorem~\ref{th:noapproximation} it holds $0<2\eps<r_0^2$, it follows from Lemma \ref{lem:distance} that there exists $\tau_3>0$, such that for all $\tau\in [0,\tau_3]$ it holds	
			\begin{equation*}
				\mathcal{Y}_{3,\delta_{3,\tau}}\subset \mathcal{Y}_{3,\delta_{3,\tau_3}}, \qquad  \mathcal{Z}_{3,\delta_{3,\tau}}\subset\mathcal{Z}_{3,\delta_{3,\tau_3}}, \qquad \mathcal{Y}_{3,\delta_{3,\tau}} \cap	\mathcal{Z}_{3,\delta_{3,\tau}} = \varnothing
			\end{equation*} 
			and 
			\begin{equation*}
				\normm{\tilde{\lambda}_i(y_3)-\tilde{\lambda}_i(z_3)}_\infty \geq 	r_0^2-2\eps- 2\tilde w\delta_{3,\tau_3} \eqqcolon \delta^\ast > 0
			\end{equation*}
			for all $y_3 \in \mathcal{Y}_{3,\delta_{3,\tau}}$ and $z_3 \in \mathcal{Z}_{3,\delta_{3,\tau}}$. 
			
			To also characterize the solution map of the initial value problem \eqref{eq:proof_specialODE}, which generates all special solutions of the DDE \eqref{eq:proof_vectorfield}, over the time interval $[t^\ast,T]$, the solution map $\overline{I}:\mathbb{R}^m\rightarrow\mathbb{R}^m$ is defined in \eqref{eq:mapIbar}. By definition, it holds $\overline{H} = \overline{I} \circ \overline{G}$, where $\overline{G}$ and $\overline{H}$ are the solution maps of the ODE \eqref{eq:proof_specialODE} over the time intervals $[0,t^\ast]$ and $[0,T]$ as defined in \eqref{eq:mapGbar} and \eqref{eq:mapHbar}.
			
			By Lemma~\ref{lem:diffeomorphisms}, the map $\overline{I}$ is a homeomorphism of $\mathbb{R}^m$. Applying the inverse $\overline{I}^{-1}$ to the estimates~\eqref{eq:Y3estimate} and \eqref{eq:Z3estimate} results for $\tau\in[0,\tau_3]$ in the following necessary conditions:
			\begin{align}
				\overline{I}^{-1}\big(\overline{H}\big(\overline{\lambda(p)}\big)\big) = \overline{G}\big(\overline{\lambda(p)}\big)&\in \mathcal{Y}_2 \coloneqq \overline{I}^{-1}( \mathcal{Y}_{3,\delta_{3,\tau_3}}) \subset\mathcal{H}_{\Psi_i(p)+\eps+\tilde w\delta_{3,\tau_3}}^{<}, \label{eq:Y2estimate}\\
				\overline{I}^{-1}(\overline{H}(\bar{y}_1)) = \overline{G}(\bar{y}_1) &\in \mathcal{Z}_2 \coloneqq \overline{I}^{-1}( \mathcal{Z}_{3,\delta_{3,\tau_3}}) \subset\mathcal{H}_{\Psi_i(p)+r_0^2-\eps-\tilde w\delta_{3,\tau_3}}^{>},\label{eq:Z2estimate}
			\end{align}
			where $\bar{y}_1$ is the initial condition of the corresponding special solution of $y(0,c_{y_1})$ with $y_1\in\partial\mathcal{K}_1$. Here, $\mathcal{H}_z^{<}$ and $\mathcal{H}_z^{>}$ are by Lemma \ref{lem:hypersurface} the unbounded open components of $\mathbb{R}^m\setminus \mathcal{H}_z$, where
			\begin{equation*}
				\mathcal{H}_z\coloneqq \overline{I}^{-1}(\tilde{\lambda}_i^{-1}(z))
			\end{equation*} 
			is for every $z\in\mathbb{R}$ a closed level set. If $y\in \mathcal{H}_z^{<}$, then $\tilde{\lambda}_i(\overline{I}(y))<z$ and if $y\in \mathcal{H}_z^{>}$, then $\tilde{\lambda}_i(\overline{I}(y))>z$. The closure of the open sets $\mathcal{H}_z^{<}$ and $\mathcal{H}_z^{>}$ is denoted by $\mathcal{H}_z^{\leq} \coloneqq  \mathcal{H}_z^{<} \cup \mathcal{H}_z $ and $\mathcal{H}_z^{\geq}\coloneqq  \mathcal{H}_z^{>} \cup \mathcal{H}_z $ respectively. The necessary conditions \eqref{eq:Y2estimate} and \eqref{eq:Z2estimate} are visualized in Figure~\ref{fig:necessaryconditions}\subref{fig:necessary_b}.
			
			\begin{figure}
				\centering
				\begin{subfigure}[t]{0.45\textwidth}
					\centering
					\resizebox{\linewidth}{!}{%
					\begin{tikzpicture}[>={Stealth[length=2.0mm]},x=1cm,y=0.9cm,scale=0.95,
					            lbl/.style={font=\small,inner sep=1.5pt},
					            arr/.style={thick,->,shorten <=0.5pt}]
					  \colorlet{cSupFull}{ddeTeal!15}
					  \colorlet{cSupTight}{ddeTeal!40}
					  \colorlet{cSupLine}{ddeTeal}
					  \colorlet{cSupText}{ddeTeal!75!black}
					  \colorlet{cSubFull}{ddeIndigo!15}
					  \colorlet{cSubTight}{ddeIndigo!40}
					  \colorlet{cSubLine}{ddeIndigo}
					  \colorlet{cSubText}{ddeIndigo}
					  \useasboundingbox (0,0) rectangle (8,6.2);
					  \fill[cSupFull]  (0,3.95) rectangle (8,6.15);
					  \fill[cSupTight] (0,4.60) rectangle (8,5.50);
					  \foreach \y in {3.95,4.60,5.50,6.15}
					    \draw[cSupLine,very thick] (0,\y) -- (8,\y);
					  \fill[cSubFull]  (0,0.05) rectangle (8,2.25);
					  \fill[cSubTight] (0,0.70) rectangle (8,1.60);
					  \foreach \y in {0.05,0.70,1.60,2.25}
					    \draw[cSubLine,very thick] (0,\y) -- (8,\y);
					  \node[lbl,cSupText] (Q)   at (0.55,3.50) {$\mathcal{Z}_3$};
					  \node[lbl,cSupText] (Qe)  at (2.25,3.50) {$\mathcal{Z}_{3,\delta_{3,\tau_3}}$};
					  \node[lbl,cSubText] (P)   at (0.55,2.70) {$\mathcal{Y}_3$};
					  \node[lbl,cSubText] (Pe)  at (2.25,2.70) {$\mathcal{Y}_{3,\delta_{3,\tau_3}}$};
					  \draw[arr,cSupLine] (Q)  -- (0.55,5.05);
					  \draw[arr,cSupLine] (Qe) -- (2.25,4.28);
					  \draw[arr,cSubLine] (P)  -- (0.55,1.15);
					  \draw[arr,cSubLine] (Pe) -- (2.25,1.93);
					  \fill[cSupText] (5.30,5.05) circle (1.6pt)
					    node[lbl,anchor=west,xshift=2.5pt] {$H(y_1)$};
					  \fill[cSupText] (3.30,4.28) circle (1.6pt)
					    node[lbl,anchor=west,xshift=2.5pt] {$\overline{H}(\bar{y}_1)$};
					  \fill[cSubText] (4.55,1.15) circle (1.6pt)
					    node[lbl,anchor=west,xshift=2.5pt] {$H(\lambda(p))$};
					  \fill[cSubText] (3.05,0.38) circle (1.6pt)
					    node[lbl,anchor=west,xshift=2.5pt] {$\overline{H}\bigl(\overline{\lambda(p)}\bigr)$};
					\end{tikzpicture}%
					}
					\vspace{2mm}
					\caption{Necessary conditions \eqref{eq:Y3estimate0} and \eqref{eq:Z3estimate0} for the map $H$ and necessary conditions \eqref{eq:Y3estimate} and \eqref{eq:Z3estimate} for the map $\overline{H}$ at time $t = T$.}
					\label{fig:necessary_a}
				\end{subfigure}
				\begin{subfigure}[t]{0.05\textwidth}
					\textcolor{white}{.}
				\end{subfigure}
				\begin{subfigure}[t]{0.45\textwidth}
					\centering
					\resizebox{\linewidth}{!}{%
					\begin{tikzpicture}[>={Stealth[length=2.0mm]},x=1cm,y=0.9cm,scale=0.95,
					            lbl/.style={font=\small,inner sep=1.5pt},
					            declare function={f(\x)=0.8*(1+tanh(1.1*(\x-4.3)))+0.08*sin(80*\x);},
					            curve/.style={smooth,samples=60,domain=-0.2:8.2}]
					  \colorlet{cSupPale}{ddeTeal!8}
					  \colorlet{cSupFull}{ddeTeal!15}
					  \colorlet{cSupLine}{ddeTeal}
					  \colorlet{cSupText}{ddeTeal!75!black}
					  \colorlet{cSubPale}{ddeIndigo!8}
					  \colorlet{cSubFull}{ddeIndigo!15}
					  \colorlet{cSubLine}{ddeIndigo}
					  \colorlet{cSubText}{ddeIndigo}
					  \useasboundingbox (0,0) rectangle (8,6.2);
					  \clip (0,0) rectangle (8,6.2);
					  \def\bLo{-0.25}\def\bHi{1.00}\def\oLo{3.10}\def\oHi{4.35}
					  \fill[cSubPale] plot[curve] (\x,{\bHi+f(\x)}) -- (8.2,-1) -- (-0.2,-1) -- cycle;
					  \fill[cSupPale] plot[curve] (\x,{\oLo+f(\x)}) -- (8.2,7) -- (-0.2,7) -- cycle;
					  \fill[cSubFull] plot[curve] (\x,{\bLo+f(\x)})
					    -- plot[curve,domain=8.2:-0.2] (\x,{\bHi+f(\x)}) -- cycle;
					  \fill[cSupFull] plot[curve] (\x,{\oLo+f(\x)})
					    -- plot[curve,domain=8.2:-0.2] (\x,{\oHi+f(\x)}) -- cycle;
					  \foreach \c in {\bLo,\bHi}
					    \draw[cSubLine,very thick] plot[curve] (\x,{\c+f(\x)});
					  \foreach \c in {\oLo,\oHi}
					    \draw[cSupLine,very thick] plot[curve] (\x,{\c+f(\x)});
					  \node[lbl,cSupText] at (1.15,3.70) {$\mathcal{Z}_2$};
					  \node[lbl,cSubText] at (1.15,0.38) {$\mathcal{Y}_2$};
					  \node[lbl,cSupText] at (1.00,5.30) {$\mathcal{H}^{>}$};
					  \node[lbl,cSubText] at (7.00,0.62) {$\mathcal{H}^{<}$};
					  \node[lbl,cSupText,anchor=west] at (0.15,2.45)
					    {$\mathcal{H}_{\Psi_i(p)+r_0^2-\eps-\tilde w\delta_{3,\tau_3}}$};
					  \node[lbl,cSubText,anchor=east] at (7.85,3.20)
					    {$\mathcal{H}_{\Psi_i(p)+\eps+\tilde w\delta_{3,\tau_3}}$};
					  \fill[cSupText] (5.40,{(\oLo+\oHi)/2+f(5.4)}) circle (1.6pt)
					    node[lbl,anchor=west,xshift=2.5pt] {$\overline{G}(\bar{y}_1)$};
					  \fill[cSubText] (5.00,{(\bLo+\bHi)/2+f(5.0)}) circle (1.6pt)
					    node[lbl,anchor=west,xshift=2.5pt] {$\overline{G}\bigl(\overline{\lambda(p)}\bigr)$};
					\end{tikzpicture}%
					}
					\vspace{2mm}
					\caption{Necessary conditions  \eqref{eq:Y2estimate} and \eqref{eq:Z2estimate} for the map $\overline{G}$ at time $t = t^\ast$, obtained after applying the map $\overline{I}^{-1}$ to the conditions \eqref{eq:Y3estimate} and \eqref{eq:Z3estimate}.}
					\label{fig:necessary_b}
				\end{subfigure}
				\caption{Necessary conditions for approximation implied by the estimate \eqref{eq:proof_approximation} for the local minimum $p \in \mathcal{X}$ and all points of the boundary $s \in \partial \mathcal{K}_0$ or  $y_1\in\partial\mathcal{K}_1$ respectively.}
				\label{fig:necessaryconditions}
			\end{figure}	
			
			By Theorem \ref{th:specialsolutionsODEs}, the initial value problem \eqref{eq:proof_specialODE} has a globally defined and in the second variable Lipschitz continuous vector field with Lipschitz constant at most $Ke$, which also holds for $\tau = 0$ by the convention $\widetilde{F} = \overline{F}$. By Grönwall's inequality, as in the proof of Lemma~\ref{lem:diffeomorphisms}, it holds
			\begin{align}
				\norm{\overline{I}(y_2)-\overline{I}(z_2)}_\infty = & \norm{\bar{y}(t^\ast,y_2)(T)-\bar{y}(t^\ast,z_2)(T)}_\infty \nonumber\\
				\leq &\norm{y_2-z_2}_\infty e^{Ke(T-t^\ast)}\leq \norm{y_2-z_2}_\infty e^{KeT} \label{eq:IbarLipschitz}
			\end{align}		
			for all $y_2,z_2\in\mathbb{R}^m$.

			\item \textbf{Deriving a contradiction:} Let the approximation assumption \eqref{eq:proof_approximation} hold and assume
			\begin{equation} \label{eq:ass_beta_tau_2}
				0<\kappa \leq r_1, \quad \beta=\beta_\kappa \coloneqq \ln\left(\frac{2C_2}{\kappa}\right)> 0, \quad 0\leq \tau< \tau_0 \coloneqq \min\left\{\tau_{1,\beta,\kappa},\frac{1}{2Ke},\frac{T}{\beta},\tau_3,T\right\},
			\end{equation}
			which are the assumptions \eqref{eq:ass_beta_tau} combined with $\tau\in[0,\tau_3]$ to guarantee that the sets $\mathcal{Y}_{3,\delta_{3,\tau_3}}$ and $\mathcal{Z}_{3,\delta_{3,\tau_3}}$ do not intersect. By Lemma \ref{lem:distance}, it holds under these assumptions that
			\begin{equation}\label{eq:distance}
				\normm{\tilde{\lambda}_i(\overline{I}(y_2))-\tilde{\lambda}_i(\overline{I}(z_2))}_\infty \geq
				\delta^\ast > 0
			\end{equation}
			for all $y_2\in	\mathcal{H}^{\leq}_{\Psi_i(p)+\eps+\tilde w\delta_{3,\tau_3}}$  and $z_2\in\mathcal{H}^{\geq}_{\Psi_i(p)+r_0^2-\eps-\tilde w\delta_{3,\tau_3}}$. 
			
			The compact set $\mathcal{K}_1 = \lambda(\mathcal{K}_0)\subset\mathbb{R}^m$, defined at time $t = 0$, is used as a subset of the phase space at time $t = t^\ast$, on which the level sets $\mathcal{H}_z$ are defined. This is justified by the estimate \eqref{eq:Gdistance}, as every point $y_1\in\mathcal{K}_1$ has distance less than $\kappa$ from the point $\overline{G}(\bar{y}_1)$ of the corresponding special solution at time $t^\ast$. By Lemma \ref{lem:intersectionpoint} for $z = \Psi_i(p)+\eps+\tilde w\delta_{3,\tau_3}$, there exists $y_1^\ast \in \partial \mathcal{K}_1$ with  $y_1^\ast \in \mathcal{H}^{\leq}_{\Psi_i(p)+\eps+\tilde w\delta_{3,\tau_3}}$. The assumption of Lemma \ref{lem:intersectionpoint} that there exists $y_1\in \mathcal{K}_1$ with $y_1\in\mathcal{H}^{<}_{\Psi_i(p)+\eps+\tilde w\delta_{3,\tau_3}}$ is fulfilled as by~\eqref{eq:Gbarlambdap} it holds $\overline{G}(\overline{\lambda(p)}) \in\mathcal{K}_1$ and by \eqref{eq:Y2estimate} it holds $\overline{G}(\overline{\lambda(p)}) \in\mathcal{H}^{<}_{\Psi_i(p)+\eps+\tilde w\delta_{3,\tau_3}}$. The points $y_1^\ast \in \partial \mathcal{K}_1$ and $\overline{G}(\overline{\lambda(p)}) \in\mathcal{K}_1$ are visualized in Figure~\ref{fig:proof}.
			
			Now fix $\kappa = \min\left\{\frac{\delta^\ast}{\tilde w e^{KeT}},r_1\right\}>0$, where $\tilde w$ is the fixed constant with $\normm{\widetilde W}_\infty \leq \tilde w$. By step~4 and step~5, the constants $C_1$ and $C_2$ depend continuously on $K$ and otherwise only on the fixed constants $A$, $T$, $w$, $b_\infty$, $m$, $r_1$, and $\mathcal{K}_0$. By Lemma~\ref{lem:distance}, the constants $\delta^\ast$ and $\tau_3$ depend continuously on $C_2$ and otherwise only on $r_0$, $\eps$, $\tilde w$, and $T$. Hence, $\kappa$, $\beta_\kappa$, $\tau_{1,\beta,\kappa}$, and therefore~$\tau_0$ depend continuously on $K$ and are independent of the specific neural DDE in the class. As $\tau_0(K)\leq\frac{1}{2Ke}$, it holds $K\tau_0(K)e\leq\frac{1}{2}<1$, so $\tau_0\in C^0((0,\infty),(0,T])$ has the properties claimed in Theorem~\ref{th:noapproximation}.
			
			In the following, we use the Lipschitz continuity of the map $\tilde \lambda_i\circ\overline{I}:\mathbb{R}^m\rightarrow \mathbb{R}$ at the considered point $y_1^\ast\in\partial\mathcal{K}_1$ with $y_1^\ast \in \mathcal{H}^{\leq}_{\Psi_i(p)+\eps+\tilde w\delta_{3,\tau_3}}$. We can estimate for all  $y_1\in\mathbb{R}^m$ with $\norm{y_1^\ast-y_1}_\infty < \kappa$ that 
			\begin{align}
				\normm{\tilde{\lambda}_i(\overline{I}(y_1^\ast))-\tilde{\lambda}_i(\overline{I}(y_1))}_\infty \leq \; &\tilde w  \norm{\overline{I}(y_1^\ast)-\overline{I}(y_1)}_\infty \nonumber \\
				\leq \; &\tilde w  e^{KeT} \norm{y_1^\ast-y_1}_\infty < \tilde w e^{KeT} \kappa \leq \delta^\ast, \label{eq:estimatecontinuity}
			\end{align}
			where we first used the affine linear structure of the map $\tilde \lambda$ and then the estimate \eqref{eq:IbarLipschitz} and the definition of $\kappa$.
			
			Now we consider the points $y_1^\ast$ and $\overline G(\bar y_1^\ast)$, for which it holds by the estimate \eqref{eq:Gdistance} that $\norm{y_1^\ast-\overline G(\bar y_1^\ast)}_2<\kappa$ and hence
			\begin{equation*}
				\norm{y_1^\ast-\overline G(\bar y_1^\ast)}_\infty  < \kappa,
			\end{equation*}
			hence $\overline G(\bar y_1^\ast) \in B_\kappa(y_1^\ast)$, as visualized in Figure \ref{fig:proof}. From this estimate, the inequality~\eqref{eq:estimatecontinuity}  implies that
			\begin{equation}\label{eq:contradiction}
				\normm{\tilde{\lambda}_i(\overline{I}(y_1^\ast))-\tilde{\lambda}_i(\overline{I}(\overline{G}(\bar y_1^\ast)))}_\infty < \delta^\ast.
			\end{equation}
			On the other hand,  it is by inclusion \eqref{eq:Z2estimate} a necessary condition for approximation that $\overline G(\bar y_1^\ast)\in \mathcal{H}^{\geq}_{\Psi_i(p)+r_0^2-\eps-\tilde w\delta_{3,\tau_3}}$ as $y_1^\ast\in\partial\mathcal{K}_1$, which is also visualized  in Figure~\ref{fig:proof}. As a consequence, for $y_2 = y_1^\ast$ and $z_2 =\overline  G(\bar y_1^\ast)$ the estimate \eqref{eq:distance} implies that 
			\begin{equation*}
				\normm{\tilde{\lambda}_i(\overline{I}(y_1^\ast))-\tilde{\lambda}_i(\overline{I}(\overline G(\bar y_1^\ast)))}_\infty \geq \delta^\ast,
			\end{equation*}
			which contradicts the continuity estimate \eqref{eq:contradiction}. The contradiction is visualized in Figure~\ref{fig:proof} as follows: for an approximation it is necessary that $\overline  G(\bar y_1^\ast)\in\mathcal{Z}_2$, but by continuity it holds  $\overline G(\bar y_1^\ast) \in B_\kappa(y_1^\ast)$. For every $\tau\in[0,\tau_0(K))$, the constant $\kappa$ is sufficiently small such that $\mathcal{Z}_2\cap B_\kappa(y_1^\ast) = \varnothing$ and under the given assumptions no approximation of $\Psi$ by a neural DDE $\overline{\Phi}$ is possible with accuracy $\eps$.
			
			\begin{figure}
				\centering
				\resizebox{0.6\textwidth}{!}{%
				\begin{tikzpicture}[>={Stealth[length=2.0mm]},x=1cm,y=0.90cm]
				          \colorlet{cSupPale}{ddeTeal!8}
				          \colorlet{cSupFull}{ddeTeal!15}
				          \colorlet{cSupLine}{ddeTeal}
				          \colorlet{cSupText}{ddeTeal!75!black}
				          \colorlet{cSubPale}{ddeIndigo!8}
				          \colorlet{cSubFull}{ddeIndigo!15}
				          \colorlet{cSubLine}{ddeIndigo}
				          \colorlet{cSubText}{ddeIndigo}
				          \colorlet{cAuxSet}{ddePurple}
				          \colorlet{cAuxBall}{ddeGrayDark}
				          \clip (0,0) rectangle (8,6.2);
				          \fill[cSubPale]
				            plot[smooth,tension=0.75] coordinates
				              {(-0.6,0.95) (0.6,0.8) (1.8,0.85) (2.8,1.3) (3.8,2.2) (4.8,2.75) (6.0,2.85) (7.2,2.6) (8.6,2.8)}
				            -- (8.6,-0.8) -- (-0.6,-0.8) -- cycle;
				          \fill[cSupPale]
				            plot[smooth,tension=0.75] coordinates
				              {(-0.6,3.0) (0.6,2.85) (1.8,2.9) (2.8,3.35) (3.8,4.25) (4.8,4.8) (6.0,4.9) (7.2,4.65) (8.6,4.85)}
				            -- (8.6,7.0) -- (-0.6,7.0) -- cycle;
				          \fill[cSubFull]
				            plot[smooth,tension=0.75] coordinates
				              {(-0.6,-0.3) (0.6,-0.45) (1.8,-0.4) (2.8,0.05) (3.8,0.95) (4.8,1.5) (6.0,1.6) (7.2,1.35) (8.6,1.55)}
				            -- (8.6,2.8)
				            -- plot[smooth,tension=0.75] coordinates
				              {(8.6,2.8) (7.2,2.6) (6.0,2.85) (4.8,2.75) (3.8,2.2) (2.8,1.3) (1.8,0.85) (0.6,0.8) (-0.6,0.95)}
				            -- cycle;
				          \draw[cSubLine,very thick] plot[smooth,tension=0.75] coordinates
				              {(-0.6,-0.3) (0.6,-0.45) (1.8,-0.4) (2.8,0.05) (3.8,0.95) (4.8,1.5) (6.0,1.6) (7.2,1.35) (8.6,1.55)};
				          \draw[cSubLine,very thick] plot[smooth,tension=0.75] coordinates
				              {(-0.6,0.95) (0.6,0.8) (1.8,0.85) (2.8,1.3) (3.8,2.2) (4.8,2.75) (6.0,2.85) (7.2,2.6) (8.6,2.8)};
				          \fill[cSupFull]
				            plot[smooth,tension=0.75] coordinates
				              {(-0.6,3.0) (0.6,2.85) (1.8,2.9) (2.8,3.35) (3.8,4.25) (4.8,4.8) (6.0,4.9) (7.2,4.65) (8.6,4.85)}
				            -- (8.6,6.1)
				            -- plot[smooth,tension=0.75] coordinates
				              {(8.6,6.1) (7.2,5.9) (6.0,6.15) (4.8,6.05) (3.8,5.5) (2.8,4.6) (1.8,4.15) (0.6,4.1) (-0.6,4.25)}
				            -- cycle;
				          \draw[cSupLine,very thick] plot[smooth,tension=0.75] coordinates
				              {(-0.6,3.0) (0.6,2.85) (1.8,2.9) (2.8,3.35) (3.8,4.25) (4.8,4.8) (6.0,4.9) (7.2,4.65) (8.6,4.85)};
				          \draw[cSupLine,very thick] plot[smooth,tension=0.75] coordinates
				              {(-0.6,4.25) (0.6,4.1) (1.8,4.15) (2.8,4.6) (3.8,5.5) (4.8,6.05) (6.0,6.15) (7.2,5.9) (8.6,6.1)};
				          \draw[cAuxSet,very thick] plot[smooth cycle,tension=0.7] coordinates
				              {(4.30,3.90) (5.10,3.40) (5.25,2.65) (4.65,2.00) (3.80,1.92) (3.35,2.55) (3.50,3.40)};
				          \draw[cAuxBall,thick] (5.25,2.65) circle (0.75cm);
				          \node[cSupText,font=\small] at (1.15,3.50) {$\mathcal{Z}_2$};
				          \node[cSubText,font=\small] at (3.05,0.85) {$\mathcal{Y}_2$};
				          \node[cSupText,font=\small] at (1.00,5.55) {$\mathcal{H}^{>}$};
				          \node[cSubText,font=\small] at (6.95,0.55) {$\mathcal{H}^{<}$};
				          \node[cAuxSet,font=\small] at (3.95,2.95) {$\mathcal{K}_1$};
				          \node[cAuxBall,anchor=west,font=\small] at (5.55,3.75) {$B_{\kappa}(y_1^\ast)$};
				          \node[cSupText,anchor=east,font=\small] at (7.80,4.25)
				            {$\mathcal{H}_{\Psi_i(p)+r_0^2-\eps-\tilde w\delta_{3,\tau_3}}$};
				          \node[cSubText,anchor=west,font=\small] at (0.10,2.0)
				            {$\mathcal{H}_{\Psi_i(p)+\eps+\tilde w\delta_{3,\tau_3}}$};
				          \fill[cSupText] (5.40,5.45) circle (1.6pt);
				          \node[cSupText,anchor=west,font=\small] at (5.55,5.45)
				            {$\overline{G}(\bar{y}_1^\ast)$};
				          \fill[cSupText] (5.50,3.10) circle (1.6pt);
				          \node[cSupText,anchor=west,font=\small] at (6.05,3.10)
				            {$\overline{G}(\bar{y}_1^\ast)$};
				          \fill[cAuxSet] (5.25,2.65) circle (1.6pt);
				          \node[cAuxSet,anchor=north west,font=\small] at (5.2,2.6) {$y_1^\ast$};
				          \fill[cSubText] (0.75,0.30) circle (1.6pt);
				          \node[cSubText,anchor=west,font=\small] at (0.90,0.30)
				            {$\overline{G}\left(\overline{\lambda(p)}\right)$};
				        \end{tikzpicture}%
				}
				\vspace{2mm}
				\caption{Visualization of the contradiction derived in step 7 of the proof of Theorem \ref{th:noapproximation} at time $t = t^\ast$. On the one hand, it is by \eqref{eq:Z2estimate} a necessary condition for approximation that $\overline{G}(\bar{y}_1^\ast)\in\mathcal{Z}_2$. On the other hand, for $\tau\in[0,\tau_0(K))$, the constant $\kappa$ is sufficiently small, such that the open ball around the considered point $y_1^\ast \in \partial \mathcal{K}_1$ with radius $\kappa$ does not intersect with $\mathcal{Z}_2$, i.e., $\mathcal{Z}_2\cap B_\kappa(y_1^\ast) = \varnothing$. Accordingly, $\overline{G}(\bar{y}_1^\ast)$ is drawn twice: once at the position forced by \eqref{eq:Z2estimate}, and once at the position it actually takes. As this is a contradiction to  $\overline{G}(\bar{y}_1^\ast)\in\mathcal{Z}_2$, no approximation of the map $\Psi$ by any considered neural DDE is possible.}
				\label{fig:proof}
			\end{figure}
			
			Hence, for $\tau \in [0,\tau_0(K))$, every neural DDE $\overline{\Phi} \in \overline{\textup{NDDE}}_{\tau,\textup{N},K}^k(\mathcal{X},\mathbb{R}^q)$, $k \geq 0$, with underlying vector field $\overline{F}:\mathbb{R}\times\mathcal{C}\rightarrow\mathbb{R}^m$ with global Lipschitz constant~$K$ in its second variable and $\norm{\overline{F}(t,0)}_\infty\leq A$ for all $t\in\mathbb{R}$ cannot approximate the map $\Psi$ with accuracy~$\eps$ on $\mathcal{K}_0$. With the argumentation of step 3, the same result follows for every neural DDE $\Phi \in \textup{NDDE}_{\tau,\textup{N},K}^k(\mathcal{X},\mathbb{R}^q)$, which satisfies the assumptions of the theorem.
			\qedhere
		\end{enumerate}
	\end{proof}

	\section{Conclusion and Outlook}
	
	In this work, we studied the approximation capability of neural DDEs depending on their memory capacity, i.e., depending on the product $K\tau $ of the Lipschitz constant and the delay of the vector field. Neural DDEs can be interpreted as an infinite-depth limit of densely connected residual neural networks (DenseResNets), which allow for feed-forward shortcut connections between all layers. DenseResNets unify the general idea of existing neural network architectures with memory terms and inter-layer connections. Studying the continuous-time neural network allows the usage of well-developed theory of delay differential equations. By an Euler discretization of the considered neural DDEs, results for discrete-time neural networks such as DenseResNets can be established. 
	
	Of significant interest are non-augmented architectures where the dimension of the DDE is not larger than the input dimension. Non-augmented MLPs, ResNets, and neural ODEs show a restricted approximation capability \cite{Hanin2018,Johnson2018,Kuehn2023,Park2020}. As neural ODEs are neural DDEs with delay $\tau = 0$, it follows that if $\tau = 0$, non-augmented neural DDEs cannot have the universal approximation property. In Theorem~\ref{th:tau0}, we have shown that if the memory capacity $K\tau$ is sufficiently small, non-augmented neural DDEs with bounded weights and vector field still cannot approximate all maps with a prescribed Lipschitz constant, however small. If $m<n$ or $m<q$, Theorem~\ref{th:bottleneck} and Theorem~\ref{th:outputbottleneck} exclude universal approximation for every memory capacity. The infinite-dimensional phase space of DDEs with $ \tau>0$ is not sufficient to approximate every element of the function space $C^j(\mathcal{X},\R^q)$, $\mathcal{X}\subset \R^n$ open, for any $j\geq 0$. The proof uses a scalar function~$\Psi_i$ with a non-degenerate local extreme point, i.e., locally, the function~$\Psi_i$ has a minimum or maximum. The topological constraints of unique solution curves in the finite-dimensional state space of ODEs prevent scalar neural ODEs from approximating local extreme points. For $K\tau$ sufficiently small, this property carries over to neural DDEs. It is important to work with non-degenerate local extreme points, as it is possible that scalar neural ODEs and neural DDEs approximate saddle points.
	
	As numerically observed for many feed-forward neural network architectures with additional inter-layer connections, the approximation capability increases with the memory capacity \cite{Huang2017,Sander2021,Srivastava2015,Wang2018}. If the memory capacity $K\tau$ is sufficiently large, we have shown in Theorem \ref{th:universal_embedding} that neural DDEs with $m\geq\max\{n,q\}$ have the universal embedding property with respect to the space of Lipschitz continuous functions on bounded open sets. We were able to show an exact representation of any Lipschitz continuous function, as the vector field of the DDE is a parameter of the network itself. If the vector field has an explicit and rich enough parameterization, parameterized neural DDEs with sufficient memory capacity have the universal approximation property. The parameter region where universal embedding is possible can be increased for augmented architectures with $m \geq n+q$, as shown in Theorem \ref{th:universalaugmented}. The proof works similarly to neural ODEs and does not use any delayed terms. Overall, we specify for large parts of the three-dimensional parameter space $(K,\tau,m)$ of Lipschitz constant, delay, and dimension of the vector field, where universal embedding and no universal approximation are possible. Nevertheless, the behavior near critical parameter values remains for future work. In that way, we provide, in the context of continuous-time models, a fundamental understanding of how the memory capacity influences the universal approximation property. 
	
	It is left for future work to transfer the established results to discrete-time architectures. In Section \ref{sec:modeling_discretization}, we stated the necessary assumptions for the discretization of a neural DDE to become a DenseResNet. For specific discrete-time architectures, the concrete parameterized form of the neural DDE needs to be specified. Furthermore, the discretization error between neural DDEs and DenseResNets needs to be studied, as done for neural ODEs and ResNets in \cite{Sander2022}. 
	
	Further interesting research areas are the implications that a large or a small memory capacity has on the training process or the generalization properties of neural DDEs and DenseResNets. In the context of learning algorithms, additional memory in a neural network can mitigate the vanishing gradient problem and influence the convergence behavior. Regarding generalization properties, it is important to analyze whether the increasing expressivity due to a larger memory capacity promotes overfitting.
	
	\vspace{12mm}

	\noindent \textbf{Acknowledgements:} This work was supported by the Deutsche Forschungsgemeinschaft (DFG) via the SPP 2298 `Theoretical Foundations of Deep Learning' and via the Research Training Group GRK 3081 `Machine Learning and Control Theory: Exploring Synergies, Complementarities, and Mutual Benefits' (METEOR). CK would like to thank the VolkswagenStiftung for support via a Lichtenberg Professorship. SVK would like to thank the Munich Data Science Institute (MDSI) for partial support via a Linde doctoral fellowship. \\[2mm]

	\vspace{10mm}

	\newpage
	\nomenclature[A,02]{ResNet}{Residual Neural Network}
	\nomenclature[A,04]{ODE}{Ordinary Differential Equation}
	\nomenclature[A,06]{DDE}{Delay Differential Equation}
	\nomenclature[A,05]{NODE}{Neural Ordinary Differential Equation}
	\nomenclature[A,04]{IVP}{Initial Value Problem}
	\nomenclature[A,01]{FNN}{Feed-Forward Neural Network}
	\nomenclature[A,07]{NDDE}{Neural Delay Differential Equation}
	\nomenclature[A,03]{DenseResNet}{Densely Connected Residual Neural Network}
	\nomenclature[A,08]{UA}{Universal Approximation}
	\nomenclature[A,09]{UE}{Universal Embedding}
	
	\nomenclature[B, 01]{\(\mathbb{R}\)}{real numbers}
	\nomenclature[B, 02]{\(\mathbb{N}\)}{natural numbers $\mathbb{N} \coloneqq \{1,2,3,\ldots \}$}
	\nomenclature[B,04]{$\partial \mathcal{U}$}{boundary of the set $\mathcal{U}$}
	\nomenclature[B,05]{$\textup{int}( \mathcal{U})$}{interior of the set $\mathcal{U}$, $\textup{int}( \mathcal{U}) \coloneqq \mathcal{U} \setminus \partial \mathcal{U}$}
	\nomenclature[B,06]{$\overline{ \mathcal{U}}$}{closure of the set $\mathcal{U}$, $\overline{\mathcal{U}} \coloneqq \textup{int}(\mathcal{U}) \cup \partial \mathcal{U}$}
	\nomenclature[B,07]{$B_r(x) \subset \R^n$}{open Euclidean ball with radius $r>0$ around the point $x \in \R^n$}
	\nomenclature[B,08]{$K_r(x) \subset \R^n$}{closed Euclidean ball with radius $r>0$ around the point $x \in \R^n$}
	\nomenclature[B,09]{$[v]_{i,\ldots,j}$}{components $\{i,\ldots,j\}\subset \{1,\ldots,n\}$ of the vector $v\in \mathbb{R}^n$}
	\nomenclature[B,10]{$A_{i,j}$}{$i,j$-th entry of a matrix $A\in \mathbb{R}^{m \times n}$}
	\nomenclature[B,11]{$\text{Id}_n$}{identity matrix $\text{Id}_n \in \mathbb{R}^{n \times n}$}
	\nomenclature[B,12]{$\text{id}_\mathcal{X}:\mathcal{X}\rightarrow \mathcal{X}$}{identity function $\text{id}(x) = x$ for $x\in\mathcal{X}$}
	\nomenclature[B,13]{$\nabla f(x)$}{gradient of a scalar function $f: \mathbb{R}^n \rightarrow \mathbb{R}$ at $x \in \mathbb{R}^n$ }
	\nomenclature[B,14]{$J_f(x)$}{Jacobian matrix of function $f: \mathbb{R}^n \rightarrow \mathbb{R}^m$ at $x \in \mathbb{R}^n$}
	\nomenclature[B,15]{$H_f(x)$}{Hessian matrix of function $f: \mathbb{R}^n \rightarrow \mathbb{R}$ at $x \in \mathbb{R}^n$}

	\nomenclature[C,01]{$C^k(\mathcal{X},\mathcal{Y})$}{space of  $k$-times continuously differentiable functions $f: \mathcal{X} \rightarrow \mathcal{Y}$}
	\nomenclature[C,02]{$C^{k,l}(\mathcal{X}_1\times \mathcal{X}_2,\mathcal{Y}$)}{space of $k$-times in $\mathcal{X}_1$ and $l$-times in $\mathcal{X}_2$ continuously differentiable functions $f: \mathcal{X}_1 \times \mathcal{X}_2 \rightarrow \mathcal{Y}$}
	\nomenclature[C,03]{$C^{0,k}_b(\mathbb{R}\times \mathcal{C},\mathbb{R}^m)$}{space of bounded, jointly continuous functions with bounded and jointly continuous Fréchet derivatives up to order $k$ with respect to the second variable}
	\nomenclature[C,05]{$\norm{x}_2$}{Euclidean norm $\norm{x}_2 \coloneqq \left(\sum_{i = 1}^n x_i^2\right)^{1/2}$ of $x\in \mathbb{R}^n$}
	\nomenclature[C,06]{$\norm{x}_\infty$}{maximum norm (max-norm) $\norm{x}_\infty \coloneqq  \max\{\abs{x_1},\ldots,\abs{x_n}\}$ of a vector $x\in \mathbb{R}^n$}
	\nomenclature[C,08]{$\norm{f}_\infty$}{sup-norm $\norm{f}_\infty \coloneqq \sup_{x \in \mathcal{X}} \norm{ f(x)}_\infty$ of a bounded function $f: \mathcal{X} \rightarrow \mathbb{R}^q$}
	\nomenclature[C,09]{$\norm{f}_k$}{$k$-norm, which takes the supremum over all derivatives up to order $k$ in the second variable, $\norm{f}_k \coloneqq \sup_{l \in \{0,\ldots,k\}, x \in \mathbb{R}\times \mathcal{C}} \norm{ f^{(0,l)}(x)}_\infty$ of a function $f \in C^{0,k}(\mathbb{R}\times \mathcal{C},\mathbb{R}^m)$}
	\nomenclature[C,10]{$\norm{f}_{\infty,\mathcal{U}}$}{sup-norm $\norm{f}_{\infty,\mathcal{U}} \coloneqq \sup_{x\in\mathcal{U}}\norm{f(x)}_\infty$ of a function $f$ on a set $\mathcal{U}$}
	\nomenclature[C,04]{$C^k_b(\R^n,\R^q)$}{space of bounded, $k$-times continuously differentiable functions with bounded derivatives up to order $k$}
	\nomenclature[C,11]{$\sigma_{\min}(A)$}{smallest singular value of a matrix $A\in\R^{m\times n}$, $m\leq n$}
	\nomenclature[F, 15]{$\textup{sat}_r:\R^n\rightarrow\R^n$}{saturation function with radius $r>0$, cf.\ Definition~\ref{def:saturation}}
	\nomenclature[C,07]{$\norm{A}_\infty$}{sup-norm of a matrix $A \in \mathbb{R}^{r \times s}$ defined by $\norm{A}_\infty = \max_{\norm{x}_\infty = 1} \norm{Ax}_\infty$, $x \in \mathbb{R}^s$}

	\nomenclature[D, 01]{\(h_l\in\R^{n_l}\)}{layers of a feed-forward neural network with $n_l$ neurons for \(l \in \{0,1,\ldots,L\}\), \\ hidden layers for \(l \in \{1,\ldots,L-1\}\)}
	\nomenclature[D, 02]{\(h_0 \)}{input layer}
	\nomenclature[D, 03]{\(h_L\)}{output layer}
	\nomenclature[D, 04]{\(\theta_l\in\R^{p_l}\)}{parameters for \(l \in \{0,1,\ldots,L-1\}\), \(\theta = (\theta_0,\ldots,\theta_{L-1})\)}
	\nomenclature[D, 05]{\(L\)}{depth of the network, number of layers $L+1$}
	\nomenclature[D, 06]{\(n_\text{max}\)}{width of the network, maximum number of neurons per layer, \(n_\text{max} = \max\{n_0,\ldots,n_L\}\)}
	\nomenclature[D, 07]{\(\Phi_\theta:\R^{n_0}\rightarrow \R^{n_L}\)}{input-output map of the FNN with weights $\theta = (\theta_0,\ldots,\theta_{L-1})$}
	\nomenclature[D, 08]{\(\Psi\)}{ground truth function to approximate}

	\nomenclature[E, 01]{\(\mathcal{C}\)}{state space $\mathcal{C} = C^0([-\tau,0],\R^m)$ of a DDE with delay $\tau \geq 0$}
	\nomenclature[E, 02]{\(y_t \in \mathcal{C}\)}{delayed function defined by $y_t(s) = y(t+s)$ for $s\in[-\tau,0]$}
	\nomenclature[E, 03]{\(c_{y_0} \in \mathcal{C}\)}{constant initial function $c_{y_0}(t) = y_0 \in \R^m$}
	\nomenclature[E, 04]{\(y(t_0,u)\)}{solution of the DDE with initial data $y(t_0,u)_{t_0} = u \in \mathcal{C}$}
	\nomenclature[E, 05]{\(\bar{y}(t_0,y_0)\)}{special solution of the DDE passing through $(t_0,y_0)\in\R \times \R^m$}
	\nomenclature[E, 06]{\(S(y(t_0,u))\)}{special solution corresponding to the solution $y(t_0,u)$, cf.\ Theorem \ref{th:specialsolutions_expattraction}}
	\nomenclature[E, 07]{\(\mathcal{I}_{y_0}\subset \R\)}{maximal time interval of existence of a DDE with constant initial data }
	\nomenclature[E, 08]{$K$}{Lipschitz constant w.r.t.\ the second variable of the vector field $F\in C^0(\R\times \mathcal{C},\R^m)$}
	\nomenclature[E, 09]{$A$}{upper bound $\norm{F(t,0)}_\infty\leq A$ of weakly nonlinear vector field $F\in C^0(\R\times \mathcal{C},\R^m)$}
	
	\nomenclature[F, 02]{\(\lambda:\R^{n}\rightarrow\R^m\)}{affine linear layer before the DDE, $\lambda(x) = Wx+b$}
	\nomenclature[F, 03]{\(\tilde \lambda:\R^m\rightarrow\R^q\)}{affine linear layer after the DDE, $\tilde \lambda(y) = \widetilde Wy+ \tilde b$}
	\nomenclature[F, 04]{\(\mathbb{V}\)}{parameter space for the weights and biases $W,\widetilde{W},b,\tilde{b}$, cf.\ Definition \ref{def:DDEweightspace}}
	\nomenclature[F, 06]{\( \textup{NDDE}^k_{\tau,i}(\mathcal{X},\R^q)\)}{Neural DDE of Definition \ref{def:NDDE},  $\mathcal{X}\subset \R^n$, augmented if $i = \text{A}$ ($m>n$), non-augmented if $i = \text{N}$ ($m\leq n$), unspecified if $i = \varnothing$}
	\nomenclature[F, 07]{$\Phi:\mathcal{X}\rightarrow\R^q$}{$\Phi \in \textup{NDDE}^k_{\tau}(\mathcal{X},\R^q)$ is the input-output map of the general neural DDE}
	\nomenclature[F, 09]{\(\theta:[0,T]\rightarrow\R^p\)}{parameter function of the parameterized vector field $f_\text{DDE}:\mathcal{C}\times \R^p \rightarrow \R^m$}
	\nomenclature[F, 10]{\(\Theta^k([0,T],\R^p)\)}{parameter space for the parameter functions $\theta$ of Definition \ref{def:NDDEparameterized}}
	\nomenclature[F, 11]{\( \textup{NDDE}^k_{\tau,\theta}(\mathcal{X},\R^q)\)}{parameterized neural DDE of Definition \ref{def:NDDEparameterized}, $\mathcal{X}\subset \R^n$}	
	\nomenclature[F, 12]{$\Phi_\theta:\mathcal{X}\rightarrow\R^q$}{$\Phi_\theta \in \textup{NDDE}^k_{\tau,\theta}(\mathcal{X},\R^q)$ is the input-output map of a parameterized neural DDE}
	\nomenclature[F, 13]{\( \overline{\textup{NDDE}}^k_{\tau}(\mathcal{X},\R^q)\)}{subset of $\textup{NDDE}^k_{\tau,i}(\mathcal{X},\R^q)$, where the underlying vector field is globally defined and weakly nonlinear, cf.\ Definition \ref{def:NDDEglobal}}
	\nomenclature[F, 14]{$\overline{\Phi}:\mathcal{X}\rightarrow\R^q$}{$\overline{\Phi} \in \overline{\textup{NDDE}}^k_{\tau}(\mathcal{X},\R^q)$ is the input-output map of globally defined weakly nonlinear neural DDE}

\setlength{\nomitemsep}{-0.1em}  
	\printnomenclature[2.7cm]

	\newpage
	
	\addcontentsline{toc}{section}{References}
	\bibliographystyle{abbrvurl}
	\bibliography{0literature}
	
\end{document}